\documentclass[12pt,twoside,reqno]{amsart}

\usepackage{amsmath, amssymb, amsthm, enumitem, esint}
\newcommand{\IR}{\mathbb{R}}

\newcommand{\IZ}{\mathbb{Z}}

\newcommand{\RR}{\mathcal{R}}
\newcommand{\XX}{\mathcal{X}}
\newcommand{\eps}{\varepsilon}
\newcommand{\ov}[1]{\overline{#1}}
\newcommand{\td}[1]{\widetilde{#1}}

\newcommand{\textQQqq}[1]{\qquad \text{#1} \qquad}

\newcommand{\textQq}[1]{\quad \text{#1} \quad}
\DeclareMathOperator*{\osc}{osc}
\DeclareMathOperator{\sn}{sn}
\DeclareMathOperator{\Int}{Int}
\DeclareMathOperator{\Ric}{Ric}
\DeclareMathOperator{\Rm}{Rm}

\DeclareMathOperator{\id}{id}

\DeclareMathOperator{\inj}{inj}

\DeclareMathOperator{\supp}{supp}

\DeclareMathOperator{\DIV}{div}
\DeclareMathOperator{\proj}{proj}

\newcommand{\tdrrm}{\td{r}_{\Rm}}
\newcommand{\rrm}{r_{\Rm}}
\newcommand{\rrrm}{r_{\Rm, r}}

\newcommand{\EMPTY}[1]{}

\newtheorem*{Claim}{Claim}
\newtheorem*{Claim1}{Claim 1}
\newtheorem*{Claim2}{Claim 2}
\newtheorem*{Claim3}{Claim 3}
\newtheorem*{Claim4}{Claim 4}
\newtheorem*{Claim5}{Claim 5}
\newtheorem*{Claim6}{Claim 6}
\newtheorem*{Claim7}{Claim 7}
\newtheorem*{Claim8}{Claim 8}
\newtheorem*{Claim9}{Claim 9}
\newtheorem*{Claim10}{Claim 10}
\newtheorem*{Claim11}{Claim 11}
\newtheorem{Theorem}{Theorem}[section]
\newtheorem{Lemma}[Theorem]{Lemma}
\newtheorem{Corollary}[Theorem]{Corollary}
\newtheorem{Proposition}[Theorem]{Proposition}
\newtheorem{Definition}[Theorem]{Definition}
\newtheorem*{PropertyA}{Property (A)}
\newtheorem*{PropertyB}{Property (B)}
\newtheorem*{PropertyC}{Property (C)}
\newtheorem*{PropertyD}{Property (D)}
\newtheorem*{PropertyE}{Property (E)}
\newtheorem*{PropertyF}{Property (F)}
\newtheorem*{PropertyG}{Property (G)}

\numberwithin{equation}{section}

\title{Structure theory of singular spaces}
\author{Richard H Bamler}
\thanks{This material is in part based upon work supported by the National Science Foundation under Grant No. DMS-1440140 while the author was in residence at the Mathematical Sciences Research Institute in Berkeley, California, during the Spring 2016 semester.}
\address{Department of Mathematics, UC Berkeley, CA 94720, USA}
\email{rbamler@math.berkeley.edu}
\date{\today}

\begin{document}

\begin{abstract}
In this paper we develop a structure theory of Einstein manifolds or manifolds with lower Ricci curvature bounds for certain singular spaces that arise as geometric limits of sequences of Riemannian manifolds.
This theory generalizes the results that were obtained by Cheeger, Colding and Naber in the smooth setting.
In the course of the paper, we will carefully characterize the assumptions that we have to impose on this sequence of Riemannian manifolds in order to guarantee that the individual results hold.

An important aspect of our approach is that we don't need impose any Ricci curvature bounds on the sequence of Riemannian manifolds leading to the singular limit.
The Ricci curvature bounds will only be required to hold on the regular part of the limit and we will not impose any (synthetic) curvature condition on its singular part.

The theory developed in this paper will have applications in the blowup analysis of certain geometric equations in which we study scales that are much larger than the local curvature scale.
In particular, this theory will have applications in the study of Ricci flows of bounded scalar curvature, which we will describe in a subsequent paper.
\end{abstract}

\maketitle
\tableofcontents

\section{Introduction and statement of the main results}
\subsection{Introduction} \label{subsec:Introduction}
Geometric limit and, in particular, blowup arguments have become popular tools in the study of geometric PDEs.
Useful applications of these tools can be found in the analysis of intrinsic objects, such as Einstein metrics and Ricci flows, or extrinsic objects, such as minimal surfaces and mean curvature flow.
The goal behind limit and blowup arguments is often to gain approximate characterizations of solutions of the geometric PDEs under investigation, at small scales.
In a broad sense, the strategy of proof is the following:
One first shows compactness of blowups of these solutions in an appropriate topology, possibly under additional assumptions.
Consequently, any sequence of solutions subconverges to a limit space, which often exhibits additional geometric properties.
If arguments are set up adequately, then these additional properties may be used to analyze these limit spaces more deeply, which may give rise to extra structural information or geometric bounds.
This extra information can then be used to derive geometric bounds or local characterizations of the actual solutions that led to the limit space.

Often, however, the conditions under which a geometric limit can be extracted are rather restrictive.
For example, in many situations it is necessary to assume that the solutions of the geometric PDE under investigation satisfy uniform curvature bounds, in order to ensure that the limit space is smooth.
These restrictions usually confine our freedom to choose blowup scales and basepoints.
In other words, limit arguments often only characterize the behavior of the solution at the ``smallest scale''.
If, on the other hand, we want to consider intermediate blowup scales, then we need to accept that limit spaces have lower regularity, which in turn may complicate their analysis.
In the study of extrinsic geometric PDEs, such as minimal surfaces and mean curvature flow, strategies of looking at an ``intermediate scale'' have been proven useful, for example via removable singularities or the theory of varifolds.

In this paper, we discuss a similar strategy for intrinsic objects, namely Riemannian manifolds.
More specifically, our paper consists of two parts:

In the first part, we establish a compactness property of Riemannian manifolds under relatively weak uniform geometric bounds.
This compactness property is a hybrid of the (metric) Gromov-Hausdorff compactness result and the (smooth) Cheeger-Gromov compactness result.
Recall that Gromov-Hausdorff compactness holds under the --- relatively weak --- condition of uniform total boundedness, but gives us very little control on the regularity of the limit space.
On the other hand, Cheeger-Gromov compactness ensures that the limit space is a Riemannian manifold, but it requires uniform bounds on the Riemannian curvature (and its derivatives) and the injectivity radius.
The compactness property that we will establish in this paper requires much weaker uniform geometric bounds than in Cheeger-Gromov compactness and it guarantees that the limit is differentiable away from a singular set of bounded Minkowski dimension.
We will also characterize the regularity of the limit in further detail, depending on various additional uniform properties that we impose on the sequence of the Riemannian manifolds.

In the second part, we analyze the geometry of the singular limit spaces, assuming that the Einstein equation or a lower bound on the Ricci curvature holds on its regular part.
It is worth mentioning, that on the regular part, the geometry of the limit space is given by a (sufficiently differentiable) Riemannian metric.
Therefore, the Einstein condition or the lower Ricci curvature bound is purely elementary.
Moreover, we will not impose a (synthetic) curvature condition on the singular part of the limit space under investigation.
Assuming this extra condition on the Ricci curvature, we will then carry out the theory of Cheeger, Colding and Naber  (see \cite{Colding-vol-conv, Cheeger-Colding-Cone, Cheeger-Naber-quantitative, Cheeger-Naber-Codim4}).
More specifically, we will show that the following results hold on the singular limit space: Colding's volume stability theorem, Cheeger and Colding's cone rigidity theorem and Cheeger and Naber's structure theory leading to $L^{p < 2}$-curvature bounds in the Einstein case.

The results of this paper can find applications in the study of geometric equations on Riemannian manifolds that guarantee Ricci curvature bounds at scales that are small and below the curvature scale.
An application, which is of particular interest here, is the structure theory of Ricci flows with bounded scalar curvature.
This theory will be developed in a subsequent paper by the author (see \cite{Bamler-Conv-RF}, see also \cite{Bamler:2015yg} for a preprint).

%Assuming relatively weak geometric bounds on these manifolds, we will establish a compactness in such a way that the limiting spaces have a moderate regularity.
%
% that satisfy weak geometric bounds exhibit a compactness 
%
%
%In this paper, we will focus on intrinsic objects.
%More specifically, we will introduce a certain class of Riemannian manifolds and prove that convergent sequences 
%
%The strategy behind such a limit or blowup argument can be summarized as follows:
%Consider a class $\mathcal{M}$ of $n$-dimensional Riemannian manifolds $(M, g)$, which satisfy a certain geometric condition.
%Let us assume for simplicity that $\mathcal{M}$ is chosen in such a way that for any $(M, g) \in \mathcal{M}$ also its blow-ups $(M, \lambda^2 g)$, $\lambda \geq 1$ are contained in $\mathcal{M}$.
%Often the class $\mathcal{M}$ is too large to allow a careful analysis.
%
%
%Let us summarize the strategy behind a blowup argument using the language of Riemannian manifolds:
%Consider a class $\mathcal{M}$ of pointed, $n$-dimensional Riemannian manifolds $(M, g, p)$, which satisfy a certain geometric condition.
%Based on this condition, we would like to derive a bound on some geometric quantity $\mathcal{Q}$, which we can understand as a function $\mathcal{Q} : \mathcal{M} \to \IR$.
%In order to obtain this bound, we now argue by contradiction: 
%Assume that there is a sequence $(M_i, g_i, p_i ) \in \mathcal{M}$ with $\mathcal{Q} (M_i, g_i, p_i) \to \infty$ as $i \to \infty$.
%
%to analyze geometric problems of elliptic and parabolic type.
%Their application ranges from  

\subsection{Geometric compactness and convergence of Riemannian manifolds} \label{subsec:characterlimit}
We will now present our first result of this paper.
For the remainder of this subsection, we consider a sequence $\{ (M_i, g_i, q_i) \}_{i=1}^\infty$ of pointed, complete, $n$-dimensional Riemannian manifolds, each of which has uniformly bounded curvature.
Let $d_{M_i}$ be the induced length metric of each $(M_i, g_i)$.
If we assume that the sequence $\{ (M_i, d_{M_i}, q_i) \}_{i=1}^\infty$ is uniformly totally bounded, then we can pass to a subsequence such that we have Gromov-Hausdorff convergence
\[ (M_i, d_{M_i}, q_i) \xrightarrow{i \to \infty} (X, d_X, q_\infty) \]
to a pointed, complete metric space.
Let us assume from now on that the limit $(X, d_X, q_\infty)$ exists.
In the following, we will analyze in what way uniform bounds on the sequence $\{ (M_i, g_i, q_i) \}_{i=1}^\infty$ affect the regularity of this limit.
These bounds occur in the study of certain geometric PDEs, for example if the $(M_i, g_i)$ are time-slices of Ricci flows with bounded scalar curvature (see \cite{Bamler-Conv-RF, Bamler:2015yg}).

Let us now fix some constants for the remainder of this subsection
\[ A , T > 0 \textQQqq{and} \mathbf{p}_0 > 0. \]
A property that we will naturally need to impose on $\{ (M_i, g_i, q_i) \}_{i=1}^\infty$ is that volumes of distance balls are uniformly bounded from above and below.
In particular, this property implies the uniform total boundedness mentioned above.

\begin{PropertyA}
For any $D < \infty$ and for large $i$ we have the following lower and upper volume bounds on geodesic balls for all $x \in B^{M_i}(q_i, D)$ and all $0 < r < \sqrt{T}$:
\[ A^{-1} r^n < \big| B^{M_i} (x,r) \big| < A r^n. \]
\end{PropertyA}

Next, we wish to ensure that the limit space $(X, d_X)$ is sufficiently differentiable away from a singular set of Minkowski dimension not greater than $\mathbf{p}_0$.
For technical reasons, we would like to require that the metric on the regular set of the limit has regularity $C^4$.
We will therefore coin the following terminology:

\begin{Definition}[curvature radius] \label{Def:curvradius}
Let $(M, g)$ be a (not necessarily complete) Riemannian manifold and let $p \in M$ be a point.
Then we define the \emph{curvature radius $\rrm (p)$ at $p$} to be the supremum over all $r > 0$ such that the ball $B(p,r)$ is relatively compact in $M$ and such that
\[ |{\Rm}| < r^{-2} \textQq{and} |{\nabla \Rm}| < r^{-3} \textQQqq{on} B(p,r). \]
We will moreover define the \emph{modified curvature radius} $\tdrrm (p)$ to be the supremum over all $r > 0$ such that the ball $B(p,r)$ is relatively compact in $M$ and such that
\[ |{\Rm}| < r^{-2}, \quad |{\nabla \Rm}| < r^{-3} \textQq{and} |{\nabla^2 \Rm}| < r^{-4} \textQQqq{on} B(p,r). \]

We will often denote by $\{ a < \rrm < b \}$ or $\{ a < \tdrrm < b \}$ the set of all points $x \in M$ such that $a < \rrm (x) < b$ or $a < \tdrrm (x) < b$, respectively.
\end{Definition}

Note that in related literature, the bounds on $|{\nabla \Rm}|$ and $|{\nabla^2 \Rm}|$ are often not required.
In the case in which $g$ is an Einstein metric or a metric that arises from a Ricci flow, which will be of most interest for us, we can in fact drop this bound and use Shi's estimates instead whenever we need a bound on $|{\nabla \Rm}|$ or $|{\nabla^2 \Rm}|$ (see \cite{MR1010165}).
Furthermore, the reason why we have introduced two different notions of the curvature radius, $\rrm$ and $\tdrrm$, is purely technical and has to do with the fact that under geometric convergence a lower bound on $\tdrrm$ implies a lower bound on $\rrm$ in the limit.

With this terminology at hand, we can finally state the second property that we will impose on the sequence $\{ (M_i, g_i, q_i) \}_{i=1}^\infty$.
This property will guarantee that the singular part in the limit has Minkowski dimension not larger than $\mathbf{p}_0$.
We will later often assume that $\mathbf{p}_0 > 3$, which will allow us to carry out standard analytic arguments on the limit of the $(M_i, g_i, q_i)$, such as gradient estimates for functions that are harmonic on the regular part of $(X, d_X)$.

\begin{PropertyB}
For any $D < \infty$ there is a $C = C(D) < \infty$ such that for large $i$ and all $x \in B^{M_i} (q_i, D)$ and all $0 < r < D$ we have the following bound on the sublevel sets of the modified curvature radius:
\[ \big| \big\{ \tdrrm^{M_i} (\cdot) < sr \big\} \cap B^{M_i} (x,r) \big| < C s^{\mathbf{p}_0} r^n. \]
\end{PropertyB}

The next property will guarantee that any pair of regular points in the limit $(X,d_X)$ can be connected by a curve within the regular part of $(X, d_X)$ whose length is arbitrarily close to their metric distance.

\begin{PropertyC}
For any $\sigma_0, \eta > 0$ and $D < \infty$ there is a $\sigma = \sigma (\sigma_0, \eta, D) > 0$ such that the following is true for sufficiently large $i$ (depending on $\sigma_0, \eta$ and $D$):
Assume that $x,y \in B^{M_i} (q_i, D) \cap \{ \tdrrm^{M_i} > \sigma_0 \}$.
Then there is a smooth curve $\gamma : [0,1] \to M_i$ between $x, y$ such that
\[ \ell^{M_i} (\gamma) < d_{M_i} (x, y) + \eta \]
and such that
\[ \tdrrm^{M_i} (\gamma(s)) > \sigma \textQQqq{for all} s \in [0,1]. \]
\end{PropertyC}

In the following we will always have to rely on the, relatively basic, properties (A)--(C).
In order to carry out more analytic arguments, we will furthermore often assume that we have a uniform log-Sobolev inequality on $\{ (M_i, g_i, q_i) \}_{i=1}^\infty$.
Inequalities of this type occur naturally in the study of Ricci flows (with bounded scalar curvature), using Perelman's monotonicity of the $\mathcal{W}$-functional (see \cite{PerelmanI}).
Log-Sobolev inequalities are useful for us as they can be used to derive a Poincar\'e inequality and Gaussian heat kernel estimates (see for example \cite{Davies_1987} or proof of Theorem \ref{Thm:tame}, item (4), in subsection \ref{subsec:tame}).

\begin{PropertyD}
We have the following uniform log-Sobolev inequality at scales less than $\sqrt{T}$:
For large $i$ and any $0 < \tau \leq 2T$ and any $f \in C^\infty (M_i)$ with
\[ \int_{M_i} (4\pi \tau)^{-n/2} e^{-f} dg_i = 1 \]
we have
\[ \int_{M_i} \big( \tau |\nabla f|^2 + f \big) (4 \pi \tau)^{-n/2} e^{-f} dg_i > - A. \]
\end{PropertyD}

Our analysis of the limit space $(X, d_X)$ will also often rely on Bishop-Gromov volume comparison (see Proposition \ref{Prop:volumecomparison}) and the segment inequality (see Proposition \ref{Prop:segmentinequ}).
Our derivation of these results in the singular setting, will require that every regular point can be connected with almost every other regular point by a minimizing geodesic that entirely lies in the regular part of $(X, d_X)$.
This property of the limit, which we will refer to in this paper as ``mildness'' (see Definition \ref{Def:mild} in subsection \ref{subsec:terminology}), is implied by the following property, which we will sometimes impose on the sequence $\{ (M_i, g_i, q_i) \}_{i=1}^\infty$.

\begin{PropertyE}
For any $\delta, \sigma_0 > 0$ and $D < \infty$ there is a $\sigma = \sigma (\delta, \sigma_0, D) > 0$ such that for any $\eta > 0$ and any sufficiently large $i$ (depending on $\delta, D$ and $\eta$) the following holds: For any $x \in B^{M_i} (q_i, D) \cap \{ \tdrrm^{M_i} ( \cdot ) > \sigma_0 \}$ there is a subset $S \subset B^{M_i} (x, D)$ such that
\[ \big| B^{M_i} (x, D) \setminus S \big| < \delta \]
and such that for any $y \in S$ there is a curve $\gamma : [0,1] \to M_i$ between $x$ and $y$ such that \[ \ell^{M_i} (\gamma) < d_{M_i} (x, y) + \eta \]
and
\[ \tdrrm^{M_i} (\gamma(s)) > \sigma \textQQqq{for all} s \in [0,1]. \]
(Note that the constant $\sigma$ is required to be independent of $\eta$.)
\end{PropertyE}

We remark that property (E) implies property (C).

We will moreover occasionally use cutoff functions with controlled gradient and Laplacian (wherever defined) on $(X, d_X)$ in order to localize certain analytic arguments.
In our paper, such cutoff functions will arise from cutoff function on the manifolds $(M_i, g_i)$ via the following property:

\begin{PropertyF}
For any $i$, any $0 < r < \sqrt{T}$ and any $x \in M_i$ there is a compactly supported function $\phi \in C^\infty_c (B^{M_i} (x,r))$ that only takes values in $[0,1]$, that satisfies $\phi \equiv 1$ on $B^{M_i} (x, r/ A)$ and that satisfies the bounds
\[ |\nabla \phi | < A r^{-1} \textQQqq{and} |\triangle \phi | < A r^{-2}. \]
\end{PropertyF}

Finally, we introduce a property, which we will assume surprisingly rarely.
This property will allow us to deduce curvature bounds based on volume bounds and is well known to hold for smooth Einstein metrics.
Moreover, this property will be used to rule out the existence of singularities at points in the limit $(X, d_X)$ whose tangent cones are almost Euclidean.

\begin{PropertyG}
For any $i$ and any $x \in M_i$ and $0 < r < \sqrt{T}$ for which
\[ \big| B^{M_i} (x, r) \big| > (\omega_n - A^{-1} ) r^n \]
we have $\tdrrm (x) > A^{-1} r$.
Here $\omega_n$ denotes the volume of the standard ball in $n$-dimensional Euclidean space.
\end{PropertyG}

Note that properties (A), (D), (F) and (G) depend on the choice of the constants $A$ and $T$.

Let us finally phrase the exact characterization of the limiting space $(X, d_X, q_\infty)$ of the sequence $\{ (M_i, d_{M_i}, q_i ) \}_{i =1}^\infty$, depending on the properties (A)--(G).
The precise terminology used in the following theorem will be introduced in subsection~\ref{subsec:terminology}.

%: Theorem: Convergence to singular space
\begin{Theorem}[convergence to singular space] \label{Thm:basicconvergence}
Assume that the sequence $\{ (M_i, \linebreak[1] g_i, \linebreak[1] q_i ) \}_{i=1}^\infty$ satisfies properties (A)--(C) above for some constants $A, T, \mathbf{p}_0 > 0$.
Then the limit space $(X, d_X)$ is part of a singular space $\XX = (X, d_X, \RR, g)$ in the sense of Definition \ref{Def:singlimitspace}, with singularities of codimension $\mathbf{p}_0$ in the sense of Definition \ref{Def:codimensionsingularities}.
After passing to a subsequence, the pointed manifolds $(M_i, g_i, q_i)$ converge to the pointed singular space $(\XX, q_\infty)$ in the sense of Definition \ref{Def:convergencescheme}.
Moreover:
\begin{enumerate}[label=(\alph*)]
\item If the sequence $\{ (M_i, g_i, q_i ) \}_{i=1}^\infty$ additionally satisfies property (E), then the singularities of $\XX$ are mild in the sense of Definition \ref{Def:mild}.
\item If the sequence $\{ (M_i, g_i, q_i ) \}_{i=1}^\infty$ additionally satisfies properties (D)--(F) and if $\mathbf{p}_0 > 1$, then $\XX$ is $Y_1$-tame at scale $c \sqrt{T}$ in the sense of Definition~\ref{Def:tameness}.
Here $Y_1 = Y_1 (n, A, \mathbf{p}_0) < \infty$ can be chosen only depending on $n$, $A$ and $\mathbf{p}_0$ and $c = c(n, A) > 0$ can be chosen only depending on $n$ and $A$.
\item If the sequence $\{ (M_i, g_i, q_i ) \}_{i=1}^\infty$ additionally satisfies property (G), then $\XX$ is $Y_2$-regular at scale $\sqrt{T}$ in the sense of Definition \ref{Def:Yregularity}.
Here $Y_2 = Y_2 (n, A) < \infty$ can be chosen only depending on $n$ and $A$.
\end{enumerate}
\end{Theorem}

The importance of the properties of the singular space $\XX = (X, d_X , \RR, g)$ that were derived in Theorem \ref{Thm:basicconvergence} (i.e. ``singularities of codimension $\mathbf{p}_0$'', ``mild singularities'', ``$Y$-tameness'' and ``$Y$-regularity'') comes from the fact that they capture all the regularity properties of $\XX$ that are necessary for our subsequent analysis.
So for the remaining results we don't need to remember that $\XX$ arises as a limit of smooth Riemannian manifolds that satisfy all or some of the properties (A)--(G).
Instead, we will simply characterize $\XX$ using those derived properties.

\subsection{Structure theory of singular spaces} \label{subsec:strcttheory}
Next, we will analyze the singular spaces $\XX = (X, d, \RR, g)$, assuming certain conditions on the Ricci curvature, which we only impose on the regular set $\RR$ of $\XX$.
In particular, we will not impose any (synthetic) curvature conditions on the singular points $X \setminus \RR$.
In terms of the sequence $\{ (M_i, g_i, q_i) \}_{i = 1}^\infty$, converging to $\XX$, this means that we will allow that the conditions on the Ricci curvature are violated in small regions of the Riemannian manifolds $(M_i, g_i)$, for example in regions around points that converge to singular points in $\XX$.

Our first result is a generalization of Colding's volume stability result (cf \cite{Colding-vol-conv}) to the singular setting.
It states that Gromov-Hausdorff closeness to Euclidean space $\IR^n$ implies volume closeness.
For technical reasons, which will become apparent later, we will phrase our result such that we allow Gromov-Hausdorff closeness to a Cartesian product of an arbitrary metric space $(Z,d_Z)$ with $\IR^n$.

\begin{Theorem}[Volume Stability] \label{Thm:volconv}
For any $n \geq 2$, $\eps > 0$, $\mathbf{p}_0 > 3$ and $Y < \infty$ there is a constant $\delta = \delta(n, \eps, \mathbf{p}_0, Y) > 0$ such that the following holds:

Assume that $\XX = (X, d, \RR, g)$ is an $n$-dimensional singular space with mild singularities of codimension $\mathbf{p}_0$ that is $Y$-tame at some scale $r > 0$.
Assume moreover that the Ricci curvature condition $\Ric \geq - \delta r^{-2}$ holds on $\RR$.
Let $x \in X$ and assume that there is a metric space $(Z, d_Z)$ and a point $z \in Z$ such that
\[ d_{GH} \big( \big( B^{X} (x, r),x \big), \big( B^{Z \times \IR^n} ( (z,0^n), r), (z, 0^n) \big) \big) < \delta r. \]
Then
\[ \big| B^X (x, \delta r) \cap \RR \big| > (\omega_n - \eps) (\delta r)^n. \]
Here $\omega_n$ denotes the volume of the $n$-dimensional Euclidean ball of radius $1$.
\end{Theorem}

As a corollary of Theorem \ref{Thm:volconv}, we will obtain a curvature bound at points that are almost Euclidean, under $Y$-regularity assumptions.

\begin{Corollary}[Gromov-Hausdorff $\eps$-regularity] \label{Cor:GHepsregularity}
For any $n \geq 2$, $\mathbf{p}_0 > 3$ and $Y < \infty$ there is a constant $\eps = \eps(n, \mathbf{p}_0, Y) > 0$ such that the following holds:

Assume that $\XX = (X, d, \RR, g)$ is an $n$-dimensional singular space with mild singularities of codimension $\mathbf{p}_0$ that is $Y$-tame and $Y$-regular at some scale $r > 0$.
Assume moreover that the Ricci curvature condition $\Ric \geq - (n-1) r^{-2}$ holds on $\RR$.
Let $x \in X$ and assume that there is a metric space $(Z, d_Z)$ and a point $z \in Z$ such that
\[ d_{GH} \big( \big( B^{X} (x, r),x \big), \big( B^{Z \times \IR^n} ( (z,0^n), r), (z, 0^n) \big) \big) < \eps r. \]
Then $\rrm ( x ) > \eps r$.
\end{Corollary}

Next, we show that Cheeger and Colding's Cone Rigidity Theorem (cf \cite{Cheeger-Colding-Cone}) holds in the singular setting:

\begin{Theorem}[Cone Rigidity] \label{Thm:conerigidityIntroduction}
For any $n \geq 2$, $\eps > 0$, $\mathbf{p}_0 > 3$ and $Y < \infty$ there is a $\delta = \delta (n, \eps, \mathbf{p}_0, Y) > 0$ such that the following holds:

Let $\XX = (X, d, \RR, g)$ be a singular space with mild singularities of codimension $\mathbf{p}_0$ that is $Y$-tame at scale $\delta^{-1} r$ for some $r > 0$.
Assume moreover that $\Ric \geq - (n-1) \kappa$ on $\RR$ for some $0 \leq \kappa \leq \delta^2 r^{-2}$.
Let $p \in X$ be a point and assume that
\[ \frac{| B(p, \delta r) \cap \RR |}{v_{- \kappa} (\delta r)} - \frac{|B(p, 32r) \cap \RR|}{v_{- \kappa} (32r)} < \delta. \]
Then there is a metric cone $(\mathcal{C}, d_{\mathcal{C}}, \ov{p})$ with vertex $\ov{p} \in \mathcal{C}$ such that
\[ d_{GH} \big( \big( B^X (p, r), p \big), \big( B^{\mathcal{C}} (\ov{p}, r), \ov{p} \big) \big) < \eps r. \]
\end{Theorem}

Finally, we generalize the structure theory of Cheeger and Naber (cf \cite{Cheeger-Naber-quantitative,Cheeger-Naber-Codim4}) to singular spaces.
The following bound is slightly more general than an $L^{p<2}$-curvature bound.

\begin{Theorem} \label{Thm:Lpbound}
For any $n \geq 2$, $\eps > 0$, $\mathbf{p}_0 > 3$ and $Y < \infty$ there is a constant $E = E(n, \eps, \mathbf{p}_0, Y) < \infty$ such that the following holds:

Assume that $\XX = (X, d, \RR, g)$ is an orientable $n$-dimensional singular space with mild singularities of codimension $\mathbf{p}_0$ that is $Y$-tame and $Y$-regular at scale $E r$ for some $r > 0$.
Assume that the Einstein equation $\Ric = \lambda g$, $|\lambda | \leq n-1$ holds on $\RR$.
Let $x \in X$ and $0 <  s < 1$.
Then
\begin{equation} \label{eq:LpboundinThm}
 \big| \big\{ \rrm < sr \big\} \cap B^X (x,r) \cap \RR \big| \leq E s^{4-\eps} r^n. 
\end{equation}
\end{Theorem}

As a consequence of (\ref{eq:LpboundinThm}) we obtain a bound of the form
\[ \int_{B^X (x,r) \cap \RR} |{\Rm}|^{2-\eps} < C(\eps, \mathbf{p}_0, Y) r^{n- 4 + 2\eps}. \]

Note that the bound (\ref{eq:LpboundinThm}) in Theorem \ref{Thm:Lpbound} is similar to the bound (\ref{eq:LpboundDefcodim}) in Definition \ref{Def:codimensionsingularities} of the codimension of the singular set (see the following subsection).
However, (\ref{eq:LpboundDefcodim}) is only used to characterize the severity of the singularities of the singular spaces $\XX$ and was allowed to depend on $\XX$.
So the content of Theorem~\ref{Thm:Lpbound} is that $E$ can, in fact, be chosen depending only on $\mathbf{p}$ and $Y$ and independently of the constant in (\ref{eq:LpboundDefcodim}).

%\section{Important terminology and conventions} \label{sec:terminologyandconventions}
\subsection{Important terminology} \label{subsec:terminology}
We now give a precise definition of the terminology that was used in the theorems of corollaries of the previous subsection and which we will use throughout this paper.
Let us first introduce the following notion:
Given a measurable subset $S \subset M$ of a Riemannian manifold $(M,g)$ we will denote by $|S| = |S|_g$ the Riemannian measure of $S$ with respect to the metric $g$.

We now define what we mean by the singular spaces that appeared as limit spaces in Theorem~\ref{Thm:basicconvergence}.
The following definition comprises the most basic notions of a metric space that is sufficiently differentiable on a generic subset.

\begin{Definition}[singular space] \label{Def:singlimitspace}
A tuple $\mathcal{X} = (X, d, \RR, g)$ is called an \emph{($n$-dimensional) singular space} if the following holds:
\begin{enumerate}[label=(\arabic*)]
\item $(X,d)$ is a locally compact, complete metric length space.
\item $\RR \subset X$ is an open and dense subset that is equipped with the structure of a differentiable $n$-manifold of regularity $C^4$ whose topology is equal to the topology induced by $X$.
\item $g$ is a Riemannian metric on $\RR$ of regularity $C^3$.
\item The length metric of $(\RR, g)$ is equal to the restriction of $d$ to $\RR$.
In other words, $(X,d)$ is the completion of the length metric on $(\RR, g)$.
\item For any compact subset $K \subset X$ and any $D < \infty$ there are constants $0 < \kappa_1(K,D) < \kappa_2(K,D) < \infty$ such that for all $x \in K$ and $0 < r < D$
\[ \kappa_1 r^n < | B(x,r) \cap \RR | < \kappa_2 r^n. \]
Here $| \cdot |$ denotes the Riemannian volume with respect to the metric $g$ and distance balls $B(x,r)$ are measured with respect to the metric $d$.
\end{enumerate}
If $q \in X$ is a point, then the tuple $(\mathcal{X}, q)$ or $(X, d, \RR, g, q)$ is called \emph{pointed singular space}.
The subset $\RR$ is called the \emph{regular part of $\XX$} and its complement $X \setminus \RR$ the \emph{singular part of $\XX$}.
We say that $\XX$ is orientable, if $\RR$ is orientable.
\end{Definition}

Decorations of $\XX$ are inherited by its members.
That is, for example, if we refer to the space $\XX'_j$, then $\RR'_j$ automatically denotes its regular part.
We also denote open annuli in $\XX$ by $A(p,r_1, r_2) = \{ x \in X \;\; : \;\; r_1 < d(x,p) < r_2 \}$.

We emphasize that the metric $d$ on $X$ is induced by the length metric of the Riemannian metric $g$ on $\RR$ (see item (4)).
So the distance between any two points in $\RR$ can be approximated arbitrarily well by the length of a differentiable connecting curve in $\RR$.

We can generalize the concept of curvature radius from Definition \ref{Def:curvradius} to singular spaces $\mathcal{X} = (X, d, \RR, g)$ by defining the function $\rrm : X \to [0, \infty]$ as follows: $\rrm |_{X \setminus \RR} \equiv 0$ and for any $x \in \RR$ let $\rrm (x)$ be the curvature radius of the (incomplete) Riemannian manifold $(\RR, g)$ as defined in Definition \ref{Def:curvradius}.

We will always imagine the singular part $X \setminus \RR$ as a closed subset of measure zero.
All analytic arguments of this paper that require local computations will be carried out on the regular part $\RR \subset X$ and, similarly, curvature conditions will only be imposed on $\RR$.
Whenever necessary, we will argue why it is enough to consider only this generic subset.
The following two characterizations of the singular set will be helpful hereby:

\begin{Definition}[mild singularities] \label{Def:mild}
A singular space $\mathcal{X} = (X, d, \RR, g)$ is said to have \emph{mild singularities} if for any $p \in \RR$ there is a closed subset $Q_p \subset \RR$ of measure zero such that for any $x \in \RR \setminus Q_p$ there is a minimizing geodesic between $p$ and $x$ whose image lies in $\RR$.
\end{Definition}

The idea behind the notion of mild singularities also occurs in the work of Cheeger and Colding (see \cite[Theorem~3.9]{Cheeger-Colding-structure-II} and Chen and Wang (see \cite[Definition 2.1]{Chen-Wang-II}).

\begin{Definition}[singularities of codimension $\mathbf{p}_0$]  \label{Def:codimensionsingularities}
A singular space $\mathcal{X} = (X, \linebreak[1] d, \linebreak[1] \RR, \linebreak[1] g)$ is said to have \emph{singularities of codimension $\mathbf{p}_0$}, for some $\mathbf{p}_0 > 0$, if for any $0 < \mathbf{p} < \mathbf{p}_0$, $x \in X$ and $r_0 > 0$ there is an $\mathbf{E}_{\mathbf{p}, x,r_0} < \infty$ (which may depend on $\XX$) such that the following holds:
For any $0 < r < r_0$ and $0 < s < 1$ we have
\begin{equation} \label{eq:LpboundDefcodim}
 | \{ \rrm < s r \} \cap B(x, r) \cap \RR | \leq \mathbf{E}_{\mathbf{p}, x,r_0} s^{\mathbf{p}} r^n. 
\end{equation}
\end{Definition}

We will often combine the last two properties and say that \emph{$\XX$ has mild singularities of codimension $\mathbf{p}_0$}.

An elementary covering argument shows that if an $n$-dimensional singular space $\mathcal{X} = (X, d, \RR, g)$ has singularities of codimension $\mathbf{p}_0$, then the Minkowski (and hence also Hausdorff) dimension of $X \setminus \RR$ is $\leq n - \mathbf{p}_0$.

Definition \ref{Def:codimensionsingularities} can be seen as an analogue of the density estimate in the work of Chen and Wang (see \cite[Definition 3.3]{Chen-Wang-II}).

We will now define the following properties of singular spaces:

\begin{Definition}[$Y$-regularity] \label{Def:Yregularity}
A singular space $\XX$ is called \emph{$Y$-regular at scales less than $a$}, for some $a, Y > 0$, if for any $p \in X$ and $0 < r < a$ the following holds:
If 
\[ | B(p, r) \cap \RR| > (\omega_n - Y^{-1}) r^n, \]
then $p \in \RR$ and $\rrm (y) > Y^{-1} r$.
\end{Definition}

The notion of $Y$-regularity is standard in the study of Einstein metrics.
A similar notion has been used in \cite{Cheeger-Colding-structure-II} and \cite{MR999661} and, in the setting of Ricci flows with bounded scalar curvature, in \cite[Definition 3.3]{Chen-Wang-II} and \cite[Theorem~2.35]{Tian-Zhang:2013}.

It can be shown that in a $Y$-regular space with singularities of codimension $\mathbf{p}_0$ (for some $\mathbf{p}_0 > 0$), any point $p \in X$ whose tangent cone is isometric to $\IR^n$, is actually contained in $\RR$.
Therefore, the regular set $\RR$ and the metric $g$ in such a space is uniquely characterized by the metric $d$.

Next, we define what we understand by convergence towards a singular space.

\begin{Definition}[convergence and convergence scheme] \label{Def:convergencescheme}
Consider a sequence $(M_i, g_i, q_i)$ of pointed $n$-dimensional Riemannian manifolds and a pointed, $n$-dimensional singular space $(\mathcal{X}, q_\infty) = (X, d, \RR, g,q_\infty)$.
Let $U_i \subset \RR$ and $V_i \subset M_i$ be open subsets and $\Phi_i : U_i \to V_i$ be (bijective) diffeomorphisms such that the following holds:
\begin{enumerate}[label=(\arabic*)]
\item $U_1 \subset U_2 \subset \ldots$
\item $\bigcup_{i=1}^\infty U_i = \RR$.
\item For any open and relatively compact $W \subset \RR$ we have $\Phi_i^* g_i \to g$ on $W$ in the $C^3$-sense.
\item There exists a sequence $q^*_i \in U_i$ such that
\[ d_{M_i} ( \Phi_i(q^*_i), q_i) \to 0. \]
\item For any $R < \infty$ and $\eps > 0$ there is an $i_{R, \eps} < \infty$ such that for all $i > i_{R, \eps}$ and $x, y \in B^X (q_\infty, R) \cap U_i$ we have
\[ \big| d_{M_i} (\Phi_i(x), \Phi_i(y)) - d^X (x,y) \big| < \eps \]
and such that for any $i > i_{R, \eps}$ and $x \in B^{M_i} (q_i, R)$ there is a $y \in V_i$ such that $d^{M_i} (x,y) < \eps$.
\end{enumerate}
Then the sequence $\{ (U_i, V_i, \Phi_i) \}_{i =1}^\infty$ is called a \emph{convergence scheme for the sequence of pointed Riemannian manifolds $(M_i, g_i, q_i)$ and the pointed singular space $(\mathcal{X}, q_\infty)$}.
We say that \emph{$(M_i, g_i, q_i)$ converges to $(\mathcal{X}, q_\infty)$} if such a convergence scheme exists.
\end{Definition}

Lastly, we introduce the $Y$-tameness properties.
These properties are rather technical and mostly ensure that the theory of Cheeger, Colding and Naber can be applied to $Y$-tame singular spaces.

\begin{Definition}[$Y$-tameness] \label{Def:tameness}
A singular space $\mathcal{X}$ is said to be \emph{$Y$-tame at scale $a$} for some $Y, a > 0$ if the following \emph{tameness properties} hold:
\begin{enumerate}[label=(\arabic*)]
\item We have the volume bounds
\[\qquad\quad Y^{-1} r^n <  |B(p,r) \cap \RR| < Y r^n \textQQqq{for all} p \in X \textQq{and} 0 < r < a. \]
\item For any $p \in X$, $q \in \RR$ and $0 < r < \min \{ a, d(p,q) \}$ the following holds:

Define the function $b (x) := d(x, q)$ on $X$.
Then there is a bounded $C^3$-function $h : B(p, r) \cap \RR \to \IR$ that satisfies $\triangle h = 0$ and
\[ \int_{B(p,r) \cap \RR} |\nabla (h - b)|^2 dg \leq 2r \int_{\ov{B(p, r)} \cap \RR} d|\mu_{\triangle b}|, \]
where the integrals on the left and right-hand sides are defined according to Proposition \ref{Prop:weakLaplaciandistance} in section \ref{sec:Riemgeometry}.
Moreover,
\[ \int_{B(p,r) \cap \RR} |h - b|^2 dg \leq Yr^3 \int_{\ov{B(p, r)} \cap \RR} d|\mu_{\triangle b}|. \]
\item For any $p \in \RR$ and $0 < r_1 < r_2 < a$ the following is true:
Assume that $n \geq 3$ and that we have the lower Ricci curvature bound $\Ric \geq - (n-1) \kappa$ on $B(p, 4r_2) \cap \RR$ for some $\kappa \geq 0$ with $\kappa  r_2^2  \leq 1$.

Then there is a $C^3$-function $h : A(p, r_1, r_2) \cap \RR \to \IR$ such that $\triangle h = 0$ and that satisfies the following bounds
\[ r_2^{2-n} \leq h \leq r_1^{2-n} \textQQqq{on} A(p, r_1, r_2) \cap \RR \]
and
\begin{multline} \label{eq:part2Ytame1}
\qquad\quad \int_{A(p, r_1, r_2) \cap \RR} |\nabla (h - d^{2-n} (\cdot, p)) |^2 dg \\
 \leq Y r_1^{2-n} \bigg( \kappa r_2^2 +  \frac{|\ov{B(p, r_1)} \cap \RR|}{v_{-\kappa} (r_1)} - \frac{|A(p, r_2, 2r_2) \cap \RR|}{v_{-\kappa} (2r_2) - v_{-\kappa} (r_2)} \bigg)
\end{multline}
and
\begin{multline} \label{eq:part2Ytame2}
\qquad\quad \int_{A(p, r_1, r_2) \cap \RR} | h - d^{2-n} (\cdot, p) |^2 dg \\
 \leq Y r_2^2 r_1^{2-n} \bigg( \kappa r_2^2 + \frac{|\ov{B(p, r_1)} \cap \RR|}{v_{-\kappa} (r_1)} - \frac{|A(p,r_2, 2r_2) \cap \RR|}{v_{-\kappa} (2r_2) - v_{-\kappa} (r_2)} \bigg) .
\end{multline}
Here $v_{-\kappa} (r)$ denotes the volume of a geodesic $r$-ball in the $n$-dimensional model space of constant sectional curvature $- \kappa$ (compare with (\ref{eq:vkappaformula}) in section \ref{sec:Riemgeometry}).
\item For every $p \in X$ and $0 < r \leq a$ there is an increasing sequence
\[ U_1 \subset U_2 \subset \ldots \subset B(p, 2r) \cap \RR \]
of open subsets and a sequence of $C^2$ function $\phi_i : U_i \to [0,1]$ such that the following holds:
\begin{enumerate}[label=(4\alph*)]
\item $\bigcup_{i=1}^\infty U_i = B(p,2r) \cap \RR$,
\item $\phi_i$ vanishes on a neighborhood of $U_i \setminus B(p,r)$,
\item $\phi_i \equiv 1$ on $U_i \cap B(p, r/2)$,
\item $|\nabla \phi_i | < Y r^{-1}$ and $|\triangle \phi_i | < Y r^{-2}$.
\end{enumerate}
\item There is a function $K : X \times X \times (0, \infty) \to [0,\infty)$ that is locally bounded on $X \times X \times (0,\infty)$ and $C^3$ on $\RR \times \RR \times (0,a^2)$ such that:
\begin{enumerate}[label=(5\alph*)]
\item $K(x,y,t) = K(y,x,t)$ for all $x, y \in X$, $t \in (0,a^2)$,
\item For all $y \in X$
\[ (\partial_t - \triangle ) K ( \cdot, y , t) = 0 \textQQqq{on} \RR, \]
\item For all $y \in \RR$ and all $r > 0$ we have
\[ \lim_{t \searrow 0} \int_{\RR \setminus B(y,r)} K(\cdot, y,t) = 0. \]
\item For all $y \in X$ and $t \in (0,a^2)$ we have
\[ \int_{\RR} K(\cdot, y, t) dg = 1. \]
\item For all $x, y \in X$ and $0 < t < a^2$, we have
\[ K (x, y, t) < \frac{Y}{t^{n/2}} \exp \bigg({ - \frac{d^2(x,y)}{Yt } }\bigg). \]
\end{enumerate}
\end{enumerate}
\end{Definition}

Note that in property (2), we do not claim that $h = b$ on $\partial B(p,r)$, since it is somewhat technically difficult to verify this condition if $\XX$ arises as a limit of smooth Riemannian manifolds.
We may imagine in the following that $h$ is the solution of the Dirichlet problem with the boundary values $b$ on $\partial B(p,r)$, but this fact will never be used.

\subsection{Acknowledgements}
I would like to thank Tristan Ozuch-Meersseman and Yalong Shi for helpful comments.

\subsection{Conventions}
In the following we will fix a dimension $n \geq 2$ and we will omit the dependence of our constants on $n$.
Moreover, when working with a manifold $M$ of regularity $C^4$, then we will call objects involving $M$ that have maximal regularity \emph{smooth}.
For example, we will call a curve $\gamma : [a,b] \to M$ smooth if it has regularity $C^4$ and a function $f : M \to \IR$ smooth if it has regularity $C^4$.

\section{Geometry of incomplete Riemannian manifolds} \label{sec:Riemgeometry}
In this section we review basic facts from Riemannian geometry.
We will phrase our results in such a way that they also hold for incomplete Riemannian mani\-folds.
So let in the following $(M, g)$ be a (not necessarily complete) Riemannian manifold of dimension $n \geq 2$.
More specifically, we assume that $M$ is of regularity $C^4$ (meaning that the chart transition maps have regularity $C^4$) and the Riemannian metric $g$ has regularity $C^3$.
Note that $g$ induces a length metric $d$ on $M$, but due to the possible incompleteness, the distance $d(x,y)$ between two points $x,y \in M$ may not be represented by a minimizing geodesic between $x, y$.
However, we have the following result:

\begin{Lemma} \label{Lem:geodincomplete}
Let $(M, g)$ be a (not necessarily complete) Riemannian manifold and $x, y \in M$.
Then there is an arclength, minimizing geodesic $\gamma : [0, l) \to M$ with $\gamma(0) = x$ such that
\[ d(x,\gamma(t)) + d(\gamma(t), y) = d(x,y) \textQQqq{for all} t \in [0, l) \]
and such that either $l = d(x,y)$ and $\gamma(t) \to y$ as $t \to l$ or $l < d(x,y)$ and $\gamma(t)$ does not converge in $M$ as $t \to l$.
\end{Lemma}

\begin{proof}
Choose $r > 0$ such that the ball $B(x,2r)$ is relatively compact in $M$ and such that the exponential map $\exp_x$ is defined on $B(0,2r) \subset T_x M$.
Choose $z  \in B(x,2r)$ such that $d(x,z) = r$ and $d(x, z) + d(z, y)$ is minimal.
Such a $z$ exists by relative compactness of $B(x,2r)$.
Since $d$ is a length metric, we have $d(x,z) + d(z, y) = d(x,y)$.
By standard Riemannian geometry, there exists an arclength, minimizing geodesic $\sigma : [0,r] \to M$ between $x,z$.
Then for all $t \in [0,r]$
\begin{multline*}
 d(x,y) \leq d(x,\sigma(t)) + d(\sigma(t), y) = d(x, z) - d(\sigma(t), z) + d(\sigma(t), y) \\
  \leq d(x,z) - d(\sigma(t), z) + d(\sigma(t), z) + d(z,y) = d(x,y). 
\end{multline*}
So
\[ d(x, \sigma(t)) + d(\sigma(t), y) = d(x,y) \textQQqq{for all} t \in [0, r]. \]
By solving the geodesic equation, we can extend $\sigma$ to an arclength geodesic $\gamma : [0,l) \to M$ for some maximal $l \leq d(x,y)$.
We now show that $\gamma$ satisfies the claim of the lemma.
Choose $t_0 \in [0, l]$ maximal with the property that $\gamma |_{[0,t_0)}$ is minimizing and with the property that $d(x, \gamma(t)) + d(\gamma(t), y) = d(x,y)$ for all $t \in [0, t_0)$.

If $t_0 = l$, then we are done, so assume $t_0 < l$ and set $w := \gamma(t_0)$.
By the same arguments as before, replacing $x$ by $w$, we can find an $r^* > 0$ and an arclength, minimizing geodesic $\sigma^* : [0,r^*] \to M$ with $\sigma^* (0) = w$ such that
\[ d(w, \sigma^*(t)) + d(\sigma^*(t), y) = d(w,y) \textQQqq{for all} t \in [0, r^*]. \]
Then
\begin{multline*}
 \ell(\gamma |_{[0,t_0]}) + \ell(\sigma^*) = d(x, w) + d(w, \sigma^*(r^*)) \\
 = d(x,y) - d(w, y) + d(w, \sigma^*(r^*)) = d(x,y) - d(\sigma^*(r^*), y) \leq d(x, \sigma^*(r^*)).  
\end{multline*}
It follows that the concatenation of $\gamma |_{[0,t_0]}$ and $\sigma^*$ is a minimizing geodesic.
This implies that $\gamma'(t_0) = (\sigma^*)'(0)$ and hence $\gamma(t_0 + t) = \sigma^* (t)$ for all $t \in [0,r^*]$.
So $\gamma|_{[0, t_0 + r^*]}$ is minimizing and for any $t \in [0,r^*]$ we have
\begin{multline*}
 d(x, y) \leq d(x , \gamma(t_0+t)) + d(\gamma(t_0+t), y) \\
 \leq d(x,w) + d(w, \sigma^*(t)) + d(\sigma^*(t), y) = d(x,w) + d(w,y) = d(x,y), 
\end{multline*}
contradicting the maximal choice of $t_0$.
This finishes the proof.
\end{proof}

Next, we generalize the concept of the exponential map and the cut locus to incomplete Riemannian manifolds.

\begin{Definition} \label{Def:expmap}
Let $(M, g)$ be a (not necessarily complete) Riemannian manifold and $p \in M$.
For every vector $v \in T_p M$ let $\gamma_v : [0, l_p (v)) \to M$ be the constant speed geodesic with $\gamma_v(0) = p$ and $\gamma'_v (0) = v$, where $l_p(v) \in (0, \infty]$ is chosen such that $[0,l_p(v))$ is the maximal interval of existence.
Define the domain
\[ \mathcal{D}_p := \big\{ v \in T_p M \;\; : \;\;  l_p (v) > 1 \big\} \subset T_p M \]
and the exponential map $\exp_p : \mathcal{D}_p \to M$ by $\exp_p (v) := \gamma_v(1)$.
We also define
\[ \mathcal{D}^*_p := \big\{ v \in \mathcal{D}_p \;\; :\;\; d(p, \exp_p ((1+\eps) v)) = |(1+\eps) v| \textQq{for some} \eps > 0 \big\}. \]
Next, we define the following two ranges in $M$:
\begin{multline*}
 \mathcal{G}_p := \big\{ x \in M \;\; : \;\; \text{there is a minimizing geodesic} \; \\ \gamma : [0,l] \to M \; \text{with} \;\gamma(0) = p, \gamma(l) = x \big\}
\end{multline*}
and
\begin{multline*}
\mathcal{G}^*_p := \big\{ x \in M \;\; : \;\; \text{there is a minimizing geodesic} \; \gamma : [0,l] \to M \; \\
\text{and some} \; 0 \leq t <  l \; \text{such that} \; \gamma(0) = p, \gamma(t) = x \big\}.
\end{multline*}
Finally, we set
\[ \mathcal{D} := \bigcup_{p \in M} \mathcal{D}_p \subset TM, \qquad \mathcal{D}^* := \bigcup_{p \in M} \mathcal{D}^*_p \subset TM \]
and
\[ \mathcal{G}^* := \bigcup_{p \in M} \{ p \} \times \mathcal{G}^*_p \subset M \times M. \]
\end{Definition}

Note that in the complete case $\mathcal{D}_p = T_p M$, $\mathcal{G}_p = M$ and $\mathcal{G}_p \setminus \mathcal{G}^*_p$ is commonly referred to as the cut locus.

We now review the following basic identities.

\begin{Proposition} \label{Prop:expincomplete}
Using the same notation as in Definition \ref{Def:expmap}, the following is true for any $p \in M$:
\begin{enumerate}[label=(\alph*)]
\item $\exp_p (\mathcal{D}_p) \supset \mathcal{G}_p$, $\exp_p (\mathcal{D}^*_p) = \mathcal{G}^*_p$ and $\exp (\mathcal{D}^*) = \mathcal{G}^*$.
\item The subsets $\mathcal{D}_p$, $\mathcal{D}$, $\mathcal{D}^*_p$, $\mathcal{D}^*$, $\mathcal{G}^*_p$ and $\mathcal{G}^*$ are open and the subsets $\mathcal{D}_p$, $\mathcal{D}^*_p$ star-shaped with respect to $0 \in \mathcal{D}_p, \mathcal{D}^*_p$.
\item The subset $\mathcal{G} \setminus \mathcal{G}^*$ has measure zero.
\item The restriction $\exp_p |_{\mathcal{D}^*_p} : \mathcal{D}^*_p \to \mathcal{G}^*_p$ is a (bijective) diffeomorphism.
Slightly more generally, for each $x \in \mathcal{G}^*_p$, there is a unique minimizing geodesic between $p, x$ and the point $x$ is not a conjugate point to $p$ along this geodesic.
\item The function $b(x) := d(x,p)$ is smooth on $\mathcal{G}^*_p$ and satisfies $|\nabla b| = 1$ there.
\item If $q \in \mathcal{G}^*_p$, then $p \in \mathcal{G}^*_q$.
So the set $\mathcal{G}^*$ is symmetric.
\item For any $(p,q) \in \mathcal{G}^*$ there is a unique arclength, minimizing geodesic $\gamma_{p,q} : [0, d(p,q)] \to M$ between $p, q$ and $\gamma_{p,q}$ depends smoothly on $p, q$.
\end{enumerate}
\end{Proposition}

\begin{proof}
For assertion (a) observe that if $x \in \mathcal{G}_p$, then there is a $v \in T_p M$ such that $l_v > 1$ and $x = \gamma_v(1)$.
So $\exp_p ( \mathcal{D}_p ) \supset \mathcal{G}_p$.
The remaining identities are clear.

Before continuing with the proof, we establish the following two useful claims:

\begin{Claim1}
Let $lv \in \mathcal{D}^*_p$, $|v| = 1$, $l \geq 0$ and set $x := \gamma_v (l) \in \mathcal{G}^*_p$.
Let $\sigma : [l-a,l] \to M$ be an arclength, minimizing geodesic such that $\sigma(l) = x$ and
\[ d(p, \sigma(l-a)) + d(\sigma(l-a), x) = d(p,x). \]
Then $\sigma$ is a subsegment of $\gamma_v$, i.e. $\sigma (t) = \gamma_v(t)$ for all $t \in [l-a,l]$.
\end{Claim1}

\begin{proof}
This claim is similar to parts of \cite[Chp 13, Proposition 2.2]{doCarmo}.
By definition, there is a constant $\eps > 0$ such that $\gamma_v |_{[0,l+\eps]}$ is minimizing.
Then
\begin{multline*}
  \ell(\gamma_v |_{[l,l+\eps]} ) + \ell(\sigma) = d(x, \gamma_v(l+\eps)) + d(\sigma(l-a), x) \displaybreak[2] \\
= d(p, \gamma_v(l+ \eps)) - d(p, x) + d(\sigma(l-a), x) = d(p, \gamma_v(l+ \eps)) - d(p, \sigma(l-a)) \\
\leq d(\sigma(l-a), \gamma_v(l+\eps)). 
\end{multline*}
It follows that the concatenation of $\sigma$ and $\gamma_v |_{[l, l+\eps]}$ is a minimizing geodesic and hence that $\sigma'(l) = \gamma_v'(l)$.
The claim follows immediately from this.
\end{proof}

\begin{Claim2}
Assume that $lv \in \mathcal{D}^*_p$, $|v| = 1$, $l \geq 0$ and set $x := \gamma_v (l) \in \mathcal{G}^*_p$.
Assume that there is a point $y \in M$ such that
\[ d(p,y) + d(y,x) = d(p,x) \]
Then $y = \gamma_v(t)$ for some $t \in [0,l]$.
\end{Claim2}

\begin{proof}
Apply Lemma \ref{Lem:geodincomplete} with $p \leftarrow x$ and $x \leftarrow y$ to obtain an arclength, minimizing geodesic $\sigma : (l',l] \to M$ for some minimal $l' \in [d(p,y), l]$ such that
\[ d(y, \sigma(t)) + d(\sigma(t), x) = d(y,x) \textQQqq{for all} t \in (l',l] \]
Then for all $t \in (l', l]$
\begin{multline*}
 d(p,x) \leq d(p, \sigma(t)) + d(\sigma(t), x) \\
  \leq d(p,y) + d(y, \sigma(t)) + d(\sigma(t), x) 
 = d(p,y) + d(y,x) = d(p,x). 
\end{multline*}
So by Claim 1 we have $\sigma = \gamma_v |_{(l',l]}$.
It follows that $l' = d(p,y)$ and hence $y = \lim_{t \to d(p,y)} \sigma(t) = \gamma_v (d(p,y))$.
\end{proof}

Let us now return to the proof of the proposition.
Assertion (d) is a direct consequence of Claim 1 and \cite[Chp 11, Corollary 2.9]{doCarmo}, modulo the fact that $\mathcal{D}^*_p$ is open (which is needed for the statement that the map is a diffeomorphism).

Consider now assertion (b).
The fact that $\mathcal{D}_p$ and $\mathcal{D}$ are open follows from the lower semi-continuity of $l_p(v)$ in $p$ and $v$.
It remains to show that $\mathcal{D}^*$ is open.
The fact that $\mathcal{G}^*$ is open will then follow using assertion (d).
Assume that $\mathcal{D}^*$ was not open, i.e. there is some $l v \in T_p M \cap \mathcal{D}^* = \mathcal{D}^*_p$, $|v|= 1$, $l \geq 0$ that is not in the interior of $\mathcal{D}^*$.
By assumption, we can find a sequence $p_i \in M$ with $p_i \to p$ and a sequence of vectors 
\[ l_i v_i \in T_{p_i} M \cap (\mathcal{D} \setminus \mathcal{D}^*) = \mathcal{D}_{p_i} \setminus \mathcal{D}^*_{p_i}, \]
$|v_i| = 1$, $l_i \geq 0$ with $l_i v_i \to l v$.
Set $x := \gamma_{v} (l) \in \mathcal{G}^*_p$, $x_i := \gamma_{v_i} (l_i)$ and $x'_i := \gamma_{v_i} (l_i + 1/i)$.
Then, since $l_i v_i \not\in \mathcal{D}^*_{p_i}$,
\begin{equation} \label{eq:dilpi}
d_i := d(p_i ,x'_i) < l_i + 1/i =: l'_i. 
\end{equation}
Note that $x_i, x'_i \to x$ and $d_i \to l$.
Using Lemma \ref{Lem:geodincomplete}, we can find a sequence of arclength, minimizing geodesics $\gamma^*_i : (l^*_i, d_i]$ such that $\gamma^*_i(d_i) = x'_i$,
\[ d(p_i, \gamma^*_i (t)) + d(\gamma^*_i(t), x'_i) = d(p_i, x'_i) = d_i \textQQqq{for all} t \in (l^*_i, d_i] \]
and such that for each $i$ either $l^*_i = 0$ and $\lim_{t \to 0} \gamma^*_i (t) = p_i$ or $l^*_i > 0$ and $\lim_{t \searrow l^*_i} \gamma^*_i(t)$ does not exist.
After passing to a subsequence, we may assume that $(\gamma^*_i)'(d_i) \to u \in T_x M$.
Since $x'_i \to x_i$, we can find a uniform $a > 0$ such that we can solve the geodesic equation starting from $x'_i$ for at least time $a$, therefore $l^*_i < d_i - a$.
This implies that the subsegments $\gamma^*_i |_{[d_i-a,d_i]}$ converge to an arclength, minimizing geodesic $\gamma^*_\infty : [l-a,l] \to M$ such that $\gamma^*_\infty (l) = x$, $(\gamma^*_\infty)'(l) = u$ and
\[ d(p, \gamma^*_\infty (t)) + d(\gamma^*_\infty (t), x) = d(p, x) = l \textQQqq{for all} t \in [l-a, l]. \]
Using Claim 1, we find that $\gamma^*_\infty = \gamma_v |_{[l-a,l]}$ and $u = \gamma_v'(l)$.
So the geodesics $\gamma^*_i$ converge to $\gamma_v$ and thus we have $l^*_i = 0$ for large $i$.
It follows that there is a sequence $v^*_i \in T_{p_i} M$ such that for large $i$ we have $\gamma^*_i = \gamma_{v^*_i} |_{(0,d_i]}$ and that $v^*_i \to v$.
So for large $i$
\[ \exp_{p_i} (d_i v^*_i) = x'_i = \exp_{p_i} (l'_i v_i) \textQQqq{and} l_i v_i, d_i v^*_i \to lv. \]
Since the differential of the exponential map is invertible at $lv$, it follows that $d_i v^*_i = l'_i v_i$ for large $i$.
This, however, contradicts (\ref{eq:dilpi}), finishing the proof of assertion (b).

So far, we have established assertions (a), (b), (d).
Assertion (e) is a direct consequence of assertion (d).
To see assertion (c), observe that by definition
\[ \mathcal{G}_p \subset \bigcap_{\eps > 0} \exp_p ((1+\eps) \mathcal{D}^*_p \cap \mathcal{D}_p). \]
So
\[ \mathcal{G}_p \setminus \mathcal{G}^*_p \subset \exp_p \bigg( \bigcap_{\eps > 0} \big( (1+\eps) \mathcal{D}^*_p \setminus \mathcal{D}^*_p \big) \cap \mathcal{D}_p \bigg). \]
Using the fact that $\mathcal{D}^*_p$ is star-shaped and Fubini's Theorem, we obtain that the subset under the parentheses has measure zero.
So $\mathcal{G}_p \setminus \mathcal{G}^*_p$ has measure zero as well.

Next, we will show assertion (f).
Assume that $q \in \mathcal{G}^*_p$.
So there is a vector $lv \in \mathcal{D}^*_p$, $|v| =1$, $l \geq 0$ such that $q = \gamma_v(l)$.
We claim that there is an $\eps > 0$ such that
\[ d(\gamma_{-v} (\eps), p) + d(p, q) = d(\gamma_{-v} (\eps), q). \]
Assume not, i.e. that for all $i$ and $p'_i := \exp_{p} (- 1/i \cdot v)$ we have
\begin{equation} \label{eq:psipqlesspsiq}
 d(p'_i, p) + d(p, q) > d(p'_i, q),
\end{equation}
Using Lemma \ref{Lem:geodincomplete} we can find a sequence of arclength, minimizing geodesics $\sigma_i : [0, l_i ) \to M$ such that $\sigma_i (0) = p'_i$, $l_i \in (0, d(p'_i, q)]$ is chosen maximal and
\[ d(p'_i, \sigma_i (t)) + d(\sigma_i(t), q) = d(p,q) \textQQqq{for all} t \in [0, l_i). \]
After passing to a subsequence, we can find an $r > 0$ such that $l_i > r$ for all $i$ and such that the restrictions $\sigma_i |_{[0,r]}$ converge to an arclength, minimizing geodesic $\sigma : [0,r] \to M$ with $\sigma(0) = p$ and
\[ d(p, \sigma(t)) + d(\sigma(t), q) = d(p,q) \textQQqq{for all} t \in [0,r]. \]
Using Claim 2, we conclude that for all $t \in [0,r]$ there is a $t' \in [0,l]$ such that $\sigma(t) = \gamma_v(t')$, i.e. $\sigma ([0,r]) \subset \gamma_v ([0,l])$ and hence $\sigma = \gamma_v |_{[0,r]}$.
It follows that there is a sequence $v'_i \in T_{p'_i}$ with $v'_i \to v$ such that $\sigma_i = \gamma_{v'_i} |_{[0,l_i)}$.
Thus $l_i = d(p'_i, q)$ for large $i$ and $q = \gamma_{v'_i} (d(p'_i, q))$.
Also, $\gamma'_{v'_i} (d(p'_i, q)) \to \gamma'_v(l)$.
Next, observe that by assertion (d) the points $p,q$ are not conjugate along $\gamma$.
So the exponential map $\exp_q$ taken at $q$ is a local diffeomorphism at $- l \gamma'_v (l)$.
Since $\exp_q ( - d(p'_i, q) \gamma'_{v'_i} (d(p'_i, q))) = \exp_q ( - (d(p'_i, p) + d(p,q))\gamma'_v(l))$, it follows that $\gamma'_{v'_i} (d(p'_i, q)) = \gamma'_v(l)$ and $d(p'_i,q) = d(p'_i, p) + d(p,q)$ for large $i$.
This, however, contradicts (\ref{eq:psipqlesspsiq}) for large $i$, finishing the proof of assertion (f).

Assertion (g) follows from assertions (d) and (f) using the implicit function theorem.
\end{proof}

Next, we will derive a local version of the Bishop-Gromov volume comparison result.
For this purpose, we will introduce the following function for any $\kappa \in \IR$
\begin{equation} \label{eq:sndef}
 \sn_\kappa (t) = \begin{cases} \kappa^{-1/2} \sin (\sqrt{\kappa} t) & \text{if $\kappa > 0$} \\ t & \text{if $\kappa = 0$} \\ (-\kappa)^{-1/2} \sinh (\sqrt{-\kappa} t) & \text{if $\kappa < 0$} \end{cases}. 
\end{equation}
Recall that $\sn_{\kappa} (t)$ models the growth of Jacobi fields in the model space of constant sectional curvature $\kappa$.
In particular, the volume $v_\kappa (r)$ of a geodesic $r$-ball in the $n$-dimensional model space of constant sectional curvature $\kappa$ is can be computed as follows if $r \leq \pi \sqrt{\kappa}$ for $\kappa > 0$:
\begin{equation} \label{eq:vkappaformula}
 v_\kappa (r) = n \omega_n \int_0^r \big( \sn_\kappa (t) \big)^{n-1} dt. 
\end{equation}

\begin{Proposition}[monotonicity of the volume element] \label{Prop:Jnonincreasing}
Let $(M^n, g)$ be a (not necessarily complete) Riemannian manifold and assume that we have $\Ric \geq (n-1) \kappa$ on $M$ for some $\kappa \in \IR$.
Let $p \in M$ and consider $\mathcal{D}^*_p$ from Definition \ref{Def:expmap}.
Let $J_p : \mathcal{D}^*_p \to (0, \infty)$ be the Jacobian of the map $\exp_p |_{\mathcal{D}^*_p}$.
Then
\[ J_p (v) \cdot \frac{|v|^{n-1}}{ \big( \sn_\kappa (|v|) \big)^{n-1}}  \]
is non-increasing on radial lines.
\end{Proposition}

\begin{proof}
This fact follows as in the classic Bishop-Gromov volume comparison.
The manifold $M$ does not need to be complete, as the proof is just a computation along minimizing geodesics.
See for example \cite[chp 9, Lemma 34]{MR2243772}.
\end{proof}

Finally, we define what we mean by a weak Laplacian.

\begin{Proposition}[existence of weak Laplacian] \label{Prop:weakLaplaciandistance}
Let $(M^n, g)$ be a (not necessarily complete) Riemannian manifold of dimension $n \geq 2$.
Consider a function $b : M \to \IR$ that is locally Lipschitz and whose Hessian is locally uniformly bounded from below in the barrier sense.
For example, the following two classes of functions (which we will consider in this paper) satisfy these properties:
\begin{enumerate}[label=(\arabic*)]
\item Let $F : (0, \infty) \to \IR$ be a smooth, non-increasing function, $p \in M$ a point and consider the function $b(x) := F(d(x,p))$.
Then $b$ satisfies these properties on $M \setminus \{ p \}$.
\item Let $(E, \langle \cdot, \cdot \rangle)$ be a Euclidean vector bundle over $M$ with a metric connection $\nabla^E$ and let $X \in C^2( M ; E)$ a section of $E$.
Then $b := |X|$ satisfies these properties on $M$
\end{enumerate}

Then the gradient $\nabla b$ exists almost everywhere and for any vector field $Z \in C^1_c (M ; TM)$ we have
\begin{equation} \label{eq:Zintbyparts}
 \int_M  \langle Z , \nabla b \rangle  dg = - \int_M (\DIV Z) b \, dg. 
\end{equation}

Moreover, there is a unique signed measure $\mu_{\triangle b}$ on $M$ of locally finite total variation such that for any compactly supported $\varphi \in C^2_c (M)$ we have
\begin{equation} \label{eq:intbypartsdist}
 \int_M b (x) \triangle \varphi (x) dg(x) = - \int_{M} \langle \nabla b (x) , \nabla \varphi (x) \rangle dg(x) = \int_{M}  \varphi (x) d \mu_{\triangle b} (x). 
\end{equation}
Also, $\mu_{\triangle b}$ can be expressed as the difference of its positive and negative part, $\mu_{\triangle b} = (\mu_{\triangle b})_+ - (\mu_{\triangle b})_- $.
The negative part $(\mu_{\triangle b})_-$ is absolutely continuous with respect to $dg$.
Lastly, $d\mu_{\triangle b} = (\triangle b) dg$ wherever $b$ is $C^2$.

If in example (1) the function $F$ is non-decreasing, then similar statements hold, which follow by replacing $F$ by $-F$.
\end{Proposition}

\begin{proof}
Let us first show that the two classes of examples fall into the desired category of functions.
Consider first example (1), i.e. $b(x) = F(d(x,p))$.
Obviously, $b$ is locally Lipschitz.
Let $U \subset M \setminus \{ p \}$ be an open subset that is relatively compact in $M \setminus \{ p \}$.
We will now show that the Hessian of $b$ is uniformly bounded from below on $U$.
To do this choose $r > 0$ small enough such that the tubular neighborhood $B (U, 3r)$ is still relatively compact in $M \setminus \{ p \}$, such that $|{\Rm}| < r^{-2} / 10$ and $\inj > 2r$ on $B (U, r)$.
Let $x \in U$.
By Lemma \ref{Lem:geodincomplete} there is a point $y \in M$ such that
\[ d(y,x) = r \textQQqq{and} d(x,p) = d(x,y) + d(y,p)  . \]
By the triangle inequality, we have for all $z \in M$
\[ b(z) = F(d(z, p)) \geq  F \big( d(z,y) + d(y,p) \big) =: b^*_y(z). \]
Moreover, $b^*_y (x) = b(x)$.
For $z$ close to $x$ we have $d(y,z) < 2r$ and hence the function $b^*_y(z)$ is a $C^2$ barrier for $b$ at $x$.
Its Hessian can be bounded from above using the sectional curvature bound on $B(U, 3r)$.
This shows that the desired fact.

To see that the functions from example (2), i.e. $b = |X|$, fall in the desired category, observe that any such $b$ can be expressed as a limit of $C^2$ functions:
\[ b = \lim_{\eps \to 0} \sqrt{|X|^2 + \eps}. \]
We can now check that
\[ \big| \nabla^E \sqrt{|X|^2 + \eps} \big| \leq \frac{|\nabla^E X | \cdot |X|}{\sqrt{|X|^2 + \eps}} \leq | \nabla^E X | \]
and that for any $v \in TM$
\begin{multline} \label{eq:KatoHessian}
 \nabla^{E,2}_{v,v} \sqrt{|X|^2 + \eps} = \frac{\langle \nabla^{E,2}_{v,v} X , X \rangle + |\nabla^E_v X |^2}{\sqrt{|X|^2 + \eps}} - \frac{\langle \nabla^E_v X , X \rangle^2}{(|X|^2 + \eps)^{3/2}} \\
 \geq - |\nabla_{v,v}^{E,2} X | + \frac{|\nabla_v^E X|^{2} |X|^2 - \langle \nabla^E_v X , X \rangle^2}{(|X|^2 + \eps)^{3/2}} \geq - |\nabla_{v,v}^{E,2} X|.
\end{multline}
So the approximators $\sqrt{|X|^2 + \eps}$ are locally uniformly Lipschitz and have local uniform lower bounds on the Hessian (in the barrier sense).
These properties pass to the limit.

We will now focus on the main part of the proof.
Let $U \subset M$ be an open subset that is relatively compact in $M$ and on which we can find coordinates $(x^1, \ldots, x^n) : U \to V \subset  \IR^n$.
Moreover let $U' \subset U$ be an open subsets such that $\ov{U'} \subset U$ and $U'' \subset U'$ an open subset such that $\ov{U''} \subset U'$.
We will first construct $\mu_{\triangle b}$ on $U''$ such that (\ref{eq:intbypartsdist}) holds for all $\varphi \in C^2_c (U'')$.

As $b$ is Lipschitz on $U$ and its Hessian is uniformly bounded from below in the barrier sense, we find a constant $C_1 < \infty$ such that if we express $b = b(x^1, \ldots, x^n)$ in terms of the coordinates, then $b$ is $C_1$-Lipschitz and its Hessian is bounded from below by $- C_1$ in the barrier sense.
Using the coordinate $(x^1, \ldots, x^n)$ and a convolution argument, we can find a family of smoothings $b_\eps : U' \to \IR$, $\eps > 0$ such that the following holds for some uniform $C_2 < \infty$:
\begin{enumerate}[label=(\roman*)]
\item $\lim_{\eps \to 0} b_\eps = b$ uniformly on $U'$,
\item $\nabla b_\eps \to \nabla b$ in $L^1 (U')$ as $\eps \to 0$.
\item $| \nabla b_\eps| < C_2$ and $\nabla^2 b_\eps > - C_2$ on $U'$ for all $\eps > 0$.
\end{enumerate}

We first conclude that for any $Z \in C^1_c (U' ; TM)$
\[
\int_M \langle Z , \nabla b \rangle dg 
= \lim_{\eps \to 0} \int_M \langle Z , \nabla b_\eps \rangle dg
= - \lim_{\eps \to 0} \int_M (\DIV Z) b_\eps \, dg
= - \int_M (\DIV Z) b \, dg.
\]
This shows (\ref{eq:Zintbyparts}) on $U'$ and by setting $Z = \nabla \varphi$, we obtain that the first part of (\ref{eq:intbypartsdist}) holds for all $\varphi \in C^2_c (U')$.

Next, observe that for any $\varphi \in C^2_c (U')$ we have
\begin{equation} \label{eq:btrianglevarphi}
  \int_M b(x)\triangle \varphi (x) dg(x) = \lim_{\eps \to 0} \int_M  b_\eps (x) \triangle \varphi(x) dg(x) = \lim_{\eps \to 0}\int_M  \triangle b_\eps (x)  \varphi(x) dg(x).  
\end{equation}
Using the lower bound on the Hessian of $b_\eps$ and (\ref{eq:btrianglevarphi}), we find that whenever $\varphi \geq 0$, we have
\begin{equation} \label{eq:functionalbound1}
  -\int_M b(x)\triangle \varphi (x) dg(x) \leq C_2 \int_{M} \varphi(x) dg(x) \leq C_2 |U| \cdot \max_{M} \varphi. 
\end{equation}
Choose and fix now a cutoff function $\psi \in C^2_c (U')$  such that $0 \leq \psi \leq 1$ everywhere and $\psi \equiv 1$ on $U''$.
Set
\[ C_3 := \bigg| \int_M b(x) \triangle \psi (x) dg(x) \bigg| . \]
Then, whenever $\varphi \in C^2_c (U'')$ and $\varphi \geq 0$, we have with $A := \max_M \varphi$
\begin{multline} \label{eq:functionalbound2}
 \int_M b(x) \triangle \varphi(x) dg(x) = A \int_M b(x) \triangle \psi(x) dg(x) - \int_M b(x) \triangle \big( A \psi(x) - \varphi(x) \big) dg(x) \\
 \leq A \int_M b(x) \triangle \psi(x) dg(x) + C_2 |U| \max_M  \big( A \psi - \varphi)  
 = \big( C_3 + C_2 |U| \big) \max_M \varphi.
\end{multline}
Combining (\ref{eq:functionalbound1}) and (\ref{eq:functionalbound2}) implies that the functional
\[ H : C^2_c (U'') \longrightarrow \IR, \qquad \varphi \longmapsto \int_M b(x) \triangle \varphi (x) dg(x) \]
satisfies 
\begin{equation} \label{eq:functionalequation}
| H(\varphi) | \leq C_4 \max_M \varphi
\end{equation}
for all $\varphi \in C^2_c (U'')$ with $\varphi \geq 0$, where $C_4 < \infty$ is a uniform constant.
By approximation, this implies that $H$ can be extended to $C^0_c(U'')$ and that (\ref{eq:functionalequation}) for all $\varphi \in C^0_c (U'')$.
So by the Riesz-Markov Theorem, there is a unique signed measure $\mu_{\triangle b}$ of finite total variation on $U''$ such that for all $\varphi \in C^2_c (U'')$ we have
\begin{equation} \label{eq:characterizationUprime}
 \int_{U''} b(x) \triangle \varphi(x) dg(x) = \int_{U''} \varphi d\mu_{\triangle b}. 
\end{equation}

The fact that $(\mu_{\triangle b})_-$ is absolutely continuous with respect to $dg$ can be seen as follows: By the first inequality of (\ref{eq:functionalbound1}) and (\ref{eq:characterizationUprime}), we get
\[ \int_{U'} \varphi d\mu_{\triangle b} \geq -C_2 \int_{U''} \varphi(x) dg(x). \]
So, again by Riesz-Markov, $ d\mu_{\triangle b} + C_2 dg$ is a non-negative measure.
It follows that $d(\mu_{\triangle b})_- \leq C_2 dg$.

Repeating the argument above on the elements $U''_i$ of an open cover $\mathcal{U} = \{ U''_i \}_{i \in I}$ of $M$ yields a singed measure $\mu_{\triangle b}$ of locally finite total variation on $M$ such that (\ref{eq:intbypartsdist}) holds for all $\varphi \in C_c^2 (U''_i)$.
Using a partition of unity, we can then establish (\ref{eq:Zintbyparts}) for all $Z \in C^1_c (M ; TM)$ and (\ref{eq:intbypartsdist}) for all $\varphi \in C_c^2 (M)$.
\end{proof}

\section{Convergence to a singular space and derived properties}
In this section we prove Theorem \ref{Thm:basicconvergence}.
In the first subsection we will show that the limit of any sequence of Riemannian manifolds that satisfy the properties of subsection \ref{subsec:characterlimit} is a singular space that has mild singularities of a certain codimension.
In the second and third subsection, we will then verify the $Y$-tameness and $Y$-regularity properties of the limit.

\subsection{Convergence to a singular space}
In this subsection we will prove the following result, which is a more detailed version of Theorem \ref{Thm:basicconvergence} minus assertions (b) and (c), which state the $Y$-tameness and $Y$-regularity properties.

\begin{Theorem} \label{Thm:detailedconvergence}
Let $\{ (M_i, g_i, q_i) \}_{i=1}^\infty$ be a sequence of pointed, complete Riemannian manifolds, each with bounded curvature that converge to a metric space $(X, \linebreak[1] d_X, \linebreak[1] q_\infty)$ in the pointed Gromov-Hausdorff sense.
Assume that $\{ (M_i, g_i, q_i) \}_{i=1}^\infty$ satisfies properties (A)--(C) of subsection \ref{subsec:characterlimit} for some constant $\mathbf{p}_0> 0$.
Then the limit space $(X, d_X)$ is part of a singular space $\XX = (X, d_X, \RR, g)$ in the sense of Definition \ref{Def:singlimitspace}, with singularities of codimension $\mathbf{p}_0$, in the sense of Definition~\ref{Def:codimensionsingularities}.
Moreover, after passing to a subsequence, we have convergence of $\{ (M_i, g_i, q_i) \}_{i=1}^\infty$ to $(\XX, q_\infty)$ according to some convergence scheme $\{ (U_i, V_i, \Phi_i ) \}_{i=1}^\infty$ (in the sense of Definition \ref{Def:convergencescheme}) such that:
\begin{enumerate}[label=(\alph*), start=1]
\item For any $x \in \RR$ and $r > 0$ we have
\begin{multline*}
\qquad\qquad \big| B^{X} (x, r) \cap \RR \big| \leq \liminf_{i \to \infty} \big| B^{M_i} (\Phi_i(x),  r) \big| \\
 \leq \limsup_{i \to \infty} \big| B^{M_i} (\Phi_i(x), r) \big| \leq  \big| \ov{B^{X} (x,r)} \cap \RR \big|. 
\end{multline*}
\item For any $x \in \RR$ we have
\begin{multline*}
 \qquad\qquad 0 < \liminf_{i \to \infty} \tdrrm^{M_i} (\Phi_i ( x)) = \limsup_{i \to \infty} \tdrrm^{M_i} (\Phi_i ( x)) \\
 \leq \rrm^\infty (x) \leq \liminf_{i \to \infty} \rrm^{M_i} (\Phi_i ( x)). 
\end{multline*}
Here $\rrm^\infty$ denotes the curvature radius on $\XX$.
\item For any $D < \infty$ and $\sigma > 0$ and sufficiently large $i$ (depending on $D$ and $\sigma$) we have
\[ \tdrrm < \sigma \textQQqq{on} B^{M_i} (q_i,  D) \setminus V_i \]
and
\[ \rrm^\infty < \sigma \textQQqq{on} B^{X} (q_\infty, D) \setminus U_i. \]
\end{enumerate}
If the sequence $\{ (M_i, g_i, q_i) \}_{i=1}^\infty$ additionally satisfies property (E) of subsection~\ref{subsec:characterlimit}, then $\XX$ has mild singularities in the sense of Definition \ref{Def:mild}.
\end{Theorem}

\begin{proof}
By assumption the pointed metric spaces $(M_i, d_{M_i}, q_i)$ converge to $(X, \linebreak[1] d_X, \linebreak[1] q_\infty)$ in the Gromov-Hausdorff sense.
This means that there are sequences $\eps_i \to 0$, $R_i \to \infty$ and maps $\Psi_i : B^X (q_\infty, R_i) \to M_i$ such that the following holds:
\begin{enumerate}[label=(\arabic*)]
\item $d_{M_i} (\Psi_i(q_\infty), q_i) < \eps_i$.
\item For any $x, y \in B^X(q_\infty ,  R_i)$ we have
\[ \big| d_{M_i} (\Psi_i (x), \Psi_i(y)) - d_X (x,y) \big| < \eps_i. \]
\item For any $i$ and any $z \in B^{M_i} (q_i,  R_i)$ there is a $z^* \in B^X (q_\infty , R_i)$ such that $d_{M_i} (\Psi_i(z^*), z) < \eps_i$.
\end{enumerate}
We will now pass to a subsequence such that the functions $\tdrrm^{M_i} (\cdot )$ converge to a function $h$ on $X$.

%: HERE
\begin{Claim1}
We may pass to a subsequence such that one of the following is true:
Either $(X, d_X)$ is isometric to Euclidean space $(\IR^n, d_{\textnormal{eucl}})$ and $\lim_{i \to \infty} \tdrrm^{M_i} (q_i) = \infty$ or the limit $h(x) := \lim_{i \to \infty} \tdrrm^{M_i} (\Psi_i (x))$ exists for all $x \in X$.
Moreover, $h$ is $1$-Lipschitz.
\end{Claim1}

\begin{proof}
Since $\tdrrm^{M_i} (\cdot )$ is $1$-Lipschitz with respect to $d_{M_i}$ we have for all $x, y \in X$
\[ \big| \tdrrm^{M_i} ( \Psi_i (x) ) - \tdrrm^{M_i} (\Psi_i (y)) \big| \leq d_{M_i} (\Psi_i (x), \Psi_i (y) ) < d_X (x,y) + \eps_i. \]
So the assertion follows from an Arzela-Ascoli type argument.
\end{proof}

In the case in which $\lim_{i \to \infty} \tdrrm^{M_i} (q_i) = \infty$, we are done.
So assume in the following that $h : X \to [0, \infty)$ exists.
Then we define
\[ \RR := h^{-1} ( (0, \infty)). \]
By continuity of $h$ it follows that the subset $\RR$ is open.
Using the non-collapsing property (A) from subsection \ref{subsec:characterlimit}, we conclude that $\RR$ can be equipped with the structure of a differentiable manifold of regularity $C^4$, whose topology coincides with the topology induced by the metric $d_X$ and we can find a $C^3$-Riemannian metric $g$ on $\RR$ such that $d_X |_{\RR}$ is locally equal to the length metric of $(\RR, g)$.
More specifically, for any $p \in \RR$ there is an open neighborhood $p \in U \subset \RR$ such that the length metric of $(\RR, g)$ restricted to $U$ is equal to $d_X |_{U}$.
Note that this implies that for any $C^1$-curve $\gamma : [0,1] \to \RR$ we have $\ell_g (\gamma) \geq d_X (\gamma(0), \gamma(1))$, where $\ell_g (\gamma)$ denotes the length of $\gamma$ with respect to $g$.

\begin{Claim2}
$\RR$ is dense in $(X,d_X)$.
\end{Claim2}

\begin{proof}
Assume not and pick a point $p \in X$ and an $r > 0$ such that $h \equiv 0$ on $B^X (p,r) \subset X$.
Let $\eps > 0$ and choose an $\eps$-net $x_1, \ldots, x_N \in B^X (p,r)$.
It follows that for large $i$ the points $\Psi_i (x_1), \ldots, \Psi_i(x_n)$ form a $2\eps$-net of $B^{M_i} (\Psi_i (p),  r)$.
Moreover, for large $i$, we have $\tdrrm^{M_i} (\Psi_i(x_j)) < \eps$ for all $j = 1, \ldots, N$.
Since $\tdrrm^{M_i} (\cdot )$ is $1$-Lipschitz with respect to $d_{M_i}$, we find that $\tdrrm^{M_i} (\cdot ) < 3\eps$ on $B^{M_i} (\Psi_i(p), r)$ for large $i$.
Choosing $\eps$ small enough, this, however, contradicts properties (A) and (B) of subsection \ref{subsec:characterlimit}.
\end{proof}

We have shown that $\XX := (X, d_X , \RR, g)$ satisfies all properties (1)--(3) of the Definition of a singular space (see Definition \ref{Def:singlimitspace}).
Using a center of gravity construction, and by passing to another subsequence, it is possible to show the existence of a scheme $\{ (U_i, V_i, \Phi_i )\}_{i = 1}^\infty$ for the convergence $(M_i, g_i, q_i)$ to $(\XX, q_\infty)$ (in the sense of Definition \ref{Def:convergencescheme}).
We may also assume that this scheme approximates $\Psi_i$ in the following sense:
There is a sequence $\eps'_i \to 0$ such that for all $i$ and $x \in U_i \cap B^X(q_\infty, R_i)$ we have $d_{M_i} (\Psi_i(x), \Phi_i(x)) < \eps'_i$.
Note that, formally, the definition of a convergence scheme requires the limiting space to be a singular space in the sense of Definition \ref{Def:singlimitspace}, meaning that it has to satisfy properties (4) and (5) of this definition as well.
However, even without these properties, the definition of a convergence scheme still makes sense.

\begin{Claim3}
For any $r > 0$ we have
\[
 \sup_{x \in B^X(q_\infty , r) \cap U_i} \big| h (x) -  \tdrrm^{M_i} (\Phi_i(x)) \big|  \to 0 
\]
as $i \to \infty$.
\end{Claim3}

\begin{proof}
This follows from the fact that $\Phi_i$ approximates $\Psi_i$, that $\tdrrm^{M_i} (\cdot )$ is $1$-Lipschitz and that $\tdrrm^{M_i} (\Psi_i(\cdot))$ converges to $h$ uniformly on compact subsets of $X$.
\end{proof}

\begin{Claim4}
For any $r, \sigma > 0$ and for sufficiently large $i$ (depending on $r, \sigma$) we have
\[
 \tdrrm^{M_i} (\cdot ) < \sigma \textQQqq{on} B^{M_i} (q_i,  r) \setminus V_i .
\]
\end{Claim4}

\begin{proof}
Assume that this was wrong for some $r, \sigma > 0$.
Then, after passing to a subsequence, there is a sequence of points $x_i \in B^{M_i} (q_i,  r) \setminus V_i$ such that $\tdrrm^{M_i} (x_i) \geq \sigma$ for all $i$.
By the definition of a convergence scheme (see Definition \ref{Def:convergencescheme}(5)), we can pick a sequence $y_i \in V_i$ such that $d_{M_i} (x_i, y_i) \to 0$.
So $\tdrrm^{M_i} (y_i) > \sigma / 2$ for large $i$.
Choose $z_i \in U_i$ such that $\Phi_i (z_i ) = y_i$.
By local compactness, we can pass to a subsequence and assume that $z_i \to z_\infty \in X$.
By Claim 3, we have $h(z_i) > \sigma / 4$ for large $i$ and hence $h(z_\infty) \geq \sigma /4 > 0$.
Thus $z_\infty \in \RR$ and so there is a $\delta > 0$ such that $B^X (z_\infty, 2\delta) \subset U_i$ for large $i$.
As $\Phi_i^* g_i \to g$ in $C^3$ on $B^X (z_\infty, 2\delta)$, we find that
\[ B^{M_i} (\Phi_i (z_\infty),  \delta) \subset \Phi_i (B^X (z_\infty, 2\delta)) \subset \Phi_i (U_i) = V_i \]
for large $i$.
But $d_{M_i} (\Phi_i (z_\infty), x_i) \leq d_{M_i} (\Phi_i (z_\infty), \Phi_i (z_i)) + d_{M_i} (\Phi_i (z_i), x_i) \to 0$ as $i \to \infty$, which implies that for large $i$ we have $x_i \in B^{M_i} (\Phi_i (z_\infty), \delta) \subset V_i$, contradicting our assumptions.
\end{proof}

Note that assertion (c) of this theorem holds due to Claim 4 and a compactness argument.

\begin{Claim5}
For any $x, y \in \RR$ and $\eps > 0$ there is a $C^1$-curve $\gamma : [0,1] \to \RR$ between $x, y$, such that $\ell (\gamma ) < d (x,y) + \eps$.
In other words, $(X,d_X)$ is the completion of the length metric on $(\RR, g)$ and $\XX$ satisfies property (4) of Definition~\ref{Def:singlimitspace}.
\end{Claim5}

\begin{proof}
Let $x, y \in \RR$ and choose $0 < \sigma_0  < \min \{ h(x), h(y) \}$.
Let $x_i := \Phi_i (x), y_i := \Phi_i(y) \in M_i$ for large $i$.
Then $\tdrrm^{M_i} (x_i), \tdrrm^{M_i} (y_i) > \sigma_0$ and $d_{M_i} (x_i, y_i) < d_X (x,y) + \eps / 4$ for large $i$.
By property (C) from subsection \ref{subsec:characterlimit}, there is a uniform $\sigma > 0$ and a sequence of curves $\gamma_i : [0,1] \to M_i$ between $x_i, y_i$ such that for large $i$ we have $\ell_{g_i} (\gamma_i) < d_X (x, y) + \eps/2$ and such that $\tdrrm^{M_i} (\gamma_i(s)) > \sigma$ for all $[0,1]$.
By Claim 4, we have $\gamma_i ([0,1]) \subset V_i$ for large $i$.
So for large $i$, the curve $\Phi_i^{-1} \circ \gamma_i : [0,1] \to U_i \subset \RR$ between $x, y$ has length $< d_X (x,y) + \eps$.
This proves the claim.
\end{proof}

\begin{Claim6}
Assertion (b) of this theorem holds and we have $\rrm^\infty \geq h$.
\end{Claim6}

\begin{proof}
First note that if $x \in X \setminus \RR$, then by definition $h(x) = 0$ and $\rrm^\infty (x) = 0$.
So in this case, there is nothing to show.
Next, for any $x \in \RR$, the first inequality and the equality in assertion (b) holds because the $\liminf$ and $\limsup$ are equal to $h$.

We now show that $\rrm^\infty \geq h$, which implies the second inequality in assertion (b).
Pick $x \in \RR$ arbitrarily and observe that $ h(x) > 0$.
Since $h$ is $1$-Lipschitz, we have $B^X (x, h(x)) \subset \RR$.
Moreover, for all $y \in B^X (x, h(x))$ we have $d_{M_i} (\Phi_i(x), \Phi_i(y)) < \rrm^{M_i} (\Phi_i(x))$ for infinitely many $i$ and
\[ |{\Rm}|(y) = \lim_{i \to \infty} |{\Rm}|(\Phi_i(y)) \leq \liminf_{i \to \infty} \big( \rrm^{M_i} (\Phi_i(x)) \big)^{-2} = h^{ -2} (x). \]
Similarly, we obtain
\[ |{\nabla \Rm}|(y) \leq h^{-3} (x). \]
This implies that for all $x \in \RR$ we have $\rrm^\infty (x) \geq h(x)$.

Next let $h_* (x) := \liminf_{i \to \infty} \rrm^{M_i} (\Phi_i (x))$ for all $x \in \RR$.
It remains to prove that $\rrm^\infty \leq h_*$ on $\RR$.
Pick $x \in \RR$ arbitrarily and observe that $\rrm^\infty (x) > 0$.
Choose $0 < r < \rrm^\infty(x)$.
Then the closure of $B^X (x,r)$ is compact and contained in $\RR$ and we have $|{\Rm}| < r^{-2}$ on $\ov{B^X(x,r)}$.
So for large $i$ we have $\ov{B^X(x,r)} \subset U_i$ and for any $\delta > 0$ we have, for large $i$, that $B^{M_i} (\Phi_i(x),  r - \delta) \subset \Phi_i (B^X(x,r))$ and $|{\Rm}| < (r - \delta)^{-2}$ on $\Phi_i (B^X(x,r))$.
So for large $i$, we have $\rrm^{M_i} (\Phi_i(x)) \geq r - \delta$.
Passing to the limit and using Claim 3, implies that $h_* (x) \geq r - \delta$.
Letting $\delta \to 0$, proves the desired result.
\end{proof}

\begin{Claim7}
For any $x \in \RR$ and any $r > 0$ we have
\begin{multline*}
 \big| B^X (x, r) \cap \RR \big| \leq \liminf_{i \to \infty} \big| B^{M_i} (\Phi_i(x),  r) \big| \\
 \leq \limsup_{i \to \infty} \big| B^{M_i}(\Phi_i(x),  r) \big| \leq \big| \ov{B^X(x,r)} \cap \RR \big|. 
\end{multline*}
So assertion (a) of this theorem holds.
\end{Claim7}

\begin{proof}
We first show the first inequality.
By the fact that
\[  \big| B^X (x, r ) \cap \RR \big| = \lim_{\delta \to 0} \big| B^X (x, r - \delta) \cap \RR \big|, \]
it suffices to show that for all $\delta > 0$ we have
\begin{equation} \label{eq:uppervolbounddelta}
 \big| B^{X} (x, r - \delta) \cap \RR \big| \leq \liminf_{i \to \infty} \big| B^{M_i} (\Phi_i(x), 0, r) \big| . 
\end{equation}
Note that for large $i$ we have
\[ \Phi_i \big( U_i \cap B^X (x, r - \delta) \big) \subset B^{M_i} (\Phi_i(x),  r). \]
So
\[ \big| U_i \cap B^X (x,r-\delta) \big| \leq (1+\eta_i) \big|B^{M_i} (\Phi_i(x), r) \big| \]
for some $\eta_i \to 0$.
Since $\bigcup_{i=1}^\infty U_i = \RR$, we obtain (\ref{eq:uppervolbounddelta}).

Next, we show that $\limsup_{i \to \infty} |B^{M_i} (\Phi_i(x),  r)| \leq |\ov{B^X(x,r)} \cap \RR|$.
To do this, we use the fact that
\[ \ov{B^X(x, r )} \cap \RR = \bigcap_{\delta > 0} \big( B^X (x, r + \delta) \cap \RR \big), \]
and hence that
\[ \lim_{\delta \to 0} \big| B^X (x, r + \delta) \cap \RR \big| = \big| \ov{B^X(x, r )} \cap \RR \big|, \]
to see that it suffices to prove
\[
 \big| B^X (x, r + \delta) \cap \RR \big| \geq \limsup_{i \to \infty} \big| B^{M_i} (\Phi_i(x),  r) \big|. 
\]
Then, as before, we can conclude that for large $i$
\[ \big| U_i \cap B^X (x, r+\delta) \big| \geq (1-\eta_i) \big| V_i \cap B^{M_i} (\Phi_i(x),  r) \big| \]
for some $\eta_i \to 0$.
So we need to show that
\begin{equation} \label{eq:volgoesto0}
\big| B^{M_i} (\Phi_i (x), r) \setminus V_i \big| \to 0. 
\end{equation}
To see this, observe that by Claim 4, for any $\sigma > 0$ we have
\[ B^{M_i} (\Phi_i(x),  r) \setminus V_i  \subset \{ \tdrrm^{M_i}  (\cdot ) < \sigma \} \]
for sufficiently large $i$.
So (\ref{eq:volgoesto0}) follows using property (B) of subsection \ref{subsec:characterlimit}.
\end{proof}

Using Claim 7 and property (A) of subsection \ref{subsec:characterlimit}, we obtain:

\begin{Claim8}
For any $D < \infty$ there are constants $\kappa_1 = \kappa_1(D), \kappa_2 = \kappa_2(D) > 0$ such that for all $x \in B^X(q_\infty, D)$ and $0 < r < D$ we have
\[
 \kappa_1 r^n < \big| B^X (x,r) \cap \RR \big| < \kappa_2 r^n. 
\]
In other words, $\XX = (X, d_X, \RR, g)$ satisfies property (5) of Definition \ref{Def:singlimitspace}.
\end{Claim8}

It follows that $\XX := (X, d_X, \RR, g)$ is a singular space in the sense of Definition~\ref{Def:singlimitspace}.

\begin{Claim9}
For any $x \in X$ and $r_0 > 0$ there is an $\mathbf{E}_{x, r_0} < \infty$ such that for any $0 < r < r_0$ and $0 < s < 1$ we have
\[ \big| \{ \rrm^\infty < sr \} \cap B^X (x,r) \cap \RR \big| \leq \mathbf{E}_{x,r_0} s^{\mathbf{p}_0} r^n. \]
So $\XX$ has singularities of codimension $\mathbf{p}_0$.
\end{Claim9}

\begin{proof}
Fix $0 < r < r_0$, $0 < s < 1$ and let $\delta > 0$ be some small constant.
Then, by compactness, we have for some large $i$ that
\[ \{ \delta \leq \rrm^\infty \leq sr \} \cap \ov{B^X (x,r)}  \subset U_i. \]
By Claim 6, we find that for large $i$
\[ \tdrrm^{M_i} (\cdot ) < sr + \delta \textQq{on} \Phi_i \big( \{ \delta \leq \rrm^\infty \leq sr \} \cap \ov{B^X (x,r)}  \big). \]
So, using property (B) from subsection \ref{subsec:characterlimit}, we can find a constant $\mathbf{E}_{x, r_0} < \infty$ such that for large $i$
\[ \big|  \{ \delta < \rrm^\infty < sr \} \cap B^X (x,r) \big| \leq (1+\delta) \big| \{ \tdrrm^{M_i} (\cdot ) < sr + \delta \} \big| \leq \mathbf{E}_{x, r_0} (s + \delta r^{-1})^{\mathbf{p}_0} r^n. \]
Letting $\delta \to 0$ yields the desired result.
\end{proof}

It only remains to establish the mildness of the singularities of $\XX$ assuming that $\{ (M_i, g_i, q_i ) \}_{i=1}^\infty$ satisfies property (E) in subsection \ref{subsec:characterlimit}.
So assume for the remainder of the proof that property (E) holds.

\begin{Claim10}
Let $x \in \RR$ and $r, \delta > 0$.
Then there is a $\sigma = \sigma (x, r, \delta) > 0$ and a sequence of numbers $\eta_i \to 0$ and a sequence of open subsets $S_i \subset B^X (x, r) \cap \RR$ such that for large $i$ we have
\[ \big| (B^X (x,r) \cap \RR) \setminus S_i \big| < \delta \]
and such that for any $y \in S_i$ there is a $C^1$-curve $\gamma : [0,1] \to \RR$ between $x, y$, $\gamma(0) = x, \gamma(1) = y$ such that
\[ \ell_g (\gamma) < d_X (x,y) + \eta_i \]
and
\[ \rrm^\infty (\gamma(s)) > \sigma \textQQqq{for all} s \in [0,1]. \]
\end{Claim10}

\begin{proof}
This is a direct consequence of property (E) in subsection \ref{subsec:characterlimit}.
\end{proof}

\begin{Claim11}
Let $x \in \RR$ and $r, \delta > 0$.
Then there is a $\sigma = \sigma (x, r, \delta) > 0$ and a subset $S^*_{x,r, \delta} \subset B^X (x,r) \cap \RR$ such that
\begin{equation} \label{eq:Sstarvolumedelta}
 \big| (B^X (x,r) \cap \RR) \setminus S^*_{x,r,\delta} \big| \leq \delta 
\end{equation}
and such that for any $y \in S^*_{x,r,\delta}$ there is a sequence $\gamma_i : [0,1] \to \RR$ of $C^1$-curves between $x, y$, $\gamma_i(0) = x, \gamma_i(1) = y$, such that
\[ \liminf_{i \to \infty} \ell_g (\gamma_i) = d_X (x,y) \]
and such that for all $i$ we have
\[ \rrm^\infty (\gamma_i (s)) > \sigma \textQQqq{for all} s \in [0,1]. \]
\end{Claim11}

\begin{proof}
We can construct $S^*_{x,r,\delta}$ as follows:
Let $S_1, S_2, \ldots$ be the subsets from Claim~10 and let $S^*_{x,r,\delta}$ be the set of all points $y \in B^X(x,r) \cap \RR$ such that $y \in S_i$ for infinitely many $i$.
It remains to prove (\ref{eq:Sstarvolumedelta}).
To see this bound, observe that
\[ \big( B^X (x,r) \cap \RR \big) \setminus S^*_{x,r,\delta} = \bigcup_{j=1}^\infty \bigcap_{i=j}^\infty \big( \big( B^X (x,r) \cap \RR \big) \setminus S_i \big). \]
The subset
\[ \bigcap_{i=j}^\infty \big( \big( B^X (x,r) \cap \RR \big) \setminus S_i \big) \]
has measure bounded by $\delta$ for each $j$ and it is monotone in $j$.
So the union of these subsets has measure bounded by $\delta$ as well.
\end{proof}

We now claim that for any $y \in S^*_{x,r,\delta}$ there is in fact a minimizing geodesic $\gamma : [0,1] \to \RR$ between $x, y$ whose image is contained in $\RR$.
To see this, observe that by Claim 11 and compactness of $\{ \rrm^\infty \geq \sigma \} \cap \ov{B^X(x,r)}$, for any $0 = s_0 < s_1 < \ldots < s_m = 1$ there are points $z_0, \ldots, z_m \in \{ \rrm^\infty \geq \sigma \} \cap \ov{B^X(x,r)} \subset \RR$ such that
\[ d_X (x, z_1) + d_X (z_1,z_2) + \ldots + d_X (z_{m-1}, y) = d_X (x,y) \]
and such that
\begin{equation} \label{eq:dzkm1zkSstar}
 d_X (z_{k-1}, z_k) = (s_k - s_{k-1}) d_X (x,y) \textQQqq{for all} k =1, \ldots, m. 
\end{equation}
The parameters $s_1, \ldots, s_{m-1}$ can be chosen such that we have
\[ d_X (z_{k-1}, z_k) < \sigma \leq \rrm^\infty (z_{k-1}) \textQQqq{for all} k =1, \ldots, m. \]
So we can choose minimizing geodesics between $z_{k-1}, z_k$ for each such $k$.
Their concatenation is a minimizing geodesic between $x,y$ due to (\ref{eq:dzkm1zkSstar}), which proves our claim.

Consider now the Riemannian manifold $(\RR, g)$ and define $\mathcal{G}_x, \mathcal{G}^*_x \subset \RR$ as in Definition \ref{Def:expmap}.
We have shown that for any $r, \delta > 0$ we have
\[ S^*_{x,r, \delta} \subset \mathcal{G}_x. \]
It follows using (\ref{eq:Sstarvolumedelta}) that $\RR \setminus \mathcal{G}_x$ has zero measure.
By Proposition \ref{Prop:expincomplete}(c), the complement $\mathcal{G}_x \setminus \mathcal{G}^*_x$ has zero measure as well and by Proposition \ref{Prop:expincomplete}, the subset $\mathcal{G}^*_x$ is open.
So $Q_x := \RR \setminus \mathcal{G}^*_x = (\RR \setminus \mathcal{G}_x) \cup (\mathcal{G}_x \setminus \mathcal{G}^*_x)$ is closed and has measure zero.
This finishes the proof.
\end{proof}

\subsection{The tameness properties} \label{subsec:tame}
We will now present the proof of Theorem \ref{Thm:basicconvergence}(b), establishing the $Y$-tameness property of the limiting singular space.
More specifically, we will prove the following theorem:

\begin{Theorem} \label{Thm:tame}
For any $A > 0$ and $\mathbf{p}_0 > 1$ there are $c = c(A) > 0$ and $Y = Y(A, \mathbf{p}_0) < \infty$ such that the following holds: 

Consider the sequence $\{ (M_i, g_i, q_i  ) \}_{i = 1}^\infty$ and the pointed limiting singular space $(X, d_X, \RR, g, q_\infty)$ from Theorem \ref{Thm:detailedconvergence}.
Assume that $\{ (M_i, g_i, q_i  ) \}_{i = 1}^\infty$ satisfies the properties (A)--(F) from subsection \ref{subsec:characterlimit} for the constants $A, \mathbf{p}_0$ and for some constant $T > 0$.
Then $\XX$ is $Y$-tame at scale $c\sqrt{T}$.
\end{Theorem}

We will split the proof of Theorem \ref{Thm:tame} into several parts, in which we establish the assumptions in Definition \ref{Def:Yregularity} of tameness separately.
In each part we consider the setting of Theorem \ref{Thm:tame} and, in particular, we fix a convergence scheme $\{ (U_i, V_i , \Phi_i \}_{i=1}^\infty$.

\begin{proof}[Proof of item (1) of the $Y$-tameness properties in Theorem \ref{Thm:tame}]
This is a direct consequence of property (A) from subsection \ref{subsec:characterlimit} and Theorem~\ref{Thm:detailedconvergence}(a).
\end{proof}

In order to show items (2) and (3) of the $Y$-tameness properties, we need the following lemma.

\begin{Lemma} \label{Lem:mutriangleb}
There is a constant $C^* < \infty$ and for any $\mathbf{p}_0 > 1$ and $A, E < \infty$ there is a constant $C = C(\mathbf{p}_0, A,  E) < \infty$ such that the following holds:

Let $(M, g )$ be a complete, $n$-dimensional Riemannian manifold with bounded curvature, $p \in M$ a point and $a, D > 0$ numbers and assume that
\begin{enumerate}[label=(\roman*)]
\item For all $x \in B(p,2D)$ and $0 < r < 2a$ we have $A^{-1} r^n < |B(x,r)| < A r^n$.
\item For any $x \in B(p,2D)$, $0 < r < 2a$ and $0 < s < 1$ we have
\[ |\{ \tdrrm < s r \} \cap B(x,r) | < E s^{\mathbf{p}_0} r^n. \]
\end{enumerate}

Set $b(x) := d (x, p)$ and denote by $\mu_{\triangle b}$ the signed measure from Proposition \ref{Prop:weakLaplaciandistance}.
Then we have the following estimates:
\begin{enumerate}[label=(\alph*)]
\item In the weak sense, we have for all $x \in M \setminus \{ p \}$
\begin{equation} \label{eq:upperLaplacebound}
 d \mu_{ \triangle b} (x) < C^* \big( \min\{ \rrm (x), b(x) \} \big)^{-1} dg. 
\end{equation}
\item For any $0 < r \leq a$ and $y \in B(p,D)$ such that $d (p,y) > 2 r$ we have
\begin{equation} \label{eq:rnminusoneE}
 \int_{B(y,r)} d |\mu_{\triangle b} | < C r^{n-1} 
\end{equation}
and for any $0 < s < 1$ we have
\begin{equation} \label{eq:normmuoverrrmsmall}
 \int_{B(y,r) \cap \{ \tdrrm  < sr \}} d |\mu_{\triangle b} | < C s^{\mathbf{p}_0 - 1} r^{n-1} 
\end{equation}
\item Assume that $n \geq 3$ and denote by $d\mu_{\triangle b^{2-n}}$ the weak Laplacian of $b^{2-n}$, as introduced in Proposition \ref{Prop:weakLaplaciandistance}.
Let $0 < r < a$ and consider an open subset $U \subset B(p,r) \setminus \{ p \}$ such that along any minimizing geodesic from $p$ to any point in $U$, we have $\Ric \geq - (n-1) \kappa$ for some $\kappa \geq 0$ with $\kappa r^2 \leq 100$.
Then
\begin{equation} \label{eq:laplaceb2minusn}
 d\mu_{\triangle b^{2-n}} \geq - \frac{C^* \kappa}{b^{n-2}} dg \textQQqq{on} U \setminus \{ p \}. 
\end{equation}
Moreover
\begin{equation} \label{eq:laplaceb2minusnintegral}
 \int_{U } d (\mu_{\triangle b^{2-n}})_- \leq  C \kappa r^2. 
\end{equation}
\end{enumerate}
\end{Lemma}

\begin{proof}
For assertion (a) we argue as follows:
By Proposition \ref{Prop:weakLaplaciandistance}, the positive part $(\mu_{\triangle b})_+$ is absolutely continuous with respect to $dg$ and by Proposition \ref{Prop:expincomplete}, the function $b$ is smooth almost everywhere.
So it suffices to check (\ref{eq:upperLaplacebound}) wherever $b$ is smooth.
Fix $x \in M \setminus \{ p \}$ such that $b$ is smooth at $x$, let $z \in M$ be a point on a minimizing geodesic $\gamma  : [0,1] \to M$ between $p, x$ such that $d ( x,z) = \frac1{10} \min \{ \rrm (x), b(x) \}$ and set $b' (x') := d (x',z) + d (z,p)$ for any $x' \in M$.
Then $x, z$ are not conjugate to each other along $\gamma$ and by Laplace comparison we have $\triangle b' (x) < C' d^{-1}(x,z)$ for some universal constant $C' < \infty$.
By the triangle inequality, we have $b' \geq b$ and $b'(x) = b(x)$.
So also $\triangle b (x) < C' d^{-1}(x,z)$, which proves our claim.

To see assertion (b), we first compute the integral of $d \mu_{\triangle}$ over a slightly larger ball.
Let $\varphi \in C^\infty_c (B(y, 1.5 r))$ be a cutoff function taking values between $0$ and $1$ such that $\varphi \equiv 1$ on $B(y,  r)$ and $|\nabla \varphi | < 10 r^{-1}$ everywhere.
Then, by Proposition~\ref{Prop:weakLaplaciandistance}, we get
\begin{equation} \label{eq:absofintegral}
 \bigg| \int_{B(y, 1.5 r)} \varphi d\mu_{\triangle b} \bigg| = \bigg| \int_{B(y, 1.5 r)} \nabla b \cdot \nabla \varphi \bigg| \leq 10 r^{-1} | B(y,  1.5 r) | \leq C  r^{n-1}, 
\end{equation}
for some generic constant $C = C(A) < \infty$.
Consider now the positive part $( \mu_{\triangle b})_+$ of the signed measure $\mu_{\triangle b}$.
Then we can estimate, using (\ref{eq:upperLaplacebound}) and the fact that $\rrm \geq \tdrrm$,
\begin{align*}
 \int_{B(y,1.5 r)} & d( \mu_{\triangle b} )_+ \\
&= \int_{B(y,1.5 r) \cap \{ \tdrrm  \geq r \} } d( \mu_{\triangle b} )_+  + \sum_{k=0}^{\infty} \int_{B(y,1.5 r) \cap \{ 2^{-k-1} r \leq \tdrrm  < 2^{-k} r \} } d( \mu_{\triangle b} )_+ \displaybreak[1] \\
&\leq | B(y,1.5 r) | \cdot C^* \big( \min \{ r, 0.5 r \} \big)^{-1} \\
&\quad+  \sum_{k=0}^{\infty} \big| B(y,1.5 r) \cap \{ \tdrrm   < 2^{-k} r \} \big|  \cdot C^*  \big( \min \{ 2^{-k} r, 0.5 r \} \big)^{-1}  \displaybreak[1] \\
&\leq C r^{n-1} + C \sum_{k = 0}^\infty E 2^{- \mathbf{p}_0 k} r^n \cdot 2^k r^{-1} \\
&< C (E+1) r^{n-1}.
\end{align*}
Observe that for the above bound it is essential that $\mathbf{p}_0 > 1$.
Combining this bound with (\ref{eq:absofintegral}) yields
\begin{multline*}
 \int_{B(y,  r)}  d |\mu_{\triangle b} | \leq \int_{B(y, 1.5 r)} \varphi d |\mu_{\triangle b}| \\
\leq \bigg| \int_{B(y, 1.5 r)} \varphi d\mu_{\triangle b} \bigg| +  \int_{B(y,1.5 r)}  d( \mu_{\triangle b} )_+ < C (E+1) r^{n-1},
\end{multline*}
where $C$ depends only on $A$.
This shows (\ref{eq:rnminusoneE}).

To see (\ref{eq:normmuoverrrmsmall}), we choose points $z_1, \ldots, z_N \in B(y,  r) \cap \{ \tdrrm < sr \}$ where $N$ is maximal with the property that the balls $B(z_1,  sr / 2), \ldots, B(z_N,  sr / 2)$ are pairwise disjoint.
Since $\tdrrm (\cdot)$ is $1$-Lipschitz, we have
\[ B(z_1,  sr/2) \cup \ldots \cup B(z_N,  sr / 2) \subset B(y,  2r) \cap \{ \tdrrm  < 2sr \}. \]
Using assumption (i) and
\[ | B(y,  2r) \cap \{ \tdrrm  < 2sr \} | < E s^{\mathbf{p}_0} (2r)^n, \]
we can find a constants $c = c (A) > 0$ and $C_0 = C_0 (A, E) < \infty$ such that
\[ N < \frac{E s^{\mathbf{p}_0} (2r)^n}{c(sr/2)^n} \leq C_0 s^{\mathbf{p}_0-n} .  \]
By the maximality of $N$, we have
\[ B(z_1,  sr) \cup \ldots \cup B(z_N,  sr ) \supset B(y,  r) \cap \{ \tdrrm < sr \}. \]
So by (\ref{eq:rnminusoneE})
\begin{multline*}
 \int_{B(y,r) \cap \{ \tdrrm  < sr \}} d |\mu_{\triangle b} | \leq \sum_{i=1}^N \int_{B(z_i,sr) } d |\mu_{\triangle b} | \\
  \leq C_0 s^{\mathbf{p}_0 - n}  \cdot C (sr)^{n-1} = C_0 C s^{\mathbf{p}_0 - 1} r^{n-1}. 
\end{multline*}
This shows (\ref{eq:normmuoverrrmsmall}) and hence assertion (b).

For assertion (c), let $\delta$ be a constant whose value we will determine later.
By Laplace comparison, we find that wherever $b$ is smooth, we have
\[ \triangle b \leq (n-1) \frac{\sn'_{-\kappa} (b)}{\sn_{-\kappa} (b)} \textQQqq{on} U. \]
Here $\sn_{-\kappa}$ is defined in (\ref{eq:sndef}).
So on $U$, and wherever $b$ is smooth, we have
\begin{multline*}
 \triangle b^{2-n} \geq (2-n)(1-n)  b^{-n} + (2-n)(n-1) \frac{\sn'_{-\kappa} (b)}{\sn_{-\kappa} (b)} b^{1-n} \\
 = (n-2) (n-1) b^{2-n} \cdot \frac{ \sn_{-\kappa} (b) - b \cdot \sn'_{-\kappa} (b) }{b^2 \cdot \sn_{-\kappa} (b) }. 
\end{multline*}
Note that
\begin{multline*}
 \frac{ \sn_{-\kappa} (b) - b \cdot \sn'_{-\kappa} (b) }{b^2 \cdot \sn_{-\kappa} (b) } = \frac{ \kappa^{-1/2} \sinh (\kappa^{1/2} b) - b \cdot \cosh (\kappa^{1/2} b) }{b^2 \cdot \kappa^{-1/2} \sinh (\kappa^{1/2} b) } \\
 = \kappa \cdot \frac{  \sinh (\kappa^{1/2} b) - \kappa^{1/2} b \cdot \cosh (\kappa^{1/2} b) }{(\kappa^{1/2} b)^2 \cdot \sinh (\kappa^{1/2} b) }
\end{multline*}
and that the function $\sinh x - x(\cosh x)$ has a zero of multiplicity $3$ at $x = 0$.
So, as long as $\kappa b^2 \leq 10$, we have
\[ \triangle b^{2-n} \geq - C^* \kappa b^{2-n} \]
for some universal $C^* < \infty$.
This shows (\ref{eq:laplaceb2minusn}).
For (\ref{eq:laplaceb2minusnintegral}), observe that, using assumption (i),
\begin{multline*}
 \int_U d (\mu_{\triangle b^{2-n}})_- \leq \int_{B(p, r)} \frac{C^* \kappa}{b^{n-2}} \\ = C^* \kappa  \bigg( \int_{B(p,r)} \frac{1}{r^{n-2}}  + \int_0^r \int_{B(p,s)} \frac{(n-2)}{s^{n-1}} dg ds \bigg) \\
 \leq C^* \kappa \bigg( A r^2 +  \int_0^r A s ds \bigg) \leq C(A) \kappa r^2.
\end{multline*}
This finishes the proof.
\end{proof}

\begin{proof}[Proof of item (2) of the $Y$-tameness properties in Theorem \ref{Thm:tame}]
Consider the convergence scheme $\{ (U_i, V_i, \Phi_i) \}_{i =1}^\infty$.
Let $p \in X$, $q \in \RR$, $0 < r < \min \{ d_X (p, q ), c \sqrt{T} \}$, where $c = c(A) > 0$ is a constant that will be determined in the course of the proof.
Define $b(x) := d_X (x,q)$.
Choose a sequence of points $p'_i \in U_i$ such that $p'_i \to p$ and set $p_i := \Phi_i (p'_i)$.
Set furthermore $q^*_i := \Phi_i (q) \in M_i$ for large $i$ and $b_i (x) := d_{M_i} (x, q^*_i)$.
Then $b_i \circ \Phi_i \to b$ as $i \to \infty$ uniformly on compact subsets of $\RR$.

Let $\Omega_i \subset M_i$ be smooth, connected domains such that
\begin{equation} \label{eq:choiceOmegai}
 B^{M_i} (p_i,  r- \eps_i) \subset \Omega_i \subset B^{M_i} (p_i,  r) 
\end{equation}
for some sequence $\eps_i \to 0$.
For each $i$ find harmonic functions $h_i \in C^\infty (\Int \Omega_i) \cap C^0 (\Omega_i)$ that solve the Dirichlet problem
\[ \triangle h_i = 0 \textQq{on} \Int \Omega_i  \textQQqq{and} h_i = b_i \textQq{on} \partial \Omega_i. \]
Note that
\begin{equation} \label{eq:oscofhi}
 \osc_{\Omega_i} h_i \leq \osc_{\partial \Omega_i} b_i \leq 2 r \textQQqq{and} \max_{\Omega_i} |h_i - b_i| \leq 2r. 
\end{equation}

We now claim that for large $i$
\begin{equation} \label{eq:nablahibi}
 \int_{\Omega_i} |\nabla (h_i - b_i) |^2 dg_i =  \int_{\Omega_i} (h_i - b_i) d \mu_{\triangle b_i} . 
\end{equation}
To see this, fix some large $i$ and observe that
\[ \int_{\Omega_i} |\nabla (h_i - b_i) |^2 dg_i \leq 2 \int_{\Omega_i} |\nabla h_i |^2 dg_i + 2 \int_{\Omega_i} |\nabla b_i |^2 dg_i < \infty. \]
So since $h_i - b_i$ vanishes on $\partial \Omega_i$, we can find a sequence $\varphi_j \in C^2_c (\Omega_i)$ of compactly supported smoothings of $(h_i - b_i)$ that converge to $h_i - b_i$ in $C^0$ and in $W^{1,2}$.
Then, using Proposition \ref{Prop:weakLaplaciandistance}, we have
\[
 \int_{\Omega_i} \nabla (h_i - b_i)  \cdot \nabla \varphi_j dg_i = \int_{\Omega_i} \nabla h_i \cdot \nabla \varphi_j dg_i - \int_{\Omega_i} \nabla b_i \cdot \nabla \varphi_j dg_i 
 =   \int_{\Omega_i} \varphi_j d\mu_{\triangle b_i} .
\]
Taking the limit $j \to \infty$ on both sides yields (\ref{eq:nablahibi}).
 
We will now estimate the $L^2$-norm of $h_i - b_i$ using the log-Sobolev inequality from property (D) of subsection \ref{subsec:characterlimit}, which states that for any $i$, any $0 < \tau \leq 2T$ and for any $f \in C^1 (M_i)$ with the property that
 \[ \int_{M_i} (4\pi \tau)^{-n/2} e^{-f} dg_i = 1 \]
 we have the bound
 \[  \int_{M_i} ( \tau |\nabla f|^2  + f ) (4 \pi \tau)^{-n/2} e^{-f} dg_i > - A. \]
So if $v \in C^1 (M_i)$ is a positive function with
\[ \int_{M_i} v^2 dg_i = 1, \]
then choosing $f = - \frac{n}2 \log (4 \pi \tau) - \log (v^2)$ yields
\[ \int_{M_i} \Big(4 \tau  |\nabla \log v |^2 - \frac{n}2 \log (4 \pi \tau) - \log (v^2) \Big) v^2 dg_i > - A. \]
So
\begin{equation} \label{eq:vlogSobolev}
 4 \tau  \int_{M_i}  |\nabla  v |^2 dg_i - \frac{n}2 \log (4 \pi \tau) - \int_{M_i} \log (v^2) v^2 dg_i > - A. 
\end{equation}
By continuity, this inequality also holds for the case in which $v$ does not have a sign and has regularity $W^{1,2} \cap L^\infty$.
Choose now
\[ v = (h_i - b_i) \bigg( \int_{\Omega_i} (h_i - b_i) \bigg)^{-1/2} \chi_{\Omega_i}. \]
Then $\int_{M_i} v^2 dg_i = 1$ and hence $v$ satisfies (\ref{eq:vlogSobolev}).
Moreover, by Jensen's inequality
\[ \int_{M_i} \log (v^2) v^2 dg_i = \int_{\Omega_i} \log (v^2) v^2 dg_i \geq |\Omega_i| \cdot \log \big( |\Omega_i|^{-1} \big) |\Omega_i|^{-1} = - \log |\Omega_i| . \]
So
\[ 4 \tau  \int_{\Omega_i}  |\nabla  v |^2 - \frac{n}2 \log (4 \pi \tau) + \log |\Omega_i| \geq - A. \]
Note that by property (A) from subsection \ref{subsec:characterlimit} we have $|\Omega_i| \leq |B^{M_i} (p_i, r)| < A r^n$ for large $i$.
So
\[ 4 \tau \int_{\Omega_i} |\nabla v|^2 \geq - A + \frac{n}2 \log (4 \pi \tau) - \log (A r^n). \]
Set now $c = c(A) := ( A^{2/n} \exp (2/n (A+1)) / 4 \pi )^{-1/2}$ and
\[ \tau := \frac{A^{2/n} \exp (2/n (A+1)) }{4\pi} \cdot r^2 \leq  T. \]
It follows that
\[ \frac{A^{2/n} \exp (2/n (A+1)) }{\pi} \cdot r^2 \int_{\Omega_i} |\nabla v|^2 \geq 1. \]
Plugging back the definition of $v$ gives us that for some constant $C_{*} = C_{*} (A) < \infty$
\begin{equation} \label{eq:Poincareineq}
 C_{*} r^2 \int_{\Omega_i} |\nabla (h_i - b_i)|^2 dg_i \geq  \int_{\Omega_i} (h_i - b_i)^2 dg_i 
\end{equation}

Let us summarize our results so far.
We have found that there is a constant $C = C(A) < \infty$ such that for large $i$
\begin{align}
 \int_{\Omega_i} |\nabla (h_i - b_i)|^2 dg_i &\leq \int_{\Omega_i} (h_i - b_i) d \mu_{\triangle b_i}, \label{eq:hibibeforelimit1} \\
 \int_{\Omega_i} (h_i - b_i)^2 dg_i &\leq C r^2 \int_{\Omega_i} |\nabla (h_i - b_i)|^2 dg_i.  \label{eq:hibibeforelimit2}
\end{align}
We will now pass to the limit as $i \to \infty$ to obtain the desired $h \in C^3 (B^X (p, r) \cap \RR)$.
To do this, observe first that $\Vert h_i \Vert_{L^\infty (\Omega_i)}$ is uniformly bounded.
So by local elliptic regularity an the local bounds on $\tdrrm$, the first, second, third and fourth derivative of the sequence $h_i \circ \Phi_i$ are locally uniformly bounded.
So, after passing to a subsequence, the sequence $h_i \circ \Phi_i$ converges to some harmonic $h \in C^3 (B^X (p, r) \cap \RR)$ in $C^3$ on compact subsets of $B^X (p, r) \cap \RR$.
It is clear that
\begin{equation} \label{eq:L2dropinlimit}
 \int_{B^X (p,r) \cap \RR} (h - b)^2 dg \leq \liminf_{i \to \infty} \int_{\Omega_i} (h_i - b_i)^2 dg_i. 
\end{equation}
Next, we show that
\begin{equation} \label{eq:nabladropinlimit}
 \int_{B^X (p,r) \cap \RR} |\nabla (h-b)|^2 dg \leq \liminf_{i \to \infty} \int_{\Omega_i} |\nabla (h_i - b_i)|^2 dg_i. 
\end{equation}
To see this, assume that the left-hand side of (\ref{eq:nabladropinlimit}) is positive, fix some small $\delta > 0$ and choose a vector field $Z \in C^2_c (B(p,r) \cap \RR)$, approximating $\nabla (h-b)$ well enough such that the following inequality holds
\begin{multline*}
 \int_{B^X (p,r) \cap \RR} \langle \nabla (h-b) , Z \rangle dg \\
 \geq (1-\delta) \bigg( \int_{B^X (p,r) \cap \RR} |\nabla (h-b)|^2 dg \bigg)^{1/2}  \bigg( \int_{B^X (p,r) \cap \RR} |Z|^2 dg \bigg)^{1/2}. 
\end{multline*}
It follows, using Proposition \ref{Prop:weakLaplaciandistance}, that
\begin{multline*}
 \lim_{i \to \infty} \int_{\Omega_i} \langle \nabla (h_i -b_i) , ((\Phi_i)_* Z) \rangle dg_i = - \lim_{i \to \infty} \int_{\Omega_i}  (h_i -b_i) \cdot \DIV_{g_i} ((\Phi_i)_* Z) dg_i \\
 = - \int_{B^X(p,r) \cap \RR} (h-b) \DIV_g Z dg = \int_{B^X (p,r) \cap \RR} \langle \nabla (h-b) , Z \rangle dg \\
 \geq (1-\delta) \bigg( \int_{B^X (p,r) \cap \RR} |\nabla (h-b)|^2 dg \bigg)^{1/2}  \bigg( \int_{B^X (p,r) \cap \RR} |Z|^2 dg \bigg)^{1/2}.
\end{multline*}
Bounding the first term from above by Cauchy-Schwarz yields
\begin{multline*}
 \liminf_{i \to \infty} \bigg( \int_{\Omega_i} |\nabla (h_i - b_i) |^2 dg_i \bigg)^{1/2} \bigg( \int_{\Omega_i} |(\Phi_i)_* Z |^2 dg_i \bigg)^{1/2}  \\
 \geq (1-\delta) \bigg( \int_{B^X (p,r) \cap \RR} |\nabla (h-b)|^2 dg \bigg)^{1/2}  \bigg( \int_{B^X (p,r) \cap \RR} |Z|^2 dg \bigg)^{1/2}. 
\end{multline*}
The second term on the left-hand side converges to the second term on the right-hand side.
So
\[ \liminf_{i \to \infty} \bigg( \int_{\Omega_i} |\nabla (h_i - b_i) |^2 dg_i \bigg)^{1/2} \geq (1- \delta) \bigg( \int_{B^X (p,r) \cap \RR} |\nabla (h-b)|^2 dg \bigg)^{1/2}. \]
Letting $\delta \to 0$ implies (\ref{eq:nabladropinlimit}).

Combining (\ref{eq:oscofhi}), (\ref{eq:hibibeforelimit1}), (\ref{eq:hibibeforelimit2}), (\ref{eq:L2dropinlimit}) and (\ref{eq:nabladropinlimit}) yields that
\[ | h - b | \leq 2 r, \]
\begin{align*} 
\int_{B^X (p,r) \cap \RR} |\nabla (h-b)|^2 dg & \leq C r^2 \liminf_{i \to \infty} \int_{\Omega_i} (h_i - b_i) d \mu_{\triangle b_i}, \displaybreak[1] \\
\int_{B^X (p, r) \cap \RR} (h-b)^2 dg &\leq  \liminf_{i \to \infty} \int_{\Omega_i} (h_i - b_i) d \mu_{\triangle b_i}  .
\end{align*}
So to finish the proof, it remains to show that
\begin{equation} \label{eq:rhsconvergence}
 \liminf_{i \to \infty} \int_{\Omega_i} (h_i - b_i) d \mu_{\triangle b_i} \leq 2r \int_{\ov{B^X(p,r)} \cap \RR} d |\mu_{\triangle b}|. 
\end{equation}
For the proof of (\ref{eq:rhsconvergence}) let $\sigma, \nu > 0$ be small constants and choose a compactly supported cutoff function $\eta \in C^2_c (B^X (p, r) \cap \RR)$ such that $0 \leq \eta \leq 1$ everywhere and
\[ \eta \equiv 1 \textQQqq{on} B^X(p,r - \nu) \cap \{ \rrm^\infty > \sigma \}. \]
For large enough $i$ we have $\supp \eta \subset U_i$.
For such $i$ set $\eta_i := \eta \circ \Phi_i^{-1} \in C^2_c (V_i) \subset C^2_c (M_i)$.

As a first step towards (\ref{eq:rhsconvergence}), we show that
\begin{equation} \label{eq:continuousmeasureconvergence}
 \lim_{i \to \infty}  \int_{\Omega_i} \eta_i (h_i - b_i) d\mu_{\triangle b_i} = \int_{B^X (p,r) \cap \RR} \eta (h-b) d\mu_{\triangle b} \leq 2r \int_{B^X (p,r) \cap \RR}  d |\mu_{\triangle b}|.
\end{equation}
To see this, fix a small $\delta > 0$ and let $\varphi \in C^2_c (\RR)$ be a smoothening of $\eta (h-b)$ such that $| \varphi - \eta (h-b) | < \delta$.
Then
\begin{equation} \label{eq:limvarphitriangle}
 \lim_{i \to \infty} \int_{M_i} (\varphi \circ \Phi_i^{-1}) d\mu_{\triangle b_i} = \lim_{i \to \infty}  \int_{M_i} \triangle (\varphi \circ \Phi_i^{-1}) \cdot b_i dg_i = \int_{\RR} \triangle \varphi \cdot b dg = \int_{\RR} \varphi d\mu_{\triangle b}. 
\end{equation}
By Lemma \ref{Lem:mutriangleb}(b) there is a constant $C''_1 < \infty$, which does not depend on $i$, such that for large $i$
\begin{equation} \label{eq:sideone1}
 \bigg| \int_{\Omega_i} \eta_i (h_i - b_i) d\mu_{\triangle b_i} - \int_{\Omega_i} (\varphi \circ \Phi_i^{-1})  d\mu_{\triangle b_i} \bigg| \leq \delta \int_{B^{M_i}(p_i, r)} d|\mu_{\triangle b_i}| < C''_1 \delta.  
\end{equation}
Furthermore, since $\supp (\eta(h-b) - \varphi)$ is compact and $\mu_{\triangle b}$ has locally finite total variation, there is a $C''_2 < \infty$ such that
\begin{equation} \label{eq:sideone2}
 \bigg| \int_{\RR} \eta (h - b) d\mu_{\triangle b} - \int_\RR \varphi d\mu_{\triangle b} \bigg| \leq \delta \int_{\supp (\eta (h-b) - \varphi)} d |\mu_{\triangle b}| < C''_2 \delta. 
\end{equation}
Combining (\ref{eq:limvarphitriangle}), (\ref{eq:sideone1}) and (\ref{eq:sideone2}) and letting $\delta \to 0$ gives us (\ref{eq:continuousmeasureconvergence}).

We will now bound
\[ \limsup_{i \to \infty} \int_{\Omega_i} (1 - \eta_i) (h_i - b_i) d\mu_{\triangle b_i}. \]
To do this, observe that by Theorem \ref{Thm:detailedconvergence}(c), we have for large $i$ that
\begin{multline*}
  \Omega_i \cap \supp (1-\eta_i)  \subset \big( B^{M_i} (p_i, r) \cap \{ \tdrrm^{M_i}  < 2\sigma \} \big) \\ \cup \big( A^{M_i} (p_i,r- 2\nu, r) \cap \{ \tdrrm^{M_i}  \geq 2 \sigma \} \big). 
\end{multline*}
Here
\[ A^{M_i} (p_i,r- 2\nu, r) := \big\{ x \in M_i \;\; : \;\; r - 2 \nu < d_{M_i} (p_i, x) < r \big\}. \]
So, since by (\ref{eq:oscofhi}) we have $| h_i - b_i | \leq 2 r$ on $\Omega_i$, we find that for large $i$
\begin{multline} \label{eq:splitotherpart}
 \int_{\Omega_i} (1 - \eta_i) (h_i - b_i) d\mu_{\triangle b_i} \leq 2r \bigg( \int_{B^{M_i} (p_i, r) \cap \{ \tdrrm^{M_i} < 2\sigma \} } d |\mu_{\triangle b_i} | \\ + \int_{A^{M_i} (p_i,r-2\nu, r) \cap \{ \tdrrm^{M_i}  \geq 2 \sigma \} } d |\mu_{\triangle b_i} | \bigg). 
\end{multline}
The first term can be bounded using Lemma \ref{Lem:mutriangleb}(b):
There is a constant $C''_3 < \infty$, which may depend on $r$ but not on $i$, such that
\begin{equation} \label{eq:estimateonrrmsmall}
  \int_{B^{M_i} (p_i, r) \cap \{ \tdrrm^{M_i} < 2\sigma \} } d |\mu_{\triangle b_i} | < C''_3 \sigma^{\mathbf{p}_0 - 1}. 
\end{equation}
In order to bound the second term on the right-hand side in (\ref{eq:splitotherpart}), let $\psi \in C^2_c ( A^X (p,  r - 4 \nu, r + 2\nu ) \cap \RR )$ be a cutoff function such that $0 \leq \psi \leq 1$ everywhere and
\[ \psi \equiv 1 \textQQqq{on} A^X (p, r - 3 \nu, r + \nu ) \cap \{ \rrm^\infty > \sigma \} \]
and set $\psi_i := \psi \circ \Phi_i^{-1} \in C^2_c(V_i) \subset C^2_c (M_i)$ for large enough $i$.
Then, using Theorem \ref{Thm:detailedconvergence}(b), we obtain that for large $i$ we have $\psi_i \equiv 1$ on $A^{M_i} (p_i,  r - 2\nu, r ) \cap \{ \tdrrm^{M_i}  \geq 2 \sigma \}$ and $\supp \psi_i \subset A^{M_i} (p_i,  r - 5\nu, r + 3\nu ) \cap \{ \rrm^{M_i} \geq \sigma / 2 \}$.
Since $\rrm^{M_i}  \geq \sigma / 2$ on $\supp \psi_i$ for large $i$, we may use Lemma~\ref{Lem:mutriangleb}(a) and conclude that there is a constant $C''_4 < \infty$, which does not depend on $i$, such for large $i$ we have in the weak sense
\[ d |\mu_{\triangle b_i}| \leq 2d (\mu_{\triangle b_i})_+ - d \mu_{\triangle b_i} \leq C''_4 \sigma^{-1} dg^i_0 - d \mu_{\triangle b_i}. \]
This implies that for large $i$
\begin{multline*}
 \int_{A^{M_i} (p_i,r-2\nu, r) \cap \{ \tdrrm^{M_i}  \geq 2 \sigma \} } d |\mu_{\triangle b_i} | 
\leq \int_{M_i } \psi_i d |\mu_{\triangle b_i} | \\
\leq
 C''_4 \sigma^{-1} | \supp \psi_i  |  - \int_{M_i } \psi_i d \mu_{\triangle b_i}   
 \leq  C''_4 \sigma^{-1} | \supp \psi_i  | - \int_{M_i} \triangle \psi_i \cdot b_i dg_i .
\end{multline*}
Taking this inequality to the limit yields
\begin{multline} \label{eq:annuluspsi}
 \limsup_{i \to \infty}  \int_{A^{M_i} (p_i,r-2\nu, r) \cap \{ \tdrrm^{M_i}  \geq 2 \sigma \} } d |\mu_{\triangle b_i} | \\ \leq C''_4 \sigma^{-1} |\supp \psi | -  \int_{\RR} \triangle \psi \cdot b dg = C''_4 \sigma^{-1} |\supp \psi | - \int_{\RR} \psi d\mu_{\triangle b}.
\end{multline}
Combining (\ref{eq:splitotherpart}), (\ref{eq:estimateonrrmsmall}) and (\ref{eq:annuluspsi}) yields
\begin{multline*}
 \limsup_{i \to \infty}  \int_{\Omega_i} (1 - \eta_i) (h_i - b_i) d\mu_{\triangle b_i} < 2 C''_3 r \sigma^{\mathbf{p}_0-1} \\ +  2 C''_4 r \sigma^{-1} \big| A^X (p, r - 4 \nu, r + 2 \nu)  \cap \RR \big| 
 + 2r \int_{A^{X} (p, r - 4 \nu, r + 2 \nu) \cap \RR} d |\mu_{\triangle b} |. 
\end{multline*}
Combining this inequality with (\ref{eq:continuousmeasureconvergence}) yields
\begin{multline} \label{eq:limsupAXrnu}
 \limsup_{i \to \infty} \int_{\Omega_i} (h_i - b_i) d\mu_{\triangle b_i} \leq 2r \int_{B^X (p, r) \cap \RR}  d|\mu_{\triangle b}| + 2 C''_3 r \sigma^{\mathbf{p}_0 - 1} \\ + 2 C''_4 r \sigma^{-1} \big| A^X (p, r - 4 \nu, r + 2 \nu)  \cap \RR \big| 
 + 2r \int_{A^{X} (p, r - 4 \nu, r + 2 \nu) \cap \RR} d |\mu_{\triangle b} |  .
\end{multline}
Since $\XX$ has mild singularities (in the sense of Definition \ref{Def:mild}) by Theorem \ref{Thm:detailedconvergence}, we have by Proposition \ref{Prop:expincomplete}
\begin{multline*}
 \big| A^X (p, r - 4 \nu, r + 2 \nu)  \cap \RR \big| = \big| A^X (p, r - 4 \nu, r + 2 \nu)  \cap (\RR \setminus Q_p) \big| \\
 \leq \int_{A^{T_p \RR} (0_p, r-4 \nu, r + 2 \nu) \cap \mathcal{G}_p} J_p (v) dv. 
\end{multline*}
It follows that
\[ \lim_{\nu \to 0}  \big| A^X (p, r - 4 \nu, r + 2 \nu)  \cap \RR \big| = 0. \]
So letting $\nu \to 0$ in (\ref{eq:limsupAXrnu}) gives us
\begin{multline*}
 \limsup_{i \to \infty}  \int_{\Omega_i}  (h_i - b_i) d\mu_{\triangle b_i} \leq 2r \int_{B^X (p, r) \cap \RR} d |\mu_{\triangle b}| + 2r \int_{\partial B^X (p, r) \cap \RR} d |\mu_{\triangle b}| \\
 +2 C''_3r \sigma^{\mathbf{p}_0-1} . 
\end{multline*}
Finally, letting $\sigma \to 0$ yields (\ref{eq:rhsconvergence}) and hence finishes the proof.
\end{proof}

As a preparation for part (3) of the tameness properties, we prove

\begin{Lemma} \label{Lem:BGmutriangleb}
For any $A < \infty$ there is a constant $C(A) < \infty$ such that the following holds:

Let $(M,g)$ be a complete Riemannian manifold of dimension $n \geq 3$, let $p \in M$ be a point, $0 < r_1 < r_2$ and $0 \leq \kappa \leq r_1^{-2}$.
Assume that $|B(p, r_1 ) | \leq A r_1^n$.

Then for $b (x) := d^{2-n}(x,p)$ we have
\[
 \int_{A(p,r_1, r_2)} d |\mu_{\triangle b} |  \leq  C \bigg( \frac{|B(p,r_1)|}{v_{-\kappa}(r_1)} - \frac{|A(p,r_2,  2r_2)|}{v_{-\kappa} (2r_2) - v_{-\kappa}(r_2)} + \kappa r_1^2 \bigg) 
 + 2 \int_{B(p,2r_2) \setminus \{p \}} d (\mu_{\triangle b}  )_-. 
\]
\end{Lemma}

\begin{proof}
Recall that
\begin{equation} \label{eq:vkformula}
 v_{-\kappa} (r) = n \omega_n \int_0^r \big({ \sn_{-\kappa} (t) }\big)^{n-1} dt. 
\end{equation}
Let
\[ H(r) := \begin{cases} \frac{v_{-\kappa} (r)}{v_{-\kappa} (r_1)} & \text{if $0 \leq r \leq r_1$} \\ 1 & \text{if $r_1 \leq r \leq r_2$} \\ \frac{v_{-\kappa} (2r_2) - v_{-\kappa} (r)  }{ v_{-\kappa} (2r_2) - v_{-\kappa} (r_2) }  & \text{if $r_2 \leq r \leq 2 r_2$} \\ 0 & \text{if $2r_2 \leq r$} \end{cases} \]
and note that $H : [0, \infty) \to [0,1]$ is continuous.
Define $f (x) := H(d(x,p))$.
We can then compute that, using (\ref{eq:vkformula}) and Proposition \ref{Prop:weakLaplaciandistance},
\begin{align*}
 \int_{B(p, 2 r_2)} f  d \mu_{\triangle b} &= - \int_{B(p, 2 r_2)} \langle \nabla f , \nabla b \rangle dg \\
 &=  \int_{B(p, 2r_2)} \frac{n-2}{(d(x,p))^{n-1}} \cdot H' (d(x,p)) dg(x) \\
 &= \frac{n(n-2) \omega_n}{v_{-\kappa} (r_1)} \int_{B(p,r_1)} \frac{(\sn_{-\kappa} (d(x,p)))^{n-1}}{(d(x,p))^{n-1}} dg(x) \\
 &\qquad
 - \frac{n(n-2)\omega_n}{v_{-\kappa} (2r_2) - v_{-\kappa} (r_2)} \int_{A(p, r_2, 2r_2)}  \frac{(\sn_{-\kappa} (d(x,p)))^{n-1}}{(d(x,p))^{n-1}} dg(x)  
\end{align*}
Now note that $\sn_{-\kappa} (r) \geq r$ and that there is some $C' < \infty$ such that for all $0 < r \leq r_1$ we have
\[ \frac{\sn_{-\kappa} (r)}{r} \leq \frac{\kappa^{-1/2} \sinh (\kappa^{1/2} r)}{r} \leq \frac{r + C' \kappa r^3 }{r} = 1 + C' \kappa r^2. \]
Therefore, using the bound $|B(p,r_1)| \leq 2 v_{-\kappa} (r_1)$,
\begin{multline*}
  \int_{B(p, 2 r_2)} f  d \mu_{\triangle b}  \leq n (n-2) \omega_n \bigg( (1 + C' \kappa r_1^2 ) \frac{|B(p,r_1)|}{v_{-\kappa} (r_1)} - \frac{|A(p, r_2, 2r_2)|}{v_{-\kappa} (2r_2) - v_{-\kappa} (r_2)} \bigg) \\
  \leq n (n-2) \omega_n \bigg(  \frac{A}{v_{-1} (1)} \cdot \kappa r_1^2 + \kappa r_1^2 + \frac{|B(p,r_1)|}{v_{-\kappa} (r_1)} - \frac{|A(p, r_2, 2r_2)|}{v_{-\kappa} (2r_2) - v_{-\kappa} (r_2)} \bigg)
\end{multline*}
On the other hand
\[  \int_{A(p, r_1, r_2)}  d |\mu_{ \triangle b}|  \leq \int_{B(p, 2 r_2) \setminus \{ p \}} f d \mu_{ \triangle b} + 2\int_{B(p,2r_2) \setminus \{ p \}} d ( \mu_{ \triangle b})_-. \]
This concludes the proof.
\end{proof}

\begin{proof}[Proof of item (2) of the $Y$-tameness properties in Theorem \ref{Thm:tame}] 
Fix $A, T > 0$ and $\mathbf{p}_0 > 1$ and assume that $n \geq 3$, $0 < r_1 < r_2 < c \sqrt{T}$ for some $c = c(A) > 0$ whose value we will determine in the course of this proof.
Let $p \in \RR$ and assume moreover that $\Ric \geq - (n-1) \kappa$ on $B^X (p, 4 r_2) \cap \RR$ for some $0 \leq \kappa \leq r_2^{-2}$.

Set $p_i := \Phi_i (p) \in M_i$ for large $i$.
Moreover, we let $b(x) := (d_X (x,p) )^{2-n}$ and $b_i (x) := (d_{M_i} (x, p_i))^{2-n}$.
Then we have uniform convergence $b_i \circ \Phi_i \to b$ on compact subsets of $\RR \setminus \{ p \}$.
By assertion (a) of Theorem \ref{Thm:detailedconvergence} we have
\begin{multline*}
 \limsup_{i \to \infty}  \bigg( \frac{|B^{M_i} (p_i,  r_1)|}{v_{-\kappa} (r_1)} - \frac{|A^{M_i} (p_i,  r_2, 2r_2)|}{v_{-\kappa} (2r_2) - v_{-\kappa} (r_2)} \bigg) \\
 = \frac{|\ov{B^X (p, r_1)} \cap \RR|}{v_{-\kappa} (r_1)} - \frac{|A^X (p, r_2, 2r_2) \cap \RR|}{v_{-\kappa} (2r_2) - v_{-\kappa} (r_2)}  =: H.
\end{multline*}
Using Lemma \ref{Lem:BGmutriangleb}, we conclude that
\begin{equation} \label{eq:totalmassbeforelimit}
 \limsup_{i \to \infty} \int_{A^{M_i} (p_i,  r_1, r_2)} d |\mu_{\triangle b_i}| \leq C (H+\kappa r_1^2) + 2  \limsup_{i \to \infty} \int_{B^{M_i} (p_i,  2r_2) \setminus \{ p_i \}} d (\mu_{\triangle b_i} )_-. 
\end{equation}

We will now show that the limit on the right-hand side is bounded by $C \kappa r_2^2$ for some constant $C = C(A, \mathbf{p}_0 ) < \infty$.
To see this, fix some compact subset $K \subset  (B^X(p, 3r_2) \cap \RR) \setminus Q_p$.
Then $K \subset U_i$ for large $i$.
We first show:

\begin{Claim}
There is a constant $\sigma = \sigma(K) > 0$ such that for large $i$ the following is true:
For any $x  \in K$, the image of any minimizing geodesic $\gamma : [0,l] \to M_i$ between $p_i$ and $\Phi_i (x)$ is contained in $\{ \tdrrm^{M_i}  > \sigma \} \cap B^{M_i} (p_i, 3 r_2)$.
\end{Claim}

\begin{proof}
Note first, that by Theorem \ref{Thm:detailedconvergence}(c) there is a sequence $\sigma_i \to 0$ such that for large $i$
\[ B^{M_i} (p_i,  3r_2) \cap \{ \tdrrm^{M_i} \geq \sigma_i \} \subset V_i  \]
Fix this sequence for the rest of the proof.
For large $i$ we have $\Phi_i (K) \subset B^{M_i} (p_i,  3r_2)$, so any minimizing geodesic between $p_i$ and a point in $\Phi_i (K)$ is contained in $B^{M_i} (p_i,  3r_2)$.

Assume now that the claim was wrong.
Then, after passing to a subsequence, we can find a sequence $x_i \in K$ of points and a sequence of minimizing arclength geodesics $\gamma_i : [0,l_i] \to M_i$ between $p_i$ and $\Phi_i (x_i)$ that are each not fully contained in $\{ \tdrrm^{M_i} > \sigma_i \}$.
By passing to a subsequence, we may assume that $x_i \to x_\infty \in K$ and that $l_\infty := \lim_{i \to \infty} l_i$ exists.
Let $0 < a_i \leq l_i$ each be maximal with the property that we have $\gamma_i ([l_i-a_i, l_i]) \subset \{ \tdrrm^{M_i}  \geq \sigma_i \}$.
So $\tdrrm^{M_i} (\gamma_i (l_i - a)) = \sigma_i$.
Since $\tdrrm^{M_i} (\cdot)$ is $1$-Lipschitz, it follows that for any $\eps > 0$
\begin{equation} \label{eq:rrmtozeroeps}
 \limsup_{i \to \infty} \tdrrm^{M_i} ( \gamma_{i} (l_i - a_i + \eps)) \leq \lim_{i \to \infty} \sigma_i + \eps = \eps. 
\end{equation}
After passing to another subsequence, we may assume that $a_\infty := \lim_{i \to \infty} a_i \in [0, l_\infty]$ exists and that we have uniform convergence of the curves $\Phi_i^{-1} \circ \gamma_i |_{[l_i - a_i , l_i]}$ to some continuous curve $\gamma_\infty : (l_\infty - a_\infty, l_\infty] \to \RR$ with $\gamma_\infty (l_\infty) = x_\infty$.
Note that $a_\infty > 0$ since by Theorem \ref{Thm:detailedconvergence}(b)
\begin{multline*}
 \liminf_{i \to \infty} \tdrrm^{M_i} (\Phi_i (x_i)) \geq \liminf_{i \to \infty} \big( \tdrrm^{M_i} (\Phi_i (x_\infty)) - d_{M_i} (\Phi_i (x_i), \Phi_i (x_\infty)) \big) \\
 =  \liminf_{i \to \infty}  \tdrrm^{M_i} (\Phi_i (x_\infty)) > 0. 
\end{multline*}
For any $s_1, s_2 \in (l_\infty - a_\infty, l_\infty]$, we have
\[ d_X (\gamma_\infty (s_1), \gamma_\infty (s_2)) = \lim_{i \to \infty} d_{M_i} (\gamma_i (s_1), \gamma_i (s_2)) = |s_1 - s_2|. \]
So $\gamma_\infty$ is a minimizing geodesic.
Similarly, for any $s \in (l_\infty - a_\infty , l_\infty]$, we have
\[ d_X (p, \gamma_\infty (s)) + d_X (\gamma_\infty (s), x_\infty) = d(p, x_\infty). \]

Next, we show that $\lim_{s \searrow l_\infty - a_\infty} \gamma_\infty (s) \not\in \RR$ and thus $a_\infty < l_\infty$.
Assume that it did and denote this limit by $z \in \RR$.
Then, using Theorem \ref{Thm:detailedconvergence}(b) and (\ref{eq:rrmtozeroeps}), we get that for all $\eps > 0$
\begin{multline*}
 0 < \limsup_{i \to \infty} \tdrrm^{M_i} (\Phi_i (z)) \\ \leq 
 \limsup_{i \to \infty} \big( \tdrrm^{M_i} (\gamma_i (l_i - a_i + \eps)) + d_{M_i} (\gamma_i (l_i - a_i + \eps), \Phi_i (z)) \big) 
 \leq 2\eps .
\end{multline*}
This gives us the desired contradiction.

Since $x_\infty \in K$ is not in $Q_p \cup \{ p \}$, and $Q_p \cup \{ p \}$ is closed, we can pick some $0 < b < a_\infty$ such that $\gamma_\infty (l_\infty - b) \in \RR \setminus (Q_p \cup \{ p \})$.
So there is a minimizing arclength geodesic $\gamma^* : [0, l_\infty - b] \to \RR$ between $p$ and $\gamma_\infty (l_\infty - b)$ whose image is contained in $\RR$.
It follows that
\[ \ell_g (\gamma^*) + \ell_g (\gamma_\infty |_{[l_\infty - b, l_\infty]}) = d_X (p, \gamma_\infty (l_\infty - b)) + d_X ( \gamma_\infty (l_\infty - b), x_\infty ) = d_X (p, x_\infty). \]
So the concatenation of $\gamma^*$ and $\gamma_\infty |_{[l_\infty - b, l_\infty]}$ is a minimizing geodesic between $p$ and $x_\infty$ and hence smooth.
It follows that $\gamma^* = \gamma_\infty$ on $(l_\infty - a_\infty, l_\infty - b]$, which contradicts $\lim_{s \searrow l_\infty - a_\infty} \gamma_\infty (s) \not\in \RR$.
\end{proof}

Fix now $\sigma > 0$ from the Claim and let $\kappa' > \kappa$ be some constant that is slightly larger than $\kappa$.
Then, since by assumption we have $\Ric \geq - (n-1) \kappa$ on $B(p, 4r_2) \cap \RR$, we must have $\Ric \geq -(n-1) \kappa'$ on $\{ \tdrrm^{M_i} > \sigma \} \cap B^X (p, 3r_2)$ for large $i$.
By the Claim, property (A) from subsection \ref{subsec:characterlimit} and Lemma \ref{Lem:mutriangleb}(c), this implies that for some $C = C(A) < \infty$ and large $i$ we have
\[ \int_{\Phi_i (K) \setminus \{ p_i \}} d (\mu_{\triangle b_i} )_- \leq C \kappa' (3r_2)^2. \]
This implies
\begin{equation} \label{eq:convto0onW}
 \limsup_{i \to \infty} \int_{\Phi_i(K) \setminus \{ p_i \}} d(\mu_{\triangle b_i} )_- \leq C \kappa r_2^2. 
\end{equation}

Next we will show that
\[ \limsup_{i \to \infty} \int_{B^{M_i} (p_i, 2r_2) \setminus \Phi_i (K)} d(\mu_{\triangle b_i} )_- \]
can be made arbitrarily small if $K$ is chosen to approximate $(B^X (p, 3r_2) \cap \RR) \setminus Q_p$ well enough.
To see this, observe first that wherever $d_{M_i} (\cdot, p_i)$ is $C^2$, we have
\[  \triangle b_i  = (n-2)(n-1) \big( d_{M_i} (\cdot, p_i) \big)^{-n} - (n-2) \triangle d_{M_i} (\cdot, p_i) \big( d_{M_i} (\cdot, p_i) \big)^{1-n}. \]
So, as $d(\mu_{\triangle b_i})_-$ is absolutely continuous with respect to $dg_i$, we have
\begin{equation} \label{eq:dmutrianglebinminus2}
 d (\mu_{\triangle b_i})_- \leq (n-2) \big( d_{M_i} (\cdot, p_i) \big)^{1-n} d(\mu_{\triangle d_{M_i} (\cdot, p_i)})_+. 
\end{equation}
Using Lemma \ref{Lem:mutriangleb}(b), we can conclude from this that there is some constant $C'_1 < \infty$, which is independent of $i$, such that for all $\nu, \sigma_0 > 0$ and all $i$ we have
\begin{equation} \label{eq:smallArrm1}
 \int_{A^{M_i} (p_i,  \nu, 2r_2 ) \cap \{ \tdrrm^{M_i}  \leq \sigma_0 \}} d (\mu_{\triangle b_i})_- \leq C'_1 \nu^{1-n} \sigma_0^{\mathbf{p}_0-1}. 
\end{equation}
On the other hand, (\ref{eq:dmutrianglebinminus2}) and the pointwise bound of Lemma \ref{Lem:mutriangleb}(a) implies that there is a constant $C'_2 < \infty$, which is independent of $K$ and $i$, such that
\begin{multline} \label{eq:smallArrm2}
 \int_{(A^{M_i} (p_i,  \nu, 2r_2) \cap \{ \tdrrm^{M_i} > \sigma_0 \} ) \setminus \Phi_i (K)} d (\mu_{\triangle b_i} )_- \\ \leq C'_2 \nu^{1-n}  (\nu^{-1} + \sigma_0^{-1}) \cdot \big| \big( A^{M_i} (p_i, \nu, 2r_2) \cap \{ \tdrrm^{M_i}  > \sigma_0 \} \big) \setminus \Phi_i (K) \big|. 
\end{multline}
Choose now $\nu > 0$ small enough such that $B^X (p, 3\nu) \subset (B^X (p, 3r_2) \cap \RR) \setminus Q_p$.
So $B^{M_i} (p_i, 0, 2\nu) \subset \Phi_i (B^X (p, 3\nu))$ for large $i$.
Combining (\ref{eq:smallArrm1}) and (\ref{eq:smallArrm2}), we find that for any compact subset
\[ B^X (p, 3\nu) \subset K \subset B^X (p, 3r_2) \cap \RR \]
we have
\begin{multline*}
 \limsup_{i \to \infty} \int_{B^{M_i} (p_i, 2r_2) \setminus \Phi_i (K)} d(\mu_{\triangle b_i} )_- \\
  \leq \limsup_{i \to \infty}  C'_2 \nu^{1-n}  (\nu^{-1} + \sigma_0^{-1}) \cdot \big| \big( A^{M_i} (p_i, \nu, 2r_2) \cap \{ \tdrrm^{M_i}  > \sigma_0 \} \big) \setminus \Phi_i (K) \big| \\
 \leq C'_2 \nu^{1-n}  (\nu^{-1} + \sigma_0^{-1}) \cdot \big| \big( B^X (p_i, \nu, 3r_2) \cap \RR \big) \setminus K \big|.
\end{multline*}
Combining this with (\ref{eq:convto0onW}) yields
\begin{multline} \label{eq:dmutrianglebminusbound}
 \limsup_{i \to \infty} \int_{B^{M_i} (p_i,  2r_2) \setminus \{p_i \}} d (\mu_{\triangle b_i})_- \\
 \leq C \kappa r_2^2 + C'_1 \nu^{1-n} \sigma_0^{\mathbf{p}_0-1} + C'_2 \nu^{1-n}  (\nu^{-1} 
 + \sigma_0^{-1}) \big| \big( B^X (p, 3r_2) \cap \RR \big) \setminus K \big|.
\end{multline}
We now claim that we can make the sum of the last two terms on the right-hand side of this inequality arbitrarily small.
To see this, choose $\sigma_0 > 0$ such that $C'_1 \nu^{1-n} \sigma_0^{\mathbf{p}_0-1}$ is small.
Next, since $K \subset  (B^X(p, 4r_2) \cap \RR) \setminus Q_p$ can be chosen arbitrarily and $Q_p$ has measure $0$, we may choose $K$ such that $B^X (p, 0, 3\nu ) \subset K$ and such that $| ( B^X (p, 3r_2) \cap \RR ) \setminus K |$ is arbitrarily small.
This makes the second term on the right-hand side of (\ref{eq:dmutrianglebminusbound}) small.
So we conclude that
\[
 \limsup_{i \to \infty} \int_{B^{M_i} (p_i, 3r_2)} d (\mu_{\triangle b_i} )_- \leq C \kappa r_2^2
\]
and thus, by (\ref{eq:totalmassbeforelimit}), we have
\begin{equation} \label{eq:totmassboundbeforelimitbetter}
 \limsup_{i \to \infty} \int_{A^{M_i} (p_i, 0, r_1, r_2)} d |\mu_{\triangle b_i}| \leq C ( \kappa r_2^2 + H) . 
\end{equation}

The rest of the proof now follows along the lines of the proof of item (2) of the tameness assumptions.
We will now go through the main steps of this proof and point out the major differences.
Similarly to (\ref{eq:choiceOmegai}) let $\Omega_i \subset M_i$ be smooth and connected domains such that for some $\eps_i \to 0$ we have
\[ A^{M_i} (p_i,  r_1 + \eps_i, r_2 - \eps_i) \subset \Omega_i \subset A^{M_i} (p_i,  r_1 - \eps_i, r_2 + \eps_i) \]
and let $h_i \in C^\infty (\Int \Omega_i) \cap C^0 (\Omega_i)$ be solutions to the Dirichlet problem
\[ \triangle h_i = 0 \textQq{on} \Omega_i  \textQQqq{and} h_i = b_i \textQq{on} \partial \Omega_i. \]
Then, as in (\ref{eq:oscofhi}),
\[ (r_2 + \eps_i) ^{2-n} \leq \min_{\partial \Omega_i} b_i \leq \min_{\Omega_i} h_i \leq \max_{\Omega_i} h_i \leq \max_{\partial \Omega_i} b_i \leq  (r_1 - \eps_i)^{2-n}. \]
As in the proof of (\ref{eq:nablahibi}) we find that
\begin{multline} \label{eq:nablahibiagain}
 \int_{\Omega_i} |\nabla (h_i - b_i)|^2 dg_i = \int_{\Omega_i} (h_i - b_i ) d\mu_{\triangle b_i} \\
 \leq 2 \big( (r_2- \eps_i)^{2-n} - (r_1+ \eps_i)^{2-n} \big) \int_{\Omega_i} d |\mu_{\triangle b_i}|. 
\end{multline}
Also, as in the proof of (\ref{eq:Poincareineq}), we find that there is a constant $C_{**} = C_{**} (A) < \infty$ such that for sufficiently small $c = c(A) > 0$
\begin{equation} \label{eq:hibiagain}
 \int_{\Omega_i} ( h_i - b_i )^2 dg_i \leq C_{**} r_2^2 \int_{\Omega_i} |\nabla (h_i - b_i)|^2 dg_i. 
\end{equation}
The harmonic function $h \in C^3 ( A^X (p, r_1, r_2) \cap \RR)$ can now be constructed by passing to the limit, as in the proof of part (2) of the $Y$-tameness properties.
The bounds (\ref{eq:part2Ytame1}) and (\ref{eq:part2Ytame2}) follow from (\ref{eq:nablahibiagain}), (\ref{eq:hibiagain}) via (\ref{eq:totmassboundbeforelimitbetter}) using a similar limit argument as in that proof.
Note that (\ref{eq:rhsconvergence}) in this proof is replaced by (\ref{eq:totmassboundbeforelimitbetter}).
\end{proof}

\begin{proof}[Proof of item (4) of the $Y$-tameness properties in Theorem \ref{Thm:tame}.]
We first estab\-lish the following improvement of property (F) from subsection \ref{subsec:characterlimit}:
There is a constant $C = C(A) < \infty$ such that for all $D < \infty$ and sufficiently large $i$ (depending on $D$) the following holds:
For all $0 < r < \sqrt{T}$ and all $x \in B^{M_i}(q_i,D)$ there is a compactly supported function $\phi \in C^\infty_c (B^{M_i} (x, 0.9 r))$ that only takes values in $[0,1]$, that satisfies $\phi \equiv 1$ on $B^{M_i} (x, 0.6 r)$ and that satisfies the bounds
\[ |\nabla \phi | < C r^{-1} \textQQqq{and} |\triangle \phi | < C r^{-2}. \]
To see this, assume without loss of generality that $A > 10$ and choose a maximal number of points $z_1, \ldots, z_N \in B^{M_i} (x, 0.6r)$ such that the balls $B^{M_i} (z_1, \frac1{8} r /  A), \linebreak[1] \ldots, \linebreak[1] B^{M_i} (z_N, \tfrac1{8} r / A)$ are pairwise disjoint.
Then, using property (A) of subsection~\ref{subsec:characterlimit} for sufficiently large $i$, we have
\[ N \leq \frac{|B^{M_i} (x,  r)|}{A^{-1} (\tfrac18 r /A)^n} \leq  \frac{A}{A^{-1} (\tfrac18 / A)^n}. \]
Moreover, we have
\[ B^{M_i} (x, 0.6 r) \subset B^{M_i} (z_1, \tfrac14 r / A) \cup \ldots \cup B^{M_i} (z_N, \tfrac14 r/ A). \]
Choose $\phi_i \in C^\infty_c ( B^{M_i} (z_i, \tfrac14 r) )$ according to property (F) in subsection \ref{subsec:characterlimit} and set
\[ \phi^* := \phi_1 + \ldots + \phi_N \]
Then $\phi^* \geq 1$ on $B^{M_i} (x, 0.6 r)$ and $\supp \phi^* \subset B^{M_i} (x, 0.8 r)$ and
\[ |\nabla \phi^*| < A N \big( \tfrac14 r \big)^{-1} \textQQqq{and} |\triangle \phi^* | < AN \big( \tfrac14 r \big)^{-2}. \]
Let now $F : [0, \infty) \to [0,1]$ be a smooth function such that $F \equiv 1$ on $[1, \infty)$ and $F \equiv 0$ on $[0, 1/2]$.
Then $\phi := F \circ \phi^*$ satisfies the desired properties for a proper choice of $C$.

Item (4) of the $Y$-tameness properties is a direct consequence of this improved version of property (F).
\end{proof}

For item (5) of the tameness properties, we need the following heat kernel bound:

\begin{Lemma} \label{Lem:DaviesGaussian}
For any $A < \infty$ and $\delta > 0$ there is a constant $C_\delta = C_\delta (A) < \infty$ such that the following holds:

Let $(M, g)$ be a complete $n$-dimensional Riemannian manifold with bounded curvature, let $T > 0$ and assume that for any $0 < \tau \leq T$ and any function $f \in C^\infty (M)$ with 
\[ \int_M (4 \pi \tau)^{-n/2} e^{-f} dg = 1 \]
we have
\[ \int_M \big( \tau |\nabla f |^2 + f \big) (4 \pi \tau)^{-n/2} e^{-f} dg > - A. \]
Let $K(x,y,t)$ be the heat kernel on $(M,g)$.
Then for all $x, y \in M$ and all $0 < t \leq \delta T$ we have
\[ K(x,y,t) < \frac{C_\delta}{t^{n/2}} \exp \Big( - \frac{d^2(x,y)}{(1+\delta) t} \Big). \]
\end{Lemma}

\begin{proof}
The statement follows essentially from \cite{Davies_1987}, with a few modifications.
First we derive a log-Sobolev inequality as in \cite[Equation (1.4)]{Davies_1987} for small scales.
Consider a scale $0 < \tau \leq T$ and a positive function $u \in C^\infty (M)$, $u > 0$ with $H := \int_M u^2 dg < \infty$ and choose $f \in C^\infty (M)$ such that
\[ \frac{u}{\sqrt{H}} =  (4 \pi \tau)^{-n/4} e^{-f / 2}. \]
Then
\[
 \int_M (4 \pi \tau)^{-n/2} e^{-f} dg = 1 
\]
and thus
\[ \int_M \big( \tau |\nabla f|^2 + f \big) (4 \pi \tau)^{-n/2} e^{-f} dg > - A. \]
Expressing this equation in terms of $u$ yields
\[ \int_M \bigg( 4 \tau \frac{|\nabla u |^2}{H} + \Big( - 2 \log \frac{u}{\sqrt{H}} - \frac{n}2 \log \tau - \frac{n}2 \log (4 \pi) \Big) \frac{u^2}{H} \bigg) dg \geq - A, \]
which implies
\[ \frac1H \int_M \Big( 4 \tau |\nabla u |^2 - 2 ( \log u ) u^2 \Big) dg + \log H  - \frac{n}2 \log \tau    \geq - A + \frac{n}2 \log (4 \pi) =: -2A'. \]
So for all positive $u \in C^\infty (M)$, $u > 0$ with $\int_M u^2 < \infty$ and all $0 < \tau \leq T$ we have the following log-Sobolev inequality
\[
 \int_M u^2 \log u  \leq 2\tau \int_M |\nabla u |^2   + \Big( A' - \frac{n}4 \log \tau \Big) \int_M u^2  + \frac12 \Big( \int_M u^2 \Big) \log \Big( \int_M u^2  \Big)  
\]
Note that this inequality implies \cite[Equation (1.4)]{Davies_1987} with $\eps = 2\tau$ and $\beta (\eps) = A' + \frac{n}4 \log 2 - \frac{n}4 \log \eps$ whenever $0 < \eps \leq 2T$.

The desired Gaussian bound now follows from \cite[Theorem 5]{Davies_1987}.
The restriction that \cite[Equation (1.4)]{Davies_1987} only holds for $0 < \eps \leq 2T$ does not create any issues as we explain in the following:
The log-Sobolev inequality is only used in \cite[Lemma 2]{Davies_1987}.
The extra restriction implies that this lemma only holds for $0 < \eps \leq 4T$.
In the proof of \cite[Theorem 3]{Davies_1987}, this lemma is applied for all $0 < \eps \leq \lambda t$, where $\frac{\lambda}{\lambda - 1} = 1+\delta$.
So we can only ensure that \cite[Theorem 3]{Davies_1987} holds for $t \leq 4 \lambda^{-1} T$.
Hence, \cite[Corollary 4]{Davies_1987}, and consequently \cite[Theorem 5]{Davies_1987}, can only be ensured for all $0 < t \leq 8 \lambda^{-1} T$.
Since we may assume $\delta < 0.01$, \cite[Theorem 5]{Davies_1987} holds for all $t$ with $0 < t \leq \delta T < 8 \lambda^{-1} T$.
\end{proof}

\begin{proof}[Proof of item (4) of the $Y$-tameness properties in Theorem \ref{Thm:tame}.]
$K$ is the \linebreak[1] li\-mit of the kernels of the heat equation with respect to the metrics $g_i$.
Properties (5a) and (5b) follow immediately.
The Gaussian bound, property (5e), follows from Lemma \ref{Lem:DaviesGaussian}.
For property (5d) observe that $\int_{\RR} K(\cdot, y, t) dg \leq 1$ is obvious.
To see (5c) and the reverse inequality, observe that, by Lemma \ref{Lem:DaviesGaussian} and by property (A) in subsection \ref{subsec:characterlimit} and a similar reasoning as in \cite[Lemma 2.1]{Bamler-Zhang-Part2}, we have for any $D < \infty$ and for sufficiently large $i$:
\[ \big| B^{M_i} (x, r) \big| < C_0 e^{C_0 r} \]
for all $x \in B^{M_i}(q_i, D)$ and $0 < r < D$, where $C_0 = C_0 (A) < \infty$ is a constant, which is independent of $i$.
So, using Lemma \ref{Lem:DaviesGaussian} and Fubini's Theorem, we find that for all $D < \infty$ we have the following estimate for sufficiently large $i$:
For all $0 < t \leq T$, $1 < r < D$ and $x \in B^{M_i}(q_i, D)$
\begin{align*}
 \int_{M_i \setminus B^{M_i} (x, r)} K(\cdot, x, t) dg_i &\leq \frac{C_2}{t^{n/2}} \int_{M_i \setminus B^{M_i} (x, r)} \exp \Big( { - \frac{d_{M_i}^2 (x, \cdot)}{2 t} }\Big) dg_i \\
& =  \frac{C_2}{t^{n/2}} \int_{M_i \setminus B^{M_i} (x, r)} \int_{d_{M_i} (y, x)}^\infty  \frac{s}{t} \exp \Big( { - \frac{s^2}{2t} } \Big) ds dg_i(y) \\
 &=  \frac{C_2}{t^{n/2}} \int_r^\infty \int_{B^{M_i} (x, s)}  \frac{s}{t} \exp \Big( { - \frac{s^2}{2t} } \Big)  dg_i(y) ds \\
 &\leq \frac{C_2A}{t^{n/2+1}} \int_r^\infty \exp \Big( { - \frac{s^2}{2t} } \Big)  \cdot C_0 e^{C_0 s}  ds \\
 &\leq \frac{C_2A}{t^{n/2+1}} \exp \Big( { - \frac{r^2}{4t} } \Big) \int_r^\infty \exp \Big( { - \frac{s^2}{4t} } \Big)  \cdot C_0 e^{C_0 s}  ds \\
 & \leq \frac{C'}{t^{n/2+1}} \exp \Big( { - \frac{r^2}{4t} }\Big),
\end{align*}
for some $C' = C'(A, T) < \infty$.
Passing this bound to the limit, yields (5c).

It remains to show that $\int_{\RR} K(\cdot, y, t) dg \geq 1$.
So for all $y \in \RR$, $r > 0$, $0 < t \leq T$ and for sufficiently large $i$ we have
\[ \int_{ B^{M_i}(\Phi_i (y), r)} K(\cdot, \Phi_i (y), t) dg_i > 1 - \frac{C'}{t^{n/2+1}} \exp \Big({ - \frac{r^2}{4t} }\Big). \]
By the pointwise bound on $K$ and property (B) from subsection \ref{subsec:characterlimit}, we obtain furthermore that for any $\sigma > 0$ we have for sufficiently large $i$
\[ \int_{ B^{M_i}(\Phi_i (y), r) \cap \{ \rrm^{M_i}  > \sigma \} } K(\cdot, \Phi_i (y), t) dg_i > 1 -  \frac{C'}{t^{n/2+1}} \exp \Big({ - \frac{r^2}{4t} }\Big) - C''  t^{-n/2} \sigma^{\mathbf{p}_0}, \]
for some $C'' < \infty$, which does not depend on $i$ or $t$.
Since the domain of this integral is contained in $V_i$ for large $i$, we find that by passing to the limit
\[ \int_\RR K(\cdot, y, t) dg > 1 -  \frac{C'}{t^{n/2+1}} \exp \Big({ - \frac{r^2}{4t} }\Big) - 2 C''  t^{-n/2} \sigma^{\mathbf{p}_0}. \]
Letting $\sigma \to 0$ and $r \to \infty$ yields the desired result.
\end{proof}

\subsection{The regularity property}
We will now establish assertion (c) of Theorem \ref{Thm:basicconvergence}, namely that under the assumption of property (G) of subsection \ref{subsec:characterlimit} the limit is $Y$-regular at scale $\sqrt{T}$ for some $Y = Y(A)$.

\begin{Theorem} \label{Thm:regular}
For any $A > 0$ there is a $Y = Y(A ) < \infty$ such that the following holds: 

Consider the sequence $\{ (M_i, g_i, q_i  ) \}_{i = 1}^\infty$ and the pointed limiting singular space $(X, d_X, \RR, g, q_\infty)$ from Theorem \ref{Thm:detailedconvergence}.
Assume that $\{ (M_i, g_i, q_i  ) \}_{i = 1}^\infty$ satisfies the properties (A)--(C) and property (G) from subsection \ref{subsec:characterlimit} for the constant $A$ and for some constant $T > 0$.
Then $\XX$ is $Y$-regular at scale $\sqrt{T}$.
\end{Theorem}

\begin{proof}
We claim that we have $Y$-regularity for any $Y < A$.
Let $p \in X$ and $0 < r < \sqrt{T}$ and assume that
\[ \big| B^X (p, r) \cap \RR \big| > (\omega_n - A^{-1} ) r^n. \]
By increasing $r$ slightly, we may perturb $p$ and assume that $p \in \RR$.
By Theorem \ref{Thm:detailedconvergence}(a) we then have
\[ \liminf_{i \to \infty} \big| B^{M_i} (\Phi_i (p), r) \big| \geq \big| B^X (p, r) \cap \RR \big| > (\omega_n - 2^{-1} ) r^n. \]
So for large $i$ we have
\[ \big| B^{M_i} (\Phi_i (p), r) \big| > (\omega_n - 2^{-1} ) r^n. \]
Using property (G) from subsection \ref{subsec:characterlimit}, we get that $\tdrrm^{M_i} (\Phi_i (p)) > A^{-1} r$ for large $i$.
So, using Theorem \ref{Thm:detailedconvergence}(b), it follows that $\rrm^\infty (p) \geq A^{-1} r$.
\end{proof}

%: Structure Theory XX
% !TEX root = curv-bound-singular-spaces-04.tex
\section{Volume comparison}
We now show that the volume comparison result of Bishop and Gromov also holds for singular spaces with mild singularities.

\begin{Proposition}[volume comparison] \label{Prop:volumecomparison}
Let $\XX$ be a singular space with mild singularities that satisfies $\Ric \geq (n-1)\kappa$ on $\RR$ for some $\kappa \in \IR$ and let $p \in X$ be a point.
Then $|B(p,r) \cap \RR|$ depends continuously on $p$ and $r$.
Moreover, the quantity
\begin{equation} \label{eq:volquotient}
   \frac{|B(p,r) \cap \RR|}{v_\kappa (r)}
\end{equation}
is non-increasing in $r$ (as long as $r \leq \pi \sqrt{\kappa}$ if $\kappa > 0$) and
\[ |B(p,r) \cap \RR| \leq v_\kappa (r). \]
Finally, for any $0 \leq r_1 \leq r_2$ and $0 \leq r'_1 \leq r'_2$ with $r_1 \leq r'_1$ and $r_2 \leq r'_2$ ($r_2 \leq \pi \sqrt{\kappa}$ if $\kappa > 0$) we have
\begin{equation} \label{eq:volquotientA}
 \frac{|A(p,r_1, r_2) \cap \RR|}{v_\kappa (r_2) - v_\kappa (r_1) } \geq \frac{|A(p,r'_1, r'_2) \cap \RR|}{v_\kappa (r'_2) - v_\kappa (r'_1)}. 
\end{equation}
\end{Proposition}

\begin{proof}
We first consider the case in which $p \in \RR$ and show (\ref{eq:volquotientA}).
This will then imply (\ref{eq:volquotient}) for $p \in \RR$.

Consider the Riemannian manifold $(\RR, g)$ and recall the subsets $\mathcal{G}_p^* \subset \mathcal{G}_p \subset \RR$ and $\mathcal{D}_p^* \subset \mathcal{D}_p \subset T_p\RR$ from Definition \ref{Def:expmap}.
Note that by the mildness assumption we have $Q_p \cup \mathcal{G}_p = \RR$ and we know that $U := \RR \setminus Q_p \subset \mathcal{G}_p$ is open.
Using Proposition \ref{Prop:expincomplete}(c), we find that $\RR \setminus (U \cap \mathcal{G}^*_p) = ( \mathcal{G}_p \setminus \mathcal{G}^*_p) \cup Q_p$ has measure zero and Proposition \ref{Prop:expincomplete}(b), (d) implies that
\[ \exp_p |_{U'} : U' := \exp_p^{-1} ( U \cap \mathcal{G}^*_p) \subset \mathcal{D}^*_p \longrightarrow U \cap \mathcal{G}^*_p \]
is a (bijective) diffeomorphism.
In particular, the complement $\mathcal{D}^*_p \setminus U' = \exp_p^{-1} ( (\mathcal{G}_p \setminus \mathcal{G}^*_p ) \cup Q_p )$ also has measure zero.

Let $J_p : \mathcal{D}^*_p \to \IR$ be the Jacobian of $\exp_p$.
Set $J^*_p := J_p \chi_{\mathcal{D}^*_p} : T_p \RR \to \IR$.
It is easy to verify, using Proposition \ref{Prop:Jnonincreasing}, that the quantity
\[ J^*_p(v) \cdot \frac{|v|^{n-1}}{\big( \sn_\kappa (|v|) \big)^{n-1}} \]
is non-increasing along radial lines.
For any $r > 0$ we have
\begin{multline*}
 |A(p,r_1,r_2) \cap \RR|  = |A(p,r_1,r_2) \cap (U \cap \mathcal{G}^*_p)| \\
 = \int_{A(0,r_1,r_2) \cap U'} J_p (v) dv = \int_{A(0,r_1,r_2) \cap \mathcal{D}^*_p} J_p (v) dv = \int_{A(0,r_1,r_2)} J^*_p (v) dv. 
\end{multline*}
So, in order to verify the monotonicity of (\ref{eq:volquotient}), it suffices to check that
\[ \frac1{v_\kappa (r_2)- v_\kappa (r_1)} \int_{A(0,r_1,r_2)} J^*_p (v) dv \]
is non-increasing in $r_1$ and $r_2$.
For this, it is enough to show that for any $v \in T_p \RR$, $|v| = 1$, the quantity
\[ \frac1{v_\kappa (r_2) - v_\kappa (r_1)} \int_{r_1}^{r_2} J^*_p (t v) t^{n-1} dt \]
is non-increasing in $r_1$ and $r_2$.
The rest follows via Fubini's theorem.
To see the monotonicity of this quantity in $r_1$, observe that its derivative in $r_1$ can be bounded as follows:
\begin{multline*}
 - \frac{1}{v_\kappa (r_2) - v_\kappa (r_1)}  J^*_p (r_1 v) r_1^{n-1}  
 + \frac{v'_\kappa (r_1)}{(v_\kappa (r_2) - v_\kappa (r_1))^2} \int_{r_1}^{r_2} J^*_p (tv) t^{n-1} dt \\
 = \frac{n \omega_n}{(v_\kappa (r_2) -  v_\kappa (r_1))^2}  \int_{r_1}^{r_2} \Big( - \big( \sn_{\kappa} (t) \big)^{n-1} J^*_p (r_1 v ) r_1^{n-1} \\
 + \big( \sn_\kappa (r_1) \big)^{n-1} J^*_p (tv) t^{n-1} \Big)  \leq 0.
\end{multline*}
Similarly, the derivative in $r_2$ can be bounded as follows
\begin{multline*}
  \frac{1}{v_\kappa (r_2) - v_\kappa (r_1)}  J^*_p (r_2 v) r_2^{n-1}  
 - \frac{v'_\kappa (r_2)}{(v_\kappa (r_2) - v_\kappa (r_1))^2} \int_{r_1}^{r_2} J^*_p (tv) t^{n-1} dt \\
 = \frac{n \omega_n}{(v_\kappa (r_2) - v_\kappa (r_1))^2}  \int_{r_1}^{r_2} \Big( \big( \sn_{\kappa} (t) \big)^{n-1} J^*_p (r_2 v ) r_2^{n-1} \\
 + \big( \sn_\kappa (r_2) \big)^{n-1} J^*_p (tv) t^{n-1} \Big)  \leq 0.
\end{multline*}
This shows the monotonicity of (\ref{eq:volquotientA}) and hence (\ref{eq:volquotient}) in the case $p \in \RR$.

We will now show that $|B(p,r) \cap \RR|$ depends continuously on $p \in X$, as $r > 0$ is kept fixed.
Let $p_1, p_2 \in X$ and set $d := d(p_1, p_2)$.
Then $B(p_1, r-d) \subset B(p_2, r) \subset B(p_1, r+d)$.
Assume first that $p_1 \in \RR$.
Then, by our previous conclusion
\begin{multline*}
  \frac{v_\kappa (r-d)}{v_\kappa (r)} |B(p_1, r) \cap \RR|  \leq |B(p_1, r-d) \cap \RR | \leq |B(p_2, r) \cap \RR | \\
  \leq |B(p_1, r+d) \cap \RR| \leq \frac{v_\kappa (r+d)}{v_\kappa (r)}  |B(p_1, r) \cap \RR|.
\end{multline*}
Letting $p_1 \to p_2$ yields the continuity of $|B(p,r) \cap \RR|$ in $p$.
The continuity in $r$ follows from volume comparison.
The continuity of the volume $|A(p,r_1, r_2) \cap \RR|$ of annuli in $p, r_1, r_2$ follows similarly.

Finally, the monotonicity of (\ref{eq:volquotientA}) and (\ref{eq:volquotient}) for arbitrary $p$ follows from the continuity of the volume.
\end{proof}

Next, we establish the Laplacian comparison theorem for distance functions on singular spaces with mild singularities.
Note in the following that the comparison still holds in the weak sense at points that cannot be connected to the basepoint by a regular geodesic.

\begin{Proposition} \label{Prop:LaplacecomparisonXX}
Let $\XX = (X, d, \RR, g)$ be a singular space with mild singularities and let $p \in \RR$ be a regular point.
Assume that $\Ric \geq (n-1)\kappa$ on $\RR$.
Then the function $b(x) := d (x,p)$ satisfies
\begin{equation} \label{eq:dmunm1b}
 d\mu_{\triangle b}  \leq (n-1) \frac{\sn'_{\kappa}(b)}{\sn_\kappa (b)} dg  
\end{equation}
on $\RR \setminus \{ p \}$.
\end{Proposition}

\begin{proof}
By Proposition \ref{Prop:weakLaplaciandistance}, the positive part $(\mu_{\triangle b})_+$ is absolutely continuous with respect to $dg$.
As such, it suffices to show that (\ref{eq:dmunm1b}) holds away from sets of measure $0$.
As $(\RR \setminus \mathcal{G}^*_p) \cup Q_p$ has measure zero, it suffices to verify (\ref{eq:dmunm1b}) on $\mathcal{G}^*_p$, where $b$ is $C^2$.
This can be achieved using standard Riemannian geometry.
\end{proof}

\section{Integration by parts}
In this section we establish two results that will allow us to perform integration by parts on singular spaces.
We first need the following technical lemma.

\begin{Lemma} \label{Lem:finddomains}
Let $\XX = (X, d, \RR, g)$ be a singular space with singularities of codimension $\mathbf{p}_0 > 1$, let $x \in X$ be a point, $r > 0$ a radius and $1 < \mathbf{p} < \mathbf{p}_0$.
Then there is a sequence of positive numbers $s_i \to 0$ and an increasing sequence of subsets
\[ U_1 \subset U_2 \subset \ldots \subset B(x,r) \cap \RR \]
such that the following holds:
\begin{enumerate}[label=(\alph*)]
\item Each $U_i$ is closed in $B(x,r)$ and the boundary portion $\partial U_i \cap B(x,r)$ of $U_i$ that is contained in $B(x,r)$ is $C^2$.
\item We have
\[ \bigcup_{i = 1}^\infty \Int U_i = B(x,r) \cap \RR \]
\item We have
\[ \rrm > s_i \textQQqq{on}  U_i \]
\item There is a constant $C < \infty$, which is independent of $i$ but which may depend on $x$ and $r$, such that for all $i$
\[ \mathcal{H}^{n-1} (\partial U_i \cap B(x,r)  ) < C s_i^{\mathbf{p}-1} \]
Here $\mathcal{H}^{n-1}$ denotes the $n-1$ dimensional Hausdorff measure.
\end{enumerate}
\end{Lemma}

\begin{proof}
By definition the function $\rrm : \RR \to (0, \infty)$ is $1$-Lipschitz.
Let $f : \RR \to (0, \infty)$ be a smoothing of $\rrm |_{\RR}$ such that
\[ \tfrac12 \rrm < f < 2 \rrm \textQq{and} |\nabla f | < 2 \textQQqq{on} \RR. \]
By the coarea formula and Definition \ref{Def:codimensionsingularities} we have for any $0 < s < 1$
\begin{multline*}
 \int_{s}^{2s} \mathcal{H}^{n-1} (\{ f = s'\} \cap B(x,r) \cap \RR ) ds' = \int_{ \{ s < f < 2s \} \cap B(x,r) \cap \RR} |\nabla f | dg \\
 \leq 2 \big| \{ s < f < 2s \} \cap B(x,r) \cap \RR \big| \\
  \leq 2 \big| \{ \tfrac12 s < \rrm < 4s \} \cap B(x,r) \cap \RR \big| 
 \leq 2 \cdot 4^{\mathbf{p}} \mathbf{E}_{\mathbf{p},x,r} s^{\mathbf{p}}.
\end{multline*}
So we can choose $s' \in (s, 2s)$ such that $\{ f = s' \} \subset \RR$ is $C^2$ and such that
\[ | \{ f = s' \} \cap B(x,r) \cap \RR | \leq 4^{\mathbf{p}+1} \mathbf{E}_{\mathbf{p},x,r} s^{\mathbf{p}-1} < 4^{\mathbf{p}+1} \mathbf{E}_{\mathbf{p},x,r} (s')^{\mathbf{p}-1}. \]
It follows that we can find a constant $C < \infty$ and a sequence $s_i \to 0$ such that
\[ | \{ f = 2s_i \} \cap B(x,r) \cap \RR | < C s_i^{\mathbf{p}-1}. \]
So if we set $U_i :=  \{ f \geq 2 s_i \} \cap B(x,r) \cap \RR$, then assertions (a)--(d) hold.
\end{proof}

We can now state the first main result of this section.

\begin{Proposition} \label{Prop:integrationbyparts}
Let $\XX = (X,d, \RR, g)$ be a singular space with singularities of codimension $\mathbf{p}_0 > 2$ and $Z$ a continuously differentiable vector field on $\RR$ that vanishes on $\RR \setminus B(x,r)$ for some large $r > 0$.
Assume that there is a constant $C < \infty$ such that
\[ |Z| < C \rrm^{-1} \textQq{and} |{\DIV Z}| < C \rrm^{-2} \textQQqq{on} B(x,r) \cap \RR. \]
Then
\[ \int_{\RR} (\DIV Z) dg = 0. \]

Next, assume that $f : \RR \to \IR$ is a continuously differentiable function and $Z$ is a continuously differentiable vector field on $\RR$ such that $fZ$ vanishes on $\RR \setminus B(x,r)$ for some large $r > 0$.
Assume that $|f|$ is uniformly bounded and that there is some constant $C < \infty$ such that
\[ |Z|, |\nabla f| < C \rrm^{-1} \textQq{and} |{\DIV Z}| < C \rrm^{-2} \textQQqq{on} B(x,r) \cap \RR. \]
Then
\[ \int_{\RR} \langle \nabla f, Z \rangle dg = - \int_{\RR} (f \DIV Z) dg. \]
\end{Proposition}

\begin{proof}
Let us first prove the first part.
Consider the numbers $s_i$ and the subsets $U_i$ from Lemma \ref{Lem:finddomains} applied to the ball $B(x,2r)$ for some $2 < \mathbf{p} < \mathbf{p}_0$.
Then, for some generic constant $C_* < \infty$, which may depend on $x$, $r$ and $\mathbf{p}$, we have
\begin{multline*}
 \int_{\{ 0 < \rrm < 1 \}} |{\DIV Z}| = \sum_{k = 1}^\infty \int_{\{ 2^{-k}  < \rrm < 2^{-k+1} \}} |{\DIV Z}| \\
 < \sum_{k=1}^\infty C_* C (2^{-k})^{-2} |\{ 2^{-k}  < \rrm < 2^{-k+1} \}| < \sum_{k=1}^\infty C_* C (2^{-k})^{-2} (2^{-k+1})^{\mathbf{p}} < \infty. 
\end{multline*}
So
\[ \int_{\RR} \DIV Z = \lim_{i \to \infty} \int_{U_i} \DIV Z = \lim_{i \to \infty} \int_{\partial U_i} Z_i d\nu = \lim_{i \to \infty} \int_{B(x,2r) \cap \partial U_i} Z_i d\nu. \]
It follows that 
\[ \bigg| \int_{\RR} \DIV Z \bigg| \leq \lim_{i \to \infty} C_* C s_i^{-1} \mathcal{H}^{n-1} (B(x,2r) \cap \partial U_i) \leq C_* \lim_{i \to \infty} s_i^{\mathbf{p} -1 - 1} = 0. \]
This establishes the first part.
The second part follows from the first part using the identity $\DIV (f Z) = f \DIV Z + \langle \nabla f, Z \rangle$.
\end{proof}

Sometimes we will have to carry out the limiting process in the proof of Proposition \ref{Prop:integrationbyparts} by hand in order to have better control on cutoff functions.
We will then use the second main result of this section.

\begin{Proposition} \label{Prop:Uiphii}
Let $\XX = (X, d, \RR, g)$ be a singular space with singularities of codimension $\mathbf{p}_0 > 1$ and let $x \in X$, $r > 0$ and $1 < \mathbf{p} < \mathbf{p}_0$.
Assume moreover that is $Y$-tame at scale $r$ for some $Y > 0$.
Then there is a sequence of positive numbers $s_i \to 0$, a sequence of subsets $U_i \subset B(x,r) \cap \RR$ and a sequence of maps $\phi_i : U_i \to [0,1]$ such that the following holds:
\begin{enumerate}[label=(\alph*)]
\item Each $U_i$ is closed in $B(x,r)$ and $\partial_s U_i :=  \partial U_i \cap B(x,r)$ is $C^2$.
\item $U_1 \subset U_2 \subset \ldots$.
\item $\bigcup_{i=1}^\infty \Int U_i = B(x,r) \cap \RR$.
\item $\rrm > s_i$ on $U_i$.
\item $\mathcal{H}^{n-1} (\partial_s U_i) < C s_i^{\mathbf{p}-1}$ for some constant $C < \infty$, which is independent of $i$, but which may depend on $x, r$.
Here $\mathcal{H}^{n-1}$ denotes the $n-1$ dimensional Hausdorff measure.
\item Each $\phi_i : U_i \to [0,1]$ is $C^2$ up to the boundary $\partial_s U_i$ and has compact support in $U_i$.
\item $\phi_i \equiv 1$ on $B(x,r/2) \cap U_i$ for all $i$.
\item $|\nabla \phi_i | < Y r^{-1}$ and $|\triangle \phi_i | < Y r^{-2}$ for all $i$.
\end{enumerate}
\end{Proposition}

\begin{proof}
Consider the subsets $U_i$ from Lemma \ref{Lem:finddomains} and the open subsets $U_i$ and $\phi_i : U_i \to [0,1]$ from Definition \ref{Def:tameness}(4), which we will denote by $U'_i$ and $\phi'_i$, in order to avoid confusion.
By compactness we can find for each $i = 1, 2, \ldots$ some $j_i = 1, 2, \ldots$ such that $\ov{U}_i \subset U'_{j_i}$.
If we now set $\phi_i := \phi'_{j_i} |_{U_{i}}$, then the proposition follows.
\end{proof}

\section{Cheng-Yau gradient estimate}
In this section we derive gradient estimates for bounded functions that are harmonic or satisfy the heat equation on the regular part $\RR$ of a singular space $\XX$ on which a lower Ricci curvature bound holds.
These gradient estimates will imply, in particular, that such functions are also continuous at the singular part of $\XX$.
In order to obtain our estimates, we will impose several conditions on the singular space $\XX$.
For example, we will assume that $\XX$ is $Y$-tame and that $\XX$ has singularities of codimension $\mathbf{p}_0 > 3$.

Note that the standard proof by Cheng and Yau (see \cite{Cheng-Yau-gradest}), via maximum principles, cannot be applied here as the maximum may be attained on the singular set $X \setminus \RR$ of $\XX$, where we have no information on the functions under investigation.
Even worse, we don't even assume a priori that we have continuity at the singular set.
We therefore carry out an integral version of the Cheng and Yau's proof.
This proof is particularly adapted to the singular setting.
In the non-singular setting, the integral version of Cheng and Yau's proof can be simplified significantly.

We remark that Corollary \ref{Cor:ChengYau} is similar to \cite[Proposition 2.14]{Chen-Wang-II}.

\begin{Lemma} \label{Lem:preCYbound}
There is a universal constant $C < \infty$ such that the following holds:

Let $\XX$ be a singular space with singularities of codimension $\mathbf{p}_0 > 2$, $x \in X$, $r > 0$ and consider a $C^2$ function $h : (B(x,r) \cap \RR) \times [-r^2, 0] \to \IR$.
Assume that $h$ that is uniformly bounded from above and below and that $\partial_t h = \triangle h$ on $(B(x,r) \cap \RR) \times [-r^2,0]$.
Then
\[ \int_{-r^2/2}^0 \int_{B(x,r/2) \cap \RR} |\nabla h |^2 (y,t) dg(y) dt \leq C |B(x,r) \cap \RR|  \Big( \osc_{(B(x,r) \cap \RR) \times [-r^2,0]} h \Big)^2 \]
\end{Lemma}

\begin{proof}
By standard parabolic regularity, there is a constant $C_* < \infty$, which depends on the $L^\infty$-bound of $h$, such that for all $(y,t) \in (B(x, 3r/4) \cap \RR) \times [-3r^2/4,0]$ we have
\begin{equation} \label{eq:stdestimate12}
 |\nabla h|(y,t) < C_* \big( \min \{ \rrm (y), r \} \big)^{-1} \textQq{and} |\nabla^2 h|(y,t) < C_* \big( \min \{ \rrm (y), r \} \big)^{-2}. 
\end{equation}
Note that for the second bound we have used the fact that by Definition \ref{Def:curvradius}, of the curvature radius, $|{\nabla \Rm}| < \rrm^{-3} (y)$ on $B(y, \rrm (y))$.

Choose now $2 < \mathbf{p} < \mathbf{p}_0$.
Using the bounds in (\ref{eq:stdestimate12}), we get that for any $t \in [-3r^2/4, 0]$ and for some generic constant $C_{**} < \infty$, which may depend on $\XX$, $h$ and $r$,
\begin{align*}
 \int_{B(x,3r/4) \cap \RR} |\nabla h|^2 (\cdot, t) 
  &\leq C_*^2 r^{-2} | B(x,3r/4) \cap \RR |  \\
  &\qquad +  2 C^2_* \int_0^{r} s^{-3} |B(x,3r/4) \cap \{ \rrm < s \} \cap \RR | ds  \\
 &\leq C_{**} r^{n-2} + C_{**} \int_0^r s^{-3} (s/r)^{\mathbf{p}} r^n ds < C_{**} < \infty. 
\end{align*}
Similarly, we get that for any $t \in [-3r^2/4]$
\begin{equation} \label{eq:nab2hL1}
 \int_{B(x,3r/4) \cap \RR} |\nabla^2 h| (\cdot, t) < C_{**} < \infty.
\end{equation}

Fix for the moment $t \in [-3r^2/4]$.
By smoothing the composition of the radial distance function $d(x, \cdot)$ with a suitable cutoff function, we can construct a function $\phi : X \to [0,2]$ with support in $B(x,3r/4)$ such that $\phi \equiv 1$ on $B(x,r/2)$, $\phi |_{\RR}$ is $C^1$ and $|\nabla \phi | < 10 r^{-1}$ on $\RR$.
Using Proposition \ref{Prop:integrationbyparts} and writing $h = h(\cdot, t)$, we find that for some uniform, generic constant $C < \infty$
\begin{multline*} 
  \int_{B(x,r) \cap \RR} \bigg( \frac{d}{dt} h^2(\cdot, t) \bigg) \phi^2 = 2 \int_{B(x,r) \cap \RR} h(\cdot, t) \triangle h (\cdot, t) \phi^2 \\
 =  \int_{B(x,r) \cap \RR} \Big( - 2|\nabla h|^2 \phi - 4 h \nabla h \phi \nabla \phi \Big) 
 \leq   \int_{B(x,r) \cap \RR} \Big( - |\nabla h|^2 \phi^2 + 4 h^2 |\nabla \phi|^2  \Big) \\
 \leq -  \int_{B(x,r/2) \cap \RR} |\nabla h|^2 (\cdot, t) + C r^{-2} |B(x,r) \cap \RR| \Big( \sup_{B(x,r) \cap \RR} h^2(\cdot, t) \Big).
\end{multline*}
Let now $\eta : [-r^2, 0] \to [0,1]$ be a cutoff function such that $\eta \equiv 0$ on $[-r^2, -3r^2/4]$, $\eta \equiv 1$ on $[-r^2/2, 0]$ and $|\eta' | < 10 r^{-2}$.
Then, by Fubini's Theorem,
\begin{align*}
- \int_{-r^2}^0 \eta'(t) & \int_{B(x,r) \cap \RR} h^2(y,t) \phi^2(y) dg(y) dt  
+ \int_{B(x,r) \cap \RR} h^2 (y, 0) \phi^2 (y, 0) dg(y) \displaybreak[2]  \\
&=  \int_{B(x,r) \cap \RR}  \int_{-r^2}^0 \eta(t)  \bigg( \frac{d}{dt}h^2(y,t) \bigg) \phi^2(y) dt  dg(y) \displaybreak[2] \\
 &\leq -  \int_{-r^2/2}^0 \int_{B(x,r/2) \cap \RR} |\nabla h|^2 \\
 &\qquad\qquad\qquad + r^2 \cdot C r^{-2} |B(x,r) \cap \RR| \Big( \sup_{(B(x,r) \cap \RR) \times [-r^2, 0]} h^2 \Big). 
\end{align*}
It follows that
\begin{multline*}
 \int_{-r^2/2}^0 \int_{B(x,r/2) \cap \RR} |\nabla h|^2 \leq C r^{-2} \int_{-r^2}^0 \int_{B(x,r) \cap \RR} h^2  \\ + C |B(x,r) \cap \RR| \Big( \sup_{(B(x,r) \cap \RR) \times [-r^2, 0]} h^2 \Big) 
  \leq C |B(x,r) \cap \RR| \Big( \sup_{(B(x,r) \cap \RR) \times [-r^2, 0]} h^2 \Big).
\end{multline*}
The desired inequality now follows by replacing $h$ by $h - \frac12 (\sup h + \inf h)$.
\end{proof}

We can now improve this $L^2$-bound on $\nabla h$ and derive a pointwise bound.

\begin{Proposition} \label{Prop:ChengYau-heatequation}
For any $\mathbf{p}_0 > 3$ and $Y < \infty$ there is a $C = C (\mathbf{p}_0, Y) < \infty$ such that the following holds:

Let $\XX$ be a singular space with singularities of codimension $\mathbf{p}_0$, $x \in X$, $r > 0$ and assume that $\Ric \geq - (n-1) r^{-2}$ on $B(x,r) \cap \RR$.
Assume moreover that $\XX$ is $Y$-tame at scale $r$.
Consider a $C^3$ function $h : (B(x,r) \cap \RR) \times [-r^2, 0] \to \IR$ that is uniformly bounded from above and below and that satisfies $\partial_t h = \triangle h$ on $(B(x,r) \cap \RR) \times [-r^2,0]$.
Then
\begin{multline} \label{eq:CYinequparabolic}
 |\nabla h |^2 \leq C r^{-n-2} \int_{-r^2/2}^0 \int_{B(x, r / 2) \cap \RR} |\nabla h |^2 (y, t) dg(y) dt \\
 \leq C^2 r^{-2} \Big( \osc_{(B(x,r) \cap \RR) \times [-r^2, 0]} h \Big)^2 \\
  \textQQqq{on} \big( B(x, r/2) \cap \RR \big) \times [-r^2/2,0]. 
\end{multline}
In particular, $h$ is locally Lipschitz on $B(x,r) \times (-r^2,0]$.
\end{Proposition}

In the case in which $h$ is constant in time $t$, Proposition \ref{Prop:ChengYau-heatequation} implies the gradient bound of Cheng and Yau for harmonic functions.

\begin{Corollary} \label{Cor:ChengYau}
For any $\mathbf{p}_0 > 3$ and $Y < \infty$ there is a $C = C (\mathbf{p}_0, Y) < \infty$ such that the following holds:

Let $\XX$ be a singular space with singularities of codimension $\mathbf{p}_0$, $x \in X$, $r > 0$ and assume that $\Ric \geq - (n-1) r^{-2}$ on $B(x,r) \cap \RR$.
Assume moreover that $\XX$ is $Y$-tame at scale $r$.
Consider a $C^3$ function $h : B(x,r) \cap \RR \to \IR$ that is uniformly bounded from above and below and that satisfies $\triangle h = 0$ on $B(x,r) \cap \RR$.
Then
\begin{multline*}
 |\nabla h |^2 \leq C r^{-n} \int_{B(x,r/2) \cap \RR} |\nabla h |^2 (y,t) dg(y) dt \\
  \leq  C^2 r^{-2} \Big( \osc_{B(x,r) \cap \RR} h \Big)^2 \textQQqq{on} B(x, r/2) \cap \RR. 
\end{multline*}
\end{Corollary}

We will hence use the following terminology for the remainder of this paper:

\begin{Definition}[harmonic function on singular Ricci flat space] \label{Def:XXharmonic}
Let $\XX$ be a singular space and $U \subset X$ an open subset.
A function $h : U \to \IR$ is called \emph{harmonic} if it is continuous on $U$ and $C^3$ on $U \cap \RR$ and satisfies $\triangle h = 0$ on $U \cap \RR$.
\end{Definition}

A direct consequence of Corollary \ref{Cor:ChengYau} is

\begin{Corollary} \label{Cor:harmonic}
Let $\XX$ be a singular space with singularities of codimension $\mathbf{p}_0 > 3$ that is $Y$-tame at scale $r$ for some $Y, r > 0$ and let $U \subset X$ be an open subset.
Assume that $\Ric \geq - C$ on $U$ for some $C < \infty$.
Consider a locally bounded function $h : U \to \IR$ such that $h$ is $C^3$ on $U \cap \RR$ and satisfies $\triangle h = 0$ there.
Then $h$ is harmonic in the sense of Definition \ref{Def:XXharmonic} and even locally Lipschitz.
\end{Corollary}

\begin{proof}[Proof of Proposition \ref{Prop:ChengYau-heatequation}]
Note that by restricting our attention to a smaller parabolic neighborhood, it suffices to consider the case $x \in \RR$ and to show that (\ref{eq:CYinequparabolic}) holds at $(x,0)$.
For the remainder of this proof fix a constant $3 < \mathbf{p} < \mathbf{p}_0$, a constant $0 < \alpha < 1/10$ such that $3 + 2 \alpha < \mathbf{p}$ and a constant $\beta > 0$ and set
\[ f := ( |\nabla h|^2 + \beta)^{\frac12 + \alpha}. \]
For the remainder of the proof we will fix $\alpha, \mathbf{p}$ and $\beta$, except for the very end, where we let $\beta \to 0$.
In the following, generic constants are allowed to depend on $\alpha$ and $\mathbf{p}$, but not on $\beta$.

Using Bochner's identity and the bound $\Ric \geq - (n-1)r^{-1}$ on $\RR$, we find
\[ \partial_t f = (1 + 2\alpha) \frac{\langle  \nabla \partial_t h, \nabla h \rangle}{(|\nabla h|^2 + \beta )^{\frac12 - \alpha}} \leq (1 + 2\alpha) \frac{\langle \triangle \nabla h, \nabla h \rangle + (n-1) r^{-2} |\nabla h|^2}{(|\nabla h|^2 + \beta )^{\frac12 - \alpha}}. \]
Also,
\[ |\nabla f| = \big( \tfrac12 + \alpha \big) \frac{|\nabla |\nabla h|^2|}{( |\nabla h|^2 + \beta)^{\frac12 - \alpha}} = ( 1 + 2\alpha ) \frac{|\nabla h | |\nabla^2 h|}{( |\nabla h|^2 + \beta)^{\frac12 - \alpha}} \]
and
\begin{align*}
 \triangle f - \partial_t f &=  \big( \tfrac12 + \alpha \big) \frac{\triangle |\nabla h|^2 -  2 \langle \triangle \nabla h, \nabla h \rangle - 2 (n-1) r^{-2} |\nabla h|^2}{( |\nabla h|^2 + \beta)^{\frac12 - \alpha}} \\
 &\qquad -  \big( \tfrac14 - \alpha^2 \big) \frac{|\nabla |\nabla h|^2|^2}{( |\nabla h|^2 + \beta)^{\frac32 - \alpha}} \\
 &\geq ( 1 + 2 \alpha ) \frac{ |\nabla^2 h|^2 - (n-1) r^{-2} |\nabla h|^2}{( |\nabla h|^2 + \beta)^{\frac12 - \alpha}} -  ( 1 - 4\alpha^2 ) \frac{|\nabla h |^2  |\nabla^2 h|^2}{( |\nabla h|^2 + \beta)^{\frac32 - \alpha}} \\
 &\geq ( 1 + 2 \alpha ) \frac{ |\nabla^2 h|^2 - (n-1) r^{-2} |\nabla h|^2}{( |\nabla h|^2 + \beta)^{\frac12 - \alpha}} -  ( 1 - 4\alpha^2 ) \frac{  |\nabla^2 h|^2}{( |\nabla h|^2 + \beta)^{\frac12 - \alpha}} \\
 &\geq (2 \alpha + 4 \alpha^2 )  \frac{ |\nabla^2 h|^2}{( |\nabla h|^2 + \beta)^{\frac12 - \alpha}} - (1+2\alpha) (n-1) r^{-2} f .
\end{align*}

Consider now the sequence $s_i \to 0$, the subsets $U_i$ and the functions $\phi_i$ from Proposition \ref{Prop:Uiphii} applied to the ball $B(x,r/2)$.
We can then compute that for any $i$ and $0 < t \leq r^2/2$
\begin{align}
 \frac{d}{dt} \int_{U_i} &K (x, \cdot, t) f(\cdot, -t) \phi^2_i = \int_{U_i} \Big( \triangle K(x, \cdot,t) f(\cdot, -t) \phi_i^2 - K (x, \cdot, t) \partial_t f(\cdot, -t) \phi_i^2 \Big) \notag \displaybreak[1] \\
 &= \int_{\partial_s U_i} \nabla K (x, \cdot, t) f(\cdot, -t) \phi_i^2 d\nu  + \int_{U_i} \Big(- \nabla K (x,\cdot, t) \nabla f(\cdot, -t) \phi_i^2 \notag \\
 &\qquad\qquad\qquad\qquad\qquad - \nabla K(x,\cdot, t) f(\cdot, -t) \nabla \phi_i^2 - K(x, \cdot, t) \partial_t f(\cdot, -t) \phi_i^2 \Big) \notag \displaybreak[1] \\
  &= \int_{\partial_s U_i} \Big( \nabla K (x, \cdot, t) f(\cdot, -t) \phi_i^2 - K(x,\cdot,t) \nabla f (\cdot, -t) \phi_i^2 \notag \\
  &\qquad\qquad\qquad\qquad\qquad\qquad\qquad\qquad\qquad\qquad - K(x,\cdot,t) f(\cdot, -t) \nabla \phi_i^2 \Big) d\nu \notag \\
   &\qquad + \int_{U_i} \Big( K (x,\cdot, t) \big( \triangle f(\cdot, -t)- \partial_t f(\cdot, -t) \big) \phi_i^2  \notag \\
   &\qquad\qquad\qquad +  K (x,\cdot, t)  f(\cdot, -t) \triangle \phi_i^2 + 4  K (x,\cdot, t) \nabla  f(\cdot, -t) \phi_i \nabla \phi_i  \Big). \label{eq:firstparttimeder}
 \end{align}
Next, set $\rrrm := \min \{ \rrm , r \}$ and $H := \Vert h \Vert^2_{L^\infty (B(x,r) \times [-r^2,0])}$.
By local parabolic regularity we can find some generic constant $C_* < \infty$ such that for any $(y,t) \in (B(x,r/2) \cap \RR) \times [-r^2/2,0]$
\begin{equation} \label{eq:fless}
 f(y,t) \leq C_* \big( \rrrm^{-2}(y) H^2  + \beta \big)^{\frac12 + \alpha} 
\end{equation}
 and
\begin{multline}
 |\nabla f |(y,t) \leq (1 + 2\alpha ) \frac{|\nabla h| |\nabla^2 h|}{( |\nabla h|^2 + \beta)^{\frac12 - \alpha}} (y,t) \\
\leq (1 + 2\alpha ) |\nabla h|^{2\alpha}(y,t) |\nabla^2 h| (y,t)
 \leq C_* \rrrm^{-2 - 2\alpha} (y) H^{1+\alpha}. \label{eq:derfless}
\end{multline}
By local parabolic regularity applied to $K(x, \cdot, \cdot)$ restricted to the (regular) parabolic neighborhood
\[ P(y,t) := B(y, \min\{ \rrm(y), (t/2)^{1/2}, r/4 \}) \times [t - \min \{ \rrm^2 (y), t/2, r^2/4 \}, t], \]
we obtain that for any $y \in B(x,r/2) \cap \RR$ and $0 < t \leq r^2$
\begin{multline*}
 |\nabla_y K (x, y, t) | \leq C_* \big( \rrm^{-1} (y) + t^{-1/2} ) \Vert K(x, \cdot, \cdot) \Vert_{L^\infty( P(y,t))} \\
 \leq C_* \big( \min \{ \rrm (y), t^{1/2} \} \big)^{-1} \sup_{(z,s) \in P(y,t)} \frac{Y}{s^{n/2}} \exp \bigg( {- \frac{d^2(x,z)}{Ys} }\bigg).
 \end{multline*}
Note that there is a constant $b > 0$ such that for large $i$ we have $B(x,b) \subset U_i$.
So we can find a constant $C'_{b,r} < \infty$, which may depend on $b$ and $r$, but not on $i$ or $t$ such that for all $y \in \partial U_i$ and $0 < t \leq r^2$ we have
\begin{equation} \label{eq:pderpless}
 K(x,y,t) \leq C'_{b,r} \textQQqq{and} |\nabla_y K(x,y,t) | \leq C'_{b,r} \rrrm^{-1}(y) \leq C'_{b,r} (s_i^{-1} + r^{-1}).
\end{equation}
So using (\ref{eq:firstparttimeder}), (\ref{eq:fless}), (\ref{eq:derfless}), (\ref{eq:pderpless}) and H\"older's inequality, we get that for large $i$ and for some generic $C_{**} < \infty$, which may depend on $\XX$ or $r$, but not on $i$ or $t$
\begin{align*}
\frac{d}{dt} \int_{U_i} & K (x, \cdot, t) f(\cdot, -t) \phi_i^2 \\
   &\geq - C_{**} |\partial_s U_i| \Big( C'_{b,r} (s_i^{-1} + r^{-1} ) \big( s_i^{-2} H^2 + \beta \big)^{\frac12 + \alpha}  \\
   &\qquad\qquad\qquad\qquad + C'_{b,r} s_i^{-2- 2\alpha} H^{1+2\alpha} + C'_{b,r} \big( s_i^{-2} H^2 + \beta \big)^{\frac12 + \alpha} \Big) \\
  &\qquad +  \int_{U_i} \Big( (2 \alpha + 4 \alpha^2)   K(x,\cdot, t)   \frac{ |\nabla^2 h|^2}{( |\nabla h|^2 + 1)^{\frac12 - \alpha}} (\cdot, -t) \phi_i^2 \\
  &\qquad\qquad\qquad - 2(n-1) r^{-2} K(x, \cdot, t) f( \cdot, - t) \phi_i^2 \\
  &\qquad\qquad\qquad +  K(x,\cdot,t) f(\cdot, -t) |\triangle \phi_i^2| \\
   &\qquad\qquad\qquad - 4 (1+ 2\alpha)   K(x,\cdot, t) \frac{ |\nabla h| |\nabla^2 h|}{(|\nabla h|^2 + \beta)^{\frac12 - \alpha}} (\cdot, -t) |\nabla \phi_i| \phi_i  \Big)  \\
   & \geq -  C_{**}  s_i^{\mathbf{p}-1} \Big(  s_i^{-1} \cdot s_i^{-1-2\alpha} H^{1+2\alpha}   + s_i^{-2- 2\alpha} H^{1+2\alpha}  \Big)  \\
   &\quad - 2 (n-1) r^{-2} \int_{U_i}   K(x, \cdot, t) f(\cdot, -t) \phi_i^2 \\
   &\quad  -  \int_{U_i} \Big( K(x,\cdot,t) f(\cdot,-t) |\triangle \phi_i^2|  \\
   &\qquad\qquad+ \frac{4(1+2\alpha)^2}{2 \alpha + 4 \alpha^2}  K(x,\cdot, t) \frac{ |\nabla h|^2}{(|\nabla h|^2 + \beta)^{\frac12 - \alpha}} (\cdot, -t) |\nabla \phi_i|^2  \Big).
\end{align*}
Note that the last integrand is supported in the annulus $A(x,  r/4, r/2)$ and that $|\nabla h|^2 / (|\nabla h|^2 + \beta)^{\frac12 - \alpha} \leq (|\nabla h|^2 + \beta)^{\frac12 + \alpha} = f$.
Moreover, since $\mathbf{p}-1-2-2\alpha  > 0$, we find a sequence $\eps_i \to 0$ such that for some $C = C(\alpha, Y) < \infty$
\begin{multline*}
   \frac{d}{dt} \bigg( e^{2(n-1) r^{-2} t} \int_{U_i}  K (x, \cdot, t) f(\cdot, -t) \phi_i^2 \bigg) \\
   \geq - \eps_i - Cr^{-2} \int_{A(x,r/4,r/2)  \cap \RR}  K(x,\cdot, t) f(\cdot, -t).
\end{multline*}
Let $\eta : [-r^2/2, 0] \to [0,1]$ be a cutoff function such that $\eta(-r^2/2) = 0$, $\eta \equiv 1$ on $[-r^2/4, 0]$ and $|\eta'| < 10r^{-2}$.
Using integration by parts, we get
 \begin{align*}
 f(x,0) &= \int_0^{r^2/2} \frac{d}{dt} \bigg( \eta (t) e^{2(n-1) r^{-2} t} \int_{U_i} K(x,y,t) f(y, - t) \phi_i^2 (y) dg (y) \bigg) dt \displaybreak[1] \\
  &= \int_0^{r^2/2} \eta'(t) e^{2(n-1) r^{-2} t} \int_{U_i} K (x,y, t) f(y, -t) \phi_i^2 (y) dg(y) dt \\
  &\qquad + \int_0^{r^2/2} \eta(t) \frac{d}{dt} \bigg( e^{2(n-1) r^{-2} t} \int_{U_i} K (x,y, t) f(y, -t) \phi_i^2 (y) dg(y) \bigg) dt \displaybreak[1] \\
  &\leq \eps_i r^2 + C r^{-2} \int_{r^2/4}^{r^2/2} \int_{U_i} K (x,y, r^2) f(y,-t) \phi_i^2 (y) dg(y) dt \\
 &\qquad + C r^{-2} \int_0^{r^2/2} \int_{A(x,r/4,r/2)  \cap \RR} K(x,y,t) f(y,-t) dg(y) dt \displaybreak[1] \\
 &\leq \eps_i r^2 + C r^{-2} \int_{r^2/4}^{r^2/2} \int_{B(x,r/2) \cap \RR} \frac{Y}{r^n} f(y,-t) dg(y) dt \\
 &\qquad + C r^{-2} \int_0^{r^2/2} \int_{A(x,r/4,r/2)  \cap \RR} \frac{Y}{t^{n/2}} \exp \bigg({ - \frac{(r/4)^2}{Yt} } \bigg) f(y,-t) dg(y) dt \displaybreak[1] \\
 &\leq \eps_i r^2 + C r^{-n-2} \int_0^{r^2/2} \int_{B(x,r/2) \cap \RR}  f(y,-t) dg(y) dt .
 \end{align*}
Letting $i \to \infty$ yields
 \[ (|\nabla h |^2(x,0) + \beta )^{\frac12 + \alpha}  = f (x) \leq  C r^{-n-2} \int_{-r^2/2}^0 \int_{B(x,r/2) \cap \RR} ( |\nabla h|^2  + \beta)^{\frac12 + \alpha}  . \]
 Letting $\beta \to 0$ and applying H\"older's inequality yields
 \begin{multline*}
  |\nabla h|^{1 + 2 \alpha} (x,0) \leq C r^{-n-2} \int_{-r^2/2}^0 \int_{B(x,r/2) \cap \RR} |\nabla h|^{1 + 2 \alpha} \\
  \leq C  \bigg( r^{-n-2} \int_{-r^2/2}^0 \int_{B(x,r/2) \cap \RR} |\nabla h|^2  \bigg)^{\frac12 + \alpha} . 
 \end{multline*}
 So, using Lemma \ref{Lem:preCYbound} and Definition \ref{Def:tameness}(1),
 \[
  |\nabla h|(x,0) \leq C \bigg( r^{-n-2} \int_{-r^2/2}^0 \int_{B(x,r/2) \cap \RR} |\nabla h|^2  \bigg)^{1/2} 
  \leq C r^{-1} \osc_{B(x,r) \times [-r^2,0]} h.
 \]
This proves the desired bound.
\end{proof}

\section{Integral estimates along families of geodesics}
In this section we reprove several integral estimates for families of geodesics on singular spaces.

\begin{Definition} \label{Def:unittangentbundle}
Let $\XX = (X, d, \RR, g)$ be a singular space.
Consider the unit tangent bundle $S\RR$ consisting of all vectors $v \in T\RR$ of norm $|v| = 1$.
For any $v \in S\RR$ denote by $\gamma_v : (- a_v, b_v) \to \RR$ be the arclength geodesic with $\gamma_v'(0) = v$ such that $a_v, b_v \in (0, \infty]$ are maximal.
We now define
\[ S^* X := S^* \RR := \{ v \in S\RR \;\; : \;\; a_v = b_v = \infty \} \]
to be the subset of all unit vectors that evolve into infinite geodesic lines.
For any subset $U \subset X$ we furthermore set
\[ S^* U := S^* \RR \cap S (U \cap \RR). \]
The standard measure on $S\RR$ is the Liouville measure, i.e. the product measure of $(\RR, g)$ with the spherical measure of total mass $1$ on each fiber and the standard measure on $S^* \RR$ is the restriction of this measure to $S\RR$.
\end{Definition}

Note that for any $v \in S^* \RR$ the geodesic $\gamma_v : \IR \to \RR$ is defined on all of $\IR$.

\begin{Lemma} \label{Lem:SstarRR}
Let $\XX$ be a singular space with mild singularities of codimension $\mathbf{p}_0 > 1$.
Then the subset $S^* \RR \subset S\RR$ is measurable and $S\RR \setminus S^* \RR$ has measure zero.
Moreover, for any $t \in \IR$ the flow $S^* \RR \to S^* \RR$, $v \mapsto \gamma'_v(t)$ is measure preserving.
\end{Lemma}

\begin{proof}
Note that
\begin{equation} \label{eq:SRmeasurezero}
 S \RR \setminus S^* \RR = \bigcup_{k=1}^\infty \{ v \in S \RR \;\; : \;\; a_v, b_v \leq k \} 
\end{equation}
is the union of closed sets and hence measurable.
Let $R > 0$ be fixed, $p \in X$ and $U := B(p, R) \cap \RR$.
Using the notation $b_v \in (0, \infty]$ from Definition~\ref{Def:unittangentbundle}, we define for any $0 \leq t_1 < t_2 \leq \infty$
\[ W_{t_1, t_2} (U) := \{ v \in SU \;\; : \;\; t_1  < b_v \leq t_2 \}. \]
We will now bound the measure of $W_{t_1, t_2} (U)$.
First note that, since the geodesic flow is volume preserving as long as it exists, for any $0 < t \leq t_1$ the map $W_{t_1, \infty} (\RR) \to W_{t_1 - t, \infty} (\RR)$, $v \mapsto \gamma_v (t)$ is volume preserving.
Setting $t = t_1$ yields
\[ | W_{t_1, t_2} (U) | = |\{ \gamma_v(t_1) \;\; : \;\; v \in W_{t_1, t_2}(U) \} |. \]
Since the function $\rrm : X \to [0, \infty)$ is $1$-Lipschitz and $\lim_{t \nearrow b_v} \rrm(\gamma_v(t)) = 0$ if $b_v < \infty$, we find that
\[ \rrm(\gamma_v(t_1)) \leq t_2 - t_1 \textQQqq{for all} v \in W_{t_1, t_2} (U). \]
So for some $1 < \mathbf{p} < \mathbf{p}_0$,
\[ | W_{t_1, t_2} (U) | \leq | B(p,R + t_1) \cap \{ 0 < \rrm \leq t_2 - t_1 \} | \leq C_{\XX, p, \mathbf{p}, R + t_1} (t_2 - t_1)^{\mathbf{p}}. \]
Here $C_{\XX, p, \mathbf{p}, R + t_1} < \infty$ is a constant that may depend on $\XX$, $p$, $\mathbf{p}$ and on $R + t_1$.
The constant $C_{\XX, p, \mathbf{p}, R + t_1}$ can be chosen to be monotone in $R+t_1$.
So for any $N \geq 1$ and $0 < t_0 < \infty$ we have
\[ | \{ v \in SU \;\; : \;\; b_v \leq t_0 \} | \leq \sum_{j = 1}^{N} |W_{(j-1)/N \cdot t_0, j/N \cdot t_0} (U)| \leq N \cdot C_{\XX, p, \mathbf{p}, R+t_0} \Big( \frac{t_0}N \Big)^{\mathbf{p}}. \]
Letting $N \to \infty$ yields that $\{ v \in S(B(p,R) \cap \RR) \;\; : \;\; b_v \leq t_0 \}$ has measure zero.
Letting $R \to \infty$ yields that $\{ v \in S\RR \;\; : \;\; b_v \leq t_0 \}$ has measure zero and using analogous arguments, we obtain that $\{ v \in S\RR \;\; : \;\; a_v, b_v \leq t_0 \}$ has measure zero.
The first claim now follows from (\ref{eq:SRmeasurezero}).
The second claim is clear.
\end{proof}

\begin{Proposition}[cf {\cite[Lemma 1.14]{Colding-vol-conv}}] \label{Prop:segmenttype}
Let $\XX$ be a singular space with mild singularities of codimension $\mathbf{p}_0 > 1$ and let $p \in X$, $l, r > 0$.
Then for any $C^2$ function $f : B(p, r + l) \cap \RR \to \IR$ and any $t \in [0, l]$ we have
\begin{equation} \label{eq:segmenttype1}
\int_{S^*B(p,r)} \Big| (f \circ \gamma_v)'(t) - \frac{f(\gamma_v(l)) - f(\gamma_v (0))}{l} \Big| dv \leq 2l  \int_{B(p, r+l) \cap \RR} |\nabla^2 f| dg
\end{equation}
and
\begin{equation} \label{eq:segmenttype2}
 \int_{S^*B(p,r)} \Big| \langle \nabla f, v \rangle - \frac{f(\gamma_v(l)) - f(\gamma_v (0))}{l} \Big| dv \leq 2l \int_{B(p, r+l) \cap \RR} |\nabla^2 f| dg.
\end{equation}
\end{Proposition}

\begin{proof}
This proposition follows similarly as \cite[Lemma 1.14]{Colding-vol-conv}.
The singularities of $\XX$ don't cause any issues due to Lemma \ref{Lem:SstarRR}.
Inequality (\ref{eq:segmenttype1}) is a direct consequence of (\ref{eq:segmenttype2}) for $t = 0$.
\end{proof}

We also obtain the following segment inequality.

\begin{Proposition}[segment inequality, cf {\cite[Theorem 2.11]{Cheeger-Colding-Cone}}] \label{Prop:segmentinequ}
For any $D < \infty$ there is a $C = C(D) < \infty$ such that the following holds:

Let $\XX$ be a singular space with mild singularities and assume that $\Ric \geq -(n-1)$ on $\RR$.
Then there is an open subset $\mathcal{G}^* \subset \RR \times \RR$ such that for any $(x_1, x_2) \in \mathcal{G}^*$ there is a unique minimizing geodesic $\gamma_{x_1, x_2} : [0, d(x_1,x_2)] \to X$ between $x_1, x_2$ and such that the following holds:
\begin{enumerate}[label=(\alph*)]
\item $(\RR \times \RR) \setminus \mathcal{G}^*$ has measure zero.
\item If $(x_1, x_2) \in \mathcal{G}^*$, then $(x_2, x_1) \in \mathcal{G}^*$.
\item $\gamma_{x_1, x_2} ( [0, d(x_1,x_2)]) \subset \RR$ for all $(x_1, x_2) \in \mathcal{G}^*$.
\item $\gamma_{x_1, x_2}$ varies continuously in $x_1, x_2$.
\item Consider open and bounded subsets $U_1, U_2 \subset X$ such that
\[  \sup_{(x_1, x_2) \in U_1 \times U_2} d(x_1, x_2) < D. \]
Let $f : \RR \to [0, \infty)$ be a non-negative, bounded and Borel measurable function, taking non-negative values.
Then we have
\begin{multline} \label{eq:segmentinequ}
\qquad\qquad \int_{((U_1 \cap \RR) \times (U_2 \cap \RR)) \cap \mathcal{G}^*} \int_0^{d(x_1, x_2)} f (\gamma_{x_1, x_2} (s)) ds dg(x_1) dg(x_2) \\
  \leq C \big( |U_1 \cap \RR| + |U_2 \cap \RR| \big) \int_{\RR} f.
\end{multline}
\end{enumerate}
\end{Proposition}

\begin{proof}
We apply the discussion of section \ref{sec:Riemgeometry} to the (possibly incomplete) Riemannian manifold $(\RR, g)$.
In the following, we will work with the subset $\mathcal{G}^* \subset \RR \times \RR$ from Definition \ref{Def:expmap}.
Then assertions (b)--(d) follow immediately.
Assertion (a) follows from Proposition \ref{Prop:expincomplete}(c), Fubini's Theorem and the fact that $\RR \setminus \mathcal{G}_p \subset Q_p$ is a set of measure zero for all $p \in \RR$.

For assertion (e) it suffices to check that
\begin{equation} \label{eq:halfsegmentinequ}
\int_{((U_1 \cap \RR) \times (U_2 \cap \RR)) \cap \mathcal{G}^*} \int_{d(x_1, x_2)/2}^{d(x_1,x_2)}  f (\gamma_{x_1, x_2} (s)) ds dg(x_1) dg(x_2)
  \leq C  |U_1 \cap \RR|  \int_\RR f,
\end{equation}
for some $C = C(D) < \infty$, as reversing the roles of $U_1, U_2$ and $x_1, x_2$ and adding both resulting inequalities yields (\ref{eq:segmentinequ}).
In order to show inequality (\ref{eq:halfsegmentinequ}) it suffices to show that for any $x_1 \in U_1 \cap \RR$ we have
\begin{equation} \label{eq:halfsegmentx1fixed}
\int_{ (U_2 \cap \RR) \cap \mathcal{G}^*_{x_1}} \int_{d(x_1, x_2)/2}^{d(x_1,x_2)}  f (\gamma_{x_1, x_2} (s)) ds dg(x_2)
  \leq C   \int_\RR f,
\end{equation}
for some $C = C(D) < \infty$.
Using the fact that the exponential map
\[ \exp_{x_1} |_{ \exp^{-1}_{x_1} (U_2 \cap \RR) \cap \mathcal{D}^*_{x_1}}  : \exp^{-1}_{x_1} (U_2 \cap \RR) \cap \mathcal{D}^*_{x_1} \to (U_2 \cap \RR) \cap \mathcal{G}^*_{x_1} \]
is a diffeomorphism (see Proposition \ref{Prop:expincomplete}(d)), we conclude that
\begin{multline*}
 \int_{ (U_2 \cap \RR) \cap \mathcal{G}^*_{x_1}} \int_{d(x_1, x_2)/2}^{d(x_1,x_2)}  f (\gamma_{x_1, x_2} (s)) ds dg(x_2) \\
 = \int_{ \exp^{-1}_{x_1} (U_2 \cap \RR) \cap \mathcal{D}^*_{x_1}} \int_{|v|/2}^{|v|} f \Big(\exp_{x_1} \Big( s \cdot \frac{v}{|v|} \Big) \Big) J_{x_1} (v) ds dv
\end{multline*}
Here $J_{x_1} : \mathcal{D}^*_{x_1} \to \IR$ denotes the Jacobian of $\exp_{x_1}$.
Using Proposition \ref{Prop:Jnonincreasing}, we conclude that for some $C = C(D) < \infty$
\begin{align*}
 \int_{ (U_2 \cap \RR) \cap \mathcal{G}^*_{x_1}} & \int_{d(x_1, x_2)/2}^{d(x_1,x_2)}  f (\gamma_{x_1, x_2} (s)) ds dg(x_2) \displaybreak[1] \\
 &\leq \int_{ \exp^{-1}_{x_1} (U_2 \cap \RR) \cap \mathcal{D}^*_{x_1}} \int_{|v|/2}^{|v|} f \Big(\exp_{x_1} \Big( s \cdot \frac{v}{|v|} \Big) \Big) J_{x_1} (v) ds dv \displaybreak[1] \\
& \leq \int_{ B(0,D) \cap \mathcal{D}^*_{x_1}} |v| \int_{1/2}^{1} f (\exp_{x_1} ( s v ) ) J_{x_1} (  v ) ds dv \displaybreak[1] \\
 &\leq  D \int_{1/2}^{1} \int_{ \mathcal{D}^*_{x_1}}  f (\exp_{x_1} ( s v ) ) \bigg( s \cdot \frac{\sn_{-1} (|v|)}{\sn_{-1} (s |v|)} \bigg)^{n-1} J_{x_1} (  sv ) dv ds \displaybreak[1] \\
& \leq  C D \int_{1/2}^{1} \int_{ \mathcal{D}^*_{x_1}}  f (\exp_{x_1} ( sv ) ) J_{x_1} ( sv ) dv ds \displaybreak[1] \\
 & \leq  2^{n-1} C D  \int_{ \mathcal{D}^*_{x_1}}  f (\exp_{x_1} ( v ) ) J_{x_1} ( v ) dv \\
 & = 2^{n-1} C D \int_{\mathcal{G}^*_{x_1}} f = 2^{n-1} C D \int_{\RR} f.
\end{align*}
This proves (\ref{eq:halfsegmentx1fixed}).
\end{proof}

\section{Almost splitting implies existence of a splitting map}
In this section, we recall the definition of an $\eps$-splitting map and generalize this notion to singular spaces.

\begin{Definition}[$\eps$-splitting, cf {\cite[Definition 1.20]{Cheeger-Naber-Codim4}}] \label{Def:epssplitting}
Let $\eps > 0$, $\XX$ a singular space, $p \in X$ and $r > 0$.
A map $u = (u^1, \ldots, u^k) : B(p,r) \to \IR^k$ is called an \emph{$\eps$-splitting} if
\begin{enumerate}[label=(\arabic*)]
\item $u^l$ is harmonic for all $l = 1, \ldots, k$ (in the sense of Definition \ref{Def:XXharmonic}).
\item For all $x \in B(p,r) \cap \RR$ we have
\begin{equation} \label{eq:epssplittinggradientestimate}
 |\nabla u |(x) := \sup_{v \in T_x \RR, |v|=1} |(D u)(v)| \leq 1 + \eps. 
\end{equation}
\item For all for all $l_1, l_2 = 1, \ldots, k$ we have
\[ r^{-n} \int_{B(p,r) \cap \RR} \big| \langle \nabla u^{l_1}, \nabla u^{l_2} \rangle - \delta_{l_1 l_2} \big|^2 dg < \eps^2. \]
\item For all $l = 1, \ldots, k$ we have
\begin{equation} \label{eq:Hessianestimate}
 r^{-n+2} \int_{B(p, r) \cap \RR} |\nabla^2 u^l |^2 dg < \eps^2. 
\end{equation}
\end{enumerate}
\end{Definition}

The goal of this section is to prove the following proposition.

\begin{Proposition}[cf {\cite[sec 2]{Colding-vol-conv}}, {\cite[Theorem 9.29]{Cheeger-degeneration-book}}, {\cite[Lemma 1.21(2)]{Cheeger-Naber-Codim4}}] \label{Prop:GHsplittingepssplitting}
For any $\eps > 0$, $\mathbf{p}_0 > 3$ and $Y < \infty$ there is a $\delta = \delta (\eps, \mathbf{p}_0, Y) > 0$ such that the following holds:

Let $\XX = (X, d, \RR, g)$ be a singular space with mild singularities of codimension $\mathbf{p}_0$ and assume that $\XX$ is $Y$-tame at scale $\delta^{-1} r$ for some $r > 0$.
Assume moreover that $\Ric \geq - \delta^2 (n-1) r^{-2}$ on $\RR$.
Let $p \in X$ be a point and $k \in \{ 1, \ldots, n \}$ and let $(Z, d_Z, z)$ be a pointed metric space.
Assume that
\begin{equation} \label{eq:GHZX}
 d_{GH} \big( \big( B^X (p, \delta^{-1}r), p \big), \big( B^{Z \times \IR^k} ((z, 0^k), \delta^{-1}r ), (z, 0^k) \big) \big) < \delta r.
\end{equation}
Then there is an $\eps$-splitting map $u : B^X(x, r) \to \IR^k$.
If $k = n$, then $u$ can be chosen such that additionally
\begin{equation} \label{eq:udeps}
 \big| |u (x) | - d(x,p) \big| < \eps r \textQQqq{for all} x \in B^X(p,r). 
\end{equation}
\end{Proposition}

As a preparation, we establish the following lemma.

\begin{Lemma} \label{Lem:trianglebpmsmall}
For any $\eps > 0$ and $Y < \infty$ there is a $\delta = \delta (\eps, Y) > 0$ such that the following holds:

Let $\XX = (X, d, \RR, g)$ be a singular space with mild singularities of codimension $\mathbf{p}_0 > 1$ and assume that $\XX$ is $Y$-tame at scale $\delta^{-1} r$ for some $r > 0$.
Assume moreover that $\Ric \geq - \delta^2 (n-1) r^{-2}$ on $\RR$.
Choose points $p \in X$ and $q_-, q_+ \in \RR$ such that
\[ d(p, q_\pm) > \delta^{-1} r \]
and
\begin{equation} \label{eq:distxqpm}
 \big| \big( d(x, q_-) + d(x, q_+ ) \big) - \big( d(p, q_-) + d(p, q_+) \big) \big| < \delta r \textQQqq{for all} x \in B(p,r).
\end{equation}
Then, for $b_{\pm} (x) := d(x, q_\pm)$ and the associated signed measure $d\mu_{\triangle b_\pm}$ (see Proposition \ref{Prop:weakLaplaciandistance}) we have
\begin{equation} \label{eq:Lapbpmmass}
 r \int_{B(p,r) \cap \RR} d|\mu_{\triangle b_\pm} | < \eps r^n. 
\end{equation}
\end{Lemma}

\begin{proof}
Without loss of generality, we may assume that $r = 1$.
Set $b_\pm (x) := d(x, q_\pm)$ and set $b_0 := b_- (p) + b_+ (p)$.
Then by (\ref{eq:distxqpm}), we have
\[ |b_+ + b_- - b_0 | < \delta \textQQqq{on} B(p,1). \]
Moreover, using Laplace comparison (see Proposition \ref{Prop:LaplacecomparisonXX}), we have for small $\delta$
\begin{equation} \label{eq:bpmLapcomp}
 d \mu_{\triangle b_\pm} \leq (n-1) \frac{\cosh (\delta b(\cdot))}{\delta^{-1} \sinh (\delta b(\cdot))} dg \leq 2 (n-1) \delta dg \textQQqq{on} B(p,2) \cap \RR.
\end{equation}
Now apply Proposition \ref{Prop:Uiphii} to $B(p,2)$ to obtain the sequence $s_i \to 0$ the subsets $U_i$ and the cutoff functions $\phi_i : U_i \to [0,1]$.
Next recall that the function $\rrm : X \to [0, \infty)$ is $1$-Lipschitz.
As in the proof of Lemma \ref{Lem:finddomains}, we can find a $C^1$ function $f \in C^1 (\RR)$ such that
\[ \tfrac12 \rrm < f < 2 \rrm \textQQqq{and} |\nabla f | < 2. \]
Let $F : [0, \infty) \to [0,1]$ be a smooth function such that $F \equiv 0$ on $[0,2]$ and $F \equiv 1$ on $[4, \infty)$.
For every $i = 1, 2, \ldots$, we define $\eta_i \in C^2 (\RR)$ by
\[ \eta_i(x) := F(s_i^{-1} f(x)) \]
Then $\supp \eta_i \cap B(p,2)  \subset \{ \rrm > s_i \} \cap B(p,2)  \subset U_i$, $\eta_i \equiv 1$ on $\{ \rrm > 8 s_i \}$ and $|\nabla \eta_i | < C s_i^{-1}$.
It follows that $\eta_i \phi_i \in C^2_c (U_i)$ and $\eta_i \phi_i \to 1$ pointwise on $B(p, 1) \cap \RR$ as $i \to \infty$.
So, by Proposition \ref{Prop:weakLaplaciandistance} we have
\begin{align*}
 \int_{\RR} \eta_i \phi_i & (d\mu_{\triangle  b_- } + d\mu_{\triangle b_+} )  = - \int_{\RR} \nabla (\eta_i \phi_i) \nabla (b_- + b_+ - b_0) \\
 &= - \int_{\RR} (\eta_i \nabla \phi_i) \nabla (b_- + b_+ - b_0) - \int_{\RR} (\nabla \eta_i) \phi_i \nabla (b_- + b_+ - b_0) \\
 &=  \int_{\RR} (\nabla \eta_i \nabla \phi_i + \eta_i \triangle \phi_i)  (b_- + b_+ - b_0) - \int_{\RR} (\nabla \eta_i) \phi_i \nabla (b_- + b_+ - b_0).
\end{align*}
Fix some $1 < \mathbf{p} < \mathbf{p}_0$.
So for a generic constant $C_* < \infty$ that may depend on $\XX$ or $p$, but not on $i$, and some uniform $C = C(Y) < \infty$ we have
\begin{multline*}
 \bigg| \int_{\RR} \eta_i \phi_i (d\mu_{\triangle  b_- } + d\mu_{\triangle b_+} ) \bigg| \leq C_* s_i^{-1} |\{ \rrm \leq 8s_i \} \cap \RR| \cdot \sup_{B(p,2)} |b_- + b_+ - b_0| \\ + C  \sup_{B(p,2)} |b_- + b_+ - b_0| + C_* s_i^{-1}  |\{ \rrm \leq 8s_i \} \cap \RR| \\
 \leq C_* s_i^{-1 + \mathbf{p}} \delta + C  \delta  + C_* s_i^{-1+\mathbf{p}}.
\end{multline*}
It follows that
\[ \limsup_{i \to \infty} \bigg| \int_{\RR}  \eta_i \phi_i (d\mu_{\triangle  b_- } + d\mu_{\triangle b_+} ) \bigg| \leq C \delta. \]
In combination with (\ref{eq:bpmLapcomp}), this yields
\[ \limsup_{i \to \infty} \int_{\RR} \eta_i \phi_i d \big| \mu_{\triangle  b_- } + \mu_{\triangle b_+} \big| \leq C \delta. \]
Since $\eta_i \phi_i \to 1$ pointwise on $B(p,1) \cap \RR$, we obtain by Fatou's Lemma that
\[ \int_{B(p,1) \cap \RR} d \big| \mu_{\triangle  b_- } + \mu_{\triangle b_+}  \big| \leq C \delta. \]
The bound (\ref{eq:Lapbpmmass}) now follows from the fact that
\[ d|\mu_{\triangle b_\pm} | \leq  d|\mu_{\triangle  b_- } + \mu_{\triangle b_+} |+ d (\mu_{\triangle b_+} )_+ + d (\mu_{\triangle b_-} )_+ \]
and (\ref{eq:bpmLapcomp}) for sufficiently small $\delta$.
Here the last two terms denote the positive parts of the measures $\mu_{\triangle b_\pm}$.
\end{proof}

We can now derive a partial version of Proposition \ref{Prop:GHsplittingepssplitting}.

\begin{Lemma} \label{Lem:exsplittingver1}
For any $\eps > 0$, $\mathbf{p}_0 > 3$ and $Y < \infty$ there is a $\delta = \delta (\eps, \mathbf{p}_0, Y) > 0$ such that the following holds:

Let $\XX = (X, d, \RR, g)$ be a singular space with mild singularities of codimension $\mathbf{p}_0$ and assume that $\XX$ is $Y$-tame at scale $\delta^{-1} r$ for some $r > 0$.
Assume moreover that $\Ric \geq - \delta^2 (n-1) r^{-2}$ on $\RR$.
Let $p \in X$ be a point, let $k \in \{ 1, \ldots, n \}$, let $(Z, d_Z, z)$ be a pointed metric space and assume that (\ref{eq:GHZX}) holds.
Then there is a continuous map $u = (u^1, \ldots, u^k) : B^X(p, r) \to \IR^k$ such that the following holds:
\begin{enumerate}[label=(\alph*)]
\item There is a map $\zeta : B^X (p,r) \to Z$ with $d(\zeta (p), z) < \eps r$ such that the map $(\zeta, u^1, \ldots, u^k) : B^X(p,r) \to Z \times \IR^k$ is an $\eps r$-Gromov-Hausdorff approximation between $B^X(p,r)$ and $B^{Z \times \IR^k}((z,u(p)), r)$.
\item $u^l$ is harmonic for each $l = 1, \ldots, k$.
\item For each $l = 1, \ldots, k$, we have
\[  r^{-n} \int_{B^X(p,r) \cap \RR} \big| |\nabla u^l |^2 - 1 | < \eps^2. \]
\end{enumerate}
\end{Lemma}

\begin{proof}
Without loss of generality, we may assume that $r = 1$.
The constant $\delta > 0$ will be determined in the course of the proof.
Choose points $q_{\pm}^1, \ldots, q_{\pm}^k \in \RR$ that correspond to points $(z,(\pm \frac12 \delta^{-1}, 0, \ldots, 0)), \ldots, (z, (0, \ldots, 0, \pm \frac12 \delta^{-1}))$ in a $\delta$-Gromov-Hausdorff approximation of $B^X(x, \delta^{-1})$ as in (\ref{eq:GHZX}) and set
\[ b_\pm^l (x) := d(x, q_{\pm}^l ) \]
for each $l = 1, \ldots, k$.
By Lemma \ref{Lem:trianglebpmsmall}, we have that
\[ \int_{\ov{B(p,10)} \cap \RR} d | \mu_{\triangle b^l_\pm}| < \Psi( \delta | Y), \]
where $\Psi (\delta | Y)$ denotes a constant whose value goes to $0$ as $\delta \to 0$ if $Y$ is kept fixed.
Using the second tameness property in Definition \ref{Def:tameness} and Corollary \ref{Cor:harmonic}, we find a harmonic $u : B^X (p, 10) \to \IR^k$ such that for all $l = 1, \ldots, k$
\begin{align}
 \int_{B(p,10) \cap \RR} |\nabla (u^l - b_+^l) |^2 &< \Psi (\delta | Y), \notag \\
  \int_{B(p,10) \cap \RR} |u^l - b_+^l |^2 &< \Psi (\delta | Y). \label{eq:ulbplclose}
\end{align}
Assertion (b) follows immediately.
For assertion (c) observe that $| \nabla b_+^l | = 1$ away from a set of measure zero.
So
\begin{multline*}
 \int_{B(p,1) \cap \RR} \big| |\nabla u^l |^2 - 1 \big| \leq \int_{B(p,1) \cap \RR} | \nabla u^l - \nabla b^l_+ |^2 + 2 \int_{B(p,1) \cap \RR} \big| \nabla b^l_+ \cdot \nabla (u^l - b^l_+) \big| \\
 \leq \Psi (\delta | Y) + C \bigg( \int_{B(p,1) \cap \RR} \big| \nabla (u^l - b^l_+) \big|^2 \bigg)^{1/2} \leq \Psi (\delta |Y)
\end{multline*}

It remains to show assertion (a).
For this, observe that by Corollary \ref{Cor:ChengYau} there is a uniform constant $C < \infty$ such that  $|\nabla u^l | < C$ on $B(p,4) \cap \RR$.
As $(X,d)$ is the metric completion of $(\RR, g)$, this implies that $u^l$ is $C$-Lipschitz on $B(p,2)$ and thus that $u^l - b^l_+$ is $(C+1)$-Lipschitz on $B(p,2)$.
It follows that if $|u^l - b^l_+| (y) > a$ for some $y \in B(p,1)$, then $|u^l - b^l_+|  > a/2$ on $B(p, \frac12 a (C+1)^{-1})$.
So, using (\ref{eq:ulbplclose}), this shows that
\[ |u^l - b^l_+| < \Psi (\delta | Y) \textQQqq{on} B^X(p,1). \]
This finishes the proof of assertion (a).
\end{proof}

In order to obtain the $L^2$-Hessian estimate (\ref{eq:Hessianestimate}), we will use the following lemma.

\begin{Lemma} \label{Lem:Hessianbound}
For any $Y < \infty$ there is a constant $C = C(Y) < \infty$ such that the following holds:

Let $\XX = (X, d, \RR, g)$ be a singular space with singularities of codimension $\mathbf{p}_0 > 2$ and assume that $\XX$ is $Y$-tame at scale $2 r$ for some $r > 0$.
Assume moreover that $\Ric \geq - (n-1) \kappa r^{-2}$ on $\RR$ for some $0 \leq \kappa \leq 1$.
Let $p \in X$ be a point and consider a harmonic function $u : B(p, 2r) \to \IR$.
Then
\[ r^{2-n} \int_{B(p,r) \cap \RR} | \nabla^2 u |^2 \leq C \inf_{a \in \IR} \bigg( r^{-n} \int_{B(p,2r) \cap \RR} \big| |\nabla u|^2 - a \big| + C \kappa a  \bigg). \]
\end{Lemma}

\begin{proof}
Without loss of generality, we may assume that $r = 1$.
By passing to a slightly smaller ball, we may assume that $|u|, |\nabla u| < C_*$ on $B(p,2) \cap \RR$ for some $C_* < \infty$.
The derivative bound can be assumed since by Corollary \ref{Cor:harmonic}, the function $u$ is locally Lipschitz.
By local elliptic regularity and by passing to a yet slightly smaller ball, we may also assume that $|\nabla^2 u | < C_* \rrm^{-1}$ on $B(p,2) \cap \RR$.

Invoke Proposition \ref{Prop:Uiphii} for the ball $B(p, 2)$ and consider the sequence $s_i \to 0$, the subsets $U_i \subset B(p,2) \cap \RR$ and the cutoff functions $\phi_i : X \to [0,1]$.
Then, using Bochner's identity,
\begin{equation} \label{eq:Bochner}
 \triangle \big( |\nabla u|^2 - a \big) = 2 |\nabla^2 u |^2 + 2 \Ric (\nabla u, \nabla u) \geq 2 |\nabla^2 u|^2 - 2(n-1) \kappa |\nabla u|^2, 
\end{equation}
we obtain for any $a \in \IR$ that for some uniform constant $C = C(Y) < \infty$ and some constant $C_{**} < \infty$, which may depend on $C_*$ and $a$, but not on $i$, and some $2 < \mathbf{p} < \mathbf{p}_0$.

\begin{align*}
2 \int_{U_i} & \big( |\nabla^2 u |^2 - (n-1) \kappa |\nabla u|^2 \big) \phi_i \leq \int_{U_i} \triangle \big( |\nabla u |^2 - a \big) \phi_i \\
&\leq \int_{\partial_s U_i} |\nabla^2 u| |\nabla u| \phi_i   - \int_{U_i}  \nabla \big( |\nabla u |^2 - a \big) \nabla \phi_i \\
&\leq \int_{\partial_s U_i} \Big( |\nabla^2 u| |\nabla u| \phi_i + \big| |\nabla u|^2 - a \big| |\nabla \phi_i| \Big)  + \int_{U_i} \big( |\nabla u |^2 - a \big) \triangle \phi_i \\
&\leq (C^2_* + (C^2_* + a)Y) \mathcal{H}^{n-1} (\partial_s U_i) (s_i^{-1}+1) + C \int_{B(p, 2) \cap \RR} \big| |\nabla u |^2 - a \big| \\
&\leq C_{**} s_i^{\mathbf{p}-1} (s_i^{-1}+1) + C \int_{B(p, 2) \cap \RR} \big| |\nabla u |^2 - a \big|.
\end{align*}
Since $\phi_i \to 1$ on $B(p,1) \cap \RR$, we obtain as $i \to \infty$
\begin{multline*}
 \int_{B(p,1) \cap \RR} |\nabla^2 u |^2 \leq C \int_{B(p,2) \cap \RR} \big| |\nabla u|^2 - a \big| + C \kappa \int_{B(p,2) \cap \RR} |\nabla u |^2  \\
 \leq C \int_{B(p,2) \cap \RR} \big| |\nabla u|^2 - a \big| + C \kappa a.
\end{multline*}
This finishes the proof.
\end{proof}

For the gradient estimate (\ref{eq:epssplittinggradientestimate}), we will use the following lemma.

\begin{Lemma} \label{Lem:stronggradientestimate}
For any $\mathbf{p}_0 > 2$ and $Y < \infty$ there is a constant $C = C(\mathbf{p}_0, Y) < \infty$ such that the following holds:

Let $\XX = (X, d, \RR, g)$ be a singular space with singularities of codimension $\mathbf{p}_0$ and assume that $\XX$ is $Y$-tame at some scale $r > 0$.
Assume moreover that $\Ric \geq - (n-1) \kappa r^{-2}$ on $\RR$ for some $0 \leq \kappa \leq 1$.
Let $p \in X$ be a point and consider a harmonic function $u : B(p, 2r) \to \IR$.
Then for any $a  \geq 0$ and $y \in B(p,r) \cap \RR$ we have
\begin{multline*}
 |\nabla u |^2 (y) \leq a + C \kappa a + C  a^{1/2} \bigg( r^{2-n} \int_{B(p,2r) \cap \RR} | \nabla^2 u |^2  \bigg)^{1/2} \\
 +  C r^{2-n} \int_{B(p,2r) \cap \RR} | \nabla^2 u |^2 + C r^{-n} \int_{B(p,2r) \cap \RR} \big| |\nabla u|^2 - a \big|. 
\end{multline*}
\end{Lemma}

\begin{proof}
Without loss of generality, we may assume that $r = 1$ and, as in the proof of Lemma \ref{Lem:Hessianbound}, we may assume that $|u|, |\nabla u|$ are uniformly bounded.
Also fix some $2 < \mathbf{p} < \mathbf{p}_0$.

Let $y \in B(p,1) \cap \RR$ and apply Proposition \ref{Prop:Uiphii} to $B(y, 1/2)$ to obtain the sequence $s_i \to 0$, the subsets $U_i \subset B(y,1/2) \cap \RR$ and the cutoff functions $\phi_i$.
We can again deduce an estimate of the form $|\nabla^2 u | < C_* \rrm^{-1}$ on $B(p, 2) \cap \RR$ and a similar local gradient estimate for the heat kernel $K (y, \cdot, t)$ on $\partial_s U_i$ for large $i$.
By Bochner's identity (\ref{eq:Bochner}), we have for some uniform generic constant $C = C(Y) < \infty$ and some generic constant $C_{**} < \infty$, which may depend on $\XX$, $p$, $\mathbf{p}$ and $C_*$, that for large $i$
\begin{align*}
 \frac{d}{dt} \int_{U_i} & K(y, \cdot, t) \big( |\nabla u |^2 - a \big) \phi_i  = \int_{U_i} \triangle K(y, \cdot, t) \big( |\nabla u |^2 - a \big) \phi_i  \\
 &\geq - \int_{\partial_s U_i} |\nabla K(y, \cdot, t) | \big| |\nabla u |^2 - a \big| \phi_i 
  - \int_{U_i} \nabla K(y, \cdot, t) \nabla \big( \big( |\nabla u |^2 - a \big) \phi_i \big) \displaybreak[1] \\
  &\geq - C_{**} \mathcal{H}^{n-1} (\partial_s U_i) s_i^{-1} - \int_{\partial_s U_i} K(y, \cdot, t)  \big|\nabla \big( \big( |\nabla u |^2 - a \big) \phi_i \big) \big| \displaybreak[1] \\
  &\qquad + \int_{U_i} K(y, \cdot, t) \triangle \big( \big( |\nabla u |^2 - a \big) \phi_i \big) \displaybreak[1] \\
  &\geq - C_{**} s_i^{\mathbf{p}-2} - C_{**} \mathcal{H}^{n-1} (\partial_s U_i) s_i^{-1} +  \int_{U_i} K(y, \cdot, t) \Big( \triangle \big( |\nabla u |^2 - a \big) \phi_i \displaybreak[1] \\
  &\qquad\qquad\qquad\qquad\qquad\qquad + 2 \nabla \big( |\nabla u |^2 - a \big) \nabla \phi_i +  \big( |\nabla u |^2 - a \big) \triangle \phi_i \Big) \displaybreak[1] \\
  &\geq - C_{**} s_i^{\mathbf{p}-2}  +  \int_{U_i} K(y, \cdot, t) \Big( 2 |\nabla^2 u|^2 \phi_i - 2(n-1) \kappa |\nabla u|^2 \phi_i  \displaybreak[1] \\
  &\qquad\qquad\qquad\qquad\qquad\qquad  - 4 | \nabla^2 u | \cdot |\nabla u | \cdot |\nabla \phi_i| -  \big| |\nabla u |^2 - a \big| \cdot | \triangle \phi_i | \Big) \displaybreak[1] \\
 &\geq - C_{**} s_i^{\mathbf{p}-2} - 2(n-1) \kappa  \int_{U_i} K(y, \cdot, t) \big( |\nabla u |^2 - a \big) \phi_i - 2 (n-1) \kappa \cdot a \\
 &\qquad\qquad - \frac{CY}{t^{n/2}} \exp \bigg({ - \frac{1}{16Yt} }\bigg) \int_{A(y, 1/4, 1/2) \cap \RR} \Big(  |\nabla u | \cdot |\nabla^2 u|   + \big| |\nabla u |^2 - a \big| \Big) \displaybreak[1] \\
 &\geq - C_{**} s_i^{\mathbf{p}-2} - 2(n-1) \kappa  \int_{U_i} K(y, \cdot, t) \big( |\nabla u |^2 - a \big) \phi_i - 2 (n-1) \kappa \cdot a \displaybreak[1] \\
 &\qquad\qquad  -  C \bigg(  \int_{B(p,2) \cap \RR} | \nabla u |^2  \bigg)^{1/2} \bigg(  \int_{B(p,2) \cap \RR} | \nabla^2 u |^2  \bigg)^{1/2} \\
 &\qquad\qquad - C  \int_{B(p,2) \cap \RR}\big| |\nabla u |^2 - a \big|  \displaybreak[1]   \\
 &\geq - C_{**} s_i^{\mathbf{p}-2} - 2(n-1) \kappa  \int_{U_i} K(y, \cdot, t) \big( |\nabla u |^2 - a \big) \phi_i - 2 (n-1) \kappa \cdot a \\
 &\qquad\qquad -  C \bigg( a  +  \int_{B(p,2) \cap \RR} \big| | \nabla u |^2 - a \big| \bigg)^{1/2} \bigg(  \int_{B(p,2) \cap \RR} | \nabla^2 u |^2  \bigg)^{1/2}  \\
 &\qquad\qquad  - C  \int_{B(p,2) \cap \RR}\big| |\nabla u |^2 - a \big|   \displaybreak[1]   \\
  &\geq - C_{**} s_i^{\mathbf{p}-2} - 2(n-1) \kappa  \int_{U_i} K(y, \cdot, t) \big( |\nabla u |^2 - a \big) \phi_i - C \kappa  a \\
  & \qquad\qquad - C a^{1/2} \bigg(  \int_{B(p,2) \cap \RR} | \nabla^2 u |^2  \bigg)^{1/2} \\
  &\qquad\qquad - C  \int_{B(p,2) \cap \RR}\Big( |\nabla^2 u|^2 + \big| |\nabla u |^2 - a \big| \Big) .
\end{align*}
So
\begin{multline*}
 \frac{d}{dt} \bigg( e^{2(n-1) \kappa t} \int_{U_i}  K(y, \cdot, t) \big( |\nabla u |^2 - a \big) \phi_i \bigg) \geq - C_{**} s_i^{\mathbf{p}-2}  - C \kappa  a \\
    - C a^{1/2} \bigg(  \int_{B(p,2) \cap \RR} | \nabla^2 u |^2  \bigg)^{1/2} 
  - C  \int_{B(p,2) \cap \RR}\Big( |\nabla^2 u|^2 + \big| |\nabla u |^2 - a \big| \Big).
\end{multline*}
Integration over $t$ from $0$ to $1$ and applying H\"older's inequality yields
\begin{multline*}
  \big( |\nabla u |^2 - a \big)(y) - e^{2(n-1) \kappa} \int_{U_i} K(y, \cdot, 1) \big( |\nabla u |^2 - a \big) \phi_i  
 \leq C_{**} s_i^{\mathbf{p}-2} + C \kappa a \\
  + C a^{1/2} \bigg(  \int_{B(p,2) \cap \RR} | \nabla^2 u |^2  \bigg)^{1/2} + C \int_{B(p,2) \cap \RR} \Big( |\nabla^2 u |^2 + \big| |\nabla u |^2 - a \big| \Big).
\end{multline*}
The claim now follows using the fact that $K(y, \cdot, 1) \leq Y$ and letting $i \to \infty$.
\end{proof}

\begin{Lemma} \label{Lem:gradienboundalldirections}
For any $\mathbf{p}_0 > 2$ and $Y < \infty$ there is a constant $C= C(\mathbf{p}_0, Y) < \infty$ such that the following holds:

Let $\XX = (X, d, \RR, g)$ be a singular space with singularities of codimension $\mathbf{p}_0$ and assume that $\XX$ is $Y$-tame at some scale $r > 0$.
Assume moreover that $\Ric \geq - (n-1) \kappa r^{-2}$ on $\RR$ for some $0 \leq \kappa \leq 1$.
Let $p \in X$ be a point and consider a vector valued function $u = (u^1, \ldots, u^k) : B(p,2r) \to \IR^k$ whose component functions $u^1, \ldots, u^k$ are harmonic.
Then for any $y \in B(p, r) \cap \RR$ we have
\begin{multline} \label{eq:gradienboundalldirections}
 |\nabla u|^2 (y) \leq 1  + C  \sum_{l=1}^k \bigg( \bigg( r^{2-n} \int_{B(p,2r) \cap \RR} |\nabla^2 u^l|^2 \bigg)^{1/2} +  r^{2-n} \int_{B(p,2r) \cap \RR} |\nabla^2 u^l|^2  \bigg) \\
  + C r^{-n} \sum_{i,j=1}^k \int_{B(p,2r) \cap \RR} \big| \langle \nabla u^i, \nabla u^j \rangle - \delta_{ij} \big| + C \kappa. 
\end{multline}
Here $|\nabla u| (y) := \max_{v \in T_y \RR, |v| =1} |(\nabla u)(v)|$.
\end{Lemma}

\begin{proof}
It suffices to show that for any $w \in \IR^k$, $|w| = 1$, setting $u^* := \langle w, u \rangle = w^1 u^1 + \ldots + w^k u^k$, the quantity $|\nabla u^* |^2 (y)$ is bounded by the right-hand side of (\ref{eq:gradienboundalldirections}).
Observe that $u^*$ is harmonic,
\[ |\nabla^2 u^* | \leq \sum_{l=1}^k |\nabla^2 u^l | \]
and
\begin{multline*}
 \big| |\nabla u^*|^2 - 1 \big| = \Big| \sum_{i,j=1}^k w_i w_j \langle \nabla u^i, \nabla u^j \rangle - 1 \Big| = \Big| \sum_{i,j=1}^k w_i w_j \big( \langle \nabla u^i, \nabla u^j \rangle - \delta_{ij} \big) \Big| \\
 \leq \sum_{i,j=1}^k \big| \langle \nabla u^i, \nabla u^j \rangle - \delta_{ij} \big|.
\end{multline*}
So the desired bound follows from Lemma \ref{Lem:stronggradientestimate}.
\end{proof}

We can finally prove Proposition \ref{Prop:GHsplittingepssplitting}.

\begin{proof}[Proof of Proposition \ref{Prop:GHsplittingepssplitting}]
Without loss of generality, we may assume that $r =1$.
Apply Lemma \ref{Lem:exsplittingver1} for $r \leftarrow 16$ to obtain $u = (u^1, \ldots, u^k) : B(p, 16) \to \IR^k$.
Replacing $u$ by $u - u(p)$, we may assume that $u(p) = 0^k$.
Using assertion (c) of Lemma \ref{Lem:exsplittingver1} and Lemma \ref{Lem:Hessianbound}, we find that
\begin{equation} \label{eq:Hessianboundinproof}
 \int_{B(p,8) \cap \RR} |\nabla^2 u |^2 < \Psi (\delta | Y). 
\end{equation}
So for sufficiently small $\delta$, this implies item (4) of Definition \ref{Def:epssplitting} for the restriction $u |_{B(p,1)}$.
It remains to verify items (2) and (3).
Next, we will show that for any $l_1, l_2 = 1, \ldots, k$, $l_1 \neq l_2$
\begin{equation} \label{eq:almostorthogonal}
 \int_{B(p,4) \cap \RR} \big| \langle \nabla u^{l_1}, \nabla u^{l_2} \rangle \big|^2 < \Psi (\delta | Y). 
\end{equation}
This bound together with Lemma \ref{Lem:exsplittingver1}(c) will then imply item (3) of Definition \ref{Def:epssplitting} and, using Lemma \ref{Lem:gradienboundalldirections}, item (2).
Thus we will find that $u |_{B(p,1)}$ is an $\eps$-splitting map for sufficiently small $\delta$.
The bound (\ref{eq:udeps}) will be verified at the end of the proof.

We follow the lines of the proof of \cite[Lemma 2.9]{Colding-vol-conv} to show (\ref{eq:almostorthogonal}).
Fix $l_1, l_2 \in \{ 1, \ldots, k \}$, $l_1 \neq l_2$.
Note first that by Proposition \ref{Prop:segmenttype}, (\ref{eq:Hessianboundinproof}) and H\"older's inequality we have for $i = 1,2$
\[ \int_{S^* B(p,4)} \Big| \langle \nabla u^{l_i} , v \rangle - \big( u^{l_i} (\gamma_v(1)) - u^{l_i} ( \gamma_v (0)) \big) \Big| < \Psi (\delta | Y). \]
Let $\theta > 0$ be a parameter whose value we will determine later and set
\[ C^*_\theta := \{ v \in S^* B(p,4) \;\; : \;\; \angle (v, \nabla u^{l_1} (\gamma_v(0)) ) < \theta \}. \]
Then also
\[ \int_{C^*_\theta} \Big| \langle \nabla u^{l_i} , v \rangle - \big( u^{l_i} (\gamma_v(1)) - u^{l_i} ( \gamma_v (0)) \big) \Big| < \Psi (\delta | Y). \]
Note that, by Lemma \ref{Lem:stronggradientestimate}, Lemma \ref{Lem:exsplittingver1}(c) and (\ref{eq:Hessianboundinproof}), the two terms in the difference in this integrand are bounded by a constant that only depends on $Y$.
So
\begin{equation} \label{eq:almost1onCtheta}
 \int_{C^*_\theta} \Big|  \langle \nabla u^{l_i} , v \rangle^2 - \big( u^{l_i} (\gamma_v(1)) - u^{l_i} ( \gamma_v (0)) \big)^2 \Big| <  \Psi (\delta  | Y).
\end{equation}
On the other hand, using the Pythagorean Theorem and Lemma \ref{Lem:exsplittingver1}(a), we have for any $y_1, y_2 \in B(p,2)$ that
\[ \big( u^{l_1} (y_1) - u^{l_1} (y_2) \big)^2 + \big( u^{l_2} (y_1) - u^{l_2} (y_2) \big)^2 <  d^2 (y_1, y_2) + \Psi (\delta | Y). \]
It follows that
\[ \int_{C^*_\theta}  \Big( \big( u^{l_1} (\gamma_v(1)) - u^{l_1} ( \gamma_v (0)) \big)^2 + \big( u^{l_2} (\gamma_v(1)) - u^{l_2} ( \gamma_v (0)) \big)^2 \Big) < |C^*_\theta| + \Psi(\delta | Y). \]
Combining this with (\ref{eq:almost1onCtheta}) yields
\[  \int_{C^*_\theta}  \big( \langle \nabla u^{l_1} , v \rangle^2 + \langle \nabla u^{l_2} , v \rangle^2 \big)  < |C^*_\theta| +  \Psi(\delta  | Y). \]
So we obtain, using $\langle \nabla u^{l_1}, v \rangle > |\nabla u^{l_1} | \cos \theta$ for $v \in C^*_\theta$ and Lemma \ref{Lem:exsplittingver1}(c) that
\begin{equation} \label{eq:Cstarnabl2}
 \int_{C^*_\theta} \langle \nabla u^{l_2} , v \rangle^2 < \Psi(\theta) |C^*_\theta| + \Psi(\delta | Y).
\end{equation}
Set now
\[ C_\theta := \{ v \in S B(p,4) \;\; : \;\; \angle (v, \nabla u^{l_1} (\gamma_v(0)) ) < \theta \}. \]
By Lemma \ref{Lem:SstarRR} the subset $C_\theta \setminus C^*_\theta$ has measure zero.
So by (\ref{eq:Cstarnabl2}) we have
\[ \int_{C_\theta} \langle \nabla u^{l_2} , v \rangle^2 < \Psi(\theta) |C_\theta| + \Psi(\delta | Y). \]
So if we choose first $\theta$ and then $\delta$ sufficiently small, we obtain (\ref{eq:almostorthogonal}).

Finally, we need to verify (\ref{eq:udeps}) in the case in which $k = n$ and $\delta$ is sufficiently small.
To see this, let $x \in B(p, 1) \cap \RR$, set $d := d(x, p)$ and assume that $\eps < 1$.
By the segment inequality, Proposition \ref{Prop:segmentinequ}, (\ref{eq:Hessianboundinproof}), (\ref{eq:almostorthogonal}) and Lemma \ref{Lem:exsplittingver1}(c), we can find points $p' \in B(p,\eps / 10) \cap \RR$ and $x' \in B(x, \eps / 10) \cap \RR$ such that the following holds:
The arclength minimizing geodesic $\gamma := \gamma_{p', x'} : [0,a] \to \RR$ between $p', q'$ is contained in $\RR$, for all $l = 1, \ldots, n$
\begin{equation} \label{eq:L1secondderivativesmall}
  \int_0^a \big| (u^l \circ \gamma)''(s) \big| ds \leq \int_0^a |\nabla^2 u^l | (\gamma(s)) ds < \Psi (\delta | \eps, Y) 
\end{equation}
and for all $l_1, l_2 = 1, \ldots, n$
\begin{equation} \label{eq:almostonb}
 \big|  \langle \nabla u^{l_1} (p'), \nabla u^{l_2} (p') \rangle - \delta_{l_1 l_2} \big| < \Psi (\delta | \eps, Y). 
\end{equation}
Since $a < 2$, the bound (\ref{eq:L1secondderivativesmall}) implies
\[ \big| u^l (x') - u^l (p') - a \langle \nabla u^l (p'), \gamma'(0) \rangle \big| < \Psi (\delta | \eps, Y). \]
So by (\ref{eq:almostonb}) we have
\[ \big| | u (x') - u(p') |^2 - a^2 \big| < \Psi (\delta | \eps, Y). \]
Now note that
\[ | a - d | = | d(x',p') - d(x,p) | <  \eps / 5. \]
and, $|u (p')| \leq (1+\eps) d(p,p') < \eps / 5$, as well as $|u(x') - u(x) | \leq (1+ \eps) d(x,x') < \eps / 5$.
So, we obtain that
\[ \big| | u(x)| - d \big| < 3 \eps / 5 + \Psi (\delta | \eps, Y). \]
This implies (\ref{eq:udeps}) for small $\delta$.
\end{proof}

\section{Volume Stability}
In this section we prove the first main property of singular spaces, Theorem~\ref{Thm:volconv}, which states that Gromov-Hausdorff closeness to a Cartesian product of a metric space with $\IR^n$ implies closeness of the volume to the volume of Euclidean space.
Our proof follows the lines of \cite[sec 2]{Colding-vol-conv}.

We first need the following lemma.

\begin{Lemma} \label{Lem:eps0splitting}
For any $Y < \infty$ there is an $\eps_0 = \eps_0 (Y) > 0$ such that the following holds:

Let $\XX$ be a singular space with mild singularities of codimension $\mathbf{p}_0 > 1$ that is $Y$-tame at some scale $r > 0$ and let $p \in X$ be a point.
If $u : B(p,r) \to \IR^n$ is an $\eps_0$-splitting map and $y \in \IR^n$ such that $|u(p)- y| \geq r$, then
\[ u(B(p,r)) \cap B(y,|u(p) - y|) \neq \emptyset. \]
\end{Lemma}

\begin{proof}
Without loss of generality, we may assume that $r = 1$ and $u (p ) = 0^n$.
Then, assuming $\eps_0$ to be small enough depending on $Y$, we need to find some $z \in B(p,r)$ such that $2\langle u(z), y \rangle > |u(z)|^2$.
Using Proposition \ref{Prop:segmenttype}, we have
\[ \int_{S^*B(p,1/2)} \Big| \langle \nabla u^l, v \rangle - 2 \big( u^l (\gamma_v ( 1/2)) - u^l (\gamma_v (0) \big) \Big| < \Psi( \eps_0 | Y) \]
for all $l =1, \ldots, n$ and
\[ \int_{B(p,1)} \big| \langle \nabla u^{l_1}, \nabla u^{l_2} \rangle - \delta_{l_1 l_2} \big| < \Psi(\eps_0) \]
for all $l_1, l_2 = 1, \ldots, n$.
Note that by $Y$-tameness, we have $|B(p, r') \cap \RR | > Y^{-1} r^{\prime n}$ for all $0 < r' < 1$.
So there is a point $x \in B(p, 1/10) \cap \RR$ and vector $v \in T_x \RR$ such that $S_x \RR \setminus S^* \RR$ is a set of measure zero (here $S_x \RR$ denotes the fiber of $S\RR$ over $x$), such that at $x$ we have
\[ \big| \langle \nabla u^{l_1}, \nabla u^{l_2} \rangle - \delta_{l_1 l_2} \big| < \Psi (\eps_0 | Y) \textQq{for all} l_1, l_2 = 1, \ldots, n\]
and
\[ \int_{S_x \RR \cap S^* \RR}  \Big| \langle \nabla u^l, v \rangle - 2 \big( u^l (\gamma_v ( 1/2)) - u^l (\gamma_v (0) \big) \Big|  < \Psi(\eps_0 | Y) \textQq{for all} l = 1, \ldots, n. \]
So for sufficiently small $\eps_0$, we can find some $v \in S_x \RR$ such that
\begin{equation} \label{eq:nabuvalmost1}
 \sum_{l=1}^n \langle \nabla u^l, v \rangle y^l > (1 - \Psi(\eps_0 | Y)) |y| 
\end{equation}
and such that for all $l = 1, \ldots, n$
\begin{equation} \label{eq:almostlinearatv}
 \Big| \langle \nabla u^l, v \rangle - 2 \big( u^l (\gamma_v ( 1/2)) - u^l (\gamma_v (0)) \big) \Big|  < \Psi(\eps_0 | Y) .
\end{equation}
Note that $x = \gamma_v(0)$.
Combining (\ref{eq:nabuvalmost1}) and (\ref{eq:almostlinearatv}) yields
\[ \sum_{l=1}^n 2 \big( u^l (\gamma_v ( 1/2)) - u^l (x) \big) y^l > - \Psi(\eps_0 | Y) |y| + \sum_{l=1}^n  \langle \nabla u^l, v \rangle y^l > (1- \Psi(\eps_0 | Y)) |y|. \]
Setting $z := \gamma_v (1/2)$ we then obtain for small $\eps_0$ (recall that $u(p) = 0^n$)
\begin{multline*}
2 \langle u(z) , y \rangle > (1 - \Psi(\eps_0 | Y)) |y| -  2 | u(x) | \cdot | y| \\
\geq (1 - \Psi(\eps_0 | Y)) |y| - 2(1+\eps_0) d(x, p) |y|  > 0.7 |y| \geq 0.7 . 
\end{multline*}
On the other hand, for small $\eps_0$
\begin{multline*}
 |u(z)| = | u(z) - u(p) |  \leq (1+\eps_0) d(z,p) \leq (1+\eps_0) ( d(z,x) + d(x, p)) \\
 < (1+\eps_0) (0.5 + 0.1) < 0.7. 
\end{multline*}
This shows that $2 \langle u(z) , y \rangle > |u(z)|^2$ for sufficiently small $\eps_0$.
\end{proof}

\begin{Lemma} \label{Lem:splittingimpliesvolbound}
For every $\eps > 0$ and $Y < \infty$ there is a $\delta = \delta (\eps, Y) > 0$ such that the following holds:

Let $\XX$ be a singular space with mild singularities of codimension $\mathbf{p}_0 > 1$ that is $Y$-tame at some scale $r > 0$.
Assume that $\Ric \geq - (n-1)$ on $\RR$.
Let $p \in X$ be a point and consider a $\delta$-splitting map $u : B(p,r) \to \IR^n$ that satisfies
\begin{equation} \label{eq:almostradiallem}
 \big| |u(x)| - d(p,x) | < \delta r \textQQqq{for all} x \in B(p,r).
\end{equation}
Then
\[ |B(p,  r) \cap \RR | > (\omega_n - \eps)  r^n. \]
\end{Lemma}

\begin{proof}
Without loss of generality, we may again assume that $r = 1$.

We will first show that for sufficiently small $\delta$ we have $|u(B(p,1)) | > \omega_n - \eps /2$.
Choose and fix $\nu > 0$ small enough such that
\begin{equation} \label{eq:nueps4}
 |B(0^n, 1) \setminus B(0^n, 1-3\nu) | < \eps / 4.
\end{equation}
Using (\ref{eq:almostradiallem}) we may assume that $\delta$ is chosen small enough such that
\[ u \big( B(p,1) \setminus B(p, 1 - \nu) \big) \cap B(0^n, 1-2\nu) = \emptyset. \]
So
\[ A := B(0^n,1-3\nu) \setminus u (B(p,1)) = B(0^n, 1-3\nu) \setminus u \big( \ov{B(p, 1-\nu)} \big) \]
is an open subset of $\IR^n$.
If $A = \emptyset$, then $|u(B(p,1)) | \geq |B(0^n, 1 - 3 \nu)| > \omega_n - \eps /4$ and we are done.
So assume that $A \neq \emptyset$.

For any $y \in A$ choose $r_y > 0$ maximal with the property that $B(y, r_y) \cap u ( B(p,1) ) = \emptyset$.
Note that $r_y$ depends continuously on $y$ and due to (\ref{eq:almostradiallem}) we have $\inf_{y \in A} r_y = 0$.
So if $r_y > \nu$ for some $y \in A$, then we can find some $y_0 \in A$ such that $r_0 := r_{y_0} = \nu$.

Assume for a moment that $r_y \leq \nu$ for all $y \in A$.
By Vitali's Covering Lemma we can choose (finitely or countably infinitely many) points $y_1, y_2, \ldots \in A$ and radii $r_1 := r_{y_1}, r_2 := r_{y_2}, \ldots > 0$ such that
\[
 B(y_i, r_i) \cap u (B(p,1)) = \emptyset, \qquad   \partial B(y_i, r_i) \cap u (B(p,1)) \neq \emptyset, 
\]
such that the balls $B(y_i, 3 r_i )$ are pairwise disjoint and such that
\[ \bigcup_i B(y_i, 15 r_i)  \supset A . \]
It follows, using Proposition \ref{Prop:volumecomparison}, that
\begin{equation} \label{eq:sumrinA}
 \sum_{i>0} \omega_n r_i^{n} \geq \frac{|A |}{100^n}. 
\end{equation}

By the definition of the radii $r_i$, we can choose a point $x_i \in B(p,1)$ such that $| u(x_i) - y_i | = r_i$.
This can be done both in the cases in which $r_y \leq \nu$ for all $y \in A$ and in which we have picked the points $y_1, y_2, \ldots$, as well as in the case in which $r_y > \nu$ for some $y \in A$ and in which we have picked the point $y_0$ with $r_{y_0} = \nu$.
Note that in both cases $|u(x_i)| \leq |y_i| + r_i < 1 - 3\nu + \nu = 1 - 2\nu$.
This implies that $x_i \in B(p, 1 - \nu)$ and hence that $B(x_i, r_i) \subset B(p, 1)$.
By Lemma \ref{Lem:eps0splitting}, the restriction $u |_{B(x_i, r_i)} : B(x_i, r_i) \to \IR^n$ cannot be an $\eps_0$-splitting map.
So we have for $i = 0$ or $i =1, 2, \ldots$
\begin{equation} \label{eq:eps0splitting}
 \int_{B(x_i, r_i) \cap \RR} \Big( r^{2}_i \sum_{l=1}^n |\nabla^2 u^l |^2 + \sum_{l_1, l_2 = 1}^n \big| \langle \nabla u^{l_1}, \nabla u^{l_2} \rangle - \delta_{l_1l_2} \big|^2 \Big) \geq \eps_0^2 r_i^{n} .
\end{equation}
Consider first the case $i = 0$ and $r_0 = \nu$.
Since $u$ is a $\delta$-splitting, we then have $\eps_0^2 \nu^n \leq \delta^2$.
This gives us a contradiction for sufficiently small $\delta$.
We may hence assume in the following consider the case $i = 1, 2, \ldots$.

We now claim that the balls $B(x_i, r_i)$ are pairwise disjoint.
Assume that two such balls, for two indices $i$, $i'$, intersect.
Then, for sufficiently small $\delta$, we have
\begin{multline*}
 |y_i - y_{i'} | \leq |y_i - u(x_i) | + |u(x_i) - u(x_{i'})| + |u(x_{i'}) - y_{i'} | \\ \leq r_i + (1+\delta) d(x_i, x_{i'})   + r_{i'}  < r_i + (1+\delta) (r_i+r_{i'}) + r_{i'} 
 < 3 r_i +3 r_{i'},
\end{multline*}
in contradiction to the fact that all $B(y_i, 3 r_i)$ are pairwise disjoint.
Using the fact that $u$ is an $\delta$-splitting map and (\ref{eq:eps0splitting}), it follows that
\[ \sum_{i} \eps_0^2 r_i^{n} \leq \delta^2. \]
Combining this with (\ref{eq:sumrinA}) yields that for sufficiently small $\delta$
\[ |A| < \eps/4. \]
Using (\ref{eq:nueps4}), it follows that
\[ |u(B(p,1))| > \omega_n - \eps/2. \]

Next, we show that $|u(B(p,1))| = |u(B(p,1) \cap \RR)|$.
Let $0 < s < 1$ and choose a minimal $s$-net $y_1, \ldots, y_N$ of $\{  \rrm < s \}  \cap B(p,1)$.
Then the balls $B(y_i, s)$ cover $\{  \rrm < s \}  \cap B(p,1)$ and the balls $B(y_i, s/2)$ are pairwise disjoint.
Since $\rrm$ is $1$-Lipschitz, we have $B(y_i, s/2) \subset \{ \rrm < 2s \} \cap B(p,2)$.
So for some generic constant $C_* < \infty$, which may depend on $\XX$ and $p$ but not on $s$, we have
\[ N < C_* (s/2)^{-n} |\{   \rrm < 2 s \} \cap B(p,2) \cap \RR| . \]
It follows that
\begin{multline*}
 |u(B(p,1) \setminus \RR)| \leq |u( \{  \rrm < s \} \cap B(p,1) )| \\
 \leq \Big| \bigcup_{i=1}^N B(u(z_i),(1+\delta) s) \Big| \leq C_* N s^n \leq C_* |\{   \rrm < 2 s \} \cap B(p,2) \cap \RR|.
\end{multline*}
Letting $s \to 0$ yields $|u(B(p,1) \setminus \RR)| = 0$ and hence
\[ |u(B(p,1) \cap \RR)| = |u(B(p,1))| > \omega_n - \eps/2. \]

Since $u$ is a $\delta$-splitting, we find a constant $C < \infty$ such that the Jacobian $|{\det du}|$ of $u$ satisfies the following estimate
\[ \int_{B(p,1) \cap \RR} \big| |{\det d u}| - 1 \big| < C \delta. \]
So
\[ \omega_n - \eps/2 < |u (B(p,1) \cap \RR) | = \int_{B(p,1) \cap \RR} |{\det du}| \leq (1+ C\delta) | B(p,1) \cap \RR | . \]
The claim now follows for sufficiently small $\delta$.
\end{proof}

We can finally prove Theorem \ref{Thm:volconv}.

\begin{proof}[Proof of Theorem \ref{Thm:volconv}]
The claim follows by combining Proposition \ref{Prop:GHsplittingepssplitting} with Lem\-ma~\ref{Lem:splittingimpliesvolbound}.
\end{proof}

\section{Cone rigidity}
The goal of this section is to reprove the Cone Rigidity Theorem of Cheeger and Colding (see \cite{Cheeger-Colding-Cone}) for singular spaces.
More specifically, we will show Theorem \ref{Thm:conerigidityIntroduction} from subsection \ref{subsec:strcttheory}.

We first establish the existence of a function $b$ on an annulus around $p$ that approximates $d^2 (p, \cdot)$ in $W^{1,2}$ and for which $| \nabla^2 b - 2 g |$ is small in an $L^2$-sense.

\begin{Lemma}[cf {\cite[Corollary 4.83]{Cheeger-Colding-Cone}}] \label{Lem:bonannulus}
For any $\eps, \eta > 0$, $\mathbf{p}_0 > 3$ and $Y < \infty$ there is a $\delta = \delta(\eps, \eta, \mathbf{p}_0, Y) < \infty$ such that the following holds:

Let $\XX$ be a singular space with mild singularities of codimension $\mathbf{p}_0$ that is $Y$-tame at scale $\delta^{-1} r$ for some $r > 0$.
Assume moreover that $\Ric \geq - (n-1) \kappa$ on $\RR$ for some $0 \leq \kappa \leq \delta r^{-2}$.
Let $p \in X$ be a point and set $f := d^2(\cdot, p)$.
Assume that
\[ \frac{|B(p, \eps r / 4) \cap \RR|}{v_{-\kappa} (\eps r / 4)} - \frac{|B(p,8r) \cap \RR|}{v_{-\kappa} (8r)} < \delta. \]
Then there is a locally Lipschitz function $b : A(p, \eps r, r) \to \IR$ such that $b$ is $C^3$ on $A(p, \eps r, r) \cap \RR$ and such that the following holds
\[  |b-f| < \eta r^2 \textQQqq{on} A(p,\eps r, r), \]
\begin{align*}
  \int_{A(p, \eps r, r) \cap \RR} |\nabla (b - f)|^2 dg &< \eta r^{n+2}, \\
 \int_{A(p, \eps r, r) \cap \RR} \big| \nabla^2 b - 2 g \big|^2 dg &< \eta r^{n}.
\end{align*}
\end{Lemma}

\begin{proof}
We follow the lines of the proof of \cite[Proposition 4.81]{Cheeger-Colding-Cone}.
Without loss of generality, we may assume that $r = 1$.

Let us first deduce the following bound, using Proposition \ref{Prop:volumecomparison}:
\begin{multline*}
 \frac{|A(p,4,8) \cap \RR|}{v_{-\kappa} (8) - v_{-\kappa} (4)} 
= \frac{|B(p,8) \cap \RR|}{v_{-\kappa} (8) - v_{-\kappa} (4)} - \frac{|B(p,4) \cap \RR|}{v_{-\kappa} (8) - v_{-\kappa} (4)} \\
\geq \frac{|B(p,8) \cap \RR|}{v_{-\kappa} (8) - v_{-\kappa} (4)} - \frac{|B(p, \eps / 4) \cap \RR| }{v_{-\kappa} (8) - v_{-\kappa} (4)} \cdot \frac{v_{-\kappa} (4)}{v_{-\kappa} (\eps / 4)} . 
\end{multline*}
So, for some constant $C< \infty$, which only depends on the dimension,
\begin{multline*}
 \frac{|B(p, \eps  / 4) \cap \RR|}{v_{-\kappa} (\eps  / 4)} - \frac{|A(p,4,8) \cap \RR|}{v_{-\kappa} (8) - v_{-\kappa} (4)} \\
\frac{v_{-\kappa} (8)}{v_{-\kappa} (8) - v_{-\kappa} (4)} \bigg( \frac{|B(p, \eps  / 4) \cap \RR|}{v_{-\kappa} (\eps / 4)} - \frac{|B(p,8) \cap \RR|}{v_{-\kappa} (8)} \bigg)
  < C \delta. 
\end{multline*}

By the third tameness property applied to the annulus $A(p, \eps  / 4, 4 )$ and Corollary \ref{Cor:harmonic}, we can find a harmonic function $h : A(p, \eps  / 4, 4) \to \IR$ that satisfies
\begin{subequations}
\begin{equation} \label{eq:hpropertiesa}
 4^{2-n} \leq h \leq (\eps /4)^{2-n}, 
\end{equation}
\begin{align}
 \int_{A(p, \eps / 4, 4) \cap \RR} |\nabla (h-d^{2-n} (p, \cdot) ) |^2 dg &< C \delta, \displaybreak[1] \\
 \int_{A(p, \eps  / 4, 4) \cap \RR} | h-d^{2-n} (p, \cdot) |^2 dg &< C \delta. 
\end{align}
Here, and in the rest of the proof, $C < \infty$ is a generic constant that only depends on $\eps$, $\mathbf{p}_0$ and $Y$.
By Corollary \ref{Cor:harmonic}, we know that $h$ is harmonic and locally Lipschitz.
Moreover, by Corollary \ref{Cor:ChengYau} we have
\begin{equation} \label{eq:hpropertiesd}
|\nabla h| < C  \textQQqq{on} A(p, 0.26 \eps , 3.9 ).
\end{equation}
\end{subequations}

Set now
\[ b := h^{\frac{2}{2-n}}. \]
Then
\begin{subequations}
\begin{equation} \label{eq:bpropertiesa}
 \triangle b = \frac{n}2 \frac{|\nabla b|^2}{b}
\end{equation}
and, using (\ref{eq:hpropertiesa})--(\ref{eq:hpropertiesd}), we find that
\begin{equation} \label{eq:bpropertiesb}
 (\eps /4)^2 \leq b \leq 4^2, 
\end{equation}
\begin{align}
 \int_{A(p, \eps / 4, 4) \cap \RR} |b - f |^2 dg <  \Psi(\delta | \eps, Y), \label{eq:bpropertiesc} \\
  \int_{A(p, \eps / 4, 4) \cap \RR} |\nabla (b - f) |^2 dg <  \Psi(\delta | \eps, Y), \label{eq:bpropertiesd}
\end{align}
\begin{equation} 
 |\nabla b| < C \textQQqq{on} A(p, 0.26 \eps , 3.9 ) \cap \RR. \label{eq:bpropertiese}
\end{equation}
Combining (\ref{eq:bpropertiesc}) with (\ref{eq:bpropertiese}) yields moreover
\begin{equation} \label{eq:bpropertiesf}
 |b - f| < C \Psi (\delta | \eps, Y) \textQQqq{on} A(p, 0.27 \eps , 3.8 ).
\end{equation}
\end{subequations}

Let $\Phi : \IR \to [0, 1]$ be a smooth cutoff function such that $\supp \Phi \subset ( (0.3 \eps)^2, 3^2)$ and such that $\Phi \equiv 1$ on $[\eps^2, 1]$.
By (\ref{eq:bpropertiesf}) we have $\supp \Phi \circ b \subset A := A(p, \eps  / 2, 2)$ for sufficiently small $\delta$.
Then
\begin{equation} \label{eq:intbyp0}
 \int_{A \cap \RR}  \bigg| \nabla^2 b - \frac1n \triangle b g \bigg|^2 \Phi (b) dg = \int_{A \cap \RR} \bigg( |\nabla^2 b|^2 - \frac1n (\triangle b)^2 \bigg) \Phi (b) dg. 
\end{equation}
By standard elliptic estimates and (\ref{eq:bpropertiesb}), (\ref{eq:bpropertiese}) there is a (non-uniform) constant $C_* < \infty$ such that we have on $A \cap \RR$
\[ |\nabla^2 b| < C_*  \rrm^{-1} \textQQqq{and}  |\nabla^3 b| < C_*  \rrm^{-2}. \]
So, using Proposition \ref{Prop:integrationbyparts}, we can conclude
\begin{align}
  \int_{A \cap \RR}  |\nabla^2 b|^2 \Phi (b) dg &= - \int_{A \cap \RR} \Big( \nabla b \nabla \triangle b \cdot \Phi (b) + \nabla_i b \nabla^2_{ij} b \nabla_j b \cdot \Phi' (b) \Big) dg,  \displaybreak[1] \\
\int_{A \cap \RR}  \nabla b \nabla \triangle b \cdot \Phi (b) dg &= - \int_{A \cap \RR} \Big(  ( \triangle b )^2 \cdot \Phi (b)  + \triangle b |\nabla b|^2 \Phi' (b) \Big) dg , 
\end{align}
\begin{multline}
 \int_{A \cap \RR}  \nabla_i b \nabla^2_{ij} b \nabla_j b \cdot \Phi' (b) dg \\
  = - \int_{A \cap \RR} \Big( \triangle b |\nabla b |^2 \Phi' (b) + \nabla_i b \nabla_j b \nabla^2_{ij} b  \cdot \Phi' (b) + |\nabla b |^4 \Phi'' (b) \Big) dg.
\end{multline}
So
\begin{equation} \label{eq:intbyp3}
 \int_{A \cap \RR}  \nabla_i b \nabla^2_{ij} b \nabla_j b \cdot \Phi' (b) dg 
  = - \frac12 \int_{A \cap \RR} \Big( \triangle b |\nabla b |^2 \Phi' (b)  + |\nabla b |^4 \Phi'' (b) \Big) dg. 
 \end{equation}
Combining (\ref{eq:intbyp0})--(\ref{eq:intbyp3}) and applying (\ref{eq:bpropertiesa}) yields
\begin{align*}
 \int_{A \cap \RR}  \bigg| \nabla^2 b & - \frac1n \triangle b g \bigg|^2  \Phi (b) dg \\
 &= \int_{A \cap \RR} \Big( ( \triangle b )^2 \cdot \Phi (b)  + \triangle b |\nabla b|^2 \Phi' (b) \\
&\qquad\qquad\qquad + \frac12 \triangle b |\nabla b |^2 \Phi' (b)  + \frac12 |\nabla b |^4 \Phi'' (b) 
 - \frac1n (\triangle b)^2 \Phi (b) \bigg) dg \displaybreak[1] \\
 &= \int_{A \cap \RR} \bigg( \frac{n^2}{4} \frac{|\nabla b|^4}{b^2} \cdot \Phi (b)  + \frac{n}2 \frac{|\nabla b|^2}{b} |\nabla b|^2 \Phi' (b) \\
&\qquad\qquad\qquad + \frac{n}4 \frac{|\nabla b|^2}{b} |\nabla b |^2 \Phi' (b)  + \frac12 |\nabla b |^4 \Phi'' (b) 
 - \frac{n}4 \frac{|\nabla b|^4}{b^2} \Phi (b) \bigg) dg \displaybreak[1] \\
 &= \int_{A \cap \RR} \bigg( \frac{n(n-1)}{4} b^{-2} \Phi (b)  + \frac{3n}4 b^{-1} \Phi' (b) + \frac12  \Phi'' (b)  \Big) |\nabla b|^4 dg.
\end{align*}
Using (\ref{eq:bpropertiesb})--(\ref{eq:bpropertiesf}) and \cite[Lemma 4.74]{Cheeger-Colding-Cone}, we find
\begin{align}
 \int_{A \cap \RR}&  \bigg| \nabla^2 b  - \frac1n \triangle b g \bigg|^2  \Phi (b) dg \notag \\
 &< \Psi(\delta | \eps, Y) +  \int_{A \cap \RR} \bigg( \frac{n(n-1)}{4} d^{-4}(\cdot, p) \Phi (d^2(\cdot, p)) \notag \\
 &\qquad\qquad\qquad + \frac{3n}4 d^{-2}(\cdot, p) \Phi' (d^2(\cdot, p)) + \frac12  \Phi'' (d^2(\cdot, p))  \bigg) \cdot 16 d^4(\cdot, p)  dg \notag \\
 &< \Psi (\delta | \eps, Y) + \int_{A \cap \RR} T( d(\cdot, p)) dg, \label{eq:nab2bint1}
\end{align}
for some universal smooth function $T : \IR \to \IR$ with $T(\eps / 2) = 0$.
We can estimate the last integral as follows.
We have
\[
 \int_{A \cap \RR} T( d(\cdot, p)) dg 
= \int_{A \cap \RR}  \int_{\eps  / 2}^{d(x,p)} T'( s) ds dg(x)  
=   \int_{\eps  / 2}^{2 } T'( s) |A(p,s, 2) \cap \RR| ds 
\]
and, using integration by parts,
\[
 \frac{n |B(p,2) \cap \RR|}{2^n}  \int_{\eps /2}^{2} T(s) s^{n-1} ds
  =  \int_{\eps  / 2}^{2} T'(s)\cdot \bigg( 1 - \Big( \frac{s}{2} \Big)^{n} \bigg)  |B(p,2) \cap \RR| ds. 
 \]
It follows that, using Proposition \ref{Prop:volumecomparison} and the assumption of this proposition,
\begin{align*}
\bigg|  \int_{A \cap \RR} & T(d)  dg  - \frac{n |B(p,2) \cap \RR|}{2^n}  \int_{\eps /2}^{2} T(s) s^{n-1} dg \bigg|   \\
&\leq   \int_{\eps  / 2}^{2} |T'(s)| \bigg| |A(p,s, 2) \cap \RR|  - \bigg( 1 - \Big( \frac{s}{2} \Big)^n \bigg) | B(p, 2) \cap \RR | \bigg| ds \notag \\
&\leq   \int_{\eps  / 2}^{2} |T'(s)| \bigg| |B(p,s) \cap \RR|   - \Big( \frac{s}{2} \Big)^n  | B(p, 2) \cap \RR | \bigg| ds \notag \\
 &< \Psi (\delta | \eps). 
\end{align*}
Using integration by parts and the definition of $T(s)$, we can check that
\begin{equation}  \label{eq:nab2bint3}
  \int_{\eps /2}^{2} T(s) s^{n-1} dg =  0. 
\end{equation}
Combining (\ref{eq:nab2bint1})--(\ref{eq:nab2bint3}) yields
\begin{equation} \label{eq:nab2bintfinal}
  \int_{A \cap \RR}  \bigg| \nabla^2 b  - \frac1n \triangle b g \bigg|^2  \Phi (b) dg < \Psi (\delta | \eps, Y). 
\end{equation}
The claim now follows from (\ref{eq:bpropertiesa}), (\ref{eq:bpropertiesc}), (\ref{eq:bpropertiesd}), (\ref{eq:bpropertiesf}) and (\ref{eq:nab2bintfinal}) for sufficiently small $\delta$.
\end{proof}

Next, we show that almost every minimizing geodesic in an annulus around $p$ encloses an angle with the vector field $\nabla d( \cdot, p)$ that can be estimated using the law of cosines.

\begin{Lemma}[cf {\cite[Corollary 2.48]{Cheeger-Colding-Cone}}] \label{Lem:almostlawcos}
For any $0 <  \eps < 1$, $\mathbf{p}_0 > 3$ and $Y < \infty$ there is a $0 < \delta = \delta (\eps, \mathbf{p}_0, Y) < \eps$ such that the following holds:

Let $\XX$ be a singular space with mild singularities of codimension $\mathbf{p}_0$ that is $Y$-tame at scale $\delta^{-1} r$ for some $r > 0$.
Assume moreover that $\Ric \geq - (n-1) \kappa$ on $\RR$ for some $0 \leq \kappa \leq \delta r^{-2}$.
Let $p \in \RR$ and assume that
\[ \frac{| B(p, \delta r) \cap \RR |}{v_{-\kappa} (\delta r)} - \frac{|B(p, 16r) \cap \RR|}{v_{-\kappa} (16r)} < \delta. \]
Let $x \in A(p, \eps r, r)$.
Then there is a point $x' \in B(x, \eps r) \cap \RR$ and an open subset $S \subset A(p, \eps r, r) \cap \RR$ such that the following holds:
\begin{enumerate}[label=(\alph*)]
\item $|(A(p, \eps r, r) \cap \RR) \setminus S| < \eps r^n$.
\item $d(\cdot, p)$ and $d(\cdot, x')$ are $C^1$ on $S$.
\item For any $y \in S$ there is a unique minimizing arclength geodesic $\gamma_{x',y} : [0, d(x', y)] \to X$ with $\gamma_{x', y} ([0, d(x', y)] ) \subset \RR$, $\gamma_{x', y}$ varies continuously in $y$, and we have
\begin{equation} \label{eq:almostlawcos}
\qquad\qquad \bigg|  \big\langle \gamma'_{x',y} (d(x', y)), \nabla d(\cdot, p) \big\rangle + \frac{d^2(x', p) - d^2(y, p) - d^2(x', y)}{2 d(y,p) d(x', y)} \bigg| < \eps . 
\end{equation}
\end{enumerate}
\end{Lemma}

\begin{proof}
We argue similarly as in the proof of \cite[Corollary 2.48]{Cheeger-Colding-Cone}.
Without loss of generality, we may assume that $r = 1$.

We first argue that for a sufficiently small choice of $0 < \rho < \eps^2 / 10$ (depending on $\eps$) the following is true:
If $x' \in B(x, \rho )$, $y \in A(p, \eps , 1)$, $\nabla d(\cdot, p)$ is $C^1$ at $y$ and $\gamma : [0, d(x,y)] \to X$ is a minimizing arclength geodesic between $x', y$ that leaves the annulus $A(p, \rho, 4)$, then
\begin{equation} \label{eq:almost1oververtex}
 \bigg| \frac{d^2(x', p) - d^2(y, p) - d^2(x', y)}{2 d(y,p) d(x', y)} + 1\bigg| < \eps / 4. 
\end{equation}
In fact, if $\rho < \eps / 2$, then $x' \in A(p, \eps / 2, 2)$.
If $\gamma$ leaves $A(p, \rho, 4)$, then there is some $s \in [0, d(x,y)]$ such that $d(p, \gamma(s)) \leq \rho$.
It follows that
\[ d(y, x') \leq d(y, p) + d(p, x') \leq d(y, \gamma(s)) + \rho + d(\gamma(s), x') + \rho = d(y,x') + 2 \rho. \]
So for sufficiently small $\rho$, we can ensure (\ref{eq:almost1oververtex}).
We will fix $\rho$ from now on.

Let $\nu > 0$ be a constant whose value we will determine at the end of the proof, depending on $\eps, Y, \rho$.
Assuming $\delta$ to be sufficiently small (depending on $\nu, \rho, \eps, \mathbf{p}_0, Y$), we apply Lemma \ref{Lem:bonannulus} with $r \leftarrow 4$, $\eps \leftarrow \rho / 2$ and $\eta \leftarrow \nu$ to obtain the function $b : A := A(p, \rho / 2, 4) \to \IR$.
Note that then
\begin{subequations}
\begin{equation} \label{eq:lawofcosida}
 | b- d^2 (\cdot, p)  | < \nu \textQQqq{on} A(p, \rho, 4).
\end{equation}

We now apply Proposition \ref{Prop:segmentinequ} with $U_1 \leftarrow B(x, \rho)$, $U_2 \leftarrow A= A(p, \eps, 1)$, $f \leftarrow |\nabla^2 b - 2 g| \chi_A$.
Doing this, and assuming $\nu$ and hence $\delta$ to be sufficiently small (depending on $\eps, Y, \rho$), we can find a point $x' \in B(x, \rho) \cap \RR$ and and open subset $S \subset A(p, \eps, 1) \cap \RR$ in such a way that assertions (a) and (b) hold and such that additionally the following is true:
\begin{equation} \label{eq:lawofcosidb}
 \big| \nabla b - \nabla d^2 (\cdot, p )  \big| < \Psi ( \nu | \eps, \eta, Y ) \textQQqq{on} S 
\end{equation}
and for any $y \in S$ there is a unique minimizing arclength geodesic $\gamma_{x',y} : [0, \linebreak[1] d (x', y)] \linebreak[1] \to X$ with $\gamma_{x', y} ([0, d(x', y)] ) \subset \RR$, that varies continuously in $y$, and for any $y \in S$
\begin{equation} \label{eq:lawofcosidc}
 \int_0^{d(x',y)} \big( \big| \nabla^2 b - 2 g \big| \chi_A \big) (\gamma_{x',y}(s)) ds < \Psi (\nu | \eps, \rho, Y) . 
\end{equation}
\end{subequations}
We will now show that (\ref{eq:lawofcosida})--(\ref{eq:lawofcosidc}) imply (\ref{eq:almostlawcos}).
So fix $y$ for the rest of the proof and set $\gamma := \gamma_{x', y}$ and $d := d(x',y)$.
Then
\begin{equation} \label{eq:1dintissmall}
 \int_0^d  \big| (b\circ \gamma)''(s) - 2 \big| \chi_A (\gamma(s)) ds < \Psi (\nu | \eps, \rho, Y). 
\end{equation}

Assume first that $\gamma ([0,d]) \subset A$.
Then by (\ref{eq:1dintissmall}), we have for any $s_1, s_2 \in [0,d]$
\begin{equation} \label{eq:deralonggamma}
 \big| \big( (b \circ \gamma)' (s_1) - 2s_1 \big) - \big( (b \circ \gamma) ' (s_2) - 2s_2 \big) \big| < \Psi (\nu | \eps, \rho, Y) \cdot d. 
\end{equation}
Integrating (\ref{eq:deralonggamma}) one more time yields
\begin{equation} \label{eq:twiceintegrated}
 \big| \big( (b \circ \gamma)(d) - d^2  \big) -  (b \circ \gamma) (0) - \big( (b \circ \gamma )' (d) - 2d \big) \cdot d    \big| < \Psi (\nu | \eps, \rho, Y) \cdot d^2. 
\end{equation}
So
\[ \bigg| (b \circ \gamma)' (d) + \frac{b(x')  - b(y) - d^2}{d} \bigg| < \Psi (\nu | \eps, \rho, Y) \cdot d. \]
Using (\ref{eq:lawofcosida}) and the fact that $d(y,p) > \eps$ we find that for sufficiently small $\nu$
\begin{equation} \label{eq:bgammaprime}
 \bigg| \frac{(b \circ \gamma)' (d)}{2 d(y,p)} + \frac{d^2 (x', p)  - d^2(y, p) - d^2(x',y)}{2 d(y,p) d(x',y) } \bigg| < \eps / 2. 
\end{equation}

Next, we verify (\ref{eq:bgammaprime}) in the case in which $\gamma ([0,d]) \not\subset A$.
Choose $s \in [0,d]$ such that $\gamma((s,d]) \subset A$ and $d(p, \gamma(s)) = \rho$.
Then (\ref{eq:deralonggamma}) holds for all $s_1, s_2 \in (s,d]$ and instead of (\ref{eq:twiceintegrated}) we get
\begin{multline*}
 \big| \big( (b \circ \gamma)(d) - d^2 \big) - \big( (b \circ \gamma)(s) - s^2 \big) \\ - \big( (b \circ \gamma)' (d) - 2d\big) \cdot (d-s) \big| < \Psi (\nu | \eps, \rho, Y ) ,
\end{multline*}
which implies
\[ \big| b(y)  - b(\gamma(s)) + (d-s)^2  - (b \circ \gamma)'(d) \cdot (d-s) \big| < \Psi (\nu | \eps, \rho, Y ). \]
Using (\ref{eq:lawofcosida}) we find
\[ \big| d^2 (y,p) - \rho^2 + (d-s)^2  - (b \circ \gamma)'(d) \cdot (d-s) \big| < \Psi (\nu | \eps, \rho, Y). \]
So by $d^2 (y,p) + (d-s)^2 - 2 d (y,p) \cdot (d - s) = ( d(y,p) - (d-s))^2 \leq \rho^2$, we get
\[ \big| 2 d (y,p) \cdot (d - s)  - (b \circ \gamma)'(d) \cdot (d-s) \big| < 2\rho^2 + \Psi (\nu | \eps, \rho, Y). \]
Since  $d-s > \eps - \rho > \eps > \rho^{1/2}$ and $d(y,p) > \eps > \rho^{1/2}$, we get
\[ \bigg|  \frac{(b \circ \gamma)'(d)}{2d(y,p)} - 1 \bigg| < \rho + \Psi (\nu | \eps, \rho, Y).  \]
Combining this inequality with (\ref{eq:almost1oververtex}) we find that (\ref{eq:bgammaprime}) also holds for small $\nu$ in the case in which $\gamma ([0,d]) \not\subset A$.

Lastly, observe that by (\ref{eq:lawofcosidb})
\begin{multline*}
 \bigg| \frac{(b \circ \gamma)'(d)}{2 d(y,p)} -  \langle \nabla_y d(y,p), \gamma'(d) \rangle \bigg| = \frac1{2 d(y,p)} \big| \langle \nabla b , \gamma' (d) \rangle - \langle \nabla_y d^2 (y,p), \gamma' (d) \rangle \big| \\
  \leq \frac1{2 \eps} \big| \nabla b - \nabla_y d^2 (y,p) \big| < \Psi(\nu | \eps, \rho, Y).
\end{multline*}
So (\ref{eq:almostlawcos}) follows from (\ref{eq:bgammaprime}) for sufficiently small $\nu$.
\end{proof}

Next, we integrate the approximate law of cosines identity from Lemma \ref{Lem:almostlawcos} to obtain an approximation of an identity that is true on metric cones.

\begin{Lemma} \label{Lem:simpleconeid}
For any $0 < \eps < 1$, $\mathbf{p}_0 > 3$ and $Y < \infty$ there is a $0 < \delta = \delta (\eps, \mathbf{p}_0, Y) < \eps$ such that the following holds:

Let $\XX$ be a singular space with mild singularities of codimension $\mathbf{p}_0$ that is $Y$-tame at scale $\delta^{-1} r$ for some $r > 0$.
Assume moreover that $\Ric \geq - (n-1) \kappa$ on $\RR$ for some $0 \leq \kappa \leq \delta r^{-2}$.
Let $p \in \RR$ and assume that
\[ \frac{| B(p, \delta r) \cap \RR |}{v_{-\kappa} (\delta r)} - \frac{|B(p, 32r) \cap \RR|}{v_{-\kappa} (32r)} < \delta. \]
Let $x, z_1, z_2 \in A(p, \eps r, r)$ such that
\[ d(p, z_2) = d(p, z_1) + d(z_1, z_2). \]
Then
\begin{equation} \label{eq:simpleconeid}
\bigg| \frac{d^2(x,z_1) - d^2(x,p) - d^2(z_1, p)}{2d(z_1, p) d(x,p)} - \frac{d^2(x,z_2) - d^2(x,p) - d^2(z_2, p)}{2d(z_2, p) d(x,p)} \bigg| < \eps. 
\end{equation}
\end{Lemma}

\begin{proof}
Assume without loss of generality that $r = 1$.
Let $\zeta > 0$ be a constant whose value will be determined and fixed at the end of this paragraph and consider first the case in which $x$ almost lies on the minimizing segment between $z_1, z_2$, in the sense that
\begin{equation} \label{eq:zetacondition1}
 d(x,z_1) + d(x, z_2) < d(z_1, z_2) + \zeta. 
\end{equation}
Then, using the triangle inequality,
\begin{multline*}
 0 \geq d(x, p) - d(z_1, p) - d(x, z_1) = d(x,p) - d(z_2, p) + d(z_1, z_2) - d(x, z_1) \\
 > d(x,p) - d(z_2, p) + d(x,z_2) - \zeta \geq - \zeta.
\end{multline*}
and
\begin{multline*}
0 \geq d(z_2, p) - d(x, z_2) - d(x, p)  \geq d(z_2, p) - d(x,z_2) - d(x,z_1) - d(z_1, p) \\
> d(z_2, p) - d(z_1, z_2) - \zeta - d(z_1, p) = - \zeta.
\end{multline*}
So for $i = 1, 2$
\[ \bigg|  \frac{d^2(x,z_i) - d^2(x,p) - d^2(z_i, p)}{2d(z_i, p) d(x,p)} + 2 \bigg|  < \Psi (\zeta | \eps ). \]
So for sufficiently small $\zeta$ the bound (\ref{eq:zetacondition1}) implies the bound (\ref{eq:simpleconeid}).
For the remainder of the proof we fix $\zeta$ to be so small and we consider from now on the case
\begin{equation} \label{eq:zetacondition2}
 d(x,z_1) + d(x, z_2) \geq d(z_1, z_2) + \zeta. 
\end{equation}

Let $0 < \nu < \eps/10$, $\mu > 0$ be constants whose values will be determined in the course of the proof, depending on $\eps, Y$.
Apply Lemma \ref{Lem:almostlawcos} for $r \leftarrow 2$, $\eps \leftarrow \nu$ to obtain the point $x' \in B(x, \nu ) \cap \RR$ and the open subset $S \subset A(p, \eps/2, 2) \cap \RR$.

For a given $\mu > 0$ apply Proposition \ref{Prop:segmentinequ} to $U_i \leftarrow B(z_i, \mu)$ and $f$ being the characteristic function of $(A(p, \eps/2, 2) \cap \RR) \setminus S$.
If $0 < \nu < \mu / 10$ is sufficiently small (depending on $\mu, \eps, Y$), then we can find points
\[ z'_1 \in B(z_1, \mu ) \cap \RR \textQQqq{and} z'_2 \in B(z_2, \mu ) \cap \RR \]
such that there is a unique minimizing arclength geodesic $\sigma : [0, l] \to X$ such that $\sigma (0) = z'_1$, $\sigma(l) = z'_2$, $\sigma ([0,l] ) \subset \RR$ and such that
\begin{equation} \label{eq:preimagealmall}
 |[0,l] \setminus \sigma^{-1}(S)| < \mu. 
\end{equation}

We will now analyze the function $\ell : [0,l] \to \IR$ defined by
\[ \ell(s) := d(x', \sigma(s)). \]
We first claim that, assuming $\mu$ to be small enough, that
\[ \ell(s) > \zeta/ 4 \textQQqq{for all} s \in [0,l]. \]
Indeed, otherwise
\begin{multline*}
 d(x,z_1) + d(x,z_2) < 4\mu + d(x', z'_1) + d(x', z'_2) \\
 \leq 4\mu + d(x', \sigma(s)) + d(\sigma(s), z'_1) + d(x', \sigma(s)) + d(\sigma(s), z'_2) \\
 \leq 4\mu + \zeta/2 + d(z'_1, z'_2) < 6\mu + \zeta/2 + d(z_1, z_2), 
\end{multline*}
contradicting (\ref{eq:zetacondition2}) for $\mu < \zeta / 2$.
Next, note that $\ell(s)$ is $1$-Lipschitz on $[0,l]$, continuous and differentiable on $\sigma^{-1} (S)$ and for all $s \in \sigma^{-1} (S)$ we have
\begin{multline*}
 \ell' (s) = \big\langle \gamma'_{x', \sigma(s)} (d (x', \sigma(s))), \sigma'(s) \big\rangle \\
 = \big\langle \gamma'_{x', \sigma(s)} (d( x', \sigma(s))), \nabla_1 d(\sigma(s), p) \big\rangle \\
 + \big\langle \gamma'_{x', \sigma(s)} (d( x', \sigma(s))), \sigma'(s) -  \nabla_1 d(\sigma(s), p) \big\rangle. 
\end{multline*}
Here $\nabla_1 d(\cdot, \cdot)$ denotes the gradient with respect to the first argument.
So by (\ref{eq:almostlawcos}), we have for all $s \in \sigma^{-1} (S)$
\[ \bigg| \ell' (s) + \frac{d^2(x',p) - d^2 (\sigma(s),p) - \ell^2 (s)}{2d(\sigma(s), p) \ell(s)} \bigg| < \nu + |\sigma'(s) -  \nabla_1 d(\sigma(s), p)|. \]

Now note that, 
\begin{multline}
 \int_{\sigma^{-1} (S)} |\sigma'(s) - \nabla d(\sigma(s), p)|^2 ds = \int_{\sigma^{-1} (S)} \big( 2- 2\langle \sigma' (s), \nabla_1 d(\sigma(s), p) \rangle \big) ds \\ 
 < 2 l - 2 \int_{\sigma^{-1} (S)} \frac{d}{ds} d(\sigma(s), p) ds
 < 2d(z'_1, z'_2) - 2 \big( d(z'_2, p) - d(z'_1, p) \big) + 2 \mu
 < 6 \mu. \label{eq:L2ssigmanab}
\end{multline}
Set
\[ U := \sigma^{-1}( S ) \cap \big\{ |\sigma' (s) - \nabla_1 d(\sigma(s), p)| <  \mu^{1/4} \big\} \subset [0,l]. \]
Then, by (\ref{eq:preimagealmall}) and (\ref{eq:L2ssigmanab}),
\begin{equation} \label{eq:Ualmostall}
 |[0,l] \setminus U | < \mu + 6 \mu^{1/2}
\end{equation}
and for all $s \in U$ we have
\begin{equation} \label{eq:ellprime1}
 \bigg| \ell' (s) + \frac{d^2(x',p) - d^2 (\sigma(s),p) - \ell^2 (s)}{2d(\sigma(s), p) \ell(s)} \bigg| < 2\mu^{1/4}.
\end{equation}

Next, observe that for all $s \in [0,l]$
\[  d(\sigma(s), p) - d(z_1, p) - d(z'_1, \sigma(s)) \leq d(z'_1, z_1) < \mu \] 
and
\begin{multline*}
 d(\sigma(s), p) - d(z_1, p) - d(z'_1, \sigma(s)) \\
   = d(\sigma(s), p) - d(z_1, p) - \big( d(z'_1, z'_2) - d(z'_2, \sigma(s) ) \big) \\
   \geq d(z'_2, p) - d(z_1, p) - d(z'_1, z'_2) \geq - d(z_1, z'_1) > - \mu.
 \end{multline*}
So
\[ \big| d(\sigma(s), p) - \big( d(p, z_1) + s \big) \big| = \big| d(\sigma(s), p) - d(z_1,p) - d(z'_1, \sigma(s)) \big| <  \mu. \]
So, by (\ref{eq:ellprime1}) and the fact that $d(p,z_1) > \eps$ and $\ell (s) > \zeta / 2$ for all $s \in [0,l]$, we have for all $s \in U$
\[  \bigg| \ell' (s) + \frac{d^2(x',p) - (d (p, z_1) +s)^2 - \ell^2 (s)}{2(d(p, z_1)+s) \ell(s)} \bigg| < \Psi (\mu). \]
It follows that on $U$
\begin{multline}
 \bigg| \frac{d}{ds} \bigg( \frac{ \ell^2(s) - d^2(x', p) - (d(p, z_1) + s)^2}{2 (d(p,z_1) + s) d(x', p)} \bigg) \bigg| \\
= \frac1{d(x',p)} \bigg| \frac{2\ell' (s) \ell (s) - 2 (d(p,z_1) + s)}{2 (d(p,z_1) + s)} - \frac{\ell^2(s) - d^2 (x', p) - (d(p,z_1) + s)^2}{2 (d(p,z_1) + s)^2} \bigg| \\
= \frac{\ell(s)}{d(x',p) (d(p,z_1) + s)} \bigg| \ell' (s) +  \frac{d^2(x',p) - (d (p, z_1) +s)^2 - \ell^2 (s)}{2(d(p, z_1)+s) \ell(s)} \bigg| < \Psi(\mu). \label{eq:deronUsmall}
\end{multline}
Since $\ell(s)$ is $1$-Lipschitz, we furthermore find a uniform $C = C(\eps) < \infty$ such that the quantity
\[ \frac{ \ell^2(s) - d^2(x', p) - (d(p, z_1) + s)^2}{2 (d(p,z_1) + s) d(x', p)} \]
is $C$-Lipschitz on $[0,l]$.
Combining this fact with (\ref{eq:Ualmostall}) and (\ref{eq:deronUsmall}) yields that
\[ \bigg| \frac{ \ell^2(0) - d^2(x', p) - d^2(p, z_1)}{2 d(p,z_1) d(x', p)} - \frac{ \ell^2(l) - d^2(x', p) - (d(p, z_1)+l)^2}{2 (d(p,z_1) + l) d(x', p)} \bigg| < \Psi(\mu) \]
Since
\begin{align*}
 | \ell (0) - d(x,z_1)| = | d(x',z'_1) - d(x,z_1)| &< 2 \mu, \\
 | \ell (l ) - d(x, z_2)| = | d(x', z'_2) - d(x,z_1)| &< 2 \mu, 
\end{align*}
\[ | d(x', p) - d(x,p) | < \mu, \]
\[ \big| \big(d(p,z_1) + l \big) - d(p, z_2) \big| = | d(p,z_1) + d(z'_1, z'_2) - d(p,z_2) | < 2 \mu, \]
we obtain (\ref{eq:simpleconeid}) for sufficiently small $\mu$.
Fixing $\mu$, we can then choose $\nu$ and eventually $\delta$.
\end{proof}

Applying Lemma \ref{Lem:simpleconeid} twice yields:

\begin{Lemma} \label{Lem:coneidentity}
For any $0 < \eps < 1$, $\mathbf{p}_0 > 3$ and $Y < \infty$ there is a $0 < \delta = \delta (\eps, \mathbf{p}_0, Y) < \eps$ such that the following holds:

Let $\XX$ be a singular space with mild singularities of codimension $\mathbf{p}_0$ that is $Y$-tame at scale $\delta^{-1} r$ for some $r > 0$.
Assume moreover that $\Ric \geq - (n-1) \kappa$ on $\RR$ for some $0 \leq \kappa \leq \delta r^{-2}$.
Let $p \in \RR$ and assume that
\[ \frac{| B(p, \delta r) \cap \RR |}{v_{-\kappa} (\delta r)} - \frac{|B(p, 32r) \cap \RR|}{v_{-\kappa} (32r)} < \delta. \]
Let $x_1, x_2, y_1, y_2 \in A(p, \eps r, r)$ such that
\begin{align*}
 d(p, x_2) &= d(p, x_1) + d(x_1, x_2), \\
 d(p, y_2) &= d(p, y_1) + d(y_1, y_2). 
\end{align*}
Then
\[ \bigg| \frac{d^2(x_1,y_1) - d^2(x_1,p) - d^2(y_1, p)}{2d(x_1, p) d(y_1,p)} - \frac{d^2(x_2,y_2) - d^2(x_2,p) - d^2(y_2, p)}{2d(x_2, p) d(y_2,p)} \bigg| < \eta. \]
\end{Lemma}

Next we show that balls around $p$ are almost star-shaped.

\begin{Lemma} \label{Lem:almoststar}
For any $\eps, \eta, a > 0$ there is a $\delta = \delta (\eps, \eta, a) > 0$ such that the following holds:

Let $\XX$ be a singular space with mild singularities, $r > 0$ and $p \in \RR$ and assume that for some $0 \leq \kappa \leq r^{-2}$ we have $\Ric \geq - (n-1) \kappa$ on $\RR$ and that
\[ \frac{| B(p, \eps r/4) \cap \RR |}{v_{-\kappa} (\eps r/4)} - \frac{|B(p, 4r) \cap \RR|}{v_{-\kappa} (4r)} < \delta. \]
Assume moreover, that $\XX$ is $a$-noncollapsed at scale $r$ in the sense that for any $0 < r' < r$ and $x \in X$ we have
\[ | B(x, r') \cap \RR | > a r^{\prime n}. \]
Then for any $x \in A(p, \eps r,r)$ there is an arclength minimizing geodesic $\gamma : [0, r] \to \RR$ with $\gamma(0) = p$ and a parameter $s \in [0,r]$ such that $d(x, \gamma(s)) \leq \eta r$.
\end{Lemma}

\begin{proof}
Without loss of generality, we may assume that $r = 1$, $\eps = 2^{-k} < 1/10$ and $\eta < \eps /10$.
We will determine $\delta$ at the end of the proof.

Assume that the conclusion was wrong for some $x \in A(p, \eps, 1)$.
Consider an arbitrary point $y \in (A(p,2, 4) \cap \RR) \setminus Q_p$ and let $\gamma^* : [0, l] \to \RR$ be an arclength minimizing geodesic between $p$ and $y$.
Then by assumption
\[ \gamma^*([0, l]) \cap B(x, \eta ) \subset \gamma^* ([0,2]) \cap B(x, \eta ) = \emptyset. \]
So, by looking at the proof of Proposition \ref{Prop:volumecomparison}, and taking into account that these geodesic segments $\gamma^*$ between $p$ and points in $A(p,2,4)$ avoid $B(x, \eta)$, we find that for any $j =  0, 1, 2, \ldots$
\[ \frac{v_{-\kappa} (2^{-j+1}) - v_{-\kappa} (2^{-j})}{v_{-\kappa} (4) - v_{-\kappa} (2)}  |A (p, 2, 4) \cap \RR | \leq | (A(p,2^{-j},2^{-j+1}) \setminus B(x, \eta)) \cap \RR|. \]
Summing this inequality over all $j = 0, \ldots, k+2$ yields
\[   \frac{v_{-\kappa} (2) - v_{-\kappa} (\eps / 4)}{v_{-\kappa} (4) - v_{-\kappa} (2)} |A (p, 2, 4) \cap \RR | \leq | A(p, \eps/4, 2) \cap \RR | - |B(x, \eta) \cap \RR|. \]
Combining this with volume comparison, Proposition \ref{Prop:volumecomparison} and the assumption of this lemma, we get
\begin{align*}
 |B(x,\eta ) \cap \RR|  
  &\leq  \frac{v_{-\kappa} (2) - v_{-\kappa} (\eps / 4)}{v_{-\kappa} (\eps/ 4)} | B(p, \eps/4) \cap \RR | \\
  &\qquad -  \frac{v_{-\kappa} (2) - v_{-\kappa} (\eps / 4)}{v_{-\kappa} (4) - v_{-\kappa} (2)}  \big( |B (p,  4) \cap \RR | - |B(p,2) \cap \RR| \big) \\
  &\leq \bigg(  \frac{v_{-\kappa} (2) - v_{-\kappa} (\eps / 4)}{v_{-\kappa} (\eps/ 4)} \\
  &\qquad\qquad + \frac{v_{-\kappa} (2) - v_{-\kappa} (\eps / 4)}{v_{-\kappa} (4) - v_{-\kappa} (2)} \cdot \frac{v_{-\kappa} (2)}{v_{-\kappa}(\eps / 4)} \bigg) | B(p, \eps/4) \cap \RR | \\
 &\qquad  - \frac{v_{-\kappa} (2) - v_{-\kappa} (\eps / 4)}{v_{-\kappa} (4) - v_{-\kappa} (2)}   |B (p,  4) \cap \RR | \\
 &= v_{-\kappa} (4) \cdot \frac{v_{-\kappa} (2) - v_{-\kappa} (\eps / 4)}{v_{-\kappa} (4) - v_{-\kappa} (2)}  \bigg( \frac{ |B(p,4) \cap \RR |}{v_{-\kappa} (4)}- \frac{| B(p, \eps / 4) \cap \RR|}{v_{-\kappa}(\eps/4)} \bigg) \\
 &< v_{-\kappa}(4) \cdot \frac{v_{-\kappa} (2) - v_{-\kappa} (\eps / 4)}{v_{-\kappa} (4) - v_{-\kappa} (2)} \cdot  \delta.
\end{align*}
Using our assumption, we get
\[ a \eta^n < |B(x, \eta) \cap \RR| \leq v_{-\kappa} (4) \cdot \frac{v_{-\kappa} (2) - v_{-\kappa} (\eps / 4)}{v_{-\kappa} (4) - v_{-\kappa} (2)} \cdot \delta. \]
So we obtain a contradiction for sufficiently small $\delta$, depending on $\eps, \eta, a$.
\end{proof}

We can finally prove the Cone Rigidity Theorem.

\begin{proof}[Proof of Theorem \ref{Thm:conerigidityIntroduction}]
It suffices to prove the following fact:
{\it Fix $Y < \infty$, $\mathbf{p}_0 > 3$ and consider a sequence $\delta_i \to 0$, a sequence of singular spaces $\XX_i$ with mild singularities of codimension $\mathbf{p}_0$ that are $Y$-tame at scale $\delta_i^{-1}$ that satisfy $\Ric \geq -(n-1) \kappa_i$ on $\RR$ for some $0 \leq \kappa_i \leq \delta_i^2$ and consider points $p_i \in X_i$ such that
\[ \frac{|B(p_i, \delta_i) \cap \RR_i |}{v_{-\kappa_i} (\delta_i)} - \frac{|B(p_i, 32) \cap \RR_i |}{v_{-\kappa_i}(32)} < \delta_i. \]
Then, after passing to a subsequence, the pointed metric spaces $(B^{X_i} (p_i, 1), p_i)$ Gromov-Hausdorff converge to a $1$-ball $(B^{\mathcal{C}}(\ov{p}, 1), \ov{p})$ around the vertex $\ov{p} \in \mathcal{C}$ of a metric cone $\mathcal{C}$.}

First, observe that we may assume without loss of generality that $p_i \in \RR_i$.
Next, note that in this setting, for any $i$ and any $\delta_i \leq r_1 \leq r_2 \leq 32$ we have
\[ \frac{|B(p_i, r_1) \cap \RR_i|}{v_{-\kappa_i} (r_1)} - \frac{|B(p_i, r_2) \cap \RR_i |}{v_{-\kappa_i} (r_2)} < \delta_i. \]
After passing to a subsequence, we may assume that the pointed metric spaces $(B^{X_i} (p_i, 1), p_i)$ converge to some locally compact, pointed metric length space $(Z, d_Z, z_0)$.
Using Lemma \ref{Lem:almoststar}, we obtain that for any $z \in Z$ and $\eta > 0$ there is a point $z' \in Z$ such that $d(z_0, z') > 1 - \eta$ and $d(z_0, z) + d(z, z') < d(z_0, z') + \eta$.
By local compactness and letting $\eta \to 0$ we can then conclude that there is even a point $z'' \in Z$ with $d(z_0, z'') = 1$ and $d(z_0, z) + d(z, z') = d(z_0, z')$.
It follows that for any $z \in Z$ there is an arclength minimizing geodesic $\gamma : [0,1] \to Z$ with $\gamma(0) = z_0$ such that $z = \gamma(s)$ for some $s \in [0,1]$.

Let $W$ be the set of all arclength minimizing geodesics $\gamma : [0,1] \to Z$ with $\gamma(0) = z_0$.
By our previous conclusion, image of the map $\Phi : W \times [0,1] \to Z, (\gamma, s) \mapsto \gamma(s)$ contains $B(z_0, 1)$.
For any $\gamma_1, \gamma_2 \in W$, we define the angle
\[ \alpha (\gamma_1, \gamma_2) := \arccos \bigg( \frac{d^2(\gamma_1(1), \gamma_2(1)) - 2}{2} \bigg) \]
according to the law of cosines.
Then, by Lemma \ref{Lem:coneidentity}, for any $\gamma_1, \gamma_2 \in W$ and $s_1, s_2 \in [0,1]$
\begin{equation} \label{eq:lawcosalpha}
 d^2 (\gamma_1(s_1), \gamma_2(s_2)) = s_1^2 + s_2^2 - 2 s_1 s_2 \alpha (\gamma_1, \gamma_2). 
\end{equation}
Using this identity, it can be seen that $\alpha$ is a pseudometric on $W$.
Let $W'$ be the metric space obtained from $(W, \alpha)$ by quotienting out the relation $d(\gamma_1, \gamma_2) = 0$ and define $(\mathcal{C}, \ov{p})$ to be the cone over $W'$.
By (\ref{eq:lawcosalpha}), the pointed metric space $(Z, z_0)$ is isometric to $(B^{\mathcal{C}}( \ov{p}, 1), \ov{p})$.
This finishes the proof.
\end{proof}

\section{Splitting maps induce almost metric splitting}
The goal of this section is to prove the following result:

\begin{Proposition} \label{Prop:splittingmapGHsplitting}
For every $\eps > 0$, $\mathbf{p}_0 > 3$ and $Y < \infty$ there is a $\delta = \delta (\eps, \mathbf{p}_0, Y) > 0$ such that the following holds:

Let $\XX$ be a singular space with mild singularities of codimension $\mathbf{p}_0$ that is $Y$-tame at scale $\delta^{-1} r$ for some $r > 0$.
Assume moreover that $\Ric \geq - (n-1) r^{-2}$ on $\RR$.
Let $p \in X$ and $k \leq n$ and consider a $\delta$-splitting map $u : B(p, 100 r) \to \IR^k$.
Then there is a pointed metric space $(Z, d_Z, z)$ such that
\[ d_{GH} \big( \big( B^X (p, r), p \big), \big( B^{Z \times \IR^k} ((z, 0^k), r), (z, 0^k) \big) \big) < \delta r. \]
\end{Proposition}

The proof of this proposition is similar to the proof of the cone splitting theorem, Theorem \ref{Thm:conerigidityIntroduction}.
As a preparation we first prove a lemma that is similar to Lemma \ref{Lem:almostlawcos}.

\begin{Lemma} \label{Lem:almostlawcossplittingcase}
For any $\eps > 0$, $\mathbf{p}_0 > 3$ and $Y < \infty$ there is a $\delta = \delta (\eps, \mathbf{p}_0, Y) > 0$ such that the following holds:

Let $\XX$ be a singular space with mild singularities of codimension $\mathbf{p}_0$ that is $Y$-tame at scale $\delta^{-1} r$ for some $r > 0$.
Assume moreover that $\Ric \geq - (n-1) r^{-2}$ on $\RR$.
Let $p \in X$ and $k \leq n$ and consider a $\delta$-splitting map $u : B(p, 3 r) \to \IR^k$.

Let $x \in B(p, r)$.
Then there is a point $x' \in B(x, \eps r)$ and an open subset $S \subset B(p, r) \cap \RR$ such that the following holds:
\begin{enumerate}[label=(\alph*)]
\item $|(B(p,r) \cap \RR) \setminus S | < \eps r^n$.
\item For any $y \in S$ there is a unique minimizing arclength geodesic $\gamma_{x', y} : [0, d(x',y)] \to X$ with $\gamma_{x', y} ([0, d(x',y)]) \subset \RR$, $\gamma_{x',y}$ varies continuously in $y$ and we have for any $l = 1, \ldots, k$
\begin{equation} \label{eq:almlinearul}
 \bigg| \langle \gamma'_{x',y} ( d(x', y)), \nabla u^l \rangle - \frac{u^l (y) - u^l(x')}{d(x', y)} \bigg| < \eps . 
\end{equation}
\end{enumerate}
\end{Lemma}

\begin{proof}
The proof is similar to that of Lemma \ref{Lem:almostlawcos}.
Instead of applying Proposition \ref{Prop:segmentinequ} to $|\nabla^2 b - 2g|$, we now need to apply this proposition to $|\nabla^2 u^1| + \ldots + |\nabla^2 u^k|$.
Then (\ref{eq:1dintissmall}) becomes
\[ \int_0^d \big| (u^l \circ \gamma)''(s) \big| ds < \Psi (\nu | Y) \textQQqq{for all} l = 1, \ldots, k. \]
Consequently, (\ref{eq:twiceintegrated}) becomes
\[ | (u^l \circ \gamma )(0) + (u^l \circ \gamma)'(d) \cdot d - (u^l \circ \gamma)(d) | < \Psi (\nu | Y) \cdot d^2 \textQq{for all} l = 1, \ldots, k. \]
The bound (\ref{eq:almlinearul}) can be derived from this inequality.
\end{proof}

Next, we prove a lemma that is similar to Lemma \ref{Lem:simpleconeid}.

\begin{Lemma} \label{Lem:u3ptid}
For any $\eps > 0$, $\mathbf{p}_0 > 3$ and $Y < \infty$ there is a $\delta = \delta(\eps, \mathbf{p}_0, Y) > 0$ such that the following holds:

Let $\XX$ be a singular space with mild singularities of codimension $\mathbf{p}_0$ that is $Y$-tame at scale $\delta^{-1} r$ for some $r > 0$.
Assume moreover that $\Ric \geq - (n-1) r^{-2}$ on $\RR$.
Let $p \in X$ and $k \leq n$ and consider a $\delta$-splitting map $u : B(p, 9 r) \to \IR^k$.

Let $x, z_1, z_2 \in B(p, r)$ such that
\begin{equation} \label{eq:dz1z2almostdeltau}
 \big| d(z_1, z_2) - |u(z_1) - u(z_2)| \big| < \delta r. 
\end{equation}
Then
\[ \big| \big( d^2(z_1, x) - |u(z_1) - u(x)|^2 \big) - \big( d^2 (z_2, x) - |u(z_2) - u(x)|^2 \big) \big| < \eps r^2. \]
\end{Lemma}

\begin{proof}
Without loss of generality, we may assume that $r= 1$.
Next, consider the function
\[ u^* := \sum_{l=1}^n \frac{u^l (z_1) - u^l(z_1)}{|u (z_1) - u (z_2)|} \cdot u^l. \]
Note that $u^* : B(p, 9) \to \IR$ is a $C \delta$-splitting for some uniform $C < \infty$ and we still have
\[ \big| d(z_1, z_2) - |u^*(z_1) - u^*(z_2) | \big| < k \delta. \]
Moreover
\begin{multline*}
 |u(z_1) - u(x)|^2 - |u(z_2) - u(x)|^2 = \big\langle u(z_1) - u(z_2), u(z_1) + u(z_2) - 2 u(x) \big\rangle  \\
 = \big( u^* (z_1) - u^*(z_2) \big) \cdot \big( u^* (z_1) + u^*(z_2) - 2 u^* (x) \big) \\
 = \big( u^* (z_1) - u^* (x) \big)^2 - \big( u^* (z_2) - u^* (x) \big)^2.
 \end{multline*}
This shows that in the following we may assume without loss of generality that $k =1$.
Moreover, we may assume without loss of generality that $u(z_1) \leq u(z_2)$.

We now follow closely the arguments in the proof of Lemma \ref{Lem:simpleconeid}, replacing $d(\cdot, p)$ by $u$.
Let $0 < \nu < \mu < \eps$ be constants whose values we will determine in the course of the proof, depending on $\eps, Y$.
Apply Lemma \ref{Lem:almostlawcossplittingcase} for $r \leftarrow 3r$ and $\eps \leftarrow \nu / 3$ to obtain a point $x' \in B(x, \nu) \cap \RR$ and the open subset $S \subset B(p, 3r)$.
As in the proof of Lemma \ref{Lem:simpleconeid}, assuming $\nu$ to be sufficiently small depending on $\mu$, we can find points
\[ z'_1 \in B(z_1, \mu) \cap \RR \textQQqq{and} z'_2 \in B(z_2, \mu) \cap \RR \]
such that there is a unique minimizing arclength geodesic $\sigma : [0,l] \to X$ such that $\sigma(0) = z'_1$, $\sigma(l) = z'_2$, $\sigma([0,l]) \subset \RR$ and such that
\begin{equation} \label{eq:sigmainverseSsplitting}
 |[0,l] \setminus \sigma^{-1} (S)| < \mu. 
\end{equation}

As in the proof of Lemma \ref{Lem:simpleconeid}, we define the function $\ell : [0,l] \to \IR$ by
\[ \ell (s) := d( x', \sigma(s)). \]
Then $\ell$ is $1$-Lipschitz on $[0,l]$, $C^1$ on $\sigma^{-1} (S)$ and we have for any $s \in \sigma^{-1} (S)$
\begin{multline*}
 \ell' (s) = \big\langle \gamma'_{x', \sigma(s)} (d(x', \sigma(s))), \sigma'(s) \big\rangle \\ = \big\langle \gamma'_{x', \sigma(s)} (d(x', \sigma(s))), \nabla u \big\rangle + \big\langle \gamma'_{x', \sigma(s)} (d(x', \sigma(s))), \sigma'(s) - \nabla u \big\rangle. 
\end{multline*}
So, using Lemma \ref{Lem:almostlawcossplittingcase}, it follows that for all $s \in \sigma^{-1} (S)$
\begin{equation} \label{eq:ellprimeuu}
 \bigg| \ell' (s) - \frac{u(\sigma(s)) - u (x') }{\ell(s)} \bigg|  < \mu + |\sigma'(s) - \nabla u(\sigma(s))|.  
\end{equation}
We can estimate that for small $\delta$, depending on $\mu$, we have for all $s \in [0,l]$
\[  u(\sigma(s)) - u(z_1) - d(z'_1, \sigma(s)) \leq (1+\delta) d(\sigma(s), z_1) - d(z'_1, \sigma(s)) < 2 \delta + \mu < 2 \mu \]
and, using (\ref{eq:dz1z2almostdeltau}),
\begin{multline*}
 u(\sigma(s)) - u(z_1) - d(z'_1, \sigma(s))=  \big( u(\sigma(s)) - u(z_2) \big) + \big( u(z_2) - u(z_1) \big) - d(z'_1, \sigma(s)) \displaybreak[1] \\
 \geq - (1+\delta) d(\sigma(s), z_2) + d(z_1, z_2) - \delta - d(z'_1, \sigma(s)) \displaybreak[1] \\
 \geq - 3 \delta - 3\mu + d(z'_1, z'_2) - d(\sigma(s), z'_2) - d(z'_1, \sigma(s)) = - 3 \delta - 3 \mu > - 4 \mu.
\end{multline*}
So for all $s \in [0,l]$
\[ \big| u(\sigma(s)) - u(z_1) - s \big| < 4 \mu. \]
Combining this with (\ref{eq:ellprimeuu}) shows that for all $s \in \sigma^{-1} (S)$
\begin{equation} \label{eq:ellprimesquare}
 \big|  \big( \ell^2 (s) \big)' - 2u(z_1) - 2s +2 u(x') \big| < 10 \mu + | \sigma' (s) - \nabla u (\sigma(s)) | . 
\end{equation}
Next, using (\ref{eq:dz1z2almostdeltau}) we find for sufficiently small $\delta$, depending on $\mu$, that
\begin{multline} \label{eq:L2smallsigmaprime}
 \int_0^l |\sigma'(s) - \nabla u (\sigma(s))|^2 ds = \int_0^l \big( 1 + |\nabla u |^2(\sigma(s)) - 2 \langle \sigma' (s), \nabla u (\sigma(s)) \rangle \big) \\
 \leq l + (1+ \delta) l  - 2 \big( u (z'_2) - u(z'_1) \big) \leq 2 d(z_1, z_2) + 4 \mu + 2 \delta - |u(z'_2) - u(z'_1) | \\
 \leq 4 \mu + 4 \delta + |u(z_1) - u(z_2)|  - | u(z'_1) - u(z'_2) | \leq 8 \mu + 4 \delta \leq 10 \mu.
\end{multline}
The bounds (\ref{eq:sigmainverseSsplitting}), (\ref{eq:ellprimesquare}), (\ref{eq:L2smallsigmaprime}) and the fact that $\ell(s)$ is $1$-Lipschitz then implies as in the proof of Lemma \ref{Lem:simpleconeid} that
\[ \big|  \ell^2(l) - \ell^2(0) - 2l u(z_1) + 2l u(x') - l^2 \big| < \Psi ( \mu). \]
It follows that
\[ \big| d^2(z_2, x) - d^2 (z_1, x) - 2 d(z_1, z_2) u(z_1) + 2 d(z_1, z_2) u(x) - d^2 (z_1, z_2) \big| < \Psi ( \mu). \]
Using (\ref{eq:dz1z2almostdeltau}), we find
\begin{multline*}
 \big| d^2(z_2, x) - d^2 (z_1, x) - 2 \big( u(z_2) - u(z_1) \big)  u(z_1) \\
  + 2 \big( u(z_2) - u(z_1) \big) u(x)  - \big( u(z_1) - u(z_2) \big)^2 \big| < \Psi(\delta, \mu). 
\end{multline*}
So
\[ \big| d^2(z_2,x) - d^2(z_1,x) + \big(u(z_1) - u(x) \big)^2  - \big( u(z_2) - u(x) \big)^2 \big| < \Psi (\delta, \mu). \]
Hence the claim follows for sufficiently small $\mu$ and $\delta$.
\end{proof}

Applying Lemma \ref{Lem:u3ptid} twice yields

\begin{Lemma} \label{Lem:u4ptid}
For any $\eps > 0$, $\mathbf{p}_0 > 3$ and $Y < \infty$ there is a $\delta = \delta(\eps, \mathbf{p}_0, Y) > 0$ such that the following holds:

Let $\XX$ be a singular space with mild singularities of codimension $\mathbf{p}_0$ that is $Y$-tame at scale $\delta^{-1} r$ for some $r > 0$.
Assume moreover that $\Ric \geq - (n-1) r^{-2}$ on $\RR$.
Let $p \in X$ and $k \leq n$ and consider a $\delta$-splitting map $u : B(p, 9 r) \to \IR^k$.

Let $x_1, x_2, y_1, y_2 \in B(p, r)$ such that
\begin{align*}
 \big| d(x_1, x_2) - |u(x_1) - u(x_2)| \big| &< \delta r, \\
 \big| d(y_1, y_2) - |u(y_1) - u(y_2)| \big| &< \delta r.
\end{align*}
Then
\[ \big| \big( d^2(x_1, y_1) - |u(x_1) - u(y_1)|^2 \big) - \big( d^2 (x_2, y_2) - |u(x_2) - u(y_2)|^2 \big) \big| < \eps r^2. \]
\end{Lemma}

We also need

\begin{Lemma} \label{Lem:pointinapproxdir}
For any $\eps > 0$ and $Y < \infty$ there is a $\delta = \delta (\eps, Y) > 0$ such that the following holds:

Let $\XX$ be a singular space with mild singularities of codimension $\mathbf{p}_0 >1$ that is $Y$-tame at some scale $r > 0$.
Let $p \in X$ and $k \in \{ 1, \ldots, n \}$ and consider a $\delta$-splitting map $u : B(p, 3r) \to \IR^k$.
Let $x \in B(p,r)$ and choose $w \in \IR^k$ such that $| w - u(x)  | < r$.
Then we can find a point $y \in B(p,3r)$ such that
\begin{equation} \label{eq:duxuy}
 \big| d(x,y) - |u(x) - u(y)| \big| < \eps r 
\end{equation}
and
\begin{equation} \label{eq:uywepsr}
 |u(y) - w| < \eps r. 
\end{equation}
\end{Lemma}

\begin{proof}
The proof is similar to that of Lemma \ref{Lem:eps0splitting}.
Without loss of generality, we may assume again that $r = 1$.
Let $0 < \mu < 1/10$ be a constant whose value will be determined in the course of the proof, depending on $\eps$ and $Y$.
Set $w' := w - u(x) \in \IR^k$ and apply Proposition \ref{Prop:segmenttype} with $l := |w'| < 1$ and $f = u^i$.
Then
\[ \int_{S^* B(x, \mu)} \Big| |w'| \cdot \langle \nabla u^i, v \rangle  - \big( u^i (\gamma_v(l)) - u^i (\gamma_v(0)) \big) \Big| dv < \Psi (\delta). \]
So we can find a point $x' \in B(x, \mu) \cap \RR$ such that $S^*_{x'} \RR \setminus S^* \RR$ has measure zero and such that for all $i = 1, \ldots, k$
\[ \int_{S^*_{x'} \RR}  \Big| |w'| \cdot \langle \nabla u^i, v \rangle  - \big( u^i (\gamma_v(l)) - u^i (x') \big) \Big| dv < \Psi (\delta | Y). \]
For small $\delta > 0$ it is then possible to find some vector $v \in  S^*_{x'} \RR$ such that
\[ \big| (w')^i - |w'| \cdot \langle \nabla u^i, v \rangle \big| < \mu  \]
and
\[ \Big| |w'| \cdot \langle \nabla u^i, v \rangle  - \big( u^i (\gamma_v(l)) - u^i (x') \big) \Big| < \Psi (\delta |Y). \]
Set $y := \gamma_v(l)$.
Then
\[ \Big| (w')^i - \big( u^i(y) - u^i(x') \big)  \Big| < \Psi (\mu, \delta | Y). \]
As $| u(x') - u(x) | \leq (1+\delta) d(x',x) < 2 \mu$ for small $\delta$, we obtain (\ref{eq:uywepsr}) for small enough $\mu$.
To see (\ref{eq:duxuy}) observe that
\[  |u(x) - u(y) | \leq (1+\delta) d(x,y) \leq d(x,y) + 2 \delta \]
and
\begin{multline*}
 |u(x) - u(y) | \geq |w - u(x) | - |u(y) - w| \geq |w'| - \Psi (\mu, \delta | Y) \\
 = d(y,x') - \Psi(\mu, \delta | Y) \geq d(x,y) - \Psi(\mu, \delta | Y). 
\end{multline*}
This finishes the proof for small $\mu$ and small $\delta$.
\end{proof}

With Lemmas \ref{Lem:u4ptid} and \ref{Lem:pointinapproxdir} in hand, we can finally prove Proposition \ref{Prop:splittingmapGHsplitting}.

\begin{proof}[Proof of Proposition \ref{Prop:splittingmapGHsplitting}.]
Without loss of generality, we may assume that $r = 1$.
It suffices to show the following claim:
{\it Let $Y < \infty$, $\mathbf{p}_0 > 3$ and $k \in \{ 1, \ldots, n \}$ be fixed numbers and consider a sequence of positive numbers $\delta_i \to 0$ and a sequence of singular spaces with mild singularities of codimension $\mathbf{p}_0$ that are $Y$-tame at scale $\delta_i^{-1}$ and on whose regular part we have $\Ric \geq - (n-1)$.
Let moreover $p_i \in X_i$ be points and consider $\delta_i$-splitting maps $u_i : B^{X_i} (p_i, 100) \to \IR^k$.
Then, after passing to a subsequence, the pointed metric spaces $(X_i, d_{X_i}, p_i)$ Gromov-Hausdorff converge to a pointed metric space $(X_\infty, d_{X_\infty}, p_\infty)$ such that there is a pointed metric space $(Z,d_Z, z)$ with the property that $(B^{X_\infty} (p_\infty, 1), p_\infty)$ is isometric to $(B^{Z \times \IR^k} ((z ,0^k), 1), (z, 0^k))$.}

Let us now prove this claim.
After passing to a subsequence, we may assume that the $(X_i, d_{X_i}, p_i)$ already Gromov-Hausdorff converge to some pointed, locally compact, metric length space $(X_\infty, d_{X_\infty}, p_\infty)$.
After passing to a subsequence once again, and after replacing $u_i$ by $u_i - u_i(p_i)$, we may also assume that the $u_i$ uniformly converge to some $u_\infty : X_\infty \to \IR^k$.
By the definition of a splitting map, $u_\infty |_{B^{X_\infty} (p_\infty, 50)}$ must be $1$-Lipschitz.

We will now analyze the geometry of $B^{X_\infty} (p_\infty,1)$ and the function $u_\infty$ further.
Using Lemma \ref{Lem:pointinapproxdir}, we obtain the following fact:
For any $y \in B^{X_\infty} (p_\infty, 2)$ and $v \in \IR^k$ with $|v| < 1$ there is a $y' \in X_\infty$ such that $u_\infty (y') = v$ and $d_{X_\infty} (y,y') = |u_\infty (y) - u_\infty (y')|$.
On the other hand, if $y, y', y'' \in B^{X_\infty} (p_\infty, 10)$ such that $u_\infty (y') = u_\infty (y'')$ and
\[ d_{X_\infty} (y,y') = d_{X_\infty} (y, y'') = |u_\infty (y) - u_\infty (y')| , \]
then Lemma \ref{Lem:u4ptid} implies that $y' = y''$.
So we can define the map
\[ \Phi : \big( \{ u_\infty = 0^k \} \cap B^{X_\infty} (p_\infty, 2) \big) \times B^{\IR^k}(0^k, 1) \to X_\infty \]
that sends each $(y,v)$ pair to the unique point $y' \in X_\infty$ with the property that $u_\infty (y') = v$ and $d_{X_\infty} (y,y') = |u_\infty (y) - u_\infty (y')|$.
Moreover, this map is injective (by Lemma \ref{Lem:u4ptid}) and its image contains $B^{X_\infty} (p_\infty,1)$ (by Lemma \ref{Lem:pointinapproxdir}).
By Lemma~\ref{Lem:u4ptid}, the map $\Phi$ induces an isometry between the $1$-ball around $(p_\infty, 0^k)$ in
\[ \big( \{ u_\infty = 0^k \} \cap B^{X_\infty} (p_\infty, 2) \big) \times \IR^k \]
and $B^{X_\infty} (p_\infty, 1)$.
\end{proof}

\section{Curvature estimates under cone regularity assumption}
In this section we derive bounds on the volume of the sublevel sets $\{ \rrm < s \}$ on $Y$-tame and $Y$-regular singular spaces with mild singularities of codimension $\mathbf{p}_0 > 3$.
These bounds will be obtained by adapting the techniques developed by Cheeger and Naber in \cite{Cheeger-Naber-quantitative} to our setting.
The bounds on the volume of the sublevel sets of $\{ \rrm < s \}$ imply $L^p$-bounds on $\rrm^{-1}$ and they can also be used to bound the Hausdorff dimension of singular sets that occur as we analyze limits of degenerations $\XX_i$ of singular spaces.
The possible values of $p$ or, equivalently, the Hausdorff dimension of the singular set depend on a working condition, which states that Gromov-Hausdorff closeness to metric products of the form $Y \times \IR^{n-e}$ implies local regularity.
This working condition is true for $e = 0$ by default, by Corollary~\ref{Cor:GHepsregularity}.
In subsection \ref{subsec:Codim4}, we will improve this working condition to $e = 2$, giving us stronger $L^p$-bounds.

\begin{Proposition} \label{Prop:LpboundXXversion1}
For any $\eps > 0$, $\mathbf{p}_0 > 3$ and $Y < \infty$ there is an $E = E(\eps, \mathbf{p}_0, Y) < \infty$ such that the following holds:

Let $\XX$ be a singular space that has mild singularities of codimension $\mathbf{p}_0$.
Assume moreover that $\XX$ is $Y$-tame at some scale $r > 0$ and that $\Ric \geq - (n-1) r^{-2}$ on $\RR$.
Assume that for some $e \in \{ 0, \ldots, n \}$ the following property holds:
If $p' \in X$, $0 < r' < r$ and if there exists a metric cone $(Z, d_Z, z)$ with vertex $z$ such that
\begin{equation} \label{eq:conecondition}
 d_{GH} \big( \big( B(p',r'), p' \big), \big( B^{Z \times \IR^{n-e}} ((z, 0^{n-e}),r'), (z, 0^{n-e} ) \big) \big) < \eps r', 
\end{equation}
then $\rrm (p') > \eps r'$.

Then for any $0 < s < 1$ and $p \in X$ we have
\begin{equation} \label{eq:eplus1condition}
 |\{ 0 < \rrm < sr \} \cap B(p,r) | \leq E s^{(e+1) - \eps} r^n. 
\end{equation}
\end{Proposition}

\begin{proof}
This fact is proved in \cite{Cheeger-Naber-quantitative}.
Note that this proof relies only on volume comparison estimates, which hold due to Proposition \ref{Prop:volumecomparison}, Cheeger and Colding's Cone-Splitting Theorem, which holds in the singular setting due to Theorem \ref{Thm:conerigidityIntroduction}, and estimates involving metric geometry, which can still be carried out in our setting.
\end{proof}

So by Corollary \ref{Cor:GHepsregularity} we have:

\begin{Corollary} \label{Cor:preliminaryLpbound}
For any $\eps > 0$, $\mathbf{p}_0 > 3$ and $Y < \infty$ there is an $E = E(\eps, \mathbf{p}_0, Y) < \infty$ such that the following holds:

Let $\XX$ be a singular space that has mild singularities of codimension $\mathbf{p}_0$.
Assume moreover that $\XX$ is $Y$-tame and $Y$-regular at some scale $r > 0$ and that $\Ric = \lambda g$ on $\RR$ for some $| \lambda | \leq (n-1) r^{-2}$.
Then for any $0 < s < 1$ and $p \in X$ we have
\[
 |\{ 0 < \rrm < sr \} \cap B(p,r) | \leq E s^{1 - \eps} r^n. 
\]
\end{Corollary}

In the case in which $e = 2$ and the singular space is Einstein in Proposition \ref{Prop:LpboundXXversion1}, we will soon replace the condition (\ref{eq:conecondition}) by an only slightly stronger condition, which will enable us to deduce (\ref{eq:eplus1condition}) for $e = 3$.
In order to do this, we need to establish the following lemma.

\begin{Lemma} \label{Lem:coneissmooth}
Let $\eps > 0$, $\mathbf{p}_0 > 3$ and $Y < \infty$ and consider a sequence $\XX_i$ of singular spaces with mild singularities of codimension $\mathbf{p}_0$ that are $Y$-tame at scale $1$.
Assume also that we have $\Ric \geq - (n-1)$ on the regular part of each $\XX_i$.
Let $e \in \{ 0, \ldots, n-1 \}$.

Assume that for any $i$, $p \in X_i$, $0 < r < 1$ and every pointed metric space $(Z, d_Z, z)$ the following holds:
If
\[ d_{GH} \big( \big( B^{X_i} (p, r), p \big), \big( B^{Z \times \IR^{n-e}} ((z,0^{n-e}), r) \big) \big) < \eps r, \]
then $\rrm (p) > \eps r$.

Let now $x_i \in X_i$ and assume that the pointed metric spaces $(X_i, d_i, x_i)$ Gromov-Hausdorff converge to a metric space of the form $(Z_0 \times \IR^{n-e-1}, (z_0, 0^{n-e-1}))$,  where $(Z_0, d_{Z_0}, z_0)$ is a metric cone with vertex $z_0$.
Then $Z_0 \setminus \{ z_0 \}$ is an $e+1$-dimensional differentiable manifold and the metric $d_{Z_0}$ on $Z_0 \setminus \{ z_0 \}$ is locally isometric to the length metric of a $C^3$-Riemannian metric $g_\infty$ on $Z_0 \setminus \{ z_0 \}$.

Moreover, for any relatively compact, open $U \subset Z_0 \setminus \{ z_0 \}$ and $r > 0$ there is a sequence of $C^3$-diffeomorphisms $\Phi_i : U \times B^{\IR^{n-e-1}} (0^{n-e-1}, r) \to \RR_i$ such that $\Phi^*_i g_i \to g_\infty + g_{\IR^{n-e-1}}$ in $C^2$ and such that the $\Phi_i$ converge to the identity map with respect to some fixed Gromov-Hausdorff convergence $(X_i, d_i, x_i) \to (Z_0 \times \IR^{n-e-1},  (z_0, 0^{n-e-1}))$.
\end{Lemma}

\begin{proof}
Let $(Z'_0, d_{Z'_0})$ be the link of the cone $Z'_0$.
So we can identify $Z_0 \setminus \{ z_0 \}$ with $Z'_0 \times (0, \infty)$.
For any $z = (z', s) \in Z_0 \setminus \{ z_0 \}$ there is a $0 < r < 1$ such that for some rescaling $(Z,d_Z)$ of $(Z'_0, d_{Z'_0})$ we have
\begin{multline*}
 d_{GH} \big( \big( B^{Z_0 \times \IR^{n-e-1}} ((z, 0^{n-e-1}), r), (z, 0^{n-e-1} ) \big), \\
   \big( B^{Z \times \IR^{n-e}} ( (z, 0^{n-e}), r), (z, 0^{n-e}) \big) \big) < \eps r/2. 
\end{multline*}
Let $p_i \in X_i$ be a sequence of points such that $p_i \to z$ with respect to some Gromov-Hausdorff convergence
\[ (X_i, d_i, x_i) \to (Z_0 \times \IR^{n-e-1}, d_{Z_0 \times \IR^{n-e-1}}, (z_0, 0^{n-e-1})). \]
Then, for sufficiently large $i$, we have
\[ d_{GH} \big( \big( B^{X_i} (p_i, r), p_i \big), \big( B^{Z \times \IR^{n-e}} ( (z, 0^{n-e}), r), (z, 0^{n-e}) \big) \big) < \eps r. \]
So $\rrm (p_i) > \eps r$ for large $i$.
It follows that $(z,0^{n-e})$ is a regular point in $Z_0 \times \IR^{n-e}$, hence $z$ is a regular point in $Z_0$.
The construction of the maps $\Phi_i$ follows by a center of gravity construction.
\end{proof}

Using Lemma \ref{Lem:coneissmooth}, we can now exclude the existence of $3$-dimensional cones and improve Proposition \ref{Prop:LpboundXXversion1}.

\begin{Proposition} \label{Prop:LpboundXXversion2}
For any $\eps > 0$, $\mathbf{p}_0 > 3$ and $Y < \infty$ there is an $E = E(\eps, \mathbf{p}_0, Y) < \infty$ such that the following holds:

Let $\XX$ be a singular space that has mild singularities of codimension $\mathbf{p}_0$.
Assume that $\XX = (X, d, \RR, g)$ is $Y$-tame and $Y$-regular at some scale $r > 0$ and that $\Ric = \lambda g$ on $\RR$ for some $|\lambda | \leq (n-1) r^{-2}$.
Assume that $\RR$ is orientable and assume that the following property holds:

If $p' \in X$, $0 < r' < r$ and if there exists a pointed metric space $(Z, d_Z, z)$ such that
\begin{equation} \label{eq:productcondition}
 d_{GH} \big( \big( B(p',r'), p' \big), \big( B^{Z \times \IR^{n-2}} ((z, 0^{n-2}),r'), (z, 0^{n-2} ) \big) \big) < \eps r', 
\end{equation}
then $\rrm (p) > \eps r'$.

Then for any $0 < s < 1$ and $p \in X$ we have
\[ |\{ 0 < \rrm < sr \} \cap B(p,r) | \leq E s^{4 - \eps} r^n. \]
\end{Proposition}

\begin{proof}
Without loss of generality, we may assume that $r = 1$.
In view of Proposition \ref{Prop:LpboundXXversion1}, we only need to verify that condition (\ref{eq:conecondition}) holds for $e = 3$ and for $\eps$ replaced by some $\eps' = \eps '( \eps, Y) > 0$.
So assume that this condition was false for some fixed $\eps, \mathbf{p}_0, Y$.
Then we can find a sequence $\XX_i$ of orientable singular spaces with mild singularities of codimension $\mathbf{p}_0$ that are $Y$-tame and $Y$-regular at scale $1$ and that satisfy $\Ric = \lambda_i g_i$ on their regular parts for some $|\lambda_i | \leq n-1$.
Moreover, the $\XX_i$ satisfy (\ref{eq:productcondition}) and there are points $p_i \in X_i$, scales $r_i > 0$ and metric cones $(Z_i, d_{Z_i})$ with vertex $z_i \in Z_i$ such that
\[ r_i^{-1} d_{GH} \big( \big( B^{X_i}(p_i, r_i) , p_i \big), \big( B^{Z_i \times \IR^{n-3}} (( z_i, 0^{n-3}, r_i)), (z_i, 0^{n-3}) \big) \big) \to 0, \]
but $r_i^{-1} \rrm (x_i) \to 0$.
After rescaling carefully, we may assume without loss of generality that $r_i \to \infty$, $\lambda_i \to 0$, $\rrm (x_i) \to 0$ and
\begin{equation} \label{eq:GHXYi}
  d_{GH} \big( \big( B^{X_i}(p_i, r_i) , p_i \big), \big( B^{Z_i \times \IR^{n-e-1}} (( z_i, 0^{n-3}, r_i)), (z_i, 0^{n-3}) \big) \big) \to 0. 
\end{equation}
After passing to a subsequence, we may further assume that $(X_i, d_{X_i}, x_i)$ Gromov-Hausdorff converge to some metric space $(Z_\infty, d_{Z_\infty}, z_\infty)$.
Due to (\ref{eq:GHXYi}), the pointed metric spaces $(Z_i \times \IR^{n-3}, d_{Z_i \times \IR^{n-3}}, (z_i, 0^{n-3}))$ also Gromov-Hausdorff converge to $(Z_\infty, d_{Z_\infty}, z_\infty)$.
Hence $(Z_\infty, d_{Z_\infty}, z_{\infty})$ is isometric to $(Z_0 \times \IR^{n-3}, \linebreak[1] d_{Z_0 \times \IR^{n-3}}, \linebreak[1] (z_0, 0^{n-3}))$, where $(Z_0, d_{Z_0}, z_0)$ is a metric cone with vertex $z_0$.

By Lemma \ref{Lem:coneissmooth}, we find that $Z_0 \setminus \{ z_0 \}$ is a $3$-dimensional differentiable manifold and $d_{Z_0}$ is locally isometric to the length metric of a $C^2$ Riemannian metric $g_0$ on $Z_0 \setminus \{ z_0 \}$.
This metric must be Ricci flat and hence, since $Z_0$ is $3$-dimensional, locally flat.
It follows that $Z_0$ is isometric to a union of copies of $\IR^3$ and/or $\IR^3 /  \IZ_2$ along their tips.
By the last part of Lemma \ref{Lem:coneissmooth} and our assumption that $\RR_i$ are orientable, we can exclude the cones $\IR^3 / \IZ_2$.

So $Z_0$ is isometric to the union of copies of $\IR^3$ along their tips and hence
\[ |B^{Z_0 \times \IR^{n-3}} ((y_0, 0^{n-3}), 1) \cap (( Z_0 \setminus \{z_0\} ) \times \IR^{n-3} ) | \geq \omega_n. \]
So, again by the last part of  Lemma \ref{Lem:coneissmooth}, we find that
\[ \liminf_{i \to \infty} |B^{X_i} (p_i, 1) \cap \RR_i | \geq \omega_n \]
So, by the $Y$-regularity of the $\XX_i$, we must have $\rrm(p_i) > Y^{-1}$ for large $i$, which contradicts our assumptions.
\end{proof}

\section{Codimension 4 and proof of Theorem \ref{Thm:Lpbound}} \label{subsec:Codim4}
In this section, we verify the condition (\ref{eq:productcondition}) in Proposition \ref{Prop:LpboundXXversion2} and prove Theorem \ref{Thm:Lpbound}.
We will largely follow the work of Cheeger and Naber (cf \cite{Cheeger-Naber-Codim4}) with a few simplifications.

We first prove the following lemma:

\begin{Lemma} \label{Lem:L1Laplacesplitting}
For every $\eps > 0$ and $Y < \infty$ there is a $\delta = \delta (\eps,  Y) > 0$ such that the following holds:

Let $\XX = (X, d, \RR, g)$ be a singular space with mild singularities of codimension $\mathbf{p}_0 > 2$.
Assume that $\XX$ is $Y$-tame at scale $\delta^{-1} r$ for some $r > 0$ and that $ |{\Ric} | \leq  (n-1) \delta r^{-2}$ on $\RR$.

Let $p \in X$ and let $u : B(p,2r) \to \IR^k$ be a $\delta$-splitting map for some $k \leq n$.
Define for all $1 \leq l \leq k$ the following $k$-form on $\RR$:
\[ \omega^l := d u^1 \wedge \ldots \wedge du^l. \]
Then
\[ r^2 \int_{B(p,r) \cap \RR} d |\mu_{\triangle |\omega^l |} | < \eps r^{n}. \]
\end{Lemma}

\begin{proof}
Without loss of generality, we may assume that $r = 1$.
We will furthermore fix $l$ and set $\omega := \omega^l$.

First observe that
\[ \nabla \omega = \nabla d u^1 \wedge \ldots \wedge d u^l + \ldots + d u^1 \wedge \ldots \wedge \nabla d u^l. \]
So, for some generic constant $C < \infty$,
\[ |\nabla \omega | \leq C \big( |\nabla^2 u^1 | + \ldots + |\nabla^2 u^l | \big). \]
It follows that
\begin{equation} \label{eq:nabomegaL2}
 \int_{B(p,2) \cap \RR} |\nabla \omega|^2 dg \leq C \delta^2. 
\end{equation}
Next, we find that for some local orthonormal frame $\{ e_j \}_{i = 1, \ldots, n}$,
\begin{multline*}
 \triangle \omega = \sum_{j_1, j_2 = 1}^l \sum_{i=1}^n d u^1 \wedge \ldots \wedge \nabla_{e_i} du^{j_1} \wedge \ldots \wedge \nabla_{e_i} du^{j_2} \wedge \ldots \wedge du^l \\
  + \Ric ( du^1 ) \wedge \ldots \wedge du^l + \ldots + du^1 \wedge \ldots \wedge \Ric (d u^l). 
\end{multline*}
So
\[ | \triangle \omega | \leq C \big( |\nabla^2 u^1 |^2 + \ldots + |\nabla^2 u^l |^2 \big) + C \delta^2. \]
Therefore,
\begin{equation} \label{eq:nab2omegaL1}
 \int_{B(p,2) \cap \RR} |\triangle \omega | \leq C \delta^2. 
\end{equation}

Let now $\phi \in C^2 ( B(p,2) \cap \RR)$ be a non-negative function such that $|\nabla \phi | < 10$, $\phi \equiv 1$ on $B(p,1) \cap \RR$ and such that $\phi \equiv 0$ outside $B(p,1.5) \cap \RR$.
Fix moreover a sequence $s_i \to 0$.
Similarly as in the proof of Lemma \ref{Lem:trianglebpmsmall}, we can construct functions $\eta_i \in C^2_c ( B(p,2 ) \cap \RR )$ such that $\supp \eta_i \subset \{ \rrm > s_i \} \cap B(p,2) \subset U_i$, $\eta_i \equiv 1$ on $\{ \rrm > 2 s_i \}$ and $|\nabla \eta_i | < C s_i^{-1}$.
So $\eta_i \phi \in C^2_c (\RR)$ and $\eta_i \phi \to 1$ pointwise on $B(p, 1) \cap \RR$ as $i \to \infty$.
Then, by Proposition \ref{Prop:weakLaplaciandistance} and (\ref{eq:nabomegaL2}),
\begin{align*}
 \int_{B(p,2) \cap \RR} \eta_i \phi d \mu_{\triangle | \omega|} &= - \int_{\RR} \nabla (\eta_i \phi ) \nabla |\omega | \\
 &\leq \int_{\RR} \eta_i | \nabla \phi | | \nabla \omega |  + \int_{\{ s_i < \rrm < 8s_i \}} \phi  | \nabla \eta_i |  | \nabla \omega | \\
 &\leq C \bigg( \int_{B(p,2) \cap \RR} |\nabla \omega |^2 \bigg)^{1/2} \\
 &\; + C s_i^{-1} \big| \{ \rrm < 2 s_i \} \cap B(p, 2) \cap \RR \big|^{1/2} \bigg( \int_{B(p,2) \cap \RR} |\nabla \omega |^2 \bigg)^{1/2} \\
 &\leq C \delta + C s_i^{-1} \big| \{ \rrm < 2 s_i \} \cap B(p, 2) \cap \RR \big|^{1/2} \delta
\end{align*}
Since $\XX$ has singularities of codimension $\mathbf{p}_0 > 2$, we find
\[ \limsup_{i \to \infty} \int_{B(p,2) \cap \RR} \eta_i \phi d \mu_{\triangle | \omega|} \leq C \delta. \]

Next, we claim that
\begin{equation} \label{eq:dmutriangleomegaKato}
 d\mu_{\triangle |\omega |} \geq - |\triangle \omega | dg 
\end{equation}
To see this, let $\psi \in C^2_c (\RR)$ be some non-negative and compactly supported function and note that by Proposition \ref{Prop:weakLaplaciandistance}
\begin{multline} \label{eq:psidmuomega}
 \int_{\RR} \psi d\mu_{\triangle |\omega |} = \int_{\RR} (\triangle \psi) |\omega | dg 
 = \lim_{\alpha \to 0} \int_{\RR} (\triangle \psi ) \sqrt{ |\omega |^2 + \alpha } \, dg \\ = \lim_{\alpha \to 0} \int_{\RR} \psi \triangle \sqrt{ |\omega |^2 + \alpha } \, dg. 
\end{multline}
For any $\alpha > 0$ we obtain as in (\ref{eq:KatoHessian}) in the proof of Proposition \ref{Prop:weakLaplaciandistance}
\[ \triangle \sqrt{ | \omega |^2 + \alpha}   \geq - |\triangle \omega |. \]
Combining this with (\ref{eq:psidmuomega}) yields (\ref{eq:dmutriangleomegaKato}).

By (\ref{eq:nab2omegaL1}) we have
\[ \limsup_{i \to \infty} \int_{B(p,2) \cap \RR} \eta_i \phi \big( 2|\triangle \omega | dg + d\mu_{\triangle |\omega|} \big) \leq C \delta + C \delta^2. \]
By (\ref{eq:dmutriangleomegaKato}), the integrand of this integral is non-negative.
\begin{multline*} 
 \int_{B(p,1) \cap \RR} d |\mu_{\triangle |\omega|}| \le \int_{B(p,1) \cap \RR} \big( 2 d (\mu_{\triangle |\omega|} )_- +  d \mu_{\triangle |\omega|}  \big) \\
 \leq \int_{B(p,1) \cap \RR} \big( 2|\triangle \omega | dg + d\mu_{\triangle |\omega|} \big) \leq C \delta + C \delta^2. 
\end{multline*}
This proves the desired result.
\end{proof}

We now define the singular scale $s^\delta_x$ as in \cite{Cheeger-Naber-Codim4}.

\begin{Definition}[cf {\cite[Definition 1.30]{Cheeger-Naber-Codim4}}] \label{Def:singscale}
Let $\XX$ be a singular space, $p \in X$ and $r > 0$.
Consider a continuous, vector-valued function $u = (u^1, \ldots, u^k) : B(p,2) \to \IR^k$ such that $u |_{B(p,2) \cap \RR}$ is $C^3$ and define for all $1 \leq l \leq k$ the following $k$-form on $\RR$:
\[ \omega^l := d u^1 \wedge \ldots \wedge du^l. \]
For any $x \in B(p,1)$ and $\delta > 0$ we define the \emph{singular scale $s^\delta_x \geq 0$} as the infimum of all radii $0 < s \leq 1/2$ such that for all $r$ with $s \leq r < 1/2$ and $1 \leq l \leq k$ we have
\begin{equation} \label{eq:singscaleequation}
 r^2 \int_{B(x,r) \cap \RR} d| \mu_{\triangle |\omega^l|} | \leq \delta \int_{B(x,r) \cap \RR} |\omega^l|. 
\end{equation}
\end{Definition}

With this definition in hand, we can prove the Transformation Theorem.

\begin{Proposition}[Transformation Theorem, cf {\cite[Theorem 1.32]{Cheeger-Naber-Codim4}}] \label{Prop:TransformationThm}
For every $\eps > 0$, $\mathbf{p}_0 > 3$ and $Y < \infty$ there is a $\delta = \delta (\eps, \mathbf{p}_0, Y) > 0$ such that the following holds:

Let $\XX = (X, d, \RR, g)$ be a singular space with mild singularities of codimension $\mathbf{p}_0$.
Assume that $\XX$ is $Y$-tame and $Y$-regular at scale $\delta^{-1}$.
Assume moreover that $\Ric = \lambda g$ on $\RR$ for some $|\lambda | \leq (n-1) \delta^2 $.

Let $p \in X$ and let $u : B(p,2) \to \IR^k$ be a $\delta$-splitting map for some $k \leq n$.
Then for any $x \in B(p,1)$ and $s^\delta_x \leq r < 1/2$ there is a lower triangular matrix $A = A(x,r) \in \IR^{k \times k}$ with positive diagonal entries such that $Au|_{B(x,r)}$ is an $\eps$-splitting.
\end{Proposition}

\begin{proof}
We argue as in \cite[subsec 3.2]{Cheeger-Naber-Codim4}.
First observe that we may assume without loss of generality that $\eps < \eps_*$ for some uniform constant $\eps_* > 0$, which we will determine in the course of the proof.
Assume by induction on $k$ that the proposition holds for $k$ replaced by $1, \ldots, k-1$ (If $k = 1$, then we make no assumption).
We will show in the following by contradiction that then the proposition also holds for $k$.
So fix $\eps > 0$, $\mathbf{p}_0 > 3$, $Y < \infty$ and assume that no $\delta$ exists for which the Proposition holds for $k$.
Then we can find a sequence $\delta_j \to 0$ and a sequence $\XX_j$ of singular spaces with mild singularities of codimension $\mathbf{p}_0$ that are $Y$-tame and $Y$-regular at scale $\delta_j^{-1}$ and that satisfy $\Ric = \lambda_j g_j$ on $\RR_j$ for some $|\lambda_j | \leq (n-1) \delta_j^2$.
Moreover, we can find points $p_j \in X_j$, $\delta_j$-splitting maps $u_j : B(p_j, 2) \to \IR^k$ for $\delta_j \to 0$, points $x_j \in B(p_j, 1)$ and $r_j \in [s^{\delta_j}_{x_j}, 1/2]$ such that there is now lower triangular matrix $A$ with positive diagonal entries such that $A u_j |_{B(x_j,r_j)}$ is an $\eps$-splitting.
By Lemma \ref{Lem:L1Laplacesplitting} we have $s^{\delta_j}_{x_j}, r_j \to 0$, since otherwise $u_j |_{B(x_j, r_j)}$ would be a $\delta'$-splitting for infinitely many $j$ and an arbitrary small $\delta' > 0$ and hence (\ref{eq:singscaleequation}) would hold for infinitely many $j$.
Without loss of generality, we may furthermore assume that each $r_j$ is chosen maximal in the following sense: for any $j$ and any $r \in [2 r_j, 1/2]$ the proposition holds, that is, there is a lower triangular matrix $A \in \IR^{k \times k}$ with positive diagonal entries such that $A u_j |_{B(x_j, r)}$ is an $\eps$-splitting.
Let us now use this assumption for $r \leftarrow 4 r_j$ and choose such matrices $A_j \in \IR^{k \times k}$ such that $A_j u_j |_{B(x_j, 4 r_j)}$ are $\eps$-splittings.

Next, we rescale each singular Ricci flat space $\XX_j$ by $r_j^{-1}$ and denote the result by $\XX'_j := r_j^{-1} \XX_j$.
From now on we will almost exclusively consider the spaces $\XX'_j$ and any geometric object or quantity will be understood with respect to this sequence of rescaled spaces.
Moreover, we consider the rescaled functions
\[ v_j := r_j^{-1} A_j (u_j - u_j(x_j) ), \]
which are harmonic on the regular parts $\RR'_j$ of $\XX'_j$.
So $v_j |_{B(x_j, 1)}$ is not an $\eps$-splitting, but $v_j |_{B(x_j, 4)}$ is.
Moreover, for any $r \in [4, r_j^{-1}/2]$ there is a lower triangular matrix $A_{r,j} \in \IR^{k \times k}$ with positive diagonal entries such that $A_{r,j} v_j |_{B(x_j, r)}$ is an $\eps$-splitting.
We may also assume that $A_{4, j} = I_k =: I$ is the identity matrix.

For the rest of the proof let $C < \infty$ be some generic constant that only depends on $Y, \mathbf{p}_0$ and $n$ and we set for all $1 \leq l \leq k$
\[ \omega^l_j := d v_j^1 \wedge \ldots \wedge d v_j^l \in \Omega_l (\RR'_j). \]

\begin{Claim1}
For each $j$ and $r \in [4, r_j^{-1}/2]$ we have
\[ | A_{4r, j} A_{r,j}^{-1} - I | < C \eps. \]
\end{Claim1}

\begin{proof}
The proof is the same as in \cite[sec 3.2, proof of Claim 1]{Cheeger-Naber-Codim4}.
Note that we have expressed the identity above in a slightly different but equivalent way.
The fact that the $\XX'_j$ are singular does not create any issues since no arguments are used in the proof that require the regular parts $\RR'_j$ to be complete.
\end{proof}

As in \cite{Cheeger-Naber-Codim4}, we conclude from Claim 1 that for all $j$ and $r \in [4, r_j^{-1}/2]$
\begin{subequations}
\begin{align}
|A_{r,j}|, |A^{-1}_{r,j} | &\leq r^{C \eps} \label{eq:boundsforvj1} \displaybreak[1] \\
\sup_{B(x_j, r) \cap \RR'_j} |\nabla v^l_j | &\leq (1+ C\eps) r^{C \eps} \label{eq:boundsforvj2} \displaybreak[1] \\
\sup_{B(x_j, r) \cap \RR'_j} |\omega^l_j| &\leq (1+ C\eps) r^{C \eps} \label{eq:boundsforvj3} \displaybreak[1] \\
r^{2-n} \int_{B(x_j, r) \cap \RR'_j} |\nabla^2 v^l_j|^2 &\leq C \eps r^{C \eps}. \label{eq:boundsforvj4}
\end{align}
\end{subequations}

\begin{Claim2}
For every $r \geq 4$ there are lower triangular matrices $A'_{r,j} \in \IR^{k \times k}$ with positive diagonal entries such that $|A'_{r,j} - I |< C \eps$ for large $j$, such that the maps $A'_{r,j} v_j |_{B(x_j, r)} : B(x_j, r) \to \IR^k$ are $C \eps$-splittings and such that the restricted maps
\[ \proj_{\IR^{k-1} \times 0} (A'_j v_j) |_{B(x_j,r)} : B(x_j, r) \longrightarrow \IR^{k-1} \]
are $\eps_j (r)$-splittings, where $\eps_j (r) \to 0$ for $j \to \infty$ and fixed $r$.
Moreover, for each $r \geq 4$ and $l = 1, \ldots, k-1$,
\begin{equation} \label{eq:nab2vlto0}
 \int_{B(x_j, r) \cap \RR'_j} |\nabla^2 v^l_j |^2 \to 0.
\end{equation}
\end{Claim2}

\begin{proof}
This fact, which is almost the same as \cite[sec 3.2, proof of Claim 2]{Cheeger-Naber-Codim4}, follows similarly as in \cite{Cheeger-Naber-Codim4}.
Fix $r \geq 4$.
By the inductive assumption, we can find lower triangular $(k-1) \times (k-1)$ matrices $\td{A}_{j,r} \in \IR^{(k-1) \times (k-1)}$ such that $\td{A}_{j,r} \proj_{\IR^{k-1} \times 0} (v_j) |_{B(x_j, r)}$ are $\td\eps_j(r)$-splittings for $\td\eps_j (r) \to 0$.
In particular, $\td{A}_{j,r} \proj_{\IR^{k-1} \times 0} (v_j) |_{B(x_j, 4)}$ are $\td\eps'_j(r)$-splittings for $\td\eps'_j (r) \to 0$.
Since the maps $\proj_{\IR^{k-1} \times 0} (v_j) |_{B(x_j, 4)}$ are $\eps$-splittings, we conclude, using similar arguments as in the proof of Claim 1, that $| \td{A}_{j,r} - I_{k-1} | < C \eps$ for large $j$, where $I_{k-1}$ denotes the identity $(k-1) \times (k-1)$-matrix.
So if we set $A'_{j,r} = \td{A}_{j,r} \oplus \id_{0^{k-1} \oplus \IR}$, then the first part of the claim follows.

Lastly, identity (\ref{eq:nab2vlto0}) holds since for each $l = 1, \ldots, k-1$
\[ \int_{B(x_j, r) \cap \RR'_j} \big| \nabla^2 \big( \td{A}_{j,r} v_j \big)^l \big|^2 \to 0 \]
and the inverse matrices $|\td{A}_{j,r}^{-1} |$ are uniformly bounded for small enough $\eps_*$.
\end{proof}

We now take a slightly different route from \cite{Cheeger-Naber-Codim4} in order to avoid having to use the theory of singular spaces with lower Ricci curvature bounds.
First, we note that due to Claim 2, we may replace $v_j$ by $A'_{4, j} v_j$.
Then we still have (\ref{eq:boundsforvj1})--(\ref{eq:boundsforvj4}) and (\ref{eq:nab2vlto0}).
Moreover, we get that $(v^1_j, \ldots, v^{k-1}_j) |_{B(x_j,4)}$ is an $\eps_j$-splitting for some sequence $\eps_j \to 0$.
Finally, we mention that due to Definition \ref{Def:singscale} and (\ref{eq:boundsforvj3}) we have
\begin{equation} \label{eq:triangleomegato0}
\int_{B(x_j, r) \cap \RR'_j} d |\mu_{\triangle |\omega^l|}| \to 0 \textQQqq{for all} r \in [4, r_j^{-1}/2], \qquad l = 1, \ldots, k.
\end{equation}

We will now prove our main claim.

\begin{Claim3}
For each $l = 1, \ldots, k-1$ we have
\begin{equation} \label{eq:navlnabvkto0}
 \int_{B(x_j, 4) \cap \RR'_j} \bigg| \langle \nabla v^l_j, \nabla v^k_j \rangle - \fint_{B(x_j, 4) \cap \RR'_j} \langle \nabla v^l_j, \nabla v^k_j \rangle \bigg| \to 0 
\end{equation}
and
\begin{equation} \label{eq:omegato0}
 \int_{B(x_j, 4) \cap \RR'_j} \bigg| |\omega^l_j| - \fint_{B(x_j, 4) \cap \RR'_j} |\omega^l_j| \bigg| \to 0. 
\end{equation}
as $j \to \infty$.
\end{Claim3}

\begin{proof}
The identities (\ref{eq:navlnabvkto0}) and (\ref{eq:omegato0}) follow using the same technique.
In order to present the proof only once, we set
\begin{subequations}
\begin{equation} \label{eq:fjcase1}
 f_j := \langle \nabla v^l_j, \nabla v^k_j \rangle,
\end{equation}
or
\begin{equation} \label{eq:fjcase2}
 f_j := |\omega^l_j|,
\end{equation}
\end{subequations}
depending on whether we want to establish (\ref{eq:navlnabvkto0}) or (\ref{eq:omegato0}).
Our goal will then be to show that in both cases we have
\begin{equation} \label{eq:Claim3fj}
 \int_{B(x_j, 4) \cap \RR'_j} \bigg| f_j - \fint_{B(x_j, 4) \cap \RR'_j} f_j \bigg| \to 0. 
\end{equation}

Let us first collect the identities for $f_j$ from which we will deduce (\ref{eq:Claim3fj}).
We claim that there is a uniform constant $C < \infty$ such that for any $S > 1$ we have for sufficiently large $j$ (depending on $S$)
\begin{subequations}
\begin{alignat}{2} 
 | f_j | &< C\big( 1+d(x_j, \cdot) \big)^{C\eps} && \textQQqq{on} B(x_j, S) \cap \RR'_j, \label{eq:fjid1} \displaybreak[1] \\
 |\nabla f_j| &< C \big( 1+\rrm^{-1}(\cdot) \big) \big( 1+d(x_j, \cdot) \big)^{C\eps} &&\textQQqq{on} B(x_j, S) \cap \RR'_j, \label{eq:fjid2}
\end{alignat}
\begin{equation} \label{eq:fjid3}
S^{-n+1} \int_{B(x_j, S) \cap \RR'_j} |\nabla f_j| \leq C S^{C\eps}.
\end{equation}
We also claim that for any $S > 1$
\begin{equation} \label{eq:fjid4}
\int_{B(x_j, S) \cap \RR'_j} d | \mu_{\triangle f_j} | \to 0 
\end{equation}
\end{subequations}
as $j \to 0$.
In fact, identity (\ref{eq:fjid1}) is a direct consequence of (\ref{eq:boundsforvj2}) and (\ref{eq:boundsforvj3}).
Identity (\ref{eq:fjid2}) follows similarly using (\ref{eq:fjid1}), the fact that $\triangle \nabla v_j^l = 0$ on $\RR'_j$ and local elliptic regularity.
In order to show (\ref{eq:fjid3}), note that by (\ref{eq:boundsforvj2}) we have in both cases, (\ref{eq:fjcase1}) and (\ref{eq:fjcase2}), that
\[ |\nabla f_j| \leq C S^{C\eps} \sum_{i=1}^k |\nabla^2 v_j^i | \textQQqq{on} B(x_j, S) \cap \RR'_j. \]
So, using (\ref{eq:boundsforvj4}), we have
\[ S^2 \fint_{B(x_j, S) \cap \RR'_j} |\nabla f_j|^2 \leq C S^{2C\eps}. \]
So (\ref{eq:fjid3}) follows using H\"older's inequality.
Finally, we show (\ref{eq:fjid4}).
In the case (\ref{eq:fjcase2}), this identity is the same as (\ref{eq:triangleomegato0}).
In the case (\ref{eq:fjcase1}), we first use Bochner's identity and (\ref{eq:boundsforvj2}) to derive
\[ | \triangle f_j | = | 2\langle \nabla^2 v_j^l, \nabla^2 v_j^k \rangle + 2 \Ric ( \nabla v^l_j, \nabla v^k_j ) | \leq 2 |\nabla^2 v^l_j | \cdot |\nabla^2 v^k_j | + 2 \delta_j^2 S^{2C\eps}. \]
So, by H\"older's inequality,
\[ \int_{B(x_j, S) \cap \RR'_j} |\triangle f_j|  \leq 2 \bigg( \int_{B(x_j, S) \cap \RR'_j} | \nabla^2 v_j^l|^2 \bigg)^{1/2} \bigg( \int_{B(x_j, S) \cap \RR'_j} | \nabla^2 v_j^k|^2 \bigg)^{1/2} + C \delta_j^2 S^{2C\eps} .  \]
By (\ref{eq:boundsforvj4}), the second factor of the first term stays bounded as $j \to 0$.
The first factor of the first term goes to $0$, by Claim 2.
The second term goes to $0$ by assumption, proving (\ref{eq:fjid4}).

From now on we will only work with the identities (\ref{eq:fjid1})--(\ref{eq:fjid4}).
Equations (\ref{eq:fjcase1}), (\ref{eq:fjcase2}) will not be used anymore.

Let $S > 8$ be a large constant whose value we will determine later.
Fix some $j$ for the moment, and choose $y \in B(x_j, 4) \cap \RR'_j$.
We now apply Proposition \ref{Prop:Uiphii} to $B(x_j, S)$ and obtain numbers $s_i \to 0$, subsets $U_i \subset \RR'_j$ and cutoff functions $\phi_i = \phi_{S,i} : U_i \to [0,1]$.
Next, we choose a smoothing $h \in C^2 (\RR'_j)$ of $\rrm$ such that
\[ \tfrac12 \rrm < h < 2 \rrm \textQQqq{and} |\nabla h | < 2, \]
fix a smooth function $H : [0, \infty) \to [0,1]$ such that $H \equiv 0$ on $[0,2]$ and $H \equiv 1$ on $[4, \infty)$ and set
\[ \eta_i (x) := H ( s_i^{-1} h(x)). \]
Then $\supp \eta_i \phi_i \subset \{ \rrm > s_i \} \cap B(x_j, S)$ and $\eta_i \equiv 1$ on $\{ \rrm > 8 s_i \}$ and $|\nabla \eta_i | < C s_i^{-1}$.
It follows that $\eta_i \phi_i$ is a $C^2$ function with compact support in $\RR'_j$ and that $\eta_i \phi_i \to 1$ pointwise on $B(x_j, \tfrac12 S)$ as $i \to \infty$.

In the following, we denote by $C_* < \infty$ a generic constant that may depend on $\eps, Y, \XX'_j$, $S$, an upper bound on $f_j$ and $\rrm(y)$, but not on $i$.
Let $K$ be the heat kernel from item (5) of the $Y$-tameness assumptions (cf. Definition \ref{Def:tameness}).
Note that due to the Gaussian bounds from these tameness properties and Proposition \ref{Prop:ChengYau-heatequation}, we can assume that $K (y, \cdot, t) < C_*$ and $|\nabla K (y, \cdot, t) | < C_*$ on $\{ 0 < \rrm < 8 s_i \} \cap B(x_j,  S)$ for all $t \in (0, S^2]$ and for large enough $i$.
We can then compute, using (\ref{eq:fjid1})--(\ref{eq:fjid3}), that for any $t \in (0, S^2]$ and large $i$
\begin{align*}
 \bigg| \frac{d}{dt} \int_{\RR'_j} & K(y, \cdot, t) f_j \eta_i \phi_i \bigg| = \bigg| \int_{\RR'_j} \triangle K(y, \cdot, t) f_j \eta_i \phi_i \bigg|  \displaybreak[1] \\
 &\leq \bigg| \int_{\RR'_j} \triangle \big( K(y, \cdot, t) \eta_i \phi_i \big) f_j \bigg| 
 + \bigg| \int_{\RR'_j} \Big( 2 \nabla K(y, \cdot, t) \nabla \eta_i \phi_i f_j \\
 &\qquad\qquad + 2 \nabla K(y, \cdot, t) \eta_i \nabla \phi_i f_j + 2 K(y, \cdot, t) \nabla \eta_i \nabla \phi_i f_j  \\
 &\qquad\qquad + K(y, \cdot, t) (\triangle \eta_i) \phi_i f_j + K(y, \cdot, t) \eta_i (\triangle \phi_i ) f_j \Big) \bigg|  \displaybreak[1] \\
 &\leq \bigg| \int_{\RR'_j}  K(y, \cdot, t) \eta_i \phi_i \, d\mu_{\triangle f_j} \bigg| 
 + \bigg| \int_{\RR'_j} \Big( 2 \nabla K(y, \cdot, t) \nabla \eta_i \phi_i f_j \\
 &\qquad\qquad - 2  K(y, \cdot, t) \nabla \eta_i \nabla \phi_i f_j - 2  K(y, \cdot, t) \eta_i \triangle \phi_i f_j \\
 &\qquad\qquad - 2  K(y, \cdot, t) \eta_i \nabla \phi_i \nabla f_j + 2 K(y, \cdot, t) \nabla \eta_i \nabla \phi_i f_j  \\
 &\qquad\qquad - \nabla K(y, \cdot, t) \nabla \eta_i \phi_i f_j -  K(y, \cdot, t) \nabla \eta_i \nabla \phi_i f_j \\
 &\qquad\qquad - K(y, \cdot, t) \nabla \eta_i \phi_i \nabla f_j + K(y, \cdot, t) \eta_i (\triangle \phi_i ) f_j \Big) \bigg| \displaybreak[1] \\ 
 &\leq  \bigg| \int_{\RR'_j} \Big(  \nabla K(y, \cdot, t) \nabla \eta_i \phi_i f_j   - K(y, \cdot, t) \nabla \eta_i \phi_i \nabla f_j  \\
 &\qquad\qquad -  K(y, \cdot, t) \nabla \eta_i \nabla \phi_i f_j \Big) \bigg|  \\
 &\qquad + \bigg| \int_{\RR'_j}  K(y, \cdot, t) \eta_i \phi_i \, d\mu_{\triangle f_j} \bigg| \\
 &\qquad + \bigg| \int_{\RR'_j}   K(y, \cdot, t) \eta_i \Big(2 \nabla \phi_i \nabla f_j + \triangle \phi_i f_j \Big) \bigg| \displaybreak[1] \\
 &\leq C_*  |\{ \rrm < 8 s_i \} \cap \RR'_j| \cdot \big( s_i^{-2} + s_i^{-1} S^{-1} \big) \\
 &\qquad   + C t^{-n/2} \int_{B(x_j, S) \cap \RR'_j} d |\mu_{\triangle f_j} | \\
 &\qquad + C t^{-n/2} \exp \bigg({ - \frac{S^2}{16Yt} }\bigg)  \bigg( S^{-1} \int_{A(x_j, S/2, S) \cap \RR'_j} |\nabla f_j| \\
 &\qquad\qquad\qquad\qquad\qquad\qquad\qquad\qquad +  S^{-2}\int_{A(x_j, S/2, S) \cap \RR'_j} |f_j| \bigg) \displaybreak[1] \\
 &\leq C_* s_i + C t^{-n/2} \int_{B(x_j, S) \cap \RR'_j} d |\mu_{\triangle f_j} | \\
 &\qquad + C t^{-n/2} \exp \bigg({- \frac{S^2}{16Yt} } \bigg) \big( S^{-1} S^{n + C \eps -1} + S^{-2} S^{n+C\eps} \big) \displaybreak[1] \\
 &\leq C_* s_i + C t^{-n/2} \int_{B(x_j, S) \cap \RR'_j} d |\mu_{\triangle f_j} |  + C t^{-n/2} S^{n+ C \eps -2} \exp \bigg( {- \frac{S^2}{16Yt} } \bigg) .
\end{align*}
Using (\ref{eq:fjid4}), we hence obtain the following statement for sufficiently small $\eps_*$:
For every $\alpha, \tau > 0$ and $T < \infty$ there are $J = J(\alpha, \tau, T), S_0 = S_0(\alpha, \tau, T) < \infty$ such that if $j > J$, $S > S_0$, then for any $y \in B(x_j, 4) \cap \RR'_j$ and for sufficiently large $i$ (depending on $\alpha, T, \tau, \rrm(y)$ etc.)
\[ \bigg| \frac{d}{dt} \int_{\RR'_j}  K(y, \cdot, t) f_j \eta_i \phi_{S,i} \bigg| < \alpha \textQQqq{for all} t \in [\tau, T]. \]
Integrating this inequality over $t$ from $\tau$ to $T$ yields the following statement:
For every $\alpha, \tau > 0$ and $T < \infty$ there are $J = J(\alpha, \tau, T), S_0 = S_0(\alpha, \tau, T) < \infty$ such that if $j > J$, $S > S_0$, then for any $y \in B(x_j, 4) \cap \RR'_j$ and for sufficiently large $i$ (depending on $y$)
\[ \bigg| \int_{\RR'_j} K(y,\cdot, \tau) f_j \eta_i \phi_{S,i} - \int_{\RR'_j} K(y, \cdot, T) f_j \eta_i \phi_{S,i} \bigg| < \alpha. \]
Next, observe that due to (\ref{eq:fjid2}) we have for any $z \in B(y, \min \{ \rrm(y)/2, 1 \}) \cap \RR'_j$
\[ |f_j(z) - f_j(y)| \leq C \rrm^{-1} (y) d(z,y). \]
So for any $0 < \rho < \rrm(y)/2$ and large $i$ we have
\begin{align*}
 \bigg| f_j(y) &- \int_{B(x_j, S) \cap \RR'_j}  K(y,\cdot, \tau) f_j \eta_i \phi_{S,i} \bigg| \\
 &=  \bigg| f_j(y) \int_{\RR'_j} K(y, \cdot, \tau) - \int_{B(x_j, S) \cap \RR'_j}  K(y,\cdot, \tau) f_j \eta_i \phi_{S,i} \bigg| \displaybreak[1] \\
 &\leq  \int_{B(x_j, \rho)} | f_j (y) - f_j (z) | K(y, z, \tau) dg(z) \\
 &\qquad+ \int_{(B(x_j, S) \cap \RR'_j) \setminus B(x_j, \rho)} |f_j (z) |  K(y, z, \tau) dg(z) \\
 &\qquad + |f_j(y) | \int_{\RR'_j \setminus B(x_j, \rho)} K(y, \cdot, \tau) \displaybreak[1] \\
 &\leq C  \rrm^{-1}(y) \rho +\int_{(B(x_j, S) \cap \RR'_j) \setminus B(x_j, \rho)} C \big( 1 + d(x_j, z) \big)^{C \eps} \\
 &\qquad\qquad\qquad\qquad \cdot \frac{Y}{\tau^{n/2}} \exp \Big( {- \frac{d^2(x_j, z)}{Y\tau}} \Big) dg(z) + C  \tau^{-n/2} \exp \Big({ - \frac{\rho^2}{2 Y \tau} } \Big) \displaybreak[1] \\
 &\leq C  \rrm^{-1}(y) \rho + C \tau^{-n/2} \int_{\rho}^\infty C(1+r)^{C\eps} r^{n-1} \exp \Big({ - \frac{r^2}{Y\tau} } \Big)dr \\
 &\qquad+ C \tau^{-n/2} \exp \Big( {- \frac{\rho^2}{10 Y \tau}} \Big).
\end{align*}
So, given $y$, we may choose $\rho$ small enough (depending on $\rrm(y)$) to make the first term on the right-hand side arbitrarily small.
Next we can choose $\tau$ small enough to make the second and third term arbitrarily small.
We then obtain the following statement:
For any $\alpha, s > 0$ and $T < \infty$ there are $J = J(\alpha, s, T), S_0=S_0(\alpha, s, T) < \infty$ such that if $j > J$ and $S > S_0$, then we have for large $i$
\begin{equation} \label{eq:statement1}
 \bigg| f_j(y) - \int_{B(x_j, S) \cap \RR'_j} K(y, \cdot, T) f_j \eta_i \phi_{S,i} \bigg| < \alpha \textQq{for all} y \in B(x_j,4) \cap \{ \rrm > s \}. 
\end{equation}

Let now $y_1, y_2 \in B(x_j, 4)$ and let $S >2$.
Since $X'_j$ is the completion of the length metric on $\RR'_j$, we can find an arclength $C^1$-curve $\gamma : [0, a] \to B(x_j, 4) \cap \RR'_j$ of length $a=\ell(\gamma) < 10$ such that $\gamma(0) = y_1$, $\gamma(a) = y_2$.
Then for any $j$ and $T > 1$ and $u \in [0,a]$ we have, by the symmetry of $K$ from the tameness conditions (Definition \ref{Def:tameness}) and by Proposition \ref{Prop:ChengYau-heatequation},
\begin{align*}
\bigg| \frac{d}{du} & \int_{B(x_j, S) \cap \RR'_j}  K(\gamma(u), \cdot, T) f_j \eta_i \phi_{S,i} \bigg| \leq |\gamma' (u)| \int_{\RR'_j} |\nabla_1 K |(\gamma(u), z, T) |f_j (z)| dg(z) \displaybreak[1] \\
&\leq \frac{C}{T^{(n+1)/2}} \int_{\RR'_j} \exp \bigg({ - \frac{d^2(\gamma(u), z)}{2YT} }\bigg) (1+ d(x_j, z))^{C \eps} dg(z) \displaybreak[1] \\
&\leq \frac{C}{T^{(n+1)/2}} \int_{\RR'_j} \exp \bigg({ - \frac{d^2(x_j, z)}{4YT} + \frac{d(x_j, \gamma(u))^2}{2YT} }\bigg) (1+ d(x_j, z))^{C \eps} dg(z) \displaybreak[1] \\
&\leq \frac{C}{T^{(n+1)/2}} \int_0^\infty \exp \bigg({ - \frac{r^2}{4YT} }\bigg) (1+r)^{C \eps} r^{n-1} dr \displaybreak[1] \\
&= \frac{C}{T^{1/2}}  \int_0^\infty \exp \bigg( {- \frac{r^2}{4Y}} \bigg) (1+T^{1/2} r)^{C \eps} r^{n-1} dr.
\end{align*}
So if $\eps_*$ is small enough such that $C \eps_* < 1$, then the right-hand side can be made arbitrarily small for large $T$.
So we obtain the following statement:
For any $\alpha, s > 0$ there are $T = T (\alpha, s), J = J(\alpha, s), S_0 = S_0 (\alpha, s) < \infty$ such that for $S > S_0$, $j > J$, any $y_1, y_2 \in B(x_j, 4) \cap \{  \rrm > s \}$ and for large $i$ (depending on $j, y_1, y_2$) we have
\begin{equation} \label{eq:statement2}
 \bigg| \int_{B(x_j, S) \cap \RR'_j}  K(y_1, \cdot, T) f_j \eta_i \phi_{S,i} - \int_{B(x_j, S) \cap \RR'_j} K(y_2, \cdot, T) f_j \eta_i \phi_{S,i} \bigg| < \alpha. 
 \end{equation}

Combining the statement involving (\ref{eq:statement1}) with the statement involving (\ref{eq:statement2}), we conclude:
For all $\alpha, s > 0$ there is $J = J(\alpha, s) < \infty$ such that for any $y_1, y_2 \in B(x_j, 4) \cap \{ \rrm > s \}$ we have
\[ |f_j(y_1) - f_j(y_2)|< \alpha. \]

We can finally derive (\ref{eq:Claim3fj}):
For any $s > 0$ and we have, using Corollary \ref{Cor:preliminaryLpbound}
\begin{align*}
|B(x_j,&4) \cap \RR'_j|  \int_{B(x_j, 4) \cap \RR'_j}  \bigg| f_j(y_1) - \fint_{B(x_j, 4) \cap \RR'_j} f_j(y_2) dg(y_2) \bigg| dg(y_1)   \\
 &\leq \int_{B(x_j, 4) \cap \RR'_j} \int_{B(x_j, 4) \cap \RR'_j} | f_j(y_1) - f_j (y_2) | dg(y_1) dg(y_2) \\
 &\leq \int_{B(x_j, 4) \cap \{ \rrm > s \}} \int_{B(x_j, 4) \cap \{ \rrm > s\}} | f_j(y_1) - f_j (y_2) | dg(y_1) dg(y_2) \\
 &\qquad + 2 \int_{B(x_j, 4) \cap \RR'_j} \int_{B(x_j, 4) \cap \{ \rrm \leq s \} \cap \RR'_j}  |f_j(y_1) - f_j(y_2) | dg(y_1) dg(y_2) \\
 &\leq |B(x_j, 4) \cap \RR'_j|^2 \sup_{(y_1, y_2) \in (B(x_j,4) \cap \{ \rrm > s \})^2 } |f_j(y_1) - f_j (y_2) |  \\
 &\qquad + C |B(x_j, 4) \cap \RR'_j| \cdot |B(x_j, 4) \cap \{ 0 < \rrm \leq s \} |  \\
 &\leq C  \sup_{(y_1, y_2) \in (B(x_j,4) \cap \{ \rrm > s \})^2 } |f_j(y_1) - f_j (y_2) | + C s^{0.9} .
\end{align*}
The second term can be made arbitrarily small by choosing $s$ small enough.
Once $s$ is chosen, the first term goes to $0$ as $j \to \infty$.
This proves (\ref{eq:Claim3fj}).
\end{proof}

\begin{Claim4}
We have
\begin{equation} \label{eq:nabalmostconst}
 \int_{B(x_j, 4) \cap \RR'_j} \bigg| |\nabla v^k_j |^2 - \fint_{B(x_j, 4) \cap \RR'_j} |\nabla v^k_j|^2 \bigg| \to 0. 
\end{equation}
and
\begin{equation} \label{eq:Hessianto0}
 \int_{B(x_j, 2) \cap \RR'_j} |\nabla^2 v^k_j |^2  \to 0. 
\end{equation}
as $j \to \infty$.
\end{Claim4}

\begin{proof}
Identity (\ref{eq:nabalmostconst}) follows using (\ref{eq:navlnabvkto0}) and (\ref{eq:omegato0}) from Claim 3.
Identity (\ref{eq:Hessianto0}) follows from Lemma \ref{Lem:Hessianbound}, using (\ref{eq:nabalmostconst}).
\end{proof}

Recall that since $v_j |_{B(x_j,2)}$ is an $\eps$-splitting we have for $l = 1, \ldots, k-1$
\[ \bigg| 1-  \fint_{B(x_j, 2) \cap \RR'_j} \big| \nabla v^l_j \big|^2 \bigg| < C \eps \textQq{and} \qquad \bigg|   \fint_{B(x_j, 2) \cap \RR'_j} \big\langle \nabla v^l_j, \nabla v^k_j \big\rangle \bigg| < C \eps. \]
So if $\eps_*$ is sufficiently small, we can find
\[ 1 - C \eps < a^1_j, \ldots, a^{k-1}_j < 1 + C \eps \]
such that
\[ \fint_{B(x_j,2) \cap \RR'_j} \big\langle \nabla v^l_j, \nabla v^k_j \big\rangle = a^l_j \fint_{B(x_j, 2) \cap \RR'_j} \big| \nabla v^l_j \big|^2 . \]
Set now $(\td{v}^1_j, \ldots, \td{v}^{k-1}_j) := (v^1_j, \ldots, v^{k-1}_j)$.
\[ \td{v}^k_j := v^k - a^1_j v^1_j - \ldots - a^{k-1}_j v^{k-1}_j . \]
Then for any $l = 1, \ldots, k-1$
\begin{multline*}
 \fint_{B(x_j,2) \cap \RR'_j} \big| \big\langle \nabla v^l_j, \nabla \td{v}^k_j \big\rangle \big|  \leq \fint_{B(x_j,2) \cap \RR'_j} \bigg| \big\langle \nabla v^l_j, \nabla v^k_j \big\rangle - \fint_{B(x_j,2) \cap \RR'_j} \big\langle \nabla v^l_j, \nabla v^k_j \big\rangle \bigg| \\
 + \fint_{B(x_j,2) \cap \RR'_j} \bigg| a^l_j \big| \nabla v^l_j \big|^2 - a^l_j  \fint_{B(x_j,2) \cap \RR'_j} \big| \nabla v^l_j \big|^2 \bigg| \\
 + \sum_{i = 1,i\neq l}^{k-1}  \fint_{B(x_j,2) \cap \RR'_j} a^i_j \big| \big\langle \nabla v^i_j, \nabla v^k_j \big\rangle \big|.
 \end{multline*}
The first term on the right-hand side goes to $0$ as $j \to \infty$ by Claim 3, the second and third term go to $0$ as $j \to \infty$ since $(v^1_j, \ldots, v^{k-1}_j) |_{B(x_j,2)}$ is an $\eps_j$-splitting for $\eps_j \to 0$ and $a^l_j, a^i_j$ are uniformly bounded.
So we obtain that for any $l_1 = 1, \ldots, k-1$ and any $l_2 = 1, \ldots, k$ we have
\begin{equation} \label{eq:km1almostorth}
 \fint_{B(x_j, 2) \cap \RR'_j} \big| \big\langle \nabla \td{v}^{l_1}_j, \nabla \td{v}^{l_2}_j \big\rangle - \delta_{ab} \big| \to 0 
\end{equation}
as $j \to \infty$.
Moreover, $\td{v}_j |_{B(x_j, 2)}$ is still a $C \eps$-splitting.
We now need to modify $\td{v}$ again to ensure that this identity also holds for $l_1=l_2=k$.
For this set
\[ b_j := \fint_{B(x_j,2) \cap \RR'_j} \big| \nabla \td{v}^k_j \big|^2  \]
and note that $| 1 - b_j | < C \eps$.
Now let $(w^1_j, \ldots, w^{k-1}_j) := (\td{v}^1_j, \ldots, \td{v}^{k-1}_j)$ and $w^k_j := b_j^{-1/2} \td{v}^k_j$.
Then $w_j |_{B(x_j, 2)}$ is a $C \eps$-splitting and
\begin{equation} \label{eq:nabwksqto0}
 \fint_{B(x_j,2) \cap \RR'_j} \bigg| \big| \nabla w^k_j \big|^2 - 1 \bigg| \leq b_j^{-2} \fint_{B(x_j,2) \cap \RR'_j} \bigg| \big| \nabla \td{v}^k_j \big|^2 - \fint_{B(x_j, 2) \cap \RR'_j} \big| \nabla \td{v}^k_j \big|^2 \bigg|. 
\end{equation}
Using (\ref{eq:km1almostorth}), (\ref{eq:omegato0}) from Claim 3 for $l = k-1$ and $l = k$ and the fact that $\omega_j^l = \nabla v^1_j \wedge \ldots \wedge \nabla v^l_j$, we find that the right-hand side of (\ref{eq:nabwksqto0}) goes to $0$ as $j \to \infty$.
It follows that for all $l_1,l_2 = 1, \ldots, k$ we have
\begin{equation} \label{eq:wjnabtodelta}
 \fint_{B(x_j, 2) \cap \RR'_j} \big| \big\langle \nabla w^{l_1}_j, \nabla w^{l_2}_j \big\rangle - \delta_{l_1 l_2} \big| \to 0. 
\end{equation}
Moreover, by Claim 4 we have for any $l = 1, \ldots, k$
\begin{equation} \label{eq:wjnab2to0}
 \fint_{B(x_j, 2) \cap \RR'_j} \big| \nabla^2 w^l_j \big|^2 \to 0.
\end{equation}
It remains to show the following claim.

\begin{Claim5}
We have $|\nabla w | < 1 + \eps_j$ in the sense of (\ref{eq:epssplittinggradientestimate}) in Definition \ref{Def:epssplitting} on $B(x_j, 1) \cap \RR'_j$, where $\eps_j \to 0$ as $j \to \infty$.
\end{Claim5}

\begin{proof}
The claim follows from Lemma \ref{Lem:stronggradientestimate} using (\ref{eq:wjnab2to0}) and (\ref{eq:wjnabtodelta}).
\end{proof}

We can finally finish the proof of Proposition \ref{Prop:TransformationThm}.
Note that $(w^1_j, \ldots, w^k_j) = A''_j (v^1_j, \ldots, v^k_j)$ for some upper triangular matrices $A''_j$ with positive diagonal entries.
Moreover, by (\ref{eq:wjnabtodelta}), (\ref{eq:wjnab2to0}) and Claim 5, we find that $(w^1_j, \ldots, w^k_j)|_{B(x_j, 1)}$ is an $\eps_j$-splitting for some $\eps_j \to 0$.
So for sufficiently large $j$ it is an $\eps$-splitting, which contradicts our assumptions.
\end{proof}

\begin{Lemma}[Weak Version of Slicing Theorem, cf {\cite[Theorem 1.23]{Cheeger-Naber-Codim4}}] \label{Lem:weakslicing}
For every $\eps, \alpha > 0$, $\mathbf{p}_0 > 3$ and $Y < \infty$ there are $\delta = \delta (\eps, \alpha, \mathbf{p}_0, Y) > 0$ such that the following holds:

Let $\XX = (X, d, \RR, g)$ be a singular space with mild singularities of codimension $\mathbf{p}_0$.
Assume that $\XX$ is $Y$-tame and $Y$-regular at scale $\delta^{-1}$.
Assume moreover that $\Ric = \lambda g$ on $\RR$ for some $|\lambda | \leq (n-1) \delta^2$.

Let $p \in X$ and $u : B(p,1) \to \IR^{n-2}$ a $\delta$-splitting map.
Then there is an $s \in \IR^{n-2}$ such that:
\begin{enumerate}[label=(\alph*)]
\item $u^{-1} (s) \cap B(p, \alpha) \neq \emptyset$,
\item for all $x \in u^{-1} (s)$ and $0 < r \leq 1/10$ there is a lower triangular matrix $A$ with positive diagonal entries such that $A u|_{B(x,r)} : B(x,r) \to \IR^{n-2}$ is an $\eps$-splitting map.
\item $u^{-1} (s) \cap B(p,1/2) \subset \RR$,
\end{enumerate}
\end{Lemma}

\begin{proof}
The proof is follows as in \cite{Cheeger-Naber-Codim4}, using Proposition \ref{Prop:TransformationThm}.
Following the lines of the proof yields a subset $U \subset \IR^{n-2}$ of non-zero measure such that assertions (a) and (b) hold for all $s \in U$.
To ensure assertion (c), it remains to check that $U \not\subset u (  B(p, 1/2)  \setminus \RR)$.
To see this, fix some $0 < \sigma < 1/10$ and consider a minimal $2\sigma$-net $z_1, \ldots, z_N \in B(p, 1/2)  \cap \{  \rrm  < \sigma \}$ of $ B(p, 1/2)  \cap \{  \rrm  < \sigma \}$.
Then $B(z_1, \sigma), \ldots, B(z_N, \sigma) \subset B(p, 1) \cap \{  \rrm < 2\sigma \}$ are pairwise disjoint and hence for some generic constant $C_* < \infty$, which may depend on $\XX$, but not on $\sigma$, we have
\[ N < C_* \sigma^{-n} | B(p, 1) \cap \{ 0 < \rrm < 2\sigma \} | < C_* \sigma^{-n} \sigma^{3}. \]
On the other hand, the balls $B(z_1, 2\sigma), \ldots, B(z_N, 2\sigma)$ cover $ B(p, 1/2)  \cap \{  \rrm  < \sigma \}$ and thus, for sufficiently small $\delta$, the balls $B(u(z_1), 4 \sigma), \ldots, B(u(z_N), 4 \sigma)$ cover $u ( B(p, 1/2)  \cap \{  \rrm  < \sigma \} )$.
It follows that
\[ \big| u \big(  B(p, 1/2)  \cap \{  \rrm  < \sigma \} \big) \big| \leq C_* \sigma^n \sigma^{-n+3} < C_* \sigma^{3}. \]
Letting $\sigma \to 0$ yields $| u (  B(p, 1/2)  \setminus \RR) | = 0$ and therefore we have indeed $U \not\subset u ( B(p, 1/2)  \setminus \RR)$.
\end{proof}

We can now verify condition (\ref{eq:productcondition}) in Proposition \ref{Prop:LpboundXXversion2}.

\begin{Lemma} \label{Lem:codim4regularity}
For every $\mathbf{p}_0 > 3$ and $Y < \infty$ there is an $\eps = \eps ( \mathbf{p}_0, Y) > 0$ such that the following holds:

Let $\XX$ be a singular space with mild singularities of codimension $\mathbf{p}_0$.
Assume that $\XX$ is $Y$-tame and $Y$-regular at scale $r$ for some $r > 0$.
Assume moreover that $\Ric = \lambda g$ on $\RR$ for some $|\lambda | \leq (n-1) r^{-2}$.

Assume that there is a pointed metric space $(Z, d_Z, z)$ such that
\begin{equation} \label{eq:closetonminus2}
 d_{GH} \big( \big( B(p,r), p \big), \big( B^{Z \times \IR^{n-2}} ((z, 0^{n-2}),r), (z, 0^{n-2} ) \big) \big) < \eps r. 
\end{equation}
Then $\rrm (p) > \eps r$.
\end{Lemma}

\begin{proof}
Using Proposition \ref{Prop:splittingmapGHsplitting} and rescaling, we may replace the assumption (\ref{eq:closetonminus2}) of the lemma by the assumption that there is an $\eps$-splitting map $u : B(p, 1) \to \IR^{n-2}$, that $\XX$ is $Y$-tame and $Y$-regular at scale $\eps^{-1}$ and that the Einstein constant satisfies $|\lambda| \leq (n-1) \eps^2$.
We now need to show that we have $\rrm (p) > \eps$.

Assume that this was not the case for some fixed $Y$ and $\mathbf{p}_0$.
Then we can find a sequence $\eps_i \to 0$ and a sequence of singular spaces $\XX_i$ with mild singularities of codimension $\mathbf{p}_0$ that are $Y$-tame and $Y$-regular at scale $\eps_i^{-1}$ and satisfy $\Ric = \lambda g_i$ on $\RR_i$ for some $|\lambda_i | \leq (n-1) \eps_i^2$.
Moreover, we can find points $p_i \in X_i$ and $\eps_i$-splitting maps $u_i : B(p_i,1) \to \IR^{n-2}$ such that $\rrm (p_i) \leq \eps_i$.
Using Lemma \ref{Lem:weakslicing}, we can find a sequence $\alpha_i \to 0$ and a sequence $s_i \in \IR^{n-2}$ such that assertions (a)--(c) hold.
Note that for any $x \in u^{-1} (s_i) \cap B(p_i, \alpha_i) \neq \emptyset$ we have $\rrm (x) < \rrm (p_i) + \alpha_i \leq \eps_i + \alpha_i$.
For each $i$ pick a point $x_i \in u_i^{-1} (s_i) \cap B(p_i, 1/2)$ where $\rrm (\cdot) (1/2 - d(p_i, \cdot))^{-1}$ attains its minimum.
Then
\[ \rrm (x_i) \big( 1/2 - d(p_i, x_i) \big)^{-1} <  (\eps_i + \alpha_i) (1/2 - \alpha_i)^{-1} \to 0. \]
In particular $r_i := \rrm (x_i) \to 0$.

Rescale each $\XX_i$ now by $r_i^{-1}$ and call the resulting space $\XX'_i$.
The corresponding $\eps_i$-splitting map will be denoted by $u'_i : B^{X'_i} (p_i, r_i^{-1}) \to \IR^{n-2}$.
By the choice of $x_i$, for every $D < \infty$ we have $\rrm > 1/2$ on $(u'_i)^{-1} (s_i) \cap B^{X'_i} (x_i, D)$ for large $i$.

Using Lemma \ref{Lem:weakslicing}(b), we can find a sequence $r'_i \to \infty$ such that $B^{X'_i} (x_i, r'_i) \subset B^{X'_i} (p_i, r_i^{-1})$ and such that there are lower triangular matrices with positive diagonal entries $A_i \in \IR^{(n-2) \times (n-2)}$ such that $u''_i := A_i u'_i |_{B^{X'_i} (x_i, r)}$ is an $\eps'_i$-splitting for all $r \in [1, r'_i]$ for some $\eps'_i \to 0$.
Set $s''_i := A_i s_i$.
Let us now pass to a subsequence and assume that we have Gromov-Hausdorff convergence $(X'_i, d_{X'_i}, x_i)$ to some pointed metric space $(X_\infty, d_{X_\infty}, x_\infty)$ and convergence of the (uniformly Lipschitz) maps $u''_i - s''_i$ to some $1$-Lipschitz $u_\infty : X_\infty \to \IR^{n-2}$.
With the help of Proposition \ref{Prop:splittingmapGHsplitting}, we conclude that $(X_\infty, d_{X_\infty})$ is isometric to a Cartesian product $Z_\infty \times \IR^{n-2}$ for some complete metric length space $(Z_\infty, d_{\infty})$.
We will henceforth assume that $X_\infty = Z_\infty \times \IR^{n-2}$.
We will moreover assume that $x_\infty$ corresponds to the point $(z_\infty, 0^{n-2})$ for some $z_\infty \in Z_\infty$ and that $u_\infty : X_\infty = Z_\infty \times \IR^{n-2} \to \IR^{n-2}$ corresponds to the projection onto the second factor.

Fix a Gromov-Hausdorff convergence 
\begin{equation} \label{eq:GHconvergencefix}
(X'_i, d_{X'_i}, x_i) \to (Z_\infty \times \IR^{n-2}, d_{Z_\infty \times \IR^{n-2}}, (z_\infty, 0^{n-2})), 
\end{equation}
e.g. by specifying a sequence of almost-isometries on larger and larger balls.
We now claim that for every $z \in Z_\infty$ there is a sequence $y_i \in X'_i$ that converges to $(z,0^{n-2})$ with respect to (\ref{eq:GHconvergencefix}) such that $\rrm (y_i) > 1/10$ for infinitely many $i$.
Call every point $z \in Z_\infty$ that satisfies this claim \emph{smooth}.
By the choice of the points $x_i$, we know that $z_\infty$ is smooth.
Consider now a smooth point $z \in Z_\infty$ and let $y_i \in X'_i$ be the corresponding sequence that converges to $(z, 0^{n-2})$.
Then, after passing to a subsequence, the Gromov-Hausdorff convergence (\ref{eq:GHconvergencefix}) is actually $C^3$ on $B^{Z_\infty \times \IR^{n-2}} ((z, 0^{n-2}), 1/100)$ and the convergence $u''_i - s''_i \to u_\infty$ happens in $C^2$.
So, since $u''_i (y_i) - s''_i \to 0$, we can find points $y'_i \in X'_i$ with $d^{X'_i} (y_i, y'_i) \to 0$ such that $u''_i (y'_i) - s''_i = 0$.
This implies $u'_i (y'_i) = s_i$ and, by our previous discussion, that $\rrm (y'_i) > 1/2$ for large $i$.
Therefore, all points in $B^{Z_\infty} (z, 1/4)$ are smooth.
Summarizing our arguments, we have shown that for any smooth $z \in Z_\infty$, all points in $B^{Z_\infty} (z, 1/4)$ are smooth.
As $Z_\infty$ is a length space, it follows that all points of $Z_\infty$ are smooth.

So $(Z_\infty, d_{Z_\infty})$ is the length space of complete $2$-dimensional Riemannian metric $g_{Z_\infty}$ of regularity $C^2$ and the Gromov-Hausdorff convergence (\ref{eq:GHconvergencefix}) is $C^3$ in a neighborhood of $Z_\infty \times \{ 0^{n-2} \}$.
As the limiting metric must be Ricci flat, we conclude that $g_{Z_\infty}$ is locally flat.
So the limit $Z_\infty \times \IR^{n-2}$ is flat as well.
Using Corollary \ref{Cor:GHepsregularity}, we conclude that the Gromov-Hausdorff convergence (\ref{eq:GHconvergencefix}) is $C^3$ everywhere.
Thus $\rrm (z_\infty, 0^{n-2}) = \lim_{i \to \infty} \rrm (x_i) = 1$, which contradicts the flatness of the limit.
\end{proof}

We can now prove the main theorem of this paper, Theorem \ref{Thm:Lpbound}.

\begin{proof}[Proof of Theorem \ref{Thm:Lpbound}]
The theorem is a consequence of Proposition \ref{Prop:LpboundXXversion2} and Lemma~\ref{Lem:codim4regularity}.
\end{proof}

\bibliography{RF-bounded-scal}{}
\bibliographystyle{amsalpha}
\end{document}